\documentclass[opre, nonblindrev]{informs3}
\OneAndAHalfSpacedXI

\usepackage{afterpage}
\usepackage{natbib}
\bibpunct[, ]{(}{)}{,}{a}{}{,}%
\def\bibfont{\small}%
\def\bibsep{\smallskipamount}%
\def\F{{\mathbb F}}
\def\B{{\mathbb{B}}}
\def\R{{\mathbb{R}}}
\def\Z{{\mathbb{Z}}}
\def\Be{\mathcal B}
\def\Ge{{\mathcal G}}

\def\Ee{\mathcal E}

\def\Pe{\mathcal P}
\def\Qe{\mathcal Q}

\def\Se{\mathcal S}
\def\Cupper{\overline C}
\def\Clower{\underline C}
\def\Vupper{\overline V}

\def\argmax{\mbox{argmax}}

\def\convP{{\rm conv}(\Pe)}
\def\convPprime{{\rm conv}(\Pe')}
\def\convPtwo{{\rm conv}({\Pe}_2)}
\def\dimP{{\rm dim}({\rm conv}(\Pe))}
\def\F1^+{F1\textsuperscript{+}}  
\def\FOneXprime{F1-X$'$}

\def\fred{\textcolor{black}}

\def\rred{\textcolor{black}}

\def\sred{\textcolor{black}}
\def\tred{\textcolor{black}}

\TheoremsNumberedThrough
\EquationsNumberedThrough

\usepackage{amsmath, comment, tikz, graphicx, algorithm, algpseudocode, multirow, multicol, subcaption, bbm, epsfig, enumerate, enumitem, amsfonts, subeqnarray, url, color, array, xcolor, booktabs, endnotes, float, mathpazo, pdflscape, anyfontsize, tabularx, longtable, diagbox, rotating, xr, xr-hyper}

\usepackage{pifont} 
\usetikzlibrary{positioning, arrows, decorations, backgrounds}
\usetikzlibrary{calc}
\usepackage{mathdots}
\usepackage[belowskip=10pt,aboveskip=10pt]{caption}
\usepackage[skip=2pt, font=footnotesize, labelfont=bf, labelsep=period]{caption}

\usepackage{pgfplots}
    \pgfplotsset{
        compat=1.9,
    }
\usepackage[backref = false, bookmarks, breaklinks=true, colorlinks = true, plainpages = false, citecolor = DarkBlue , urlcolor = DarkBlue, filecolor = DarkBlue, linkcolor = DarkBlue]{hyperref}

\usepackage{orcidlink}

\definecolor{DarkBlue}{rgb}{0,0.08,0.45}

\let\footnote=\endnote

\renewcommand{\theARTICLETOP}{}
\renewcommand{\theRRHFirstLine}{}
\renewcommand{\theRRHSecondLine}{}
\renewcommand{\theLRHFirstLine}{}
\renewcommand{\theLRHSecondLine}{}
\newcommand\sscriptsize{\fontsize{7.5pt}{7.5pt}\selectfont}

\TITLE{A Polyhedral Study on Unit Commitment with a Single Type of Binary Variables}

\begin{document}


\ARTICLEAUTHORS{
\AUTHOR{Bin Tian$^{{\textup a}}$, Kai Pan$^{{\textup a},*}$, Chung-Lun Li$^{\textup b}$}
\AFF{$^{\textup a}$Faculty of Business, The Hong Kong Polytechnic University, Kowloon, Hong Kong}
\AFF{$^{\textup b}$College of Business, City University of Hong Kong, Kowloon, Hong Kong}
\AFF{$^*$Corresponding author}
\AFF{Contact: \EMAIL{bin01.tian@connect.polyu.hk} (BT), \EMAIL{kai.pan@polyu.edu.hk} (KP), \EMAIL{chung-lun.li@connect.polyu.hk} (CLL)}}

\ABSTRACT{Efficient power production scheduling is a crucial concern for power system operators aiming to minimize operational costs. Previous mixed-integer linear programming formulations for unit commitment (UC) problems have primarily used two or three types of binary variables. The investigation of strong formulations with a single type of binary variables has been limited, as it is believed to be challenging to derive strong valid inequalities using fewer binary variables, and the reduction of the number of binary variables is often accompanied by a compromise in tightness. To address these issues, this paper considers a formulation for unit commitment using a single type of binary variables and develops strong valid inequality families to enhance the tightness of the formulation. Conditions under which these strong valid inequalities serve as facet-defining inequalities for the single-generator UC polytope are provided. For those large-size valid inequality families, the existence of efficient separation algorithms for determining the most violated inequalities is also discussed. The effectiveness of the proposed single-binary formulation and strong valid inequalities is demonstrated through computational experiments on network-constrained UC problems. The results indicate that the strong valid inequalities presented in this paper are effective in solving UC problems and can also be applied to UC formulations that contain more than one type of binary variables.}


\maketitle
\vspace{-1cm}
\section{Introduction} \label{sec:introduction}
With the increased prevalence of extreme weather events such as \fred{droughts}, wildfires, and flooding around the world in recent years, power consumption in many areas hit an all-time high in the summer of 2022.
In addition, the continued retirements of coal-fired generating plants, relatively high coal prices, and lower-than-average coal stocks at power plants have limited coal consumption \citep{EIAaug2022}.
As a result, the efficient production and distribution of electricity have been identified as the biggest concern for power producers around the globe. 

The unit commitment (UC) problem, which involves the scheduling of power generators, has been a challenging optimization problem in the power industry for many years.
It often needs to be solved multiple times per day by system operators \citep{xavier2021learning}.
Because of such practical needs, it has received \fred{considerable} attention over the past decades.
\fred{The UC problem involves} scheduling a group of generators at \fred{a} possibly minimal operational cost, subject to their physical and system constraints over a finite time horizon.
\rred{Physical constraints specify the technical properties of generators, and they may vary depending on the type of generation unit, such as hydro, thermal, and wind units \citep{van2018large}.
Unless explicitly stated otherwise, all subsequent discussions are about thermal units.}

\rred{The most common physical constraints for thermal units are the ramp-up/-down rate, generation lower/upper bound, and minimum up/down time constraints.}
System constraints typically include the load (demand) requirement, system reserve constraint, and transmission flow limit.
All generation units are coupled by the system constraints to ensure the reliability of the entire system.
The operational cost comprises various components, including the generation (fuel) cost and the startup and shutdown costs of generators.
The generation cost is a significant component \citep{padhy2004unit}, and \fred{it} is generally assumed to be an increasing quadratic convex function of the generation amount \citep{takriti2000using}.
\fred{Start-up} and \fred{shut-down} costs are incurred \fred{each} time the status of a generator changes \citep{sen1998optimal}.

Many different UC problems \fred{can arise} depending on the structure of the electrical power system.
\fred{The ability to solve} UC problems \fred{efficiently can have} a great impact on both society and individual consumers.
Even small improvements in the quality of solutions for UC problems can affect the electricity price over large regions and lead to millions of \fred{U.S.} dollars of savings per day \citep{damci2016polyhedral}.
Although the \rred{single-generator UC (single-UC)} problem with only physical constraints and a quadratic generation cost function is proven to be polynomial-time solvable \citep{frangioni2006solving}, general UC problems with system constraints have not yet been solved satisfactorily.
They are often formulated as large-scale mixed-integer programming problems.
Such problems are typically NP-hard and difficult to solve when the problem sizes are large \citep{zheng2015stochastic, knueven2020mixed, tejada2020unit}.
For example, a UC problem with an arbitrary number of generators that contain only \rred{minimum up/down}, generation lower/upper bound, and demand constraints is classified as NP-hard even considering a single operational period \citep{bendotti2019complexity}.
In a day-ahead deregulated electricity market, an independent system operator (ISO) is expected to determine the generation schedule for a power system within a very short time.
Such a generation schedule involves hundreds of thermal units, thousands of transmission lines, and 24--48 operating hours.
Therefore, over the past decades, numerous approaches have been devised from both formulation and algorithm perspectives in order to solve UC problems efficiently.

Priority list and heuristic algorithms are among the earliest solution approaches used to solve UC problems \citep{kazarlis1996genetic}.
\rred{The former lists all units by their operational costs and then economically dispatches the system load to generators by a pre-determined order.
Heuristic algorithms, such as genetic algorithms and simulated annealing, are also frequently adopted as they can be converted to work on parallel computers.}
However, these approaches usually lead to suboptimal solutions. 
Dynamic programming (DP)-based approaches, \fred{in contrast}, are exact solution techniques for UC problems and are widely applied in the early period \citep{wang1993effects, baldick1995generalized}.
\rred{Using these approaches, UC problems are decomposed by time, and the state for each time period is represented by the combinations of units.
However, the number of states increases dramatically as the number of units and time periods grows.
Therefore,
t}hese approaches are also integrated with heuristic methods to reduce the search space \citep{ouyang1991intelligent}. 
Recently, \cite{frangioni2006solving} and \cite{guan2018polynomial} propose DP algorithms to solve single-UC problems and provide theoretical complexity results.
\rred{They show that single-UC problems with only physical constraints and suitable generation cost functions, such as quadratic cost functions, can be solved in polynomial time.} 
In addition, \fred{given} the large size of UC problems, Lagrangian relaxation (LR)-based approaches are proposed \citep{muckstadt1977application, abdul1996practical,  takriti2000using, lu2005unit}.
These approaches relax system constraints, such as load requirements, and integrate them into the objective function through Lagrangian multipliers. 
The resulting problem is then decomposed into subproblems either by units or by time periods, and these subproblems are solved iteratively until the primal-dual gap is \fred{difficult} to shrink. 
However, LR-based approaches suffer from slow and unsteady convergence, and the feasibility of the final solution cannot be guaranteed \citep{ma1999unit}.
\rred{To overcome the convergence problem, quadratic terms are incorporated into the objective function to penalize the violation of demand constraints and improve its convexity, resulting in augmented Lagrangian relaxation based approaches.
A comprehensive examination of solution approaches for UC problems can be found in \cite{van2018large}.}


The extensive advancement of \fred{mixed-integer linear programming} (MILP) solvers has led to the widespread application of MILP-based approaches in formulating and solving UC \rred{problems.
MILP-based} approaches can guarantee convergence to the optimal solution while providing a flexible and accurate modeling framework.
Moreover, the optimality gap is easy to \rred{obtain.
ISOs} are \fred{therefore} increasingly adopting MILP-based approaches over LR-based approaches \fred{to solve} large-scale UC problems \citep{hedman2009analyzing, wu2011tighter, ostrowski2012tight, li2021extreme}.

Two primary factors are generally considered in \fred{evaluations of} MILP formulations for UC problems: \textit{compactness} and \textit{tightness}.
Compactness refers to the size of the problem, which can be quantified by the number of constraints, decision variables, \rred{or nonzero coefficients;} tightness refers to the proximity of the linear programming (LP) relaxation of the problem to the convex hull of its feasible region (\citealt{morales2013tight, knueven2020mixed}).
For UC problems, compactness can be achieved by reducing the number of \rred{integer} variables in an MILP formulation, as fewer integer variables may lead to a reduction in the number of nodes of the search tree for the branch-and-cut method. 
Note that for UC formulations, all integer variables are binary. 
In terms of \rred{the number of binary variables used}, UC formulations can be categorized into \rred{two} categories.
A single-binary formulation uses a single set of binary variables to denote the on/off status of all units.
A three-binary formulation uses two additional sets of binary variables to represent the start-up and shut-down decisions.
As an important variant of the three-binary formulation, a two-binary formulation can be obtained by expressing the shut-down variables in terms of the on/off and start-up variables.
For a UC problem, a three-binary formulation can generally be tighter than a single-binary one because the addition of the start-up and shut-down binary variables in the three-binary formulation may facilitate the improvement of the physical constraints, potentially leading to a better LP bound.
However, the large size of its search tree and the difficulty of solving the subproblem at each branching node may increase the solution time.
On the other hand, a single-binary formulation uses fewer binary variables, which reduce the size of the search tree and may decrease the solution time.
Nevertheless, it usually has a weaker LP bound than the three-binary one.
Tightness can be achieved by deriving strong valid inequalities for the MILP formulation to tighten its LP relaxation.
Most strong valid inequalities are obtained by studying the physical constraints of a single generator (see, e.g., \citealt{lee2004min}, \citealt{rajan2005minimum}, \citealt{morales2013tight}, \citealt{damci2016polyhedral}, \citealt{pan2016polyhedral}, and \citealt{bendotti2018min}).

Three-binary formulations for UC problems are the most widely studied.
\cite{garver1962power} is the first \fred{to propose} an MILP formulation for a UC problem.
In this three-binary formulation, the generation cost function is assumed to be linear with respect to the generation amount.
\cite{arroyo2000optimal} introduce a three-binary formulation for a \rred{single}-UC problem.
They approximate the exponential start-up cost function and the nonconvex generation cost function using stairwise and piecewise functions, respectively.
\cite{chang2001experiences} present a three-binary formulation for a short-term hydro scheduling UC problem.
\cite{chang2004practical} put forward a new three-binary formulation for UC problems\rred{, and} they approximate the cubic generation cost function using a piecewise linear one with three breakpoints.
\cite{li2005price} compare the LR-based approach with the MILP-based approach in solving a price-based UC problem with various types of generators based on the formulation of \cite{chang2004practical}.
\fred{Their numerical} results indicate that the MILP-based approach exhibits superior performance \fred{on small scale problems} compared \fred{with} the LR-based approach, and the MILP formulation \fred{must} be tightened to improve its performance \fred{on} large-scale problems.
\cite{ostrowski2012tight} consider the formulation of \cite{arroyo2000optimal} \fred{and replace} their \rred{minimum up/down} constraints with those of \cite{rajan2005minimum} because the latter can reduce the computational time significantly.
By studying the physical constraints of a single generator, \fred{they derive} a class of strong valid inequalities to tighten the MILP formulation.
\rred{Their computational results indicate that their tightened formulation is more effective than \citeauthor{carrion2006computationally}'s (\citeyear{carrion2006computationally}) single-binary formulation.}
\cite{morales2013tight} propose an alternative three-binary formulation based on the formulation of \cite{ostrowski2012tight}.
The generation cost function is represented as a linear function with respect to the generation amount.
They introduce generation \fred{limit} constraints to substitute \fred{for} those in \cite{ostrowski2012tight}.
\fred{They show that the} resulting formulation \rred{has smaller size and better LP bound} than that of \cite{ostrowski2012tight}.
\cite{morales2015tight} establish a three-binary formulation based on \fred{that of} \cite{morales2013tight} by considering different start-up/shut-down trajectories, which are ignored in conventional research.
\cite{damci2016polyhedral} conduct a polyhedral study of the physical constraints \fred{based on the work} of \cite{ostrowski2012tight}.
\fred{They derive a} convex hull for the two-period case and strong valid inequalities for the multi-period case to tighten the MILP formulation.
\fred{Because} the number of these strong valid inequalities can be exponential, polynomial separation algorithms are provided to apply them in the solution process.
Computational results demonstrate that \fred{this formulation} outperforms the strong formulation of \cite{ostrowski2012tight}.
\cite{atakan2018state} develop a state-transition formulation for UC problems based on \fred{the} formulations of \cite{ostrowski2012tight} and \cite{morales2013tight}.
Transmission constraints are not considered in their formulation.
\fred{Their test} results demonstrate that the proposed formulation has a shorter computational time for long-horizon problems than \fred{the formulations of} \cite{ostrowski2012tight} and \cite{morales2013tight}.
\tred{
\cite{gentile2017tight} provide the convex hull description for the single-UC polytope considering the minimum up/down time constraints, start-up/shut-down capabilities constraints, and generation limit constraints.} 

Other variants of three-binary formulations have also received considerable attention.
\cite{rajan2005minimum} study \rred{a single-UC problem with minimum up/down times constraints using three types of binary variables.}
They provide a complete description of the convex hull of the polytope.
\rred{Based on \citeauthor{rajan2005minimum}'s three-binary formulation, \cite{pan2016strengthened} derive several families of strong valid inequalities for UC problems with gas turbine generators.}
Their strong valid inequalities are facet-defining for the polytope of physical constraints under specific conditions.
\cite{pan2016polyhedral} conduct a polyhedral study of physical constraints based on the formulation \fred{of} \cite{pan2016strengthened}.
They derive the complete convex hull descriptions for the two- and three-period polytopes under different parameter settings.
They also develop strong valid inequalities for the multi-period case and provide polynomial-time separation algorithms for \fred{exponentially large} valid inequality families. 
\cite{bendotti2018min} analyze the \rred{minimum up/down} polytope of multiple generators based on the formulation \fred{of} \cite{pan2016polyhedral}; \fred{their} generation cost function is linear \fred{in} the generation amount.
\fred{They obtain} \rred{\it up-set} and \rred{\it interval up-set} valid inequalities to accelerate the branch-and-cut algorithm.
However, given a fractional solution, the problems of separating these two types of inequalities are NP-complete and NP-hard, respectively.
\cite{pan2022polyhedral} perform a polyhedral study of a single generator by incorporating fuel constraints.
They prove that the \rred{single-UC} problem with a fuel constraint is NP-hard, and they derive strong valid inequalities to improve the computational performance.
\tred{\cite{dupin2017tighter} present two main formulations for a UC problem with min-stop ramping constraints based on the definitions of the so-called \emph{state} and \emph{level} variables.
The two formulations are compared and exhibit an isomorphism.}


Studies on single-binary formulations are limited.
\cite{lee2004min} investigate the \rred{minimum up/down} polytope using a single type of binary variables.
They give a complete \rred{convex hull} description of the polytope, obtain valid inequalities, and design an efficient separation procedure for using these valid inequalities.
\cite{carrion2006computationally} propose a single-binary MILP formulation for UC problems.
\fred{They approximate the} generation cost function and the exponential start-up cost function using linear \fred{functions} as in \cite{arroyo2000optimal}.
\fred{They also establish} new \rred{minimum up/down} constraints.
\fred{They then compare the} proposed formulation with the three-binary formulation of \cite{arroyo2000optimal}, as well as \fred{with} its variant in which one type of binary variables in \cite{arroyo2000optimal} is relaxed.
\rred{They computationally demonstrate that their single-binary formulation outperforms the other two formulations significantly.}
\cite{frangioni2006perspective} derive perspective cuts for the mixed-integer quadratic programming problem with semi-continuous variables.
They test these cuts by solving a single-binary UC formulation with a quadratic generation cost function.
These cuts can \fred{substantially} improve the performance of the branch-and-cut method.
However, \fred{their formulation does not consider} the ramp-up/-down, \rred{system} reserve, and transmission constraints.
\cite{frangioni2009tighter} apply the perspective cuts of \cite{frangioni2006perspective} to provide a new piecewise linear approximation of the generation cost function for a short-term UC problem with hydro and thermal generators.
\cite{brandenberg2017summed} consider the summed start-up cost across all time periods in a single-binary UC formulation by introducing a single continuous variable for each unit.
They derive the H-representation of its epigraph and provide an exact linear separation algorithm.

Generally, enhancing the tightness of a formulation for a UC problem necessitates the inclusion of additional variables or constraints, which may lead to an increase in the solution time since the solver needs to solve larger LP subproblems repeatedly; conversely, improving the compactness of a formulation often comes at the expense of weakening tightness, resulting in a weak lower bound \citep{morales2013tight, knueven2020mixed}.
Therefore, in practice, tightness and compactness \fred{must be balanced}.
\rred{In most cases, improvement in tightness for UC formulations is preferred over that in compactness because of the potential reduction in solution time, despite the increased complexity resulting from additional binary variables \citep{hedman2009analyzing, ostrowski2012tight}.}
Moreover, \fred{a lack of} binary variables for start-up/shut-down decisions makes it \fred{difficult} to generate strong valid inequalities \citep{ostrowski2012tight}.
Thus, few studies examine formulations with a single type of binary variables and derive strong valid inequalities to improve their tightness.
To bridge this gap, this paper studies a single-binary formulation for a UC problem with thermal units and derives strong valid inequalities to speed up the solution process.
\tred{The studied single-binary formulation also offers an alternative solution approach for UC problems, which can be leveraged in both theoretical research and practical applications.}
The main contributions of this study are summarized as follows:
\begin{itemize}
\item Through an investigation of the physical constraints, we provide a complete \rred{convex hull} description of the two-period single-UC polytope of the \rred{single-binary} formulation.
\item We develop strong valid inequality families for the multi-period single-UC polytope, and we derive the conditions under which the strong valid inequalities are facet-defining.
We also develop efficient separation algorithms for determining \rred{a} most violated inequality in each valid inequality family.
\item \fred{We demonstrate the} effectiveness of our strong valid inequalities in tightening our \rred{single-binary} formulation through computational experiments.
The results indicate that our strong valid inequalities are effective in solving UC problems and can also be applied to UC formulations that contain more than one type of binary variables.
\end{itemize}

The rest of the paper is organized as follows.
Section~\ref{sec:model} describes the UC problem \fred{under study}, presents a \rred{single-binary} MILP formulation for it, and introduces the single-UC polytope.
Section~\ref{sec:two-period} provides the complete description of the convex hull for the two-period case and discusses its importance \fred{for} solving our UC problem.
Section~\ref{sec:multi-period} presents various strong valid inequalities to tighten the single-UC polytope and discusses the existence of efficient separation algorithms.
Section~\ref{sec:comp-exper} reports the results of a computational study conducted to assess the effectiveness of our strong valid inequalities in tightening the \rred{single-binary} MILP formulation and speeding up the solution process.
Section~\ref{Conclusion} \fred{concludes the paper and offers} suggestions for future research.
\rred{All mathematical proofs are provided in Online Appendix~\ref{apx:A}.
Some additional computational results are provided in Online Appendix~\ref{apx:B}.}
\vskip0pt

\section{MILP Model for Unit Commitment}\label{sec:model}

In this section, we first present an MILP model for the UC problem \rred{with thermal units,} and then \rred{present some important properties of the single-generator case}. 
In the UC problem being studied, a system operator plans the generation schedule of a set of generators $\Ge$ for a number of time periods at minimal operating costs while satisfying physical and system constraints.
The system includes a set of buses $\Be$ and a set of transmission lines $\Ee$ that link the buses, allowing surplus power to be distributed.
Each bus can be equipped with multiple generators and is responsible for the load requirement of a geographical region.
Surplus power at one bus can be transferred to neighboring buses through transmission lines to satisfy the load requirements of other regions.
The power flow on each transmission line should not exceed the line's capacity.
To ensure the reliability of the power supply, some generation capacity should be reserved for outages.
All \fred{of the} generators should operate without violating their physical configurations.
Every time a generator starts up or shuts down, a fixed cost is incurred.
For each time period, an operational cost is incurred depending on the generation amount and the online/offline status of the generators.

We let 
$\R^n$ denote the $n$-dimensional real vector space, 
$\R_+^n$ denote the $n$-dimensional nonnegative real vector space, and 
$\B^n$ denote the $n$-dimensional binary vector space.
Given any nonnegative integers $a$ and $b$, we let $[a,b]_\Z$ denote the set of all integers between $a$ and $b$;
that is, $[a,b]_\Z=\{a,a+1,\ldots,b\}$ if $a\le b$, and $[a,b]_\Z=\emptyset$ if $a>b$.

Let $T$ be the number of time periods in the operation horizon.
For each generator $g\in\Ge$, let $L^g>0$ and $\ell^g>0$ be the \rred{minimum up} and \rred{minimum down} time \fred{requirements}, respectively.
That is, once the generator starts up, it must stay online for at least $L^g$ time periods, and once it shuts down, it must stay offline for at least $\ell^g$ time periods.
For each $g\in\Ge$, let $\Cupper^g$ and $\Clower^g$ be the generation upper and lower bounds, where $\Cupper^g>\Clower^g>0$.
For each $g\in\Ge$, let $V^g>0$ be the maximum change \fred{in} the generation amount between two consecutive online time periods\rred{, i.e., the ramp-up rate is assumed to be equal to the ramp-down rate.}
\rred{For each $g\in\Ge$, let $\Vupper^g>0$ be the start-up/shut-down ramp limit, i.e., the start-up ramp limit is assumed to be equal to the shut-down ramp limit.}
Thus, when a generator $g$ is online, its generation amount should be within the range $[\Clower^g,\Cupper^g]$.
When the generator starts up, its generation amount in the start-up period should be within the range $[\Clower^g,\Vupper^g]$.
When the generator shuts down, its generation amount in the previous time period should also be within the range $[\Clower^g,\Vupper^g]$.
We assume that $\Vupper^g+V^g\le\Cupper^g$ for all $g\in\Ge$.
This condition guarantees that a generator can ramp up at its full rate $V^g$ for at least one period after it starts up.
\tred{We also assume that $\Clower^g<\Vupper^g<\Clower^g+V^g$ for all $g\in\Ge$, which holds in most industrial settings as indicated by \cite{morales2015tight}, \cite{damci2016polyhedral}, \cite{pan2016strengthened}, and \cite{gentile2017tight}.}
For each bus $b\in\Be$, let $\Ge_b$ be the set of generators at bus $b$ (note: $\bigcup_{b\in\Be}\Ge_b=\Ge$ and $\Ge_b\cap\Ge_{b'}=\emptyset$ for all $b,b'\in\Be$ such that $b\ne b'$).
The other parameters are defined as follows:
\begin{list}{\textbullet}{}
\item $f^g(\cdot)$: \fred{The generation} cost function for generator $g$ (for each $g\in\Ge$, $f^g(\cdot)$ is a non-decreasing convex piecewise linear function with a fixed number of linear segments).
\item $c^g$: \fred{The fixed} cost incurred if generator $g$ is online ($c^g\ge 0$ for all $g\in\Ge$). 
\item $\phi^g$: \fred{The fixed} start-up cost of generator $g$ ($\phi^g\ge 0$ for all $g\in\Ge$).
\item $\psi^g$: \fred{The fixed} shut-down cost of generator $g$ ($\psi^g\ge 0$ for all $g\in\Ge$).
\item $d^b_t$: The load (demand) at bus $b$ in time period $t$ ($d^b_t\ge 0$ for all $t\in[1,T]_\Z$ and $b\in\Be$).    
\item $C_e$: \fred{The capacity} limit of transmission line $e$ ($C_e\ge 0$ for all $e\in\Ee$).
\item $K^b_e$: Line flow distribution factor for the flow on transmission line $e$ contributed by the net injection at bus $b$ ($K^b_e\ge 0$ for all $e\in\Ee$ and $b\in\Be$).
\item $r_t$: \fred{The system} reserve factor of time period \fred{$t$} ($r_t\ge 0$ for all $t\in[1,T]_\Z$).
\end{list}
Here, the non-decreasing convex piecewise linear generation cost function $f^g(x)$ is used to approximate the convex quadratic cost function $a^gx^2+b^gx$;
see \cite{carrion2006computationally} and \cite{pan2022polyhedral} for similar approximations.
Our model has the following decision variables:
\begin{list}{\textbullet}{}
\item $x^g_t$: The generation amount of generator $g\in\Ge$ in period $t\in[1,T]_\Z$.
\item $y^g_t$: The online/offline status of generator $g\in\Ge$ in period $t\in[1,T]_\Z$, where $y^g_t=1$ if $g$ is online in period $t$, and $y^g_t=0$ otherwise. 
\item $u^g_t$: The start-up cost of generator $g\in\Ge$ in period $t\in[1,T]_\Z$.
\item $v^g_t$: The shut-down cost of generator $g\in\Ge$ in period $t\in[1,T]_\Z$.
\end{list}
Variables $x^g_t$, $u^g_t$, and $v^g_t$ are continuous, \fred{whereas} variable $y^g_t$ is binary.
We assume that the values of $y^g_{-\max\{L^g,\ell^g\}+1},y^g_{-\max\{L^g,\ell^g\}+2},\ldots,y^g_{-1},y^g_0$, and $x^g_{0}$ (for all $g\in\Ge$) are given as initial conditions.
The UC problem is formulated as follows:
{\small \begin{subeqnarray}
\mbox{\hskip-25pt Problem 1:\ }&\mbox{min\ } & \textstyle \sum_{g\in\Ge}\sum_{t=1}^T(u^g_t+v^g_t+c^gy^g_t+f^g(x^g_t)) \slabel{prob1:obj} \\[-1pt]
&\mbox{s.t.\ } & -\!y_{t-1}^g\!+\!y_t^g\!-\!y_k^g\le 0,\ \forall t\in[-L^g\!+\!2,T]_{\Z},\,\forall k\in[t,\min\{T,t\!+\!L^g\!-\!1\}]_\Z,\,\forall g\in\Ge, \slabel{prob1:min-up} \\[-1pt]
&& y_{t-1}^g-y_t^g+y_k^g\le 1,\ \forall t\in[-\ell^g\!+\!2,T]_\Z,\ \forall k\in[t,\min\{T,t\!+\!\ell^g\!-\!1\}]_\Z,\,\forall g\in\Ge, \slabel{prob1:min-down} \\[-1pt]
&& -x_{t}^{g}+\Clower^{g}y_{t}^{g}\leq0,\ \forall t\in[1,T]_\Z,\ \forall g\in\Ge, \slabel{prob1:gen-lower-bound}\\[-1pt] 
&& x_{t}^{g}-\Cupper^{g}y_{t}^{g}\leq0,\ \forall t\in[1,T]_\Z,\ \forall g\in\Ge, \slabel{prob1:gen-upper-bound}\\[-1pt]
&& x_t^g-x_{t-1}^g\leq V^g y_{t-1}^g+\Vupper^g(1-y_{t-1}^g),\ \forall t\in[1,T]_\Z,\,\forall g\in\Ge, \slabel{prob1:ramp-up}\\[-1pt]
&& x_{t-1}^g-x_t^g\leq V^g y_t^g+\Vupper^g(1-y_t^g),\ \forall t\in[1,T]_\Z,\,\forall g\in\Ge, \slabel{prob1:ramp-down}\\[-1pt]
&& u_{t}^{g}\geq\phi^{g}(y_{t}^{g}-y_{t-1}^{g}),\ \forall t\in[1,T]_\Z,\,\forall g\in\Ge, \slabel{prob1:startup-cost}\\[-1pt]
&& v_{t}^{g}\geq\psi^{g}(y_{t-1}^{g}-y_{t}^{g}),\ \forall t\in[1,T]_\Z,\,\forall g\in\Ge, \slabel{prob1:shutdown-cost}\\[-1pt]
&& \textstyle \sum_{g\in\Ge}x_t^g=\sum_{b\in\Be}d_t^b,\ \forall t\in[1,T]_\Z, \slabel{prob1:load}\\[-1pt]
&& \textstyle \sum_{g\in\Ge}\Cupper^g y_t^g\geq(1+r_t)\sum_{b\in\Be}d_t^b,\ \forall t\in[1,T]_\Z, \slabel{prob1:spinning-reserve}\\[-1pt]
&& \textstyle -C_e\le\sum_{b\in\Be}K^b_e\bigl(\sum_{g\in\Ge_b}x_t^g-d_t^b\bigr)\le C_e,\ \forall t\in[1,T]_\Z,\,\forall e\in\Ee, \slabel{prob1:transmission-constraint}\\[-1pt]
&& y_t^g\in\{0,1\}, x_t^g\ge 0, u_t^g\ge 0, v_t^g\ge 0,\ \forall t\in[1,T]_\Z,\,\forall g\in\Ge. \slabel{prob1:nonnegativity}
\end{subeqnarray}}%

Objective function \eqref{prob1:obj} minimizes the total cost, which includes the start-up costs, shut-down costs, and fixed and variable generation costs.
Constraint \eqref{prob1:min-up} states the \rred{minimum up} requirement for generator $g$.
It requires generator $g$ to stay online in periods $[t,\min\{T,t+L^g-1\}]_\Z$ if it starts up in period $t$.
Constraint \eqref{prob1:min-down} states the \rred{minimum down} requirement for generator $g$.
It requires generator $g$ to stay offline in periods $[t,\min\{T,t+\ell^g-1\}]_\Z$ if it shuts down in period $t$.
Constraints \eqref{prob1:gen-lower-bound} and \eqref{prob1:gen-upper-bound} ensure that the generation amount of generator $g$ in period $t$ is $0$ if the generator is offline and is within the range $[\Clower^g,\Cupper^g]$ if the generator is online.
Constraints \eqref{prob1:ramp-up} and \eqref{prob1:ramp-down} guarantee that generator $g$ ramps up/down within its limit $V^g$ between two consecutive online time periods.
They also guarantee that generator $g$ ramps up by no more than $\Vupper^g$ units when it starts up and ramps down by no more than $\Vupper^g$ units when it shuts down.
Constraint \eqref{prob1:startup-cost} and objective function \eqref{prob1:obj}, together with the nonnegativity constraint of $u^g_t$, imply that the start-up cost for generator $g$ in period $t$ is $\phi^g$ if the generator starts up in period $t$, and $0$ otherwise.
Constraint \eqref{prob1:shutdown-cost} and \rred{the} objective function \eqref{prob1:obj}, together with the nonnegativity constraint of $v^g_t$, imply that the shut-down cost for generator $g$ in period $t$ is $\psi^g$ if the generator shuts down in period $t$, and $0$ otherwise.
Constraint \eqref{prob1:load} is the load balance constraint in period $t$, which requires the total generation amount to satisfy the total demand in the period.
Constraint \eqref{prob1:spinning-reserve} is the system reserve requirement, which requires the total generation capacity of all online generators to exceed the load requirement by a system reserve factor to deal with demand variations.
Constraint \eqref{prob1:transmission-constraint} states the transmission flow limit.
In the distribution process, a bus $b$ contributes a factor $K^b_e$ of its net injection $\sum_{g\in\Ge_b}x_t^g-d_t^b$ to each transmission line $e$,
and constraint \eqref{prob1:transmission-constraint} requires the absolute value of the total net injection contributed by all buses to each transmission line to stay below its capacity limit to prevent it from being overloaded; see \cite{ma1999unit}, \citet[p.~290]{shahidehpour2002market}, and \cite{xavier2021learning} for similar settings.
Constraint \eqref{prob1:nonnegativity} states the nonnegativity and binary requirements of the decision variables.
\sred{Note that Problem 1 uses only a single type of binary variables, $y_t^g$, and thus is a single-binary formulation. However, in the optimal solution, the continuous variables $u_t^g$ and $v_t^g$ have only two possible values. Hence, Problem~1 can also be formulated using two additional vectors of discrete variables. Because the polytope $\convP$ that we are analyzing in this paper is independent of $u_t^g$ and $v_t^g$ (see the definition of set $\Pe$ below), the polytope is a single-binary polytope regardless of whether $u_t^g$ and $v_t^g$ are declared as continuous or discrete.}
Note also that the objective function of Problem~1 is piecewise linear.
Following the literature \citep[see, e.g.,][]{arroyo2000optimal}, Problem~1 can be converted into an MILP.

\cite{bendotti2019complexity} consider a UC problem in which there is a linear generation cost, a minimum demand requirement, no ramp-up, ramp-down, start-up, and shut-down limits, no system reservation requirement, no transmission flow limit, and some initial conditions. 
They prove that the problem is strongly NP-hard. 
It is easy to verify that their NP-hardness proof remains valid when applied to our UC problem. 
Thus, Problem~1 is also NP-hard in the strong sense.

In Problem~1, constraints \eqref{prob1:min-up}--\eqref{prob1:ramp-down} specify the physical properties of the generators.
Constraints \eqref{prob1:startup-cost} and \eqref{prob1:shutdown-cost} determine the start-up and shut-down costs, respectively.
Once the $y_t^g$ values are determined for all $g\in\Ge$ and $t\in[1,T]_\Z$, the $u_t^g$ and $v_t^g$ values can be easily obtained by these constraints.
Constraints \eqref{prob1:load}--\eqref{prob1:transmission-constraint} are the coupling constraints, or system constraints, \fred{that} link all \fred{of the} generators.
\fred{Because of} the scale and complexity of the UC problem, one way of reducing the solving time is to decompose the problem into smaller subproblems with one subproblem corresponding to each generator \citep[see][]{knueven2020novel}.
For example, in the Lagrangian relaxation method, the coupling constraints can be integrated into the objective function through Lagrangian multipliers, and the resulting problem is decomposed into subproblems that contain only the physical constraints \citep{baldick1995generalized, takriti2000using}.
Thus, most improvements in UC models \fred{result} from studying the properties of an individual generator's feasible region \citep{knueven2018ramping}.
Moreover, strong valid inequalities for the physical constraints are valid for Problem~1 and can be used \fred{to tighten} its linear relaxation.
A tighter linear relaxation can often improve computational efficiency by reducing the amount of enumeration required to find and prove an optimal solution \citep{knueven2020mixed}.
Hence, in the mathematical analysis presented in Sections \ref{sec:two-period} and \ref{sec:multi-period}, we focus on deriving strong valid inequalities for the physical constraints for the generators in Problem~1.
Because all \fred{of} the generators have the same set of physical constraints, it suffices to concentrate on the physical constraints of a single generator, and the results obtained can be applied to all \fred{of} other generators.
Therefore, in the following analysis, the superscript $g$ in the parameters and decision variables is dropped.

Denote $\BFx=(x_1,\ldots,x_T)$ and $\BFy=(y_1,\ldots,y_T)$.
Thus, the vector $(\BFx,\BFy)\in\R_{+}^T\times\B^T$ contains the generation amount and on/off status of the generator in the $T$ time periods.
The set of $(\BFx,\BFy)$ values that satisfy the physical constraints of Problem~1 is given as
{\small \begin{subeqnarray}
&& \Pe=\big\{(\BFx,\BFy)\in\R_+^{T}\times\B^{T}:\nonumber\\[-1pt]
&& \qquad\qquad -y_{t-1}+y_t-y_k\le0,\,\forall t\in[2,T]_{\Z},\,\forall k\in[t,\min\{T,t+L-1\}]_{\Z},\slabel{eqn:p-minup}\\[-1pt]    
&& \qquad\qquad y_{t-1}-y_t+y_k \le1,\,\forall t\in[2,T]_{\Z},\,\forall k\in[t,\min\{T,t+\ell-1\}]_{\Z},\slabel{eqn:p-mindn}\\[-1pt] 
&& \qquad\qquad -x_t+\Clower y_t\le0,\,\forall t\in[1,T]_{\Z},\slabel{eqn:p-lower-bound}\\[-1pt]                                     
&& \qquad\qquad x_t-\Cupper y_t\le0,\,\forall t\in[1,T]_{\Z},\slabel{eqn:p-upper-bound} \\[-1pt]                                     
&& \qquad\qquad x_t-x_{t-1}\le Vy_{t-1}+\Vupper(1-y_{t-1}),\,\forall t\in[2,T]_{\Z},\slabel{eqn:p-ramp-up}\\[-1pt]                   
&& \qquad\qquad x_{t-1}-x_t\le Vy_t+\Vupper(1-y_t),\,\forall t\in[2,T]_{\Z}\slabel{eqn:p-ramp-down}\big\}.                     
\end{subeqnarray}}%
Here, the assumptions $\Cupper>\Clower>0$, $V>0$, $\Vupper+V\le\Cupper$, and $\Clower<\Vupper<\Clower+V$ remain valid. 
Note that inequalities \eqref{eqn:p-minup}--\eqref{eqn:p-ramp-down} in $\Pe$ are the same as inequalities \eqref{prob1:min-up}--\eqref{prob1:ramp-down} in Problem~1 for a specific generator $g$, except that $t$ is restricted to the range $[2,T]_\Z$ in \eqref{eqn:p-minup}, \eqref{eqn:p-mindn}, \eqref{eqn:p-ramp-up}, and \eqref{eqn:p-ramp-down} (i.e., constraints dependent on the initial conditions are not included in $\Pe$).

Let $\convP$ denote the convex hull of $\Pe$, and we refer to $\convP$ as the single-UC polytope.
\tred{Note that $\convP$ is full dimensional, which is shown in Appendix~\ref{apx:A-prop-1}.}
Obviously, a valid inequality for $\Pe$ is also valid for Problem~1 for any generator $g$.
Hence, the strong valid inequalities developed for $\convP$ can be used to tighten the formulation of Problem~1.
The following two lemmas provide some important properties of $\Pe$.

\begin{lemma}\label{lem:lookbackward}
Consider any point $(\BFx,\BFy)\in\Pe$ and $t\in[2,T]_{\Z}$.
(i)~If $y_t=0$, then $y_{t-j}-y_{t-j-1}\le 0$ for all $j\in[0,\min\{t-2,L-1\}]_{\Z}$.
(ii)~If $y_t=1$, then there exists at most one $j\in[0,\min\{t-2,L\}]_{\Z}$ such that $y_{t-j}-y_{t-j-1}=1$.
\end{lemma}

\begin{lemma}\label{lem:Pprime}
Denote $\BFx'=(x'_1,\ldots,x'_T)$ and $\BFy'=(y'_1,\ldots,y'_T)$.
Let
{\small \begin{subeqnarray*}
&& \Pe'=\big\{(\BFx',\BFy')\in\R_+^{T}\times\B^{T}:\nonumber\\[-2pt]
&& \qquad\qquad -y'_{T-t+2}+y'_{T-t+1}-y'_{T-k+1}\le0,\,\forall t\in[2,T]_{\Z},\,\forall k\in[t,\min\{T,t+L-1\}]_{\Z},\\[-2pt]
&& \qquad\qquad y'_{T-t+2}-y'_{T-t+1}+y'_{T-k+1}\le1,\,\forall t\in[2,T]_{\Z},\,\forall k\in[t,\min\{T,t+\ell-1\}]_{\Z},\\[-2pt]
&& \qquad\qquad -x'_{T-t+1}+\Clower y'_{T-t+1}\le0,\,\forall t\in[1,T]_{\Z},\\[-2pt]
&& \qquad\qquad x'_{T-t+1}-\Cupper y'_{T-t+1}\le0,\,\forall t\in[1,T]_{\Z},\\[-2pt]
&& \qquad\qquad x'_{T-t+1}-x'_{T-t+2}\le Vy'_{T-t+2}+\Vupper(1-y'_{T-t+2}),\,\forall t\in[2,T]_{\Z},\\[-2pt]
&& \qquad\qquad x'_{T-t+2}-x'_{T-t+1}\le Vy'_{T-t+1}+\Vupper(1-y'_{T-t+1}),\,\forall t\in[2,T]_{\Z}\big\}.
\end{subeqnarray*} }%
Then, $\Pe=\Pe'$.
\end{lemma}

\sred{Lemma \ref{lem:lookbackward} states a relationship among the $y_t$ variables.
This relationship, derived from the minimum up time requirement, allows us to simplify the validity proofs of the inequality families presented in Section~\ref{sec:multi-period}.
Lemma \ref{lem:Pprime} states that if variables $x_t$ and $y_t$ are replaced by $x_{T-t+1}$ and $y_{T-t+1}$, respectively, then the set $\Pe$ remains unchanged. This property enables us to show that if a given inequality family is known to be valid and facet-defining, then the corresponding “mirror image” of that inequality family is also valid and facet-defining.}

\vskip0pt

\section{The Two-period Convex Hull}\label{sec:two-period}
In this section, we investigate the properties of \rred{the} set $\Pe$ when there are only two periods.
\rred{Then, we demonstrate that the strong valid inequalities resulting from our investigation can be used not only to tighten the single-UC polytope $\convP$ but also to derive other forms of strong valid inequalities for $\convP$.}

Consider any two consecutive periods $t-1$ and $t$, where $t\in[2,T]_\Z$.
Denote
{\small \begin{subeqnarray}\label{set:P2}
&&\Pe_2=\big\{(x_{t-1},x_t,y_{t-1},y_t)\in\R^2_{+}\times\B^2:\nonumber\\[-2pt]
&&\qquad\qquad-x_i+\Clower y_i\le 0, \ \forall i\in\{t-1,t\}, \slabel{eqn:P2-a}\\[-2pt] 
&&\qquad\qquad x_i-\Cupper y_i\le 0, \ \forall i\in\{t-1,t\}, \slabel{eqn:P2-b}\\[-2pt] 
&&\qquad\qquad x_t-x_{t-1}\le Vy_{t-1}+\Vupper(1-y_{t-1}),    \slabel{eqn:P2-c}\\[-2pt] 
&&\qquad\qquad x_{t-1}-x_t\le Vy_t+\Vupper(1-y_t)\big\}.      \slabel{eqn:P2-d}   
\end{subeqnarray} }%
Note that when $t=2$, the set $\Pe_2$ is the same as the set $\Pe$ with $T=2$.
Note also that when $T=2$, inequalities \eqref{eqn:p-minup} and \eqref{eqn:p-mindn} become redundant, which significantly \sred{simplifies} the set $\Pe_2$.
Let $\convPtwo$ denote the convex hull of $\Pe_2$.
The following theorem provides a complete description of $\convPtwo$.
\begin{theorem}\label{the:convPtwo}
Denote
{\small \begin{subeqnarray}\label{eqn-q2}
&& \Qe_2 = \big\{(x_{t-1},x_t,y_{t-1},y_t)\in\R^4: \nonumber \\[-0.5pt]
&& \qquad\qquad y_i\le 1, \ \forall i\in\{t-1,t\},                          \slabel{eqn-q2:y-t} \\[-0.5pt]        
&& \qquad\qquad \Clower y_i\le x_i\le \Cupper y_i, \ \forall i\in\{t-1,t\}, \slabel{eqn-q2:capacity} \\[-0.5pt]   
&& \qquad\qquad x_{t-1}\le\Vupper y_{t-1}+(\Cupper-\Vupper)y_t,             \slabel{eqn-q2:x1-ub} \\[-0.5pt]      
&& \qquad\qquad x_t\le(\Cupper-\Vupper)y_{t-1}+\Vupper y_t,                 \slabel{eqn-q2:x2-ub} \\[-0.5pt]      
&& \qquad\qquad x_t-x_{t-1}\le(\Clower+V)y_t-\Clower y_{t-1},               \slabel{eqn-q2:x2-x1-ub-1} \\[-0.5pt] 
&& \qquad\qquad x_t-x_{t-1}\le\Vupper y_t-(\Vupper-V)y_{t-1},               \slabel{eqn-q2:x2-x1-ub-2} \\[-0.5pt] 
&& \qquad\qquad x_{t-1}-x_t\le(\Clower+V)y_{t-1}-\Clower y_t,               \slabel{eqn-q2:x1-x2-ub-1} \\[-0.5pt] 
&& \qquad\qquad x_{t-1}-x_t\le\Vupper y_{t-1}-(\Vupper-V) y_t\big\}.        \slabel{eqn-q2:x1-x2-ub-2}    
\end{subeqnarray} }%
Then, $\Qe_2=\convPtwo$.
\end{theorem}



\sred{Note that for every} $t\in[2,T]_\Z$, any inequality in \eqref{eqn-q2:y-t}--\eqref{eqn-q2:x1-x2-ub-2} is valid for $\convP$.
Note also that inequalities \eqref{eqn-q2:x1-ub}--\eqref{eqn-q2:x1-x2-ub-2} do not exist in the description of $\Pe$.
Hence, they can be added to the constraint set of $\Pe$ to tighten the linear relaxation of $\Pe$.
In particular, because $\Vupper y_t-(\Vupper-V)y_{t-1}\le Vy_{t-1}+\Vupper(1-y_{t-1})$ for any $y_t\le 1$, 
the right-hand side of \eqref{eqn-q2:x2-x1-ub-2} is no greater than the right-hand side of \eqref{eqn:p-ramp-up}, 
and thus inequality \eqref{eqn-q2:x2-x1-ub-2} dominates inequality \eqref{eqn:p-ramp-up} and can effectively tighten the linear relaxation of $\Pe$.
Similarly, inequality \eqref{eqn-q2:x1-x2-ub-2} dominates inequality \eqref{eqn:p-ramp-down} and can effectively tighten the linear relaxation of $\Pe$.
Therefore, the following inequality families can be used as valid \fred{inequalities} for $\convP$:
\vskip-25pt
{\small
\begin{align}
&x_t\le\Vupper y_t+(\Cupper-\Vupper)y_{t+1},\ \forall t\in[1,T-1]_\Z; \label{multi:eqn-q2:x1-ub}\\[-2pt]
&x_t\le(\Cupper-\Vupper)y_{t-1}+\Vupper y_t,\ \forall t\in[2,T]_\Z;   \label{multi:eqn-q2:x2-ub}\\[-2pt]
&x_t-x_{t-1}\le(\Clower+V)y_t-\Clower y_{t-1},\ \forall t\in[2,T]_\Z; \label{multi:eqn-q2:x2-x1-ub-1}\\[-2pt]
&x_t-x_{t-1}\le\Vupper y_t-(\Vupper-V)y_{t-1},\ \forall t\in[2,T]_\Z; \label{multi:eqn-q2:x2-x1-ub-2}\\[-2pt]
&x_t-x_{t+1}\le(\Clower+V)y_t-\Clower y_{t+1},\ \forall t\in[1,T-1]_\Z; \label{multi:eqn-q2:x1-x2-ub-1}\\[-2pt]
&x_t-x_{t+1}\le\Vupper y_t-(\Vupper-V)y_{t+1},\ \forall t\in[1,T-1]_\Z. \label{multi:eqn-q2:x1-x2-ub-2}
\end{align}}%
These valid inequalities provide 
upper bounds on the generation amount $x_t$ for each time period $t$,
upper bounds on $x_t-x_{t-1}$ for each pair of consecutive time periods $t$ and $t-1$, and
upper bounds on $x_t-x_{t+1}$ for each pair of consecutive time periods $t$ and $t+1$.

Inequalities \eqref{multi:eqn-q2:x1-ub}--\eqref{multi:eqn-q2:x1-x2-ub-2} also enable us to develop other strong valid inequalities for $\convP$.
We demonstrate this by presenting a strong valid inequality derived from \eqref{multi:eqn-q2:x2-x1-ub-1}.
Consider any point $(\BFx,\BFy)\in\Pe$.
For any $k\in[1,T-1]_\Z$ and any $t\in[k+1,T]_\Z$, because inequality \eqref{multi:eqn-q2:x2-x1-ub-1} is valid for $\convP$, we have
{\small $$ \sum_{\tau=t-k+1}^t(x_\tau-x_{\tau-1})
\le\sum_{\tau=t-k+1}^t\big[(\Clower+V)y_\tau-\Clower y_{\tau-1}\big]
=V\sum_{\tau=t-k+1}^ty_\tau+\Clower\sum_{\tau=t-k+1}^t(y_\tau-y_{\tau-1}),$$}%
which implies that
{\small \begin{equation}\label{multi:eqn:two_var-k-ub-2}
  x_t-x_{t-k}\le V\sum_{\tau=t-k+1}^t y_{\tau}+\Clower y_t-\Clower y_{t-k}.
\end{equation}}%
If $y_t=1$, then $\sum_{\tau=t-k+1}^t y_{\tau}\le ky_t$, and by \eqref{multi:eqn:two_var-k-ub-2}, $x_t-x_{t-k}\le(\Clower+kV)y_t-\Clower y_{t-k}$.
If $y_t=0$, then by \eqref{eqn:p-lower-bound} and \eqref{eqn:p-upper-bound}, $-x_{t-k}\le -\Clower y_{t-k}$ and $x_t=0$, 
which also imply that $x_t-x_{t-k}\le(\Clower+kV)y_t-\Clower y_{t-k}$.
Thus, in both cases,
{\small \begin{equation}\label{multi:eqn:two_var-k-ub-3}
  x_t-x_{t-k}\le (\Clower+kV)y_t-\Clower y_{t-k}.
\end{equation} }%
Hence, \eqref{multi:eqn:two_var-k-ub-3} is a valid inequality for $\convP$ for any $k\in[1,T-1]_\Z$ and $t\in[k+1,T]_\Z$.
It is worth noting that inequality \eqref{eqn-T:kth-ramp-2-1} presented in Proposition~\ref{prop-T:kth-ramp-2} in Section~\ref{subsec:multi-period-two} is reduced to inequality \eqref{multi:eqn:two_var-k-ub-3} when $m=0$ and $\Se=\emptyset$.
Therefore, inequality \eqref{multi:eqn:two_var-k-ub-3} is a special case of the facet-defining valid inequality \eqref{eqn-T:kth-ramp-2-1}.
\vskip0pt


\section{Multi-period Strong Valid Inequalities} \label{sec:multi-period}
\rred{When there are more than two periods, the set $\Pe$ is significantly more complex than the set $\Pe_2$ because it involves not only more time periods but also minimum up/down constraints.}
In this section, we present a collection of strong valid inequalities that can effectively enhance the tightness of Problem~1.
We provide the validity proofs for these inequalities, and we identify the conditions under which these inequalities are facet-defining for $\convP$.
For each family of valid inequalities, we also show that for any given point $(\BFx,\BFy)\in\R_+^{2T}$, an efficient separation algorithm exists for determining \rred{a} most violated inequality.
\vspace{-2mm}

\subsection{Valid Inequalities with a Single Continuous Variable}\label{subsec:multi-period-single}

In this subsection, we present strong valid inequalities that provide upper bounds on the generation amount $x_t$ for each time period $t$.
Families of such inequalities appear in pairs.
The first family consists of inequalities \fred{for which} the upper bound on $x_t$ depends mainly on the values of $y_{t-s}-y_{t-s-1}$ for some $s\ge 0$, 
\fred{and} the second family consists of inequalities \fred{for which} the upper bound on $x_t$ depends mainly on the values of $y_{t+s}-y_{t+s+1}$ for some $s\ge 0$.
The following proposition presents a pair of such inequality families.

\begin{proposition} \label{prop-T:x_t-ub-1}
Consider any $\Se\subseteq[0, \min\{L-1, T-2, \lfloor (\Cupper - \Vupper)/V \rfloor\}]_{\Z}$.
For any $t\in[1,T]_\Z$ such that $t\ge s+2$ for all $s\in\Se$, the inequality
{\small \begin{equation}\label{eqn-T:x_t-ub-1-1}
x_t \leq \Cupper y_t - \sum_{s\in\Se} (\Cupper - \Vupper - sV) (y_{t-s} - y_{t-s-1})
\end{equation}}%
is valid and facet-defining for $\convP$.
For any $t\in[1,T]_\Z$ such that $t\le T-s-1$ for all $s\in\Se$, the inequality
{\small \begin{equation}\label{eqn-T:x_t-ub-1-2}
x_t \leq \Cupper y_t - \sum_{s\in\Se} (\Cupper - \Vupper - sV) (y_{t+s} - y_{t+s+1})
\end{equation} }%
is valid and facet-defining for $\convP$.
\end{proposition}

In Proposition \ref{prop-T:x_t-ub-1}, inequalities \eqref{eqn-T:x_t-ub-1-1} and \eqref{eqn-T:x_t-ub-1-2} provide 
upper bounds on the generation amount $x_t$.
These upper bounds can be explained as follows.
Let $s_{\max}$ denote the largest element of $\Se$.
The condition ``$\Se\subseteq[0,\min\{L-1,T-2,\lfloor(\Cupper-\Vupper)/V\rfloor\}]_\Z$" implies that $s_{\max}\le L-1$,
which in turn implies that there is at most one startup and at most one shutdown during the time interval $[t-s_{\max},t]$, 
and that there is at most one shutdown and at most one startup during the time interval $[t+1,t+s_{\max}+1]$.
Consider the \rred{situation in which} a generator starts up in period $t-s_1$, stays online until period $t+s_2$, and shuts down in period $t+s_2+1$, 
where $s_1,s_2\in[0,s_{\max}]_\Z$, $t-s_1\ge 2$, and $t+s_2+1\le T$.
Then, $y_{t-s_1-1}=0$, $y_{t-s_1}=y_{t-s_1+1}=\cdots=y_{t+s_2}=1$, and $y_{t+s_2+1}=0$.
If $s_1\in\Se$ and none of the time periods in $\{t-s\ge 2:s\in\Se\}$ is a shut-down period, 
then the right-hand side of inequality \eqref{eqn-T:x_t-ub-1-1} becomes $\Cupper-(\Cupper-\Vupper-s_1V)$.
This upper bound limits the value of $x_t$ to be no more than $\Vupper+s_1V$, which is smaller than the generation upper bound \rred{$\Cupper$.}
Similarly, if $s_2\in\Se$ and none of the periods in $\{t+s+1\le T:s\in\Se\}$ is a start-up period, 
then the right-hand side of inequality \eqref{eqn-T:x_t-ub-1-2} becomes $\Cupper-(\Cupper-\Vupper-s_2V)$.
This upper bound limits the value of $x_t$ to be no more than $\Vupper+s_2V$, which is smaller than the generation upper bound \rred{$\Cupper$.}

In Proposition~\ref{prop-T:x_t-ub-1}, the set $\Se$ only contains elements that are less than $L$.
The following proposition states that under certain conditions, inequalities \eqref{eqn-T:x_t-ub-1-1} and \eqref{eqn-T:x_t-ub-1-2} remain valid
and facet-defining when $\Se$ contains some elements that are greater than or equal to $L$.

\begin{proposition} \label{prop-T:x_t-ub-2}
Consider any integers $\alpha$, $\beta$, and $s_{\max}$ such that
(a)~$L\le s_{\max}\le\min\{T-2,\lfloor(\Cupper-\Vupper)/V\rfloor\}$,
(b)~$0\le\alpha<\beta\le s_{\max}$, and
(c)~$\beta=\alpha+1$ or $s_{\max}\le L+\alpha$.
Let $\Se=[0,\alpha]_{\Z}\cup[\beta,s_{\max}]_{\Z}$.
For any $t\in[s_{\max}+2,T]_{\Z}$, inequality \eqref{eqn-T:x_t-ub-1-1} is valid and facet-defining for $\convP$.
For any $t\in[1,T-s_{\max}-1]_{\Z}$, inequality \eqref{eqn-T:x_t-ub-1-2} is valid and facet-defining for $\convP$.
\end{proposition}

\begin{example}\label{examp:prop-T:x_t-ub-1}
Let $T=16$, $\Cupper=80$, $\Clower=8$, $L=\ell=5$, $\Vupper=15$, and $V=10$.
Then, $\lfloor(\Cupper-\Vupper)/V\rfloor=6$.
By Proposition~\ref{prop-T:x_t-ub-1}, we obtain the following pair of valid inequalities if we set $\Se=\{0,2,4\}$ and $t=8$:
{\small \begin{align*}
\left\{
\begin{array}{ll}
x_8\le 25y_3-25y_4+45y_5-45y_6+65y_7+15y_8;\\[2pt]
x_8\le 15y_8+65y_9-45y_{10}+45y_{11}-25y_{12}+25y_{13}.
\end{array}
\right.
\end{align*} }%
By Proposition~\ref{prop-T:x_t-ub-2}, we obtain the following pair of valid inequalities if we set $\Se=\{0,1,2,5,6\}$ (i.e., $\alpha=2$, $\beta=5$, and $s_{\max}=6$) and $t=8$:
{\small
\begin{align*}
\left\{
\begin{array}{ll}
x_8\le 5y_1+10y_2-15y_3+45y_5+10y_6+10y_7+15y_8;\\[2pt]
x_8\le 15y_8+10y_9+10y_{10}+45y_{11}-15y_{13}+10y_{14}+5y_{15}.
\end{array}
\right.
\end{align*} }%
\end{example}

The next proposition extends Proposition~\ref{prop-T:x_t-ub-1} and presents another two families of strong valid inequalities.

\begin{proposition}\label{prop-T:x_t-ub-3}
Consider any set $\Se\subseteq[0,\min\{L-1,T-3,\lfloor(\Cupper-\Vupper)/V\rfloor\}]_{\Z}$ and any real number $\eta$ such that $0\le\eta\le\min\{L-1,(\Cupper-\Vupper)/V\}$.
For any $t\in[1,T-1]_\Z$ such that $t\ge s+2$ for all $s\in\Se$, the inequality
{\small \begin{equation}\label{eqn-T:x_t-ub-3-1}
    x_t \leq (\Cupper - \eta V) y_t + \eta V y_{t+1} - \sum_{s\in\Se} (\Cupper - \Vupper - sV) (y_{t-s} - y_{t-s-1})
\end{equation} }%
is valid for $\convP$.
For any $t\in[2,T]_\Z$ such that $t\le T-s-1$ for all $s\in\Se$, the inequality
{\small \begin{equation}\label{eqn-T:x_t-ub-3-2}
    x_t \leq (\Cupper - \eta V) y_t + \eta V y_{t-1} - \sum_{s\in\Se} (\Cupper - \Vupper - sV) (y_{t+s} - y_{t+s+1})
\end{equation} }%
is valid for $\convP$.
Furthermore, inequalities \eqref{eqn-T:x_t-ub-3-1} and \eqref{eqn-T:x_t-ub-3-2} are facet-defining for $\convP$ when $\eta\in\{0, (\Cupper-\Vupper)/V\}$ or $\eta=L-1\in\Se$.
\end{proposition}

When $t\ne T$, inequality \eqref{eqn-T:x_t-ub-3-1} is a generalization of inequality \eqref{eqn-T:x_t-ub-1-1}.
Specifically, the right-hand side of \eqref{eqn-T:x_t-ub-3-1} differs from the right-hand side of \eqref{eqn-T:x_t-ub-1-1} by $\eta Vy_{t+1}-\eta Vy_t$, 
and this difference is zero if $\eta=0$.
Similarly, when $t\ne 1$, inequality \eqref{eqn-T:x_t-ub-3-2} is a generalization of inequality \eqref{eqn-T:x_t-ub-1-2},
and the right-hand side of \eqref{eqn-T:x_t-ub-3-2} differs from the right-hand side of \eqref{eqn-T:x_t-ub-1-2} by $\eta Vy_{t-1}-\eta Vy_t$.
In Proposition~\ref{prop-T:x_t-ub-3}, the set $\Se$ only contains elements that are less than $L$.
The following proposition, which extends Proposition~\ref{prop-T:x_t-ub-2}, states that under certain conditions, 
inequalities \eqref{eqn-T:x_t-ub-3-1} and \eqref{eqn-T:x_t-ub-3-2} remain valid and facet-defining when $\Se$ contains some elements that are greater than or equal to $L$.

\begin{proposition} \label{prop-T:x_t-ub-4}
Consider any real number $\eta$ such that $0\le\eta\le\min\{L-1,(\Cupper-\Vupper)/V\}$ and any integers $\alpha$, $\beta$, and $s_{\max}$ such that
(a)~$L\le s_{\max}\le\min\{T-3,\lfloor(\Cupper-\Vupper)/V\rfloor\}$,
(b)~$0\le\alpha<\beta\le s_{\max}$, and
(c)~$\beta=\alpha+1$ or $s_{\max}\le L+\alpha$.
Let $\Se=[0,\alpha]_{\Z}\cup[\beta,s_{\max}]_{\Z}$.
For any $t\in[s_{\max}+2,T-1]_\Z$, inequality \eqref{eqn-T:x_t-ub-3-1} is valid for $\convP$.
For any $t\in[2,T-s_{\max}-1]_\Z$, inequality \eqref{eqn-T:x_t-ub-3-2} is valid for $\convP$.
Furthermore, \eqref{eqn-T:x_t-ub-3-1} and \eqref{eqn-T:x_t-ub-3-2} are facet-defining for $\convP$ when $\eta\in\{0,(\Cupper-\Vupper)/V\}$ or $\eta=L-1\in\Se$.
\end{proposition}

\begin{example}[continuation of example \ref{examp:prop-T:x_t-ub-1}]\label{examp-prop-T:x_t-ub-2}
In Example~\ref{examp:prop-T:x_t-ub-1}, if we set $\eta=2.5$, $\Se=\{0,2,4\}$, and $t=8$,
then by Proposition~\ref{prop-T:x_t-ub-3}, we obtain the following pair of valid inequalities:
{\small \begin{align*}
\left\{
\begin{array}{ll}
x_8\le 25y_3-25y_4+45y_5-45y_6+65y_7-10y_8+25y_9;\\[2pt]
x_8\le 25y_7-10y_8+65y_9-45y_{10}+45y_{11}-25y_{12}+25y_{13}.
\end{array}
\right.
\end{align*} }%
Note that the right-hand sides of the first and second inequalities differ from those in the first pair of inequalities in Example~\ref{examp:prop-T:x_t-ub-1}
by $\eta Vy_{t+1}-\eta Vy_t$ (i.e., $25y_9-25y_8$) and $\eta Vy_{t-1}-\eta Vy_t$ (i.e., $25y_7-25y_8$), respectively.
If we set $\eta=2.5$, $\Se=\{0,1,2,5,6\}$ (i.e., $\alpha=2$, $\beta=5$, and $s_{\max}=6$), and $t=8$, 
then by Proposition~\ref{prop-T:x_t-ub-4}, we obtain the following pair of valid inequalities:
{\small \begin{align*}
\left\{
\begin{array}{ll}
x_8\le 5y_1+10y_2-15y_3+45y_5+10y_6+10y_7-10y_8+25y_9;\\[2pt]
x_8\le 25y_7-10y_8+10y_9+10y_{10}+45y_{11}-15y_{13}+10y_{14}+5y_{15}.
\end{array}
\right.
\end{align*} }%
Similarly, the right-hand sides of the first and second inequalities differ from those in the second pair of inequalities in Example~\ref{examp:prop-T:x_t-ub-1}
by $\eta Vy_{t+1}-\eta Vy_t$ (i.e., $25y_9-25y_8$) and $\eta Vy_{t-1}-\eta Vy_t$ (i.e., $25y_7-25y_8$), respectively.
\end{example}

The next proposition also extends Proposition~\ref{prop-T:x_t-ub-1} and presents another two families of strong valid inequalities.

\begin{proposition} \label{prop-T:x_t-ub-5}
Consider any $\Se\subseteq[1,\min\{L,T-2,\lfloor(\Cupper-\Vupper)/V\rfloor\}]_{\Z}$ and any real number $\eta$ such that $0\le\eta\le\min\{L,(\Cupper-\Vupper)/V\}$.
For any $t\in[2,T]_\Z$ such that $t\ge s+2$ for all $s\in\Se$, the  inequality
{\small \begin{equation}\label{eqn-T:x_t-ub-5-1}
    x_t\le(\Vupper+\eta V)y_t+(\Cupper-\Vupper-\eta V)y_{t-1}-\sum_{s\in\Se}(\Cupper-\Vupper-sV)(y_{t-s}-y_{t-s-1})
\end{equation} }%
is valid for $\convP$.
For any $t\in[1,T-1]_\Z$ such that $t\le T-s-1$ for all $s\in\Se$, the inequality
{\small \begin{equation}\label{eqn-T:x_t-ub-5-2}
    x_t\le(\Vupper+\eta V)y_t+(\Cupper-\Vupper-\eta V)y_{t+1}-\sum_{s\in\Se}(\Cupper-\Vupper-sV)(y_{t+s}-y_{t+s+1})
\end{equation} }%
is valid for $\convP$.
Furthermore, inequalities \eqref{eqn-T:x_t-ub-5-1} and \eqref{eqn-T:x_t-ub-5-2} are facet-defining for $\convP$ when $\eta\in\{0,(\Cupper-\Vupper)/V\}$ or $\eta=L\in\Se$.
\end{proposition}

In Proposition~\ref{prop-T:x_t-ub-5}, the set $\Se$ only contains elements that are less than or equal to $L$.
The following proposition states that under certain conditions, inequalities \eqref{eqn-T:x_t-ub-5-1} and \eqref{eqn-T:x_t-ub-5-2} remain valid and facet-defining when $\Se$ contains some elements that are greater than $L$.

\begin{proposition}\label{prop-T:x_t-ub-6}
Consider any integers $\alpha$, $\beta$, and $s_{\max}$ such that
(a)~$L+1\le s_{\max}\le\min\{T-2,\lfloor(\Cupper-\Vupper)/V\rfloor\}$,
(b)~$1\le\alpha<\beta\le s_{\max}$, and
(c)~$\beta=\alpha+1$ or $s_{\max}\le L+\alpha$.
Let $\Se=[1,\alpha]_{\Z}\cup[\beta,s_{\max}]_{\Z}$.
For any $t\in[s_{\max}+2,T]_{\Z}$, inequality \eqref{eqn-T:x_t-ub-5-1} is valid for $\convP$.
For any $t\in[1,T-s_{\max}-1]_{\Z}$, inequality \eqref{eqn-T:x_t-ub-5-2} is valid for $\convP$.
Furthermore, \eqref{eqn-T:x_t-ub-5-1} and \eqref{eqn-T:x_t-ub-5-2} are facet-defining for $\convP$ when $\eta\in\{0, (\Cupper-\Vupper)/V\}$ or $\eta=L\in\Se$.
\end{proposition}

\begin{example}[continuation of example \ref{examp:prop-T:x_t-ub-1}]\label{examp-prop-T:x_t-ub-3}
In Example~\ref{examp:prop-T:x_t-ub-1}, if we set $\eta=2.5$, $\Se=\{1,3,5\}$, and $t=8$, 
then by Proposition~\ref{prop-T:x_t-ub-5}, we obtain the following pair of valid inequalities:
{\small \begin{align*}
\left\{
\begin{array}{ll}
x_8\le 15y_2-15y_3+35y_4-35y_5+55y_6-15y_7+40y_8;\\[2pt]
x_8\le 40y_8-15y_9+55y_{10}-35y_{11}+35y_{12}-15y_{13}+15y_{14}.
\end{array}
\right.
\end{align*} }%
If we set $\eta=2.5$, $\Se=\{1,2,5,6\}$ (i.e., $\alpha=2$, $\beta=5$, and $s_{\max}=6$), and $t=8$,
then by Proposition~\ref{prop-T:x_t-ub-6}, we obtain the following pair of valid inequalities:
{\small \begin{align*}
\left\{
\begin{array}{ll}
x_8\le 5y_1+10y_2-15y_3+45y_5+10y_6-15y_7+40y_8;\\[2pt]
x_8\le 40y_8-15y_9+10y_{10}+45y_{11}-15y_{13}+10y_{14}+5y_{15}.
\end{array}
\right.
\end{align*} }%
\end{example}
\vskip0pt

Propositions \ref{prop-T:x_t-ub-1}--\ref{prop-T:x_t-ub-6} present different families of valid inequalities.
For each family of valid inequalities and any given point $(\BFx,\BFy)$ with non-binary $y$ values, 
it is important to have an efficient separation algorithm that can identify a most violated inequality in the family, if such a violated inequality exists.
\rred{In Propositions \ref{prop-T:x_t-ub-1}, \ref{prop-T:x_t-ub-3}, and \ref{prop-T:x_t-ub-5}, the number of combinations of $\Se$ is exponential in $T$.
Furthermore, in Propositions \ref{prop-T:x_t-ub-3} and \ref{prop-T:x_t-ub-5}, $\eta$ is a real value.
However, the next proposition states that given any point $(\BFx,\BFy)$ with non-binary $y$ values, 
a most violated inequality in each of the inequality families stated in Propositions \ref{prop-T:x_t-ub-1}, \ref{prop-T:x_t-ub-3}, and \ref{prop-T:x_t-ub-5} can be determined in linear time.}

\begin{proposition}\label{prop:separation-A}
\rred{
For any given point $(\BFx,\BFy)\in\R_{+}^{2T}$, \tred{the} most violated inequalities 
\eqref{eqn-T:x_t-ub-1-1}--\eqref{eqn-T:x_t-ub-1-2}, \eqref{eqn-T:x_t-ub-3-1}--\eqref{eqn-T:x_t-ub-3-2}, and \eqref{eqn-T:x_t-ub-5-1}--\eqref{eqn-T:x_t-ub-5-2}
in Propositions \ref{prop-T:x_t-ub-1}, \ref{prop-T:x_t-ub-3}, and \ref{prop-T:x_t-ub-5}, respectively,
can be determined in $O(T)$ time if such violated inequalities exist.}
\end{proposition}

In Propositions \ref{prop-T:x_t-ub-2}, \ref{prop-T:x_t-ub-4}, and \ref{prop-T:x_t-ub-6}, the number of combinations of $\alpha$, $\beta$, $s_{\max}$, and $t$ is $O(T^4)$.
Furthermore, in Propositions \ref{prop-T:x_t-ub-4} and \ref{prop-T:x_t-ub-6}, $\eta$ is a real value.
However, the next proposition states that given any point $(\BFx,\BFy)$ with non-binary $y$ values, 
a most violated inequality in each of the inequality families stated in Propositions \ref{prop-T:x_t-ub-2}, \ref{prop-T:x_t-ub-4}, and \ref{prop-T:x_t-ub-6} can be determined in $O(T^3)$ time.

\begin{proposition}\label{prop:separation-B}
For any given point $(\BFx,\BFy)\in\R_{+}^{2T}$, \tred{the} most violated inequalities 
\eqref{eqn-T:x_t-ub-1-1}--\eqref{eqn-T:x_t-ub-1-2}, \eqref{eqn-T:x_t-ub-3-1}--\eqref{eqn-T:x_t-ub-3-2}, and \eqref{eqn-T:x_t-ub-5-1}--\eqref{eqn-T:x_t-ub-5-2}
in Propositions \ref{prop-T:x_t-ub-2}, \ref{prop-T:x_t-ub-4}, and \ref{prop-T:x_t-ub-6}, respectively,
can be determined in $O(T^3)$ time if such violated inequalities exist.
\end{proposition}

\subsection{Valid Inequalities with Two Continuous Variables}\label{subsec:multi-period-two}
In this subsection, we present strong valid inequalities that provide upper bounds on $x_t-x_{t-k}$ (respectively $x_t-x_{t+k}$) for each pair of time periods $t$ and $t-k$ (respectively $t$ and $t+k$).
The following proposition presents a pair of such inequality families.

\begin{proposition} \label{prop-T:kth-ramp-2}
Consider any $k\in[1,T-1]_\Z$ such that $\Cupper-\Clower-kV>0$, any $m\in[0,k-1]_\Z$, and any $\Se\subseteq[0,\min\{k-1,L-m-1\}]_\Z$.
For any $t\in[k+1,T-m]_\Z$, the inequality
{\small \begin{align}\label{eqn-T:kth-ramp-2-1}
x_t-x_{t-k}\le (\Clower+(k-m)V)y_t+V\sum_{i=1}^m y_{t+i}-\Clower y_{t-k}-\sum_{s\in\Se}(\Clower+(k-s)V-\Vupper)(y_{t-s}-y_{t-s-1})
\end{align} }%
is valid for $\convP$.
For any $t\in[m+1,T-k]_\Z$, the inequality
{\small \begin{align}\label{eqn-T:kth-ramp-2-2}
x_t-x_{t+k}\le (\Clower+(k-m)V)y_t+V\sum_{i=1}^m y_{t-i}-\Clower y_{t+k}-\sum_{s\in\Se}(\Clower+(k-s)V-\Vupper)(y_{t+s}-y_{t+s+1})
\end{align} }%
is valid for $\convP$.
Furthermore, \eqref{eqn-T:kth-ramp-2-1} and \eqref{eqn-T:kth-ramp-2-2} are facet-defining for $\convP$ when $m=0$ and $s\ge\min\{k-1,1\}$ for all $s\in\Se$.
\end{proposition}

In Proposition \ref{prop-T:kth-ramp-2}, the number of combinations of $\Se$, $t$, $k$, and $m$ is exponential in $T$.
Thus, the sizes of the inequality families \eqref{eqn-T:kth-ramp-2-1} and \eqref{eqn-T:kth-ramp-2-2} are exponential in $T$.
However, the next proposition states that given any point $(\BFx,\BFy)$ with non-binary $y$ values, 
\tred{the} most violated inequalities \eqref{eqn-T:kth-ramp-2-1} and \eqref{eqn-T:kth-ramp-2-2} can be determined in polynomial time.

\begin{proposition} \label{prop-T:kth-ramp-2-separation}
For any given point $(\BFx,\BFy)\in\R^{2T}_+$, \tred{the} most violated inequalities \eqref{eqn-T:kth-ramp-2-1} and \eqref{eqn-T:kth-ramp-2-2} can be determined in $O(T^3)$ time if such violated inequalities exist.
\end{proposition}

\begin{proposition} \label{prop-T:kth-ramp-3}
Consider any $k\in[1,T-1]_\Z$ such that $\Cupper-\Clower-kV>0$, any $m\in[0,k-1]_\Z$, and any $\Se\subseteq[0,\min\{k-1,L-m-2\}]_\Z$.
For any $t\in[k+1,T-m-1]_\Z$, the inequality
{\small \begin{align}\label{eqn-T:kth-ramp-3-1}
x_t-x_{t-k}&\le (\Clower+(k-m)V-\Vupper)y_{t+m+1}+V\sum_{i=1}^m y_{t+i}+\Vupper y_t-\Clower y_{t-k}\nonumber\\
           &\quad -\sum_{s\in\Se}(\Clower+(k-s)V-\Vupper)(y_{t-s}-y_{t-s-1})
\end{align} }%
is valid and facet-defining for $\convP$.
For any $t\in[m+2,T-k]_\Z$, the inequality
{\small \begin{align}\label{eqn-T:kth-ramp-3-2}
x_t-x_{t+k}&\le (\Clower+(k-m)V-\Vupper)y_{t-m-1}+V\sum_{i=1}^m y_{t-i}+\Vupper y_t-\Clower y_{t+k}\nonumber\\
           &\quad -\sum_{s\in\Se}(\Clower+(k-s)V-\Vupper)(y_{t+s}-y_{t+s+1})
\end{align} }%
is valid and facet-defining for $\convP$.
\end{proposition}

In Proposition \ref{prop-T:kth-ramp-3}, the number of combinations of $\Se$, $t$, $k$, and $m$ is exponential in $T$.
Thus, the sizes of the inequality families \eqref{eqn-T:kth-ramp-3-1} and \eqref{eqn-T:kth-ramp-3-2} are exponential in $T$.
However, the next proposition states that given any point $(\BFx,\BFy)$ with non-binary $y$ values, 
\tred{the} most violated inequalities \eqref{eqn-T:kth-ramp-3-1} and \eqref{eqn-T:kth-ramp-3-2} can be determined in polynomial time.

\begin{proposition}\label{prop-T:kth-ramp-3-separation}
For any given point $(\BFx,\BFy)\in\R^{2T}_+$, \tred{the} most violated inequalities \eqref{eqn-T:kth-ramp-3-1} and \eqref{eqn-T:kth-ramp-3-2} can be determined in $O(T^3)$ time if such violated inequalities exist.
\end{proposition}
\vskip0pt

\section{Computational Experiments} \label{sec:comp-exper}
We conduct a computational study to evaluate the effectiveness of our strong valid inequalities in tightening the proposed \rred{single-binary} MILP formulation for the UC problem.
In Section~\ref{subsec:instance-settings}, we describe the problem instances that we use in this computational study.
In Section~\ref{subsec:results}, we present the computational results.

All \fred{of the} computational experiments are performed on a computer node with Intel(R) Xeon(R) CPU E5-2699 v3 at 2.30GHz and 16 cores.
The addressable memory is 32GB.
IBM ILOG CPLEX 22.1 is used as the MILP solver to run all \fred{of the} experiments.
The MILP solver is called through its Python application programming interface under \fred{the} default settings.
Note that the performance of a formulation is affected by the inherent random component of the heuristic process used in solvers \citep{tejada2020unit}.
Thus, to accurately evaluate the effectiveness of our strong valid inequalities, ``traditional branch-and-cut" is chosen to be the search strategy.

\subsection{Test Instances}\label{subsec:instance-settings}

We conduct three computational experiments.
\fred{These experiments} are based on a network-constrained UC problem.
Recall that in Sections \ref{sec:two-period} and \ref{sec:multi-period}, 
the superscript $g$ was omitted when we focused on deriving strong valid inequalities for the polytope $\convP$ that consists of a single generator.
In the test instances of these three experiments, we reinstate the superscript $g$ in the strong valid inequalities that are \fred{used} to tighten the UC formulations.
Thus, when a family of strong valid inequalities is added to a UC formulation in these three experiments, it will be added to each generator with its corresponding parameters at the same time.
In all three experiments, the non-decreasing convex piecewise cost function is obtained by approximating the given quadratic cost function $a^gx^2+b^gx$.
We apply the method developed by \cite{frangioni2009tighter} to perform this piecewise linear approximation, using nine line segments with the $x$-coordinates of the breakpoints spread evenly between the lower bound $\Clower$ and the upper bound $\Cupper$.

In the first experiment, we use the data obtained from \cite{ostrowski2012tight} and \cite{pan2016polyhedral}.
\fred{Because of} the absence of transmission flow data in this data set, the transmission constraint \eqref{prob1:transmission-constraint} is not considered in this experiment.
The removal of the transmission constraint \fred{does not have a} major impact on our computational study \fred{because} we focus primarily on evaluating the effectiveness of the strong valid inequalities in tightening the \rred{single-binary} formulation.

The system contains eight types of generators.
Table~\ref{tab:8-gen-data} contains the data of these eight generator types.
The generation cost function for generator $g$ is $a^gx^2+b^gx$, where the values of $a^g$ and $b^g$ are provided in the 10th and 11th columns, respectively, of the table.
The data set comprises 20 test instances, as shown in Table~\ref{tab:ost-inst-comb}.
For each instance, the operation horizon is set equal to 24 hours, i.e., $T=24$, and the system reserve factor is set equal to 3\% for all periods, i.e., $r_t=0.03$ for all $t\in[1,T]_\Z$.
The system load $\sum_{b\in\Be}d^b_t$ in each period $t$ is shown in Table~\ref{tab:ost-load}, and it is expressed as a percentage of the total generation capacity $\sum_{g\in\Ge}\Cupper^g$.

\begin{table}[h]
  \centering
  \footnotesize
  \renewcommand{\arraystretch}{1}
  \caption{Generator Data \citep{ostrowski2012tight,pan2016polyhedral}}
    \begin{tabular}{|*{12}{c|}}
    \hline
    \multirow{2}{*}{Generator Type} & $\Clower^g$ & $\Cupper^g$ & $L^g$ & $\ell^g$ & $V^g$ & $\Vupper^g$ & $\phi^g$ & $\psi^g$ & $a^g$ & $b^g$ & $c^g$ \\
    & (MW) &(MW) & (h) & (h) & (MW/h) & (MW/h) & (\verb|$|/h) & (\verb|$|/h) & (\verb|$|/M$W^{2}$h) & (\verb|$|/MWh) & (\verb|$|/h) \\
    \hline
    1     & 150   & 455   & 8   & 8   & 91    & 180   & 2000  & 2000  & 0.00048   & 16.19   & 1000 \\
    2     & 150   & 455   & 8   & 8   & 91    & 180   & 2000  & 2000  & 0.00031   & 17.26   & 970 \\
    3     & 20    & 130   & 5   & 5   & 26    & 35    & 500   & 500   & 0.00200   & 16.6    & 700 \\
    4     & 20    & 130   & 5   & 5   & 26    & 35    & 500   & 500   & 0.00211   & 16.5    & 680 \\
    5     & 25    & 162   & 6   & 6   & 32.4  & 40    & 700   & 700   & 0.00398   & 19.7    & 450 \\
    6     & 20    & 80    & 3   & 3   & 16    & 28    & 150   & 150   & 0.00712   & 22.26   & 370 \\
    7     & 25    & 85    & 3   & 3   & 17    & 33    & 200   & 200   & 0.00079   & 27.74   & 480 \\
    8     & 10    & 55    & 1   & 1   & 11    & 15    & 60    & 60    & 0.00413   & 25.92   & 660 \\
    \hline
    \end{tabular}
  \label{tab:8-gen-data}
  \vspace{-2mm}
\end{table}

\begin{table}[h]
    \centering
    \footnotesize
    \caption{Problem Instances \citep{ostrowski2012tight,pan2016polyhedral}}
    \renewcommand{\arraystretch}{1}
    \setlength\tabcolsep{4pt}
    \begin{tabular}{|c|*{8}{c}|c|}
    \hline
    \multirow{2}{*}{Instance} & \multicolumn{8}{c|}{Number of generators} & \multirow{2}{*}{\parbox{48pt}{\centering Total no.~of generators}} \\
    \cline{2-9}
    & Type 1 & Type 2 & Type 3 & Type 4 & Type 5 & Type 6 & Type 7 & Type 8 & \\ \hline
    1   & 12  & 11  & 0   & 0   & 1   & 4   & 0   & 0   & 28 \\
    2   & 13  & 15  & 2   & 0   & 4   & 0   & 0   & 1   & 35 \\
    3   & 15  & 13  & 2   & 6   & 3   & 1   & 1   & 3   & 44 \\
    4   & 15  & 11  & 0   & 1   & 4   & 5   & 6   & 3   & 45 \\
    5   & 15  & 13  & 3   & 7   & 5   & 3   & 2   & 1   & 49 \\
    6   & 10  & 10  & 2   & 5   & 7   & 5   & 6   & 5   & 50 \\
    7   & 17  & 16  & 1   & 3   & 1   & 7   & 2   & 4   & 51 \\
    8   & 17  & 10  & 6   & 5   & 2   & 1   & 3   & 7   & 51 \\
    9   & 12  & 17  & 4   & 7   & 5   & 2   & 0   & 5   & 52 \\
    10  & 13  & 12  & 5   & 7   & 2   & 5   & 4   & 6   & 54 \\
    11  & 46  & 45  & 8   & 0   & 5   & 0   & 12  & 16  & 132 \\
    12  & 40  & 54  & 14  & 8   & 3   & 15  & 9   & 13  & 156 \\
    13  & 50  & 41  & 19  & 11  & 4   & 4   & 12  & 15  & 156 \\
    14  & 51  & 58  & 17  & 19  & 16  & 1   & 2   & 1   & 165 \\
    15  & 43  & 46  & 17  & 15  & 13  & 15  & 6   & 12  & 167 \\
    16  & 50  & 59  & 8   & 15  & 1   & 18  & 4   & 17  & 172 \\
    17  & 53  & 50  & 17  & 15  & 16  & 5   & 14  & 12  & 182 \\
    18  & 45  & 57  & 19  & 7   & 19  & 19  & 5   & 11  & 182 \\
    19  & 58  & 50  & 15  & 7   & 16  & 18  & 7   & 12  & 183 \\
    20  & 55  & 48  & 18  & 5   & 18  & 17  & 15  & 11  & 187 \\
    \hline
    \end{tabular}
  \label{tab:ost-inst-comb}
  \vspace{-4mm}
\end{table}

\begin{table}[h]
    \centering
    \footnotesize
    \renewcommand{\arraystretch}{1.2}
    \setlength\tabcolsep{7.5pt}
    \caption{System Load---Percentage of Total Generation Capacity \citep{ostrowski2012tight,pan2016polyhedral}}
    \begin{tabular}{|c|*{12}{c}|}
    \hline
    Period      & 1    & 2    & 3    & 4    & 5    & 6    & 7    & 8    & 9    & 10   & 11   & 12   \\
    \hline
    System Load & 71\% & 65\% & 62\% & 60\% & 58\% & 58\% & 60\% & 64\% & 73\% & 80\% & 82\% & 83\% \\
    \hline\hline
    Period      & 13   & 14   & 15   & 16   & 17   & 18   & 19   & 20   & 21   & 22   & 23   & 24   \\
    \hline
    System Load & 82\% & 80\% & 79\% & 79\% & 83\% & 91\% & 90\% & 88\% & 85\% & 84\% & 79\% & 74\% \\
    \hline
    \end{tabular}
  \label{tab:ost-load}
  \vspace{-4mm}
\end{table}

In this experiment, we compare the following two formulations:
{\small \begin{subeqnarray*}
  \mbox{F1: \ } & \mbox{minimize \ } & \mbox{objective function \eqref{prob1:obj}}\\[-4pt]
  & \mbox{subject to \ } & \mbox{constraints \eqref{prob1:min-up}--\eqref{prob1:spinning-reserve}, \eqref{prob1:nonnegativity}.}\\[-32pt]
\end{subeqnarray*}
\begin{subeqnarray*}
  \mbox{F1-X: \ } & \mbox{minimize \ } & \mbox{objective function \eqref{prob1:obj}}\\[-4pt]
  & \mbox{subject to \ } & \mbox{constraints \eqref{prob1:min-up}--\eqref{prob1:spinning-reserve}, \eqref{prob1:nonnegativity};}\\[-4pt]
  && \mbox{constraints \eqref{multi:eqn-q2:x1-ub}--\eqref{multi:eqn-q2:x1-x2-ub-2};}\\[-4pt]
  && \mbox{\rred{cutting planes} \eqref{eqn-T:x_t-ub-1-1}--\eqref{eqn-T:kth-ramp-3-2}.}
\end{subeqnarray*} }%
Formulation F1 is the original formulation of Problem~1 with the transmission constraint \eqref{prob1:transmission-constraint} removed.
In \mbox{F1-X}, the strong valid inequality families \eqref{multi:eqn-q2:x1-ub}--\eqref{multi:eqn-q2:x1-x2-ub-2} obtained from the two-period single-UC polytope are added to the formulation as constraints, \fred{and} the multi-period strong valid inequality families \rred{\eqref{eqn-T:x_t-ub-1-1}--\eqref{eqn-T:kth-ramp-3-2}} derived in Section~\ref{sec:multi-period} are added to the user cut pool of the CPLEX optimizer and are applied at any stage of the optimization.
Note that each of the inequality families \eqref{eqn-T:x_t-ub-1-1}--\eqref{eqn-T:kth-ramp-3-2} contains a large number of inequalities.
Thus, for each of these inequality families, only some of the inequalities are added to \mbox{F1-X} as user cuts
\sred{, because the use of separation algorithms through the callbacks of CPLEX will slow down the solution process.}
\tred{This is attributable to the frequent invocation of the separation algorithms during the solution process.}
Specifically, for each of \fred{the} inequality families \eqref{eqn-T:x_t-ub-1-1}--\eqref{eqn-T:kth-ramp-3-2}, $\Se$ is restricted to the empty set and the set that contains all \fred{of the} elements in its range, \fred{and} the other parameters such as $t$, $k$, and $m$ are allowed to take any values in their respective ranges such that the inequality obtained is facet-defining for $\convP$.
For example, for inequality family \eqref{eqn-T:kth-ramp-2-1}, 
we consider each $k\in[1,T-1]_\Z$, \rred{$m=0$, $\Se=\{\emptyset,[0,\min\{k-1,L-1\}]_\Z\}$}, and $t\in[k+1,T]_\Z$ such that $s\ge\min\{k-1,1\}$ for all $s\in\Se$. 

In the second experiment, we use the same data as in the first experiment.
We compare the effectiveness of our strong valid inequalities with that of the valid inequalities in \cite{pan2016polyhedral} in tightening \citeauthor{pan2016polyhedral}'s formulation.
We also test the effectiveness of our strong valid inequalities when they are combined with the valid inequalities in \cite{pan2016polyhedral}.
To do so, we solve the following four formulations:
{\small \begin{subeqnarray*}
  \mbox{F2: \ } & \mbox{minimize \ } & \mbox{objective function (38a) in \cite{pan2016polyhedral}}\\[-4pt]
  & \mbox{subject to \ } & \mbox{constraints (38b)--(38i), (38k) in \cite{pan2016polyhedral}.}\\[-32pt]
\end{subeqnarray*}
\begin{subeqnarray*}
  \mbox{F2-X: \ } & \mbox{minimize \ } & \mbox{objective function (38a) in \cite{pan2016polyhedral}}\\[-4pt]
  & \mbox{subject to \ } & \mbox{constraints (38b)--(38i), (38k) in \cite{pan2016polyhedral};}\\[-4pt]
  && \mbox{constraints \eqref{multi:eqn-q2:x1-ub}--\eqref{multi:eqn-q2:x1-x2-ub-2} in this paper;}\\[-4pt]
  && \mbox{\rred{cutting planes} \eqref{eqn-T:x_t-ub-1-1}--\eqref{eqn-T:kth-ramp-3-2} in this paper.}\\[-32pt]
\end{subeqnarray*}
\begin{subeqnarray*}
  \mbox{F2-Y: \ } & \mbox{minimize \ } & \mbox{objective function (38a) in \cite{pan2016polyhedral}}\\[-4pt]
  & \mbox{subject to \ } & \mbox{constraints (38b)--(38i), (38k) in \cite{pan2016polyhedral};}\\[-4pt]
  && \mbox{constraints (2d)--(2g) in \cite{pan2016polyhedral};}\\[-4pt]
  && \mbox{\rred{cutting planes} (4)--(7), (10)--(13), (24d), (24h)--(24i),}\\[-4pt]
  && \quad\mbox{(24o)--(24r), (28)--(36) in \cite{pan2016polyhedral}.}\\[-50pt]
\end{subeqnarray*}
\rred{\begin{subeqnarray*}
  \mbox{F2-Z: \ } & \mbox{minimize \ } & \mbox{objective function (38a) in \cite{pan2016polyhedral}}\\[-4pt]
  & \mbox{subject to \ } & \mbox{constraints (38b)--(38i), (38k) in \cite{pan2016polyhedral};}\\[-4pt]
  && \mbox{constraints (2d)--(2g) in \cite{pan2016polyhedral};}\\[-4pt]
  && \mbox{constraints \eqref{multi:eqn-q2:x1-ub}--\eqref{multi:eqn-q2:x1-x2-ub-2} in this paper;}\\[-4pt]
  && \mbox{cutting planes (4)--(7), (10)--(13), (24d), (24h)--(24i),}\\[-4pt]
  && \quad\mbox{(24o)--(24r), (28)--(36) in \cite{pan2016polyhedral}.}\\[-4pt]
  && \mbox{cutting planes \eqref{eqn-T:x_t-ub-1-1}--\eqref{eqn-T:kth-ramp-3-2} in this paper.}\\[-28pt]
\end{subeqnarray*} }%
}Formulation F2 is the two-binary UC formulation in \cite{pan2016polyhedral}, \fred{except} that the transmission constraint (38j) has been excluded.
In formulation \mbox{F2-X}, the valid inequalities \eqref{multi:eqn-q2:x1-ub}--\eqref{multi:eqn-q2:x1-x2-ub-2} are added as constraints, 
and the valid inequalities \eqref{eqn-T:x_t-ub-1-1}--\eqref{eqn-T:kth-ramp-3-2} are added as user cuts in the same way as in the first experiment.
In formulation F2-Y, \fred{the} strong valid inequalities in \cite{pan2016polyhedral} are used the same way as in Pan and Guan's computational study to tighten formulation F2.
Specifically, valid inequalities in the two-period convex hull, (2d)--(2g), are added as constraints, \fred{and} other valid inequalities are added as user cuts.
For inequality families that have an exponential size, the $\Se$ set is restricted to the empty set and the set that contains all \fred{of the} elements in its range.
The other parameters, such as $t$, $m$, and $n$, are allowed to take any values in their respective ranges such that the inequality obtained is facet-defining for $\convP$.
\rred{Formulation F2-Z contains all the valid inequalities in formulations F2-X and F2-Y.}

The third experiment examines a network-constrained UC problem based on the modified IEEE 118-bus system.
The system comprises 54 thermal generators, 118 buses, and 186 transmission lines.
System data such as $\Cupper^g$, $\Clower^g$, $L^g$, $\ell^g$, $a^g$, $b^g$, $c^g$, etc., as well as the load amount of each load bus, are obtained from \mbox{\url{http://motor.ece.iit.edu/data/SCUC_118/}}.
Each instance has a 24-hour operation horizon, i.e., $T=24$.
The system reserve factor of each time period is set equal to 3\%, as in the first experiment.
The maximum hourly load of the system is randomly generated from a uniform distribution on $[0.5\sum_{g\in\Ge}\Cupper^g,\sum_{g\in\Ge}\Cupper^g]$.
The maximum hourly load of each load bus is then obtained by allocating the maximum hourly load of the system to each load bus in proportion to their load amounts.
For each load bus, the loads in different time periods are then obtained by following a daily load profile such that the maximum load of the day is equal to the maximum hourly load.
This daily load profile is obtained from \mbox{\url{https://www.pjm.com/markets-and-operations/data-dictionary}}, which was generated based on the average values of the actual hourly electricity demand over 30 days in the western market.
Twenty instances with randomly generated loads are created using this method.
Each instance is solved using the following formulations:
{\small \begin{subeqnarray*}
  \mbox{\rred{\F1^+}: \ } & \mbox{minimize \ } & \mbox{objective function \eqref{prob1:obj}}\\[-4pt]
  & \mbox{subject to \ } & \mbox{constraints \eqref{prob1:min-up}--\eqref{prob1:nonnegativity}.}\\[-35pt]
\end{subeqnarray*}
\begin{subeqnarray*}
  \mbox{\rred{\F1^+-X}: \ } & \mbox{minimize \ } & \mbox{objective function \eqref{prob1:obj}}\\[-4pt]
  & \mbox{subject to \ } & \mbox{constraints \eqref{prob1:min-up}--\eqref{prob1:nonnegativity};}\\[-4pt]
  && \mbox{constraints \eqref{multi:eqn-q2:x1-ub}--\eqref{multi:eqn-q2:x1-x2-ub-2};}\\[-4pt]
  && \mbox{\rred{cutting planes} \eqref{eqn-T:x_t-ub-1-1}--\eqref{eqn-T:kth-ramp-3-2}.}
\end{subeqnarray*} }%
Formulations \rred{\F1^+} and \rred{\mbox{\F1^+-X}} resemble formulations F1 and \mbox{F1-X}, respectively, with the transmission constraint \eqref{prob1:transmission-constraint} reinstated.
In \rred{\mbox{\F1^+-X}}, \fred{the} strong valid inequalities \eqref{multi:eqn-q2:x1-ub}--\eqref{multi:eqn-q2:x1-x2-ub-2} and \eqref{eqn-T:x_t-ub-1-1}--\eqref{eqn-T:kth-ramp-3-2} are added as constraints and user cuts, respectively, in the same way as in the first experiment.

\subsection{Computational Results} \label{subsec:results}

In this subsection, we report the computational results of the three experiments described in Section~\ref{subsec:instance-settings}.
In these experiments, each test instance is executed once using each of the formulations in the experiment, and the time limit for each execution is set to one hour.
Tables \ref{tab:ucresult-1}--\ref{tab:ucresult-3} summarize the computational results.
In these tables, the ``\# var," ``\# bin var," and ``\# cstr" columns report the total number of decision variables, the number of binary decision variables, and the number of constraints (excluding nonnegativity and binary constraints), respectively, in the formulation.
The ``IGap" columns report the root node integrality gaps of \fred{the} different formulations, where IGap is given as 
$|Z^*-Z_{\text{LP}}|/Z^*\times 100\%$, where
$Z^*$ is the best objective function value obtained by solving the formulations in the experiment and
$Z_{\text{LP}}$ is the optimal objective function value of the LP relaxation of the formulation concerned.
Note that for each strong formulation, the objective function value of its LP relaxation (i.e., $Z_{\text{LP}}$) is obtained such that there is no violation of all valid inequality families.
For example, to obtain $Z_{\text{LP}}$ for F1-X, we solve its LP relaxation with a subset of valid inequalities \eqref{eqn-T:x_t-ub-1-1}--\eqref{eqn-T:kth-ramp-3-2} added as constraints first.
When an optimal solution is obtained, the separation algorithms are invoked to check for any violation of valid inequalities \eqref{eqn-T:x_t-ub-1-1}--\eqref{eqn-T:kth-ramp-3-2}.
If there exists a violation, the most violated inequality will be added to the LP problem as constraints and the new problem is solved again to obtain an optimal solution.
This process is repeated until there is no violation of valid inequalities \eqref{eqn-T:x_t-ub-1-1}--\eqref{eqn-T:kth-ramp-3-2}.
Thus, the obtained solution satisfies valid inequalities \eqref{eqn-T:x_t-ub-1-1}--\eqref{eqn-T:kth-ramp-3-2}, and its objective function value $Z_{\text{LP}}$ is used to calculate the integrality gap.
This integrality gap measures the tightness of the formulation.
To evaluate the effectiveness of the strong valid inequalities in tightening the formulation,
we report the percentage reduction in integrality gap in the ``Pct.~reduction" columns, where $\mbox{Pct.~reduction}=\frac{\text{IGap}_{\text{without cuts}}-\text{IGap}_{\text{with \rred{{cuts}}}}}{\text{IGap}_{\text{\rred{without cuts}}}}\times 100\%,$
$\text{IGap}_{\text{\rred{without cuts}}}$ is the integrality gap of the formulation \rred{without our derived strong valid inequalities}, and
$\text{IGap}_{\text{with \rred{cuts}}}$ is the integrality gap of the current formulation.
The ``CPU time [TGap]" columns report the computational time (in seconds) required to solve the instance to optimality (with a default optimality gap of $0.01\%$). 
Instances that could not be solved to optimality within one hour are marked with ``**," and the terminating gaps of those instances are reported (enclosed in square brackets).
The ``\# nodes" columns report the number of branch-and-cut nodes explored.
The ``\# user cuts" columns report the number of user cuts added to each formulation.

\begin{table}[h!]
  \renewcommand{\arraystretch}{1}
  \centering
  \caption{\tred{Performance of MIP Formulations in the First Experiment}}
  \setlength\tabcolsep{5pt}
  \fontsize{8}{12}\selectfont
  \begin{tabular}{|*{10}{c|}}
    \hline
    \multirow{2}{*}{Instance} & \multirow{2}{*}{\# var} & \multirow{2}{*}{\# bin var} & \multicolumn{2}{c|}{\# cstr} & \multicolumn{2}{c|}{CPU time [TGap]} 
    & \multicolumn{2}{c|}{\# nodes} & \# user cuts \\
    \cline{4-10}
       &       &      & F1     & F1-X   & F1          & F1-X        & F1      & F1-X    & F1-X \\
    \hline
     1 & 3360  & 672  & 19422  & 23286  & 666.8 &        38.6 &  289339 &   11563 &  211 \\
     2 & 4200  & 840  & 24514  & 29344  & ** [0.06\%] &       301.8 & 1009424 &  137851 &  367 \\
     3 & 5280  & 1056 & 29556  & 35628  & ** [0.03\%] &       231.2 &  596629 &   53267 &  554 \\
     4 & 5400  & 1080 & 29304  & 35514  & ** [0.01\%] &      1335.1 & 1584979 &  494770 &  299 \\
     5 & 5880  & 1176 & 32812  & 39574  & ** [0.06\%] &       703.7 &  481348 &  149614 &  525 \\
     6 & 6000  & 1200 & 31562  & 38462  & ** [0.07\%] &       190.5 &  604379 &   94790 &  520 \\
     7 & 6120  & 1224 & 33610  & 40648  & ** [0.05\%] &      3194.2 &  463539 &  903454 &  445 \\
     8 & 6120  & 1224 & 32932  & 39970  & ** [0.09\%] &       662.4 &  504831 &  171291 &  518 \\
     9 & 6240  & 1248 & 34346  & 41522  & ** [0.09\%] &      1630.6 &  456636 &  425148 &  556 \\
    10 & 6480  & 1296 & 34356  & 41808  & ** [0.09\%] &      2119.2 &  410349 &  433936 &  697 \\
    11 & 15840 & 3168 & 87526  & 105742 & ** [0.12\%] & ** [0.04\%] &  274195 &  255185 & 1006 \\
    12 & 18720 & 3744 & 102148 & 123676 & ** [0.09\%] & ** [0.02\%] &  212690 &  279909 & 1923 \\
    13 & 18720 & 3744 & 102174 & 123702 & ** [0.10\%] &       927.5 &  172698 &   59414 & 1457 \\
    14 & 19800 & 3960 & 113314 & 136084 & ** [0.12\%] & ** [0.01\%] &   88150 &  123673 & 2899 \\
    15 & 20040 & 4008 & 109194 & 132240 & ** [0.17\%] & ** [0.01\%] &  114417 &  150664 & 1566 \\
    16 & 20640 & 4128 & 112994 & 136730 & ** [0.11\%] & ** [0.01\%] &  240391 &  234442 & 2214 \\
    17 & 21840 & 4368 & 120156 & 145272 & ** [0.13\%] &       660.0 &  177862 &   15758 & 1208 \\
    18 & 21840 & 4368 & 119936 & 145052 & ** [0.13\%] &      2880.9 &  207181 &  130559 & 1533 \\
    19 & 21960 & 4392 & 120816 & 146070 & ** [0.09\%] &      1958.5 &  208972 &   54786 & 2312 \\
    20 & 22440 & 4488 & 122468 & 148274 & ** [0.12\%] &       920.3 &  257848 &   19072 & 1198 \\
    \hline
  \end{tabular}
  \label{tab:ucresult-1}
  \vspace{-4mm}
\end{table}

\begin{table}[h!]
  \renewcommand{\arraystretch}{1}
  \centering
  \caption{\tred{The Strength of LP Relaxations of MIP Formulations in the First Experiment}}
  \setlength\tabcolsep{5pt}
  \fontsize{8}{12}\selectfont
  \begin{tabular}{|c|c|*{10}{c}|}
    \hline
    \multirow{6}{*}{IGap} & Instance & 1 & 2 & 3 & 4 & 5 & 6 & 7 & 8 & 9 & 10\\
    \cline{2-12}
     & F1 & 0.45\% & 0.39\% & 0.41\% & 0.36\% & 0.51\% & 0.61\% & 0.34\% & 0.55\% & 0.51\% & 0.64\% \\
     & F1-X & 0.20\% & 0.14\% & 0.08\% & 0.06\% & 0.05\% & 0.04\% & 0.07\% & 0.06\% & 0.06\% & 0.04\% \\
    \cline{2-12}
     & Instance & 11 & 12 & 13 & 14 & 15 & 16 & 17 & 18 & 19 & 20\\
    \cline{2-12}
     & F1 & 0.32\% & 0.32\% & 0.40\% & 0.38\% & 0.50\% & 0.29\% & 0.45\% & 0.42\% & 0.37\% & 0.43\% \\
     & F1-X & 0.06\% & 0.02\% & 0.02\% & 0.03\% & 0.02\% & 0.02\% & 0.02\% & 0.02\% & 0.02\% & 0.02\% \\
    \hline
    \multirow{4}{*}{Pct. reduction} & Instance & 1 & 2 & 3 & 4 & 5 & 6 & 7 & 8 & 9 & 10\\
    \cline{2-12}
    & F1-X & 55.7\% & 63.8\% & 80.4\% & 82.3\% & 90.0\% & 93.3\% & 78.7\% & 89.4\% & 89.1\% & 93.1\% \\
    \cline{2-12}
    & Instance & 11 & 12 & 13 & 14 & 15 & 16 & 17 & 18 & 19 & 20\\
    \cline{2-12}
    & F1-X & 82.6\% & 94.0\% & 95.2\% & 90.9\% & 96.4\% & 94.1\% & 96.1\% & 95.9\% & 94.1\% & 96.1\%\\
    \hline
  \end{tabular}
  \label{tab:ucresult-1-2}
  \vspace{-4mm}
\end{table}

The computational results of the first experiment are presented in Tables~\ref{tab:ucresult-1} and \ref{tab:ucresult-1-2}.
The integrality gaps generated by formulation \mbox{F1-X} are considerably smaller than those generated by formulation F1,
particularly for large instances (i.e., instances 11--20, where the total number of generators exceeds 100 and the number of binary decision variables exceeds 2000).
This suggests that formulation \mbox{F1-X} is tighter than formulation F1.
Using formulation F1, CPLEX is able to solve only one of the 20 test instances to optimality within one hour.
In contrast, using formulation \mbox{F1-X}, CPLEX is able to solve 15 instances to optimality within the same time limit.
For instances that cannot be solved to optimality using formulation \mbox{F1-X}, the terminating gaps are all within $0.05\%$ and \fred{are} much smaller than those using formulation F1.
Formulation \mbox{F1-X} tends to explore fewer nodes than formulation F1, and
the number of user cuts added by \mbox{F1-X} in the solution process is small compared \fred{with} the total number of constraints in formulations F1 and \rred{\mbox{F1-X}}.
These results demonstrate that our proposed strong valid inequalities can significantly tighten the single-binary formulation of the network-constrained UC problem and thus speed up the solution process.

Tables~\ref{tab:ucresult-2} and \ref{tab:ucresult-2-2} present the computational results of the second experiment, 
in which \rred{four} formulations, namely F2, \mbox{F2-X}, \rred{\mbox{F2-Y}, and \mbox{F2-Z}}, are used \fred{to solve} the network-constrained UC problem.
Using formulation F2, CPLEX solves only two of the 20 instances within the one-hour time limit.
Using formulation \mbox{F2-X}, which includes our proposed valid inequalities, CPLEX is able to solve 16 instances to optimality.
Using formulation \mbox{F2-Y}, which includes the valid inequalities developed by \cite{pan2016polyhedral}, CPLEX is also able to solve 16 instances to optimality.
The integrality gaps and CPU time of formulations \mbox{F2-X} and \mbox{F2-Y} are significantly smaller than those of formulation F2, \fred{whereas} the integrality gaps and CPU time of formulation \mbox{F2-X} are comparable to those of formulation F2-Y.

\begin{table}[h!]
  \renewcommand{\arraystretch}{1.1}
  \centering
  \caption{\tred{Performance of MIP Formulations in the Second Experiment}}
  \setlength\tabcolsep{2.5pt}
  \fontsize{6.5}{10}\selectfont
  \begin{tabular}{|*{18}{c|}}
    \hline
    \multirow{2}{*}{Instance} & \multirow{2}{*}{\# var} & \multirow{2}{*}{\# bin var} & \multicolumn{4}{c|}{\# cstr} & \multicolumn{4}{c|}{CPU time [TGap]} 
    & \multicolumn{4}{c|}{\# nodes} & \multicolumn{3}{c|}{\# user cuts} \\
    \cline{4-18}
       &       &       & F2     & F2-X   & F2-Y   & F2-Z  & F2     & F2-X   & F2-Y   & F2-Z    & F2     & F2-X    & F2-Y  & F2-Z   & F2-X   & F2-Y   & F2-Z\\
    \hline
     1 & 3304  & 1960  & 12300  & 16164  & 14876  & 18096 & 710.8  & 26.0  &  33.8  & 17.3  & 365652  & 12926  & 25641  & 11508  & 168  & 98  & 203\\
     2 & 4130  & 2450  & 15350  & 20180  & 18570  & 22595 & ** [0.07\%] & 819.0  & 316.2  & 292.0  & 900272 &  577704 &  321040  & 213721  & 294  & 126 & 307\\
     3 & 5192  & 3080  & 19354  & 25426  & 13402  & 28462 & ** [0.02\%] & 134.2  & 104.5  & 234.3  & 666918 & 57024  & 44507  & 86655  & 494  & 252 & 439\\
     4 & 5310  & 3150  & 19842  & 26052  & 23982  & 29157 & ** [0.01\%] & 222.7  & 1675.7  & 497.4  & 760760 &  260255 & 1732340  & 505512  & 203  & 130  & 234\\
     5 & 5782  & 3430  & 21556  & 28318  & 26064  & 31699  & ** [0.05\%] & 127.1  & 458.9  & 406.6  & 516797  & 65305  & 309318  & 244711  & 448  & 243 & 463\\
     6 & 5900  & 3500  & 22098  & 28998  & 26698  & 32448  & 1908.6  & 219.9  & 147.1  & 151.2  & 375616 &  268065 &  213464  & 151584  & 515  & 202  & 475\\
     7 & 6018  & 3570  & 22458  & 29496  & 27150  & 33015  & ** [0.04\%] & ** [0.01\%] & 417.0  & 564.4  & 618447 & 3089892 & 357940  & 456912  & 291  & 132 & 317\\
     8 & 6018  & 3570  & 22496  & 29534  & 27188  & 33053  & ** [0.05\%] & 419.0  & 238.9  & 72.5  & 526007 & 240221 &  147181 & 27169 & 407  & 208 & 344\\
     9 & 6136  & 3640  & 22896  & 30072  & 27680  & 33660 & ** [0.05\%] & 930.8  & 1034.6  & 985.1  & 434524 &  440662 &  495795  &  397759  & 535  & 261 & 517\\
    10 & 6372  & 3780  & 23846  & 31298  & 28814  &  35024  & ** [0.08\%] & 1252.6  & ** [0.02\%]  & 2206.9  & 310975 &  455369 & 1538372  & 787578  & 449  & 233 & 507\\
    11 & 15576 & 9240  & 58012  & 76228  & 70156  & 85336  & ** [0.12\%] & ** [0.03\%] & ** [0.03\%]  & 1036.6  & 180807 &  256996 &  570824  & 184432  & 806  & 353 & 836\\
    12 & 18408 & 10920 & 68630  & 90158  & 82982  & 100922  & ** [0.08\%] & ** [0.01\%] & 1739.3 & **[0.01\%]  & 108631 &  285196 &  276194  & 382696  & 1484 & 815 & 1763\\
    13 & 18408 & 10920 & 68630  & 90158  & 82982  & 100922  & ** [0.09\%] & 733.4 & 887.2 & 1205.6  & 84351 & 44990 &   92311  & 125922  & 1420 & 810 & 1603\\
    14 & 19470 & 11550 & 72312  & 95082  & 87492  & 106467  & ** [0.11\%] & ** [0.01\%] & ** [0.01\%] & 1187.2  &  62170 &  179443 &  412241  & 69854  & 2285 & 1570 & 2619\\
    15 & 19706 & 11690 & 73482  & 96528  & 88846  & 108051  & ** [0.13\%] & 539.7  & 180.6 & 438.4  &  50839 &   15365 &    6643  & 26615  & 1410 & 919 & 1717\\
    16 & 20296 & 12040 & 75640  & 99376  & 91464  & 111244  & ** [0.10\%] & 2790.1 & ** [0.01\%] & 1498.7  & 142997 &  191482 &  249999  & 157638  & 1690 & 762 & 1630\\
    17 & 21476 & 12740 & 80014  & 105130 & 96758  & 117688  & ** [0.10\%] & 444.3 & 180.4 & 244.3  & 70436  &   14938 & 7444  & 7710  & 1294 & 757 & 1753\\
    18 & 21476 & 12740 & 80026  & 105142 & 96770  & 117700  & ** [0.10\%] & 1686.3 & 239.2 & 225.3  & 41514 &  148584 & 24456  & 7093  & 1515 & 1067 & 1923\\
    19 & 21594 & 12810 & 80450  & 105704 & 97286  & 118331  & ** [0.09\%] & 1685.8 & 242.5 & 632.2  & 93893 &  136577 &   19084  & 51614  & 1703 & 1050 & 1977\\
    20 & 22066 & 13090 & 82264  & 108070 & 99468  & 120973  & ** [0.12\%] & 293.2 & 245.7 & 187.7  &  62970 &   19072 &   10870  & 4069  & 1299 & 970 & 1746\\
    \hline
  \end{tabular}
  \label{tab:ucresult-2}
  \vspace{-4mm}
\end{table}

This demonstrates that strong valid inequalities developed for a single-binary formulation can be used for a formulation with more than a single type of binary variables and can achieve comparable effectiveness.
Comparing the results presented in \tred{Tables~\ref{tab:ucresult-2} and \ref{tab:ucresult-2-2}} with the results presented in \tred{Tables~\ref{tab:ucresult-1} and \ref{tab:ucresult-1-2}},
we observe that formulations \mbox{F1-X}, \mbox{F2-X}, and \mbox{F2-Y} have similar performance.
This shows that a single-binary formulation has a similar performance as a formulation that uses more than a single type of binary variables when strong valid inequalities are added to these formulations.
It also shows that the strong valid inequalities obtained from our single-binary formulation have an effectiveness similar to those strong valid inequalities obtained from a formulation that uses more than a single type of binary variables.
However, formulation F2-Z outperforms formulations F1-X, F2-X, and F2-Y, as CPLEX can solve 19 instances to optimality with F2-Z, and the instance that is not solved optimally has a TGap of only 0.01\%.
This indicates that our strong valid inequalities derived from a single-binary formulation can be used in conjunction with those valid inequalities obtained from a formulation with more than a single type of binary variables to achieve better performance.

\vspace{-2mm}
\begin{table}[h]
  \renewcommand{\arraystretch}{1}
  \centering
  \caption{\tred{The Strength of LP Relaxations of MIP Formulations in the Second Experiment}}
  \setlength\tabcolsep{5pt}
  \fontsize{8}{12}\selectfont
  \begin{tabular}{|c|c|*{10}{c}|}
    \hline
    \multirow{10}{*}{IGap} & Instance & 1 & 2 & 3 & 4 & 5 & 6 & 7 & 8 & 9 & 10\\
    \cline{2-12}
    & F2 & 0.45\% & 0.39\% & 0.38\% & 0.35\% & 0.47\% & 0.60\% & 0.34\% & 0.52\% & 0.48\% & 0.63\% \\
    & F2-X & 0.20\% & 0.14\% & 0.07\% & 0.06\% & 0.05\% & 0.04\% & 0.07\% & 0.06\% & 0.06\% & 0.04\% \\
    & F2-Y & 0.20\% & 0.14\% & 0.08\% & 0.05\% & 0.05\% & 0.05\% & 0.08\% & 0.06\% & 0.05\% & 0.05\% \\
    & F2-Z & 0.20\% & 0.14\% & 0.07\% & 0.05\% & 0.05\% & 0.04\% & 0.07\% & 0.06\% & 0.05\% & 0.04\%\\
    \cline{2-12}
    & Instance & 11 & 12 & 13 & 14 & 15 & 16 & 17 & 18 & 19 & 20\\
    \cline{2-12}
    & F2 & 0.32\% & 0.31\% & 0.36\% & 0.34\% & 0.47\% & 0.29\% & 0.41\% & 0.40\% & 0.34\% & 0.41\% \\
    & F2-X & 0.05\% & 0.02\% & 0.02\% & 0.02\% & 0.02\% & 0.02\% & 0.02\% & 0.02\% & 0.02\% & 0.02\% \\
    & F2-Y & 0.06\% & 0.02\% & 0.02\% & 0.02\% & 0.02\% & 0.02\% & 0.02\% & 0.01\% & 0.02\% & 0.01\%\\
    & F2-Z & 0.05\% & 0.02\% & 0.02\% & 0.02\% & 0.02\% & 0.02\% & 0.02\% & 0.01\% & 0.02\% & 0.01\%\\
    \hline
    \multirow{8}{*}{Pct. reduction} & Instance & 1 & 2 & 3 & 4 & 5 & 6 & 7 & 8 & 9 & 10\\
    \cline{2-12}
    & F2-X & 55.7\% & 64.0\% & 81.2\% & 82.2\% & 89.3\% & 93.2\% & 79.4\% & 88.8\% & 88.3\% & 93.0\% \\
    & F2-Y & 55.7\% & 63.9\% & 80.3\% & 84.3\% & 90.3\% & 92.0\% & 77.5\% & 88.1\% & 89.0\% & 92.4\%\\
    & F2-Z & 55.7\% & 64.0\% & 81.2\% & 84.3\% & 90.3\% & 93.5\% & 79.4\% & 89.0\% & 89.0\% & 93.0\%\\
    \cline{2-12}
    & Instance & 11 & 12 & 13 & 14 & 15 & 16 & 17 & 18 & 19 & 20\\
    \cline{2-12}
    & F2-X & 84.0\% & 93.7\% & 94.7\% & 94.6\% & 96.1\% & 94.5\% & 95.6\% & 95.7\% & 93.2\% & 95.7\%\\
    & F2-Y & 82.5\% & 92.1\% & 94.5\% & 94.1\% & 96.4\% & 92.0\% & 96.2\% & 96.7\% & 94.4\% & 96.3\%\\
    & F2-Z & 84.0\% & 93.9\% & 94.8\% & 94.7\% & 96.5\% & 94.5\% & 96.2\% & 96.7\% & 94.4\% & 96.4\%\\
    \hline
  \end{tabular}
  \label{tab:ucresult-2-2}
  \vspace{-4mm}
\end{table}

\tred{Tables~\ref{tab:ucresult-3} and \ref{tab:ucresult-3-2} present} the computational results of the third experiment,
in which formulations \rred{\F1^+} and \mbox{\F1^+-X} are used \fred{to solve} the network-constrained UC problem based on the modified IEEE 118-bus system.
The integrality gaps generated by formulation \rred{\mbox{\F1^+-X}} are 44.1\% to 66.3\% smaller than those generated by formulation \rred{\F1^+}.
Using formulation \F1^+, CPLEX is able to solve only one of the 20 instances to optimality within one hour.
In contrast, using formulation \mbox{\F1^+-X}, CPLEX is able to solve 17 instances to optimality within the same time limit.
Formulation \mbox{\F1^+-X} explores fewer nodes than formulation \F1^+, and
the number of user cuts added by \mbox{\F1^+-X} in the solution process is quite small.
These results demonstrate the effectiveness of the strong formulation \mbox{\F1^+-X}.
Some additional computational results on formulations \mbox{\F1^+} and \mbox{\F1^+-X} using a more congested demand setting and a less congested demand setting are presented in Online Appendix B, which show that formulation \mbox{\F1^+-X} is more effective when the demand is more congested.

\begin{table}[htbp]
  \renewcommand{\arraystretch}{1}
  \centering
  \caption{\tred{Performance of MIP Formulations in the Third Experiment}}
  \fontsize{8}{12}\selectfont
  \setlength\tabcolsep{5pt}
  \begin{tabular}{|*{10}{c|}}
    \hline
    \multirow{2}{*}{Instance} & \multirow{2}{*}{\# var} & \multirow{2}{*}{\# bin var} & \multicolumn{2}{c|}{\# cstr} & \multicolumn{2}{c|}{CPU time [TGap]} 
    & \multicolumn{2}{c|}{\# nodes} & \# user cuts \\
    \cline{4-10}
    &&& \F1^+ & \F1^+-X & \F1^+ & \F1^+-X & \F1^+ & \F1^+-X & \F1^+-X \\
    \hline
     1 & \multirow{20}{*}{6372} & \multirow{20}{*}{1296} & \multirow{20}{*}{36124} & \multirow{20}{*}{43576} & ** [0.15\%] &       571.7 &  109080 &   14069 &     316 \\
     2 &&&&& ** [0.06\%] &      2324.7 &   97424 &   84034 &     474 \\
     3 &&&&& ** [0.02\%] &      2251.6 &  158482 &   31631 &     445 \\
     4 &&&&& ** [0.11\%] &      2803.9 &  131942 &   43108 &     544 \\
     5 &&&&& ** [0.01\%] &       288.5 &  175810 &    7220 &     394 \\
     6 &&&&& ** [0.17\%] &      1653.5 &  116598 &   24813 &     441 \\
     7 &&&&& ** [0.27\%] &      1800.0 &  122958 &   43624 &     545 \\
     8 &&&&& ** [0.11\%] &      1799.3 &  180499 &   22310 &     591 \\
     9 &&&&& ** [0.13\%] &       851.4 &  250053 &   19961 &     464 \\
    10 &&&&&      2938.4 &       850.7 &  181549 &   18889 &     449 \\
    11 &&&&& ** [0.08\%] &      2425.3 &  196002 &   59158 &     463 \\
    12 &&&&& ** [0.15\%] & ** [0.05\%] &  287079 &   44675 &     619 \\
    13 &&&&& ** [0.14\%] &      3351.3 &  164549 &   82738 &     563 \\
    14 &&&&& ** [0.17\%] &      1969.8 &  235344 &   37421 &     486 \\
    15 &&&&& ** [0.30\%] & ** [0.03\%] &  266093 &  111169 &     696 \\
    16 &&&&& ** [0.10\%] &      1666.2 &  112045 &  42146  &     438 \\
    17 &&&&& ** [0.13\%] &       461.9 &  139242 &  13578  &     448 \\
    18 &&&&& ** [0.14\%] &      2433.3 &  139375 &  61670  &     531 \\
    19 &&&&& ** [0.15\%] & ** [0.02\%] &  169572 &  133348 &     742 \\
    20 &&&&& ** [0.21\%] &      3298.6 &  158536 &  50471  &     615 \\
    \hline
  \end{tabular}
  \label{tab:ucresult-3}
  \vspace{-4mm}
\end{table}

\begin{table}[htbp]
  \renewcommand{\arraystretch}{1}
  \centering
  \caption{\tred{The Strength of LP Relaxations of MIP Formulations in the Third Experiment}}
  \setlength\tabcolsep{5pt}
  \fontsize{8}{12}\selectfont
  \begin{tabular}{|c|c|*{10}{c}|}
    \hline
    \multirow{6}{*}{IGap} & Instance & 1 & 2 & 3 & 4 & 5 & 6 & 7 & 8 & 9 & 10\\
    \cline{2-12}
     & \F1^+ & 0.85\% & 1.03\% & 0.76\% & 0.93\% & 0.86\% & 0.88\% & 1.04\% & 0.92\% & 0.91\% & 0.80\% \\
     & \F1^+-X & 0.36\% & 0.39\% & 0.39\% & 0.52\% & 0.33\% & 0.36\% & 0.35\% & 0.42\% & 0.38\% & 0.40\% \\
    \cline{2-12}
     & Instance & 11 & 12 & 13 & 14 & 15 & 16 & 17 & 18 & 19 & 20\\
    \cline{2-12}
     & \F1^+ & 0.90\% & 1.17\% & 0.87\% & 0.91\% & 1.02\% & 0.89\% & 0.84\% & 0.85\% & 0.95\% & 0.89\% \\
     & \F1^+-X & 0.40\% & 0.47\% & 0.37\% & 0.43\% & 0.44\% & 0.40\% & 0.40\% & 0.44\% & 0.39\% & 0.40\% \\
    \hline
    \multirow{4}{*}{Pct. reduction} & Instance & 1 & 2 & 3 & 4 & 5 & 6 & 7 & 8 & 9 & 10\\
    \cline{2-12}
    & \F1^+-X & 57.4\% & 62.2\% & 49.0\% & 44.1\% & 61.9\% & 59.7\% & 66.3\% & 54.8\% & 58.8\% & 52.2\% \\
    \cline{2-12}
    & Instance & 11 & 12 & 13 & 14 & 15 & 16 & 17 & 18 & 19 & 20\\
    \cline{2-12}
    & \F1^+-X & 55.9\% & 60.0\% & 57.5\% & 52.2\% & 57.4\% & 55.6\% & 53.0\% & 48.4\% & 59.0\% & 55.6\%\\
    \hline
  \end{tabular}
  \label{tab:ucresult-3-2}
  \vspace{-4mm}
\end{table}

\section{Conclusions}\label{Conclusion}

This paper considers a UC formulation with a single type of binary variables.
By analyzing the physical constraints of a single generator, we obtain the convex hull description of the two-period single-UC polytope, which can be \fred{used} to tighten the original MILP formulation and derive other strong valid inequalities.
For the multi-period single-UC polytope, we derive strong valid inequalities with one and two continuous variables.
Conditions under which these valid inequalities are facet-defining for the multi-period single-UC polytope are provided.
Because the number of inequalities in each valid inequality family is very large, efficient separation algorithms are provided to identify the most violated inequalities.
The effectiveness of the proposed strong valid inequalities is demonstrated in solving network-constrained UC problems.
Computational results show that our valid inequalities can speed up the solution process significantly.
Moreover, these strong valid inequalities exhibit effectiveness comparable to two-binary valid inequalities and thus can be used to tighten two/three-binary formulations.

\fred{Various} intriguing research directions can be pursued following this \fred{line of work}.
First, it would be interesting to investigate the complete convex hull descriptions of the multi-period single-UC polytopes, such as the three-period polytope, and derive strong valid inequalities with more than two continuous variables to further tighten Problem~1.
\fred{In addition}, the discussion \fred{of} the single-UC polytope can be extended to different parameter settings, such as the case where $\Vupper \geq \Clower+V$ or the case where $\Vupper+V>\Cupper$.
Second, it would be appealing to incorporate different start-up/shut-down trajectories of generators into the physical constraints to accurately represent the operation of units and \fred{to} conduct a polyhedral study on the obtained single-UC polytope to derive strong valid inequalities.
Third, considering the demand and electricity price fluctuations that often occur in practice when dealing with UC problems, it would be interesting to formulate the corresponding stochastic UC problems to better reflect real-world scenarios.
Fourth, given that different types of electrical generators (e.g., pumped storage hydro units) may have different physical constraints in addition to those considered in this paper, it would be interesting to derive strong valid inequalities for the UC problems with these specific generators.
\rred{For example, consider the hydro UC problem described in \cite{cheng2016hydro} in which there are constraints that impose restrictions on the power output in the safe operating zones. By employing a similar approach to the derivation of inequality \eqref{multi:eqn:two_var-k-ub-3}, we can obtain a valid inequality by deriving an upper bound for ``$p_{i,t}-p_{i,t-k}$," where $p_{i,t}$ and $p_{i,t-k}$ are the power output of unit $i$ in periods $t$ and $t-k$, respectively.}
Fifth, as mentioned in Section~\ref{sec:introduction}, some studies have utilized LR-based methods to tackle complex UC problems by decomposing the multiple-generator problem into independent single-generator subproblems. 
The strong valid inequalities derived in this paper can be applied to these subproblems, thus enhancing their effectiveness. 
It would be useful to develop mathematical techniques that can integrate them into the decomposition procedure to improve its efficiency.
Sixth, it would be intriguing to make a theoretical comparison between our strong valid inequalities and those derived in prior studies, such as \cite{pan2016polyhedral} and \cite{damci2016polyhedral}, to reveal the relationships and distinctions between them.
We leave these issues for future research.


{\bibliographystyle{apalike}
\SingleSpacedXI
\setlength\bibsep{5pt}
\bibliography{duc}}

\newpage
\renewcommand{\thepage}{ec.\arabic{page}}
\setcounter{page}{1}
\renewcommand{\theequation}{EC.\arabic{equation}}
\setcounter{equation}{0}
\renewcommand{\thetable}{EC.\arabic{table}}
\setcounter{table}{0}
\noindent
\textbf{\footnotesize Online Supplement for ``A Polyhedral Study on Unit Commitment with a Single Type of Binary Variables''}

\begin{APPENDICES}{}

\SingleSpacedXI
\small

\section{Supplement to Section~\ref{sec:model}}\label{apx:A}


\subsection{Proof of Lemma \ref{lem:lookbackward}}\label{apx:lem:lookbackward}

\noindent{\bf Lemma \ref{lem:lookbackward}.} {\it
Consider any point $(\BFx,\BFy)\in\Pe$ and $t\in[2,T]_{\Z}$.
(i)~If $y_t=0$, then $y_{t-j}-y_{t-j-1}\le 0$ for all $j\in[0,\min\{t-2,L-1\}]_{\Z}$.
(ii)~If $y_t=1$, then there exists at most one $j\in[0,\min\{t-2,L\}]_{\Z}$ such that $y_{t-j}-y_{t-j-1}=1$.}
\vskip8pt

\noindent{\bf Proof.}
(i)~Consider any $(\BFx,\BFy)\in\Pe$ and any $t\in[2,T]_{\Z}$ such that $y_t=0$.
Suppose, to the contrary, that there exists $j\in[0,\min\{t-2,L-1\}]_{\Z}$ such that $y_{t-j}-y_{t-j-1}=1$.
Then, $t-j\in[2,T]_{\Z}$.
Thus, by \eqref{eqn:p-minup}, $y_k=1$ for all $k\in[t-j,\min\{T,t-j+L-1\}]_{\Z}$.
It is easy to check that $t\in[t-j,\min\{T,t-j+L-1\}]_{\Z}$.
Hence, $y_t=1$, which is a contradiction.
Therefore, $y_{t-j}-y_{t-j-1}\le0$ for all $j\in[0,\min\{t-2,L-1\}]_{\Z}$.

(ii)~Consider any $(\BFx,\BFy)\in\Pe$ and any $t\in[2,T]_{\Z}$ such that $y_t=1$.
Suppose, to the contrary, that there exist $j_1,j_2\in[0,\rred{\min}\{t-2,L\}]_{\Z}$ such that $j_1<j_2$ and $y_{t-j_1}-y_{t-j_1-1}=y_{t-j_2}-y_{t-j_2-1}=1$.
Because $t-j_2\in[2,T]_{\Z}$, and $y_{t-j_2}-y_{t-j_2-1}=1$, by \eqref{eqn:p-minup}, $y_k=1$ for all $k\in[t-j_2,\min\{T,t-j_2+L-1\}]_{\Z}$.
Note that $t-j_1-1\ge t-j_2$, $t-j_1-1\le T$, and $t-j_1-1\le t-1\le t-j_2+L-1$.
Thus, $t-j_1-1\in[t-j_2,\min\{T,t-j_2+L-1\}]_{\Z}$.
Hence, $y_{t-j_1-1}=1$, which contradicts that $y_{t-j_1}-y_{t-j_1-1}=1$.
Therefore, there exists at most one $j\in[0,\min\{t-2,L\}]_{\Z}$ such that $y_{t-j}-y_{t-j-1}=1$.
\Halmos


\subsection{Proof of Lemma \ref{lem:Pprime}}\label{apx:lem:Pprime}

\noindent{\bf Lemma \ref{lem:Pprime}.} {\it
Denote $\BFx'=(x'_1,\ldots,x'_T)$ and $\BFy'=(y'_1,\ldots,y'_T)$.
Let
\begin{subeqnarray}
&& \Pe'=\big\{(\BFx',\BFy')\in\R_+^{T}\times\B^{T}:\nonumber\\
&& \qquad\qquad -y'_{T-t+2}+y'_{T-t+1}-y'_{T-k+1}\le0,\,\forall t\in[2,T]_{\Z},\,\forall k\in[t,\min\{T,t+L-1\}]_{\Z},\slabel{eqn:pprime-minup}\\
&& \qquad\qquad y'_{T-t+2}-y'_{T-t+1}+y'_{T-k+1}\le1,\,\forall t\in[2,T]_{\Z},\,\forall k\in[t,\min\{T,t+\ell-1\}]_{\Z},\slabel{eqn:pprime-mindn}\\
&& \qquad\qquad -x'_{T-t+1}+\Clower y'_{T-t+1}\le0,\,\forall t\in[1,T]_{\Z},\slabel{eqn:pprime-lower-bound}\\
&& \qquad\qquad x'_{T-t+1}-\Cupper y'_{T-t+1}\le0,\,\forall t\in[1,T]_{\Z},\slabel{eqn:pprime-upper-bound}\\
&& \qquad\qquad x'_{T-t+1}-x'_{T-t+2}\le Vy'_{T-t+2}+\Vupper(1-y'_{T-t+2}),\,\forall t\in[2,T]_{\Z},\slabel{eqn:pprime-ramp-up}\\
&& \qquad\qquad x'_{T-t+2}-x'_{T-t+1}\le Vy'_{T-t+1}+\Vupper(1-y'_{T-t+1}),\,\forall t\in[2,T]_{\Z}\slabel{eqn:pprime-ramp-down}\big\}.
\end{subeqnarray}
Then, $\Pe=\Pe'$.}
\vskip8pt

\noindent{\bf Proof.}
First, consider any element $(\BFx,\BFy)$ of $\Pe$, and we show that $(\BFx,\BFy)\in\Pe'$.

Consider inequality \eqref{eqn:pprime-minup} and any $t\in[2,T]_{\Z}$.
Obviously, $(\BFx,\BFy)$ satisfies \eqref{eqn:pprime-minup} if $-y_{T-t+2}+y_{T-t+1}\le 0$.
Consider the case where $-y_{T-t+2}+y_{T-t+1}>0$ (i.e., $y_{T-t+2}=0$ and $y_{T-t+1}=1$).
Suppose, to the contrary, that $y_{T-k+1}=0$ for some $k\in[t,\min\{T,t+L-1\}]_{\Z}$.
Then, because $y_{T-k+1}=0$ and $y_{T-t+1}=1$, there exists $p\in[t,k-1]_\Z$ such that $y_{T-p}=0$ and $y_{T-p+1}=1$.
This implies that $-y_{T-p}+y_{T-p+1}-y_{T-t+2}=1$.
Note that $T-p+1\in[2,T]_\Z$ and $T-t+2\in[T-p+1,\min\{T,(T-p+1)+L-1\}]_{\Z}$.
Thus, $(\BFx,\BFy)$ violates inequality \eqref{eqn:p-minup}, which contradicts that $(\BFx,\BFy)\in\Pe$.
Hence, $y_{T-k+1}=1$ for all $k\in[t,\min\{T,t+L-1\}]_{\Z}$.
Thus, $(\BFx,\BFy)$ satisfies inequality \eqref{eqn:pprime-minup}.

Consider inequality \eqref{eqn:pprime-mindn} and any $t\in[2,T]_{\Z}$.
Obviously, $(\BFx,\BFy)$ satisfies \eqref{eqn:pprime-mindn} if $y_{T-t+2}-y_{T-t+1}\le 0$.
Consider the case where $y_{T-t+2}-y_{T-t+1}>0$ (i.e., $y_{T-t+2}=1$ and $y_{T-t+1}=0$).
Suppose, to the contrary, that $y_{T-k+1}=1$ for some $k\in[t,\min\{T,t+\ell-1\}]_{\Z}$.
Then, because $y_{T-k+1}=1$ and $y_{T-t+1}=0$, there exists $p\in[t,k-1]_\Z$ such that $y_{T-p}=1$ and $y_{T-p+1}=0$.
This implies that $y_{T-p}-y_{T-p+1}+y_{T-t+2}=2$.
Note that $T-p+1\in[2,T]_\Z$ and $T-t+2\in[T-p+1,\min\{T,(T-p+1)+\ell-1\}]_{\Z}$.
Thus, $(\BFx,\BFy)$ violates inequality \eqref{eqn:p-mindn}, which contradicts that $(\BFx,\BFy)\in\Pe$.
Hence, $y_{T-k+1}=0$ for all $k\in[t,\min\{T,t+\ell-1\}]_{\Z}$.
Thus, $(\BFx,\BFy)$ satisfies inequality \eqref{eqn:pprime-mindn}.

Consider inequalities \eqref{eqn:pprime-lower-bound} and \eqref{eqn:pprime-upper-bound}.
For any $t\in[1,T]_{\Z}$, 
because $(\BFx,\BFy)$ satisfies inequalities \eqref{eqn:p-lower-bound} and \eqref{eqn:p-upper-bound},
$(\BFx,\BFy)$ also satisfies inequalities \eqref{eqn:pprime-lower-bound} and \eqref{eqn:pprime-upper-bound}.

Consider inequality \eqref{eqn:pprime-ramp-up} and any $t\in[2,T]_{\Z}$.
Because $T-t+2\in[2,T]_{\Z}$, by \eqref{eqn:p-ramp-down}, $x_{T-t+1}-x_{T-t+2}\le Vy_{T-t+2}+\Vupper(1-y_{T-t+2})$.
Thus, $(\BFx,\BFy)$ satisfies inequality \eqref{eqn:pprime-ramp-up}.

Consider inequality \eqref{eqn:pprime-ramp-down} and any $t\in[2,T]_{\Z}$.
Because $T-t+2\in[2,T]_{\Z}$, by \eqref{eqn:p-ramp-up}, $x_{T-t+2}-x_{T-t+1}\le Vy_{T-t+1}+\Vupper(1-y_{T-t+1})$.
Thus, $(\BFx,\BFy)$ satisfies inequality \eqref{eqn:pprime-ramp-down}.

Summarizing the above analysis, we conclude that $(\BFx,\BFy)\in\Pe'$.

Next, consider any element $(\BFx',\BFy')$ of $\Pe'$, and we show that $(\BFx',\BFy')\in\Pe$.

Consider inequality \eqref{eqn:p-minup} and any $t\in[2,T]_{\Z}$.
Obviously, $(\BFx',\BFy')$ satisfies \eqref{eqn:p-minup} if $-y'_{t-1}+y'_t\le 0$.
Consider the case where $-y'_{t-1}+y'_t>0$ (i.e., $y'_{t-1}=0$ and $y'_t=1$).
Suppose, to the contrary, that $y'_k=0$ for some $k\in[t,\min\{T,t+L-1\}]_{\Z}$.
Then, because $y'_t=1$ and $y'_k=0$, there exists $p\in[t,k-1]_\Z$ such that $y'_p=1$ and $y'_{p+1}=0$.
This implies that $-y'_{T-(T-p+1)+2}+y'_{T-(T-p+1)+1}-y'_{T-(T-t+2)+1}=1$.
Note that $T-p+1\in[2,T]_\Z$ and $T-t+2\in[T-p+1,\min\{T,(T-p+1)+L-1\}]_{\Z}$.
Thus, $(\BFx',\BFy')$ violates inequality \eqref{eqn:pprime-minup}, which contradicts that $(\BFx',\BFy')\in\Pe'$.
Hence, $y'_k=1$ for all $k\in[t,\min\{T,t+L-1\}]_{\Z}$.
Thus, $(\BFx',\BFy')$ satisfies inequality \eqref{eqn:p-minup}.

Consider inequality \eqref{eqn:p-mindn} and any $t\in[2,T]_{\Z}$.
Obviously, $(\BFx',\BFy')$ satisfies \eqref{eqn:p-mindn} if $y'_{t-1}-y'_t\le 0$.
Consider the case where $y'_{t-1}-y'_t>0$ (i.e., $y'_{t-1}=1$ and $y'_t=0$).
Suppose, to the contrary, that $y'_k=1$ for some $k\in[t,\min\{T,t+\ell-1\}]_{\Z}$.
Then, because $y'_t=0$ and $y'_k=1$, there exists $p\in[t,k-1]_\Z$ such that $y'_p=0$ and $y'_{p+1}=1$.
This implies that $y'_{T-(T-p+1)+2}-y'_{T-(T-p+1)+1}+y'_{T-(T-t+2)+1}=2$.
Note that $T-p+1\in[2,T]_\Z$ and $T-t+2\in[T-p+1,\min\{T,(T-p+1)+\ell-1\}]_{\Z}$.
Thus, $(\BFx',\BFy')$ violates inequality \eqref{eqn:pprime-mindn}, which contradicts that $(\BFx',\BFy')\in\Pe'$.
Hence, $y'_k=0$ for all $k\in[t,\min\{T,t+\ell-1\}]_{\Z}$.
Thus, $(\BFx',\BFy')$ satisfies inequality \eqref{eqn:p-mindn}.

Consider inequalities \eqref{eqn:p-lower-bound} and \eqref{eqn:p-upper-bound}.
For any $t\in[1,T]_{\Z}$, 
because $(\BFx',\BFy')$ satisfies inequalities \eqref{eqn:pprime-lower-bound} and \eqref{eqn:pprime-upper-bound},
$(\BFx',\BFy')$ also satisfies inequalities \eqref{eqn:p-lower-bound} and \eqref{eqn:p-upper-bound}.

Consider inequality \eqref{eqn:p-ramp-up} and any $t\in[2,T]_{\Z}$.
Because $T-t+2\in[2,T]_{\Z}$, by \eqref{eqn:pprime-ramp-down}, 
$x'_{T-(T-t+2)+2}-x'_{T-(T-t+2)+1}\le Vy'_{T-(T-t+2)+1}+\Vupper(1-y'_{T-(T-t+2)+1})$.
Hence, $x'_t-x'_{t-1}\le Vy'_{t-1}+\Vupper(1-y'_{t-1})$.
Thus, $(\BFx',\BFy')$ satisfies inequality \eqref{eqn:p-ramp-up}.

Consider inequality \eqref{eqn:p-ramp-down} and any $t\in[2,T]_{\Z}$.
Because $T-t+2\in[2,T]_{\Z}$, by \eqref{eqn:pprime-ramp-up}, 
$x'_{T-(T-t+2)+1}-x'_{T-(T-t+2)+2}\le Vy'_{T-(T-t+2)+2}+\Vupper(1-y'_{T-(T-t+2)+2})$.
Hence, $x'_{t-1}-x'_t\le Vy'_t+\Vupper(1-y'_t)$.
Thus, $(\BFx',\BFy')$ satisfies inequality \eqref{eqn:p-ramp-down}.

Summarizing the above analysis, we conclude that $(\BFx',\BFy')\in\Pe$.
Therefore, $\Pe=\Pe'$.
\Halmos

\section{Supplement to Section~\ref{sec:two-period}}
\subsection{Proof of Theorem \ref{the:convPtwo}}\label{apx:the:convPtwo}

\noindent{\bf Theorem \ref{the:convPtwo}.} {\it
Denote
\begin{subequations}
\begin{align}
& \Qe_2 = \big\{(x_{t-1},x_t,y_{t-1},y_t)\in\R^4: \nonumber \\
& \qquad\qquad y_i\le 1, \ \forall i\in\{t-1,t\},                          \tag{\ref{eqn-q2:y-t}} \\        
& \qquad\qquad \Clower y_i\le x_i\le \Cupper y_i, \ \forall i\in\{t-1,t\}, \tag{\ref{eqn-q2:capacity}} \\   
& \qquad\qquad x_{t-1}\le\Vupper y_{t-1}+(\Cupper-\Vupper)y_t,             \tag{\ref{eqn-q2:x1-ub}} \\      
& \qquad\qquad x_t\le(\Cupper-\Vupper)y_{t-1}+\Vupper y_t,                 \tag{\ref{eqn-q2:x2-ub}} \\      
& \qquad\qquad x_t-x_{t-1}\le(\Clower+V)y_t-\Clower y_{t-1},               \tag{\ref{eqn-q2:x2-x1-ub-1}} \\ 
& \qquad\qquad x_t-x_{t-1}\le\Vupper y_t-(\Vupper-V)y_{t-1},               \tag{\ref{eqn-q2:x2-x1-ub-2}} \\ 
& \qquad\qquad x_{t-1}-x_t\le(\Clower+V)y_{t-1}-\Clower y_t,               \tag{\ref{eqn-q2:x1-x2-ub-1}} \\ 
& \qquad\qquad x_{t-1}-x_t\le\Vupper y_{t-1}-(\Vupper-V) y_t\big\}.        \tag{\ref{eqn-q2:x1-x2-ub-2}}    
\end{align}
\end{subequations}
Then, $\Qe_2=\convPtwo$.}
\vskip8pt

\noindent{\bf Proof.}
We divide the proof into two parts. For the sake of simplicity, we let $t=2$ in $\Pe_2$ and $Q_2$.
\vskip8pt

\noindent Part I: $\convPtwo \subseteq \Qe_2$.

We prove that the linear inequalities \eqref{eqn-q2:y-t}--\eqref{eqn-q2:x1-x2-ub-2} are valid for $\convPtwo$. 
To do so, it suffices to show that \eqref{eqn-q2:y-t}--\eqref{eqn-q2:x1-x2-ub-2} are valid for $\Pe_2$. 
Clearly, inequalities \eqref{eqn-q2:y-t} and \eqref{eqn-q2:capacity} hold for any element of $\Pe_2$. 
In the following, we show that \eqref{eqn-q2:x1-ub}--\eqref{eqn-q2:x1-x2-ub-2} hold for any element of $\Pe_2$.

For inequality \eqref{eqn-q2:x1-ub}, consider any element $(x_1,x_2,y_1,y_2)$ of $\Pe_2$. 
We consider three different cases. 
Case~1: $y_1=0$.
In this case, by \eqref{eqn:P2-b}, $x_1=0$.
Thus, $(x_1,x_2,y_1,y_2)$ satisfies inequality \eqref{eqn-q2:x1-ub}.
Case~2: $y_1=1$ and $y_2=0$.
In this case, by \eqref{eqn:P2-b}, $x_2=0$.
Then, by \eqref{eqn:P2-d}, $x_1\le\Vupper$.
Thus, $(x_1,x_2,y_1,y_2)$ satisfies inequality \eqref{eqn-q2:x1-ub}.
Case~3: $y_1=1$ and $y_2=1$.
In this case, by \eqref{eqn:P2-b}, $x_1\le\Cupper$.
Thus, $(x_1,x_2,y_1,y_2)$ satisfies inequality \eqref{eqn-q2:x1-ub}.

For inequality \eqref{eqn-q2:x2-ub}, consider any element $(x_1,x_2,y_1,y_2)$ of $\Pe_2$.
We consider three different cases.
Case~1: $y_2=0$.
In this case, by \eqref{eqn:P2-b}, $x_2=0$.
Thus, $(x_1,x_2,y_1,y_2)$ satisfies inequality \eqref{eqn-q2:x2-ub}.
Case~2: $y_2=1$ and $y_1=0$.
In this case, by \eqref{eqn:P2-b}, $x_1=0$.
Then, by \eqref{eqn:P2-c}, $x_2\le\Vupper$.
Thus, $(x_1,x_2,y_1,y_2)$ satisfies inequality \eqref{eqn-q2:x2-ub}.
Case~3: $y_2=1$ and $y_1=1$.
In this case, by \eqref{eqn:P2-b}, $x_2\le\Cupper$.
Thus, $(x_1,x_2,y_1,y_2)$ satisfies inequality \eqref{eqn-q2:x2-ub}.

For inequality \eqref{eqn-q2:x2-x1-ub-1}, consider any element $(x_1,x_2,y_1,y_2)$ of $\Pe_2$.
We consider four different cases.
Case~1: $y_1 = y_2 = 0$.
In this case, by \eqref{eqn:P2-b}, $x_1 = x_2 = 0$.
Thus, $(x_1,x_2,y_1,y_2)$ satisfies inequality \eqref{eqn-q2:x2-x1-ub-1}.
Case~2: $y_1 = y_2 = 1$.
In this case, by \eqref{eqn:P2-c}, $x_2-x_1 \le V$.
Thus, $(x_1,x_2,y_1,y_2)$ satisfies inequality \eqref{eqn-q2:x2-x1-ub-1}.
Case~3: $y_1 = 1$ and $y_2 = 0$.
In this case, by \eqref{eqn:P2-b}, $x_2 = 0$.
By \eqref{eqn:P2-a}, $\Clower \le x_1$.
Thus, $(x_1,x_2,y_1,y_2)$ satisfies inequality \eqref{eqn-q2:x2-x1-ub-1}.
Case~4: $y_1 = 0$ and $y_2 = 1$.
In this case, by \eqref{eqn:P2-b}, $x_1 = 0$.
By \eqref{eqn:P2-c}, $x_2\le\Vupper<\Clower+V$.
Thus, $(x_1,x_2,y_1,y_2)$ satisfies inequality \eqref{eqn-q2:x2-x1-ub-1}.

For inequality \eqref{eqn-q2:x2-x1-ub-2}, consider any element $(x_1,x_2,y_1,y_2)$ of $\Pe_2$.
We consider four different cases.
Case~1: $y_1=y_2=0$.
In this case, by \eqref{eqn:P2-b}, $x_1 = x_2 = 0$.
Thus, $(x_1,x_2,y_1,y_2)$ satisfies inequality \eqref{eqn-q2:x2-x1-ub-2}.
Case~2: $y_1=y_2=1$.
In this case, by \eqref{eqn:P2-c}, $x_2-x_1 \le V$.
Thus, $(x_1,x_2,y_1,y_2)$ satisfies inequality \eqref{eqn-q2:x2-x1-ub-2}.
Case~3: $y_1 = 1$ and $y_2 = 0$.
In this case, by \eqref{eqn:P2-b}, $x_2 = 0$.
By \eqref{eqn:P2-a}, $\Clower \le x_1$, which implies that
$-x_1 < V - \Vupper$ (because $\Vupper<\Clower+V$).
Thus, $(x_1,x_2,y_1,y_2)$ satisfies inequality \eqref{eqn-q2:x2-x1-ub-2}.
Case~4: $y_1=0$ and $y_2=1$.
In this case, by \eqref{eqn:P2-b}, $x_1=0$.
Then, by \eqref{eqn:P2-c}, $x_2 \le \Vupper$.
Thus, $(x_1,x_2,y_1,y_2)$ satisfies inequality \eqref{eqn-q2:x2-x1-ub-2}.

For inequality \eqref{eqn-q2:x1-x2-ub-1}, consider any element $(x_1,x_2,y_1,y_2)$ of $\Pe_2$.
We consider four different cases.
Case~1: $y_1=y_2=0$.
In this case, by \eqref{eqn:P2-b}, $x_1=x_2=0$.
Thus, $(x_1,x_2,y_1,y_2)$ satisfies inequality \eqref{eqn-q2:x1-x2-ub-1}.
Case~2: $y_1=y_2=1$.
In this case, by \eqref{eqn:P2-d}, $x_1-x_2 \le V$.
Thus, $(x_1,x_2,y_1,y_2)$ satisfies inequality \eqref{eqn-q2:x1-x2-ub-1}.
Case~3: $y_1=1$ and $y_2=0$.
In this case, by \eqref{eqn:P2-b}, $x_2=0$.
Then, by \eqref{eqn:P2-d}, $x_1\le\Vupper<\Clower+V$.
Thus, $(x_1,x_2,y_1,y_2)$ satisfies inequality \eqref{eqn-q2:x1-x2-ub-1}.
Case~4: $y_1=0$ and $y_2=1$.
In this case, by \eqref{eqn:P2-b}, $x_1=0$.
By \eqref{eqn:P2-a}, $\Clower \le x_2$.
Thus, $(x_1,x_2,y_1,y_2)$ satisfies inequality \eqref{eqn-q2:x1-x2-ub-1}.

For inequality \eqref{eqn-q2:x1-x2-ub-2}, consider any element $(x_1,x_2,y_1,y_2)$ of $\Pe_2$.
We consider four different cases.
Case~1: $y_1=y_2=0$.
In this case, by \eqref{eqn:P2-b}, $x_1=x_2=0$.
Thus, $(x_1,x_2,y_1,y_2)$ satisfies inequality \eqref{eqn-q2:x1-x2-ub-2}.
Case~2: $y_1=y_2=1$.
In this case, by \eqref{eqn:P2-d}, $x_1-x_2 \le V$.
Thus, $(x_1,x_2,y_1,y_2)$ satisfies inequality \eqref{eqn-q2:x1-x2-ub-2}.
Case~3: $y_1=1$ and $y_2=0$.
In this case, by \eqref{eqn:P2-b}, $x_2=0$.
Then, by \eqref{eqn:P2-d}, $x_1\le \Vupper$.
Thus, $(x_1,x_2,y_1,y_2)$ satisfies inequality \eqref{eqn-q2:x1-x2-ub-2}.
Case~4: $y_1=0$ and $y_2=1$.
In this case, by \eqref{eqn:P2-b}, $x_1=0$.
By \eqref{eqn:P2-a}, $\Clower \le x_2$, which implies that
$-x_2 < V - \Vupper$ (because $\Vupper<\Clower+V$).
Thus, $(x_1,x_2,y_1,y_2)$ satisfies inequality \eqref{eqn-q2:x1-x2-ub-2}.
\vskip10pt

\noindent Part II: $\Qe_2 \subseteq \convPtwo$.

Consider any given $(\bar{x}_1,\bar{x}_2,\bar{y}_1,\bar{y}_2)\in\Qe_2$.
We have
\begin{align}
&\bar{y}_1\le 1; \label{xbarybar:e1}\\
&\bar{y}_2\le 1; \label{xbarybar:e2}\allowdisplaybreaks\\
&\Clower\bar{y}_1\le\bar{x}_1\le\Cupper\bar{y}_1; \label{xbarybar:e3}\allowdisplaybreaks\\
&\Clower\bar{y}_2\le\bar{x}_2\le\Cupper\bar{y}_2; \label{xbarybar:e4}\allowdisplaybreaks\\
&\bar{x}_1\le\Vupper\bar{y}_1+(\Cupper-\Vupper)\bar{y}_2; \label{xbarybar:e5}\allowdisplaybreaks\\
&\bar{x}_2\le(\Cupper-\Vupper)\bar{y}_1+\Vupper\bar{y}_2; \label{xbarybar:e6}\allowdisplaybreaks\\
&\bar{x}_2-\bar{x}_1\le(\Clower+V)\bar{y}_2-\Clower\bar{y}_1; \label{xbarybar:e7}\allowdisplaybreaks\\
&\bar{x}_2-\bar{x}_1\le\Vupper\bar{y}_2-(\Vupper-V)\bar{y}_1; \label{xbarybar:e8}\allowdisplaybreaks\\
&\bar{x}_1-\bar{x}_2\le(\Clower+V)\bar{y}_1-\Clower\bar{y}_2; \label{xbarybar:e9}\\
&\bar{x}_1-\bar{x}_2\le\Vupper\bar{y}_1-(\Vupper-V)\bar{y}_2. \label{xbarybar:e10}
\end{align}
We prove that $(\bar{x}_1,\bar{x}_2,\bar{y}_1,\bar{y}_2)$ can be expressed as a convex combination of some elements of $\Pe_2$.
Specifically, we prove that there exist real numbers $\rho_1,\rho_2,\rho_3,\rho_4,\lambda_1,\lambda_2,\lambda_3,\lambda_4\ge 0$ such that
\begin{equation}\label{rhoconvcomb}
(\bar{x}_1,\bar{x}_2,\bar{y}_1,\bar{y}_2)
=\lambda_1(\rho_1,\rho_2,1,1)+\lambda_2(\rho_3,0,1,0)+\lambda_3(0,\rho_4,0,1)+\lambda_4(0,0,0,0),
\end{equation}
$(\rho_1,\rho_2,1,1),(\rho_3,0,1,0),(0,\rho_4,0,1),(0,0,0,0)\in\Pe_2$, and $\lambda_1+\lambda_2+\lambda_3+\lambda_4=1$.
We set 
\begin{align*}
&\lambda_1=\min\{\bar{y}_1,\bar{y}_2\};\allowdisplaybreaks\\
&\lambda_2=\bar{y}_1-\lambda_1;\allowdisplaybreaks\\
&\lambda_3=\bar{y}_2-\lambda_1;\allowdisplaybreaks\\
&\lambda_4=1-\bar{y}_1-\bar{y}_2+\lambda_1.
\end{align*}
Clearly, $\lambda_1+\lambda_2+\lambda_3+\lambda_4=1$.
By \eqref{xbarybar:e1}--\eqref{xbarybar:e4}, we have $0\le\bar{y}_1\le 1$ and $0\le\bar{y}_2\le 1$.
It is easy to verify that $\lambda_1,\lambda_2,\lambda_3,\lambda_4\ge 0$.
We consider five different cases.

Case~1: $\bar{y}_1=0$.
In this case, $\lambda_1=0$, $\lambda_2=0$, $\lambda_3=\bar{y}_2$, and $\lambda_4=1-\bar{y}_2$.
We set $\rho_1=\rho_2=\rho_3=\Clower$ and 
\begin{align*}
\rho_4=
\left\{
\begin{array}{ll}
\bar{x}_2/\bar{y}_2,& \hbox{ if $\bar{y}_2>0$;} \\[2pt]
\Clower,& \hbox{ if $\bar{y}_2=0$.}
\end{array}
\right.
\end{align*}
By \eqref{xbarybar:e3}, $\bar{x}_1=0$.
It is easy to verify that equation \eqref{rhoconvcomb} holds and that $(\rho_1,\rho_2,1,1),(\rho_3,0,1,0),(0,0,0,0)\in\Pe_2$.
Therefore, it suffices to show that $(0,\rho_4,0,1)\in\Pe_2$.
Clearly, $(0,\rho_4,0,1)$ satisfies \eqref{eqn:P2-d}.
By \eqref{xbarybar:e4}, $\bar{x}_2\ge\Clower\bar{y}_2$, which implies that $\rho_4\ge\Clower$.
Thus, $(0,\rho_4,0,1)$ satisfies \eqref{eqn:P2-a}.
By \eqref{xbarybar:e4}, $\bar{x}_2\le\Cupper\bar{y}_2$, which implies that $\rho_4\le\Cupper$.
Thus, $(0,\rho_4,0,1)$ satisfies \eqref{eqn:P2-b}.
By \eqref{xbarybar:e8}, $\bar{x}_2\le\Vupper\bar{y}_2$, which implies that $\rho_4\le\Vupper$.
Thus, $(0,\rho_4,0,1)$ satisfies \eqref{eqn:P2-c}.
Therefore, $(0,\rho_4,0,1)\in\Pe_2$.

Case~2: $\bar{y}_2=0$.
In this case, $\lambda_1=0$, $\lambda_2=\bar{y}_1$, $\lambda_3=0$, and $\lambda_4=1-\bar{y}_1$.
We set $\rho_1=\rho_2=\rho_4=\Clower$ and 
\begin{align*}
\rho_3=
\left\{
\begin{array}{ll}
\bar{x}_1/\bar{y}_1,& \hbox{if $\bar{y}_1>0$;} \\[2pt]
\Clower,& \hbox{if $\bar{y}_1=0$.}
\end{array}
\right.
\end{align*}
By \eqref{xbarybar:e4}, $\bar{x}_2=0$.
It is easy to verify that equation \eqref{rhoconvcomb} holds and that $(\rho_1,\rho_2,1,1),(0,\rho_4,0,1),(0,0,0,0)\in\Pe_2$.
Therefore, it suffices to show that $(\rho_3,0,1,0)\in\Pe_2$.
Clearly, $(\rho_3,0,1,0)$ satisfies \eqref{eqn:P2-c}.
By \eqref{xbarybar:e3}, $\bar{x}_1\ge\Clower\bar{y}_1$, which implies that $\rho_3\ge\Clower$.
Thus, $(\rho_3,0,1,0)$ satisfies \eqref{eqn:P2-a}.
By \eqref{xbarybar:e3}, $\bar{x}_1\le\Cupper\bar{y}_1$, which implies that $\rho_3\le\Cupper$.
Thus, $(\rho_3,0,1,0)$ satisfies \eqref{eqn:P2-b}.
By \eqref{xbarybar:e10}, $\bar{x}_1\le\Vupper\bar{y}_1$, which implies that $\rho_3\le\Vupper$.
Thus, $(\rho_3,0,1,0)$ satisfies \eqref{eqn:P2-d}.
Therefore, $(\rho_3,0,1,0)\in\Pe_2$.

Case~3: $0<\bar{y}_1<\bar{y}_2$.
In this case, $\lambda_1=\bar{y}_1$, $\lambda_2=0$, $\lambda_3=\bar{y}_2-\bar{y}_1$, and $\lambda_4=1-\bar{y}_2$.
We set
\begin{align*}
&\rho_1=\frac{\bar{x}_1}{\bar{y}_1};\allowdisplaybreaks\\
&\rho_2=\frac{1}{\bar{y}_1}[\bar{x}_2-(\bar{y}_2-\bar{y}_1)\rho_4];\allowdisplaybreaks\\
&\rho_3=\Clower;\allowdisplaybreaks\\
&\rho_4=\min\bigg\{\frac{\bar{x}_2-\Clower\bar{y}_1}{\bar{y}_2-\bar{y}_1},\frac{(\bar{x}_2-\bar{x}_1)+V\bar{y}_1}{\bar{y}_2-\bar{y}_1},\Vupper\bigg\}.
\end{align*}
By \eqref{xbarybar:e3}, $\rho_1\ge\Clower$.
Note that $\rho_4\le\frac{\bar{x}_2-\Clower\bar{y}_1}{\bar{y}_2-\bar{y}_1}$,
which implies that $\rho_2\ge\frac{1}{\bar{y}_1}[\bar{x}_2-(\bar{y}_2-\bar{y}_1)\frac{\bar{x}_2-\Clower\bar{y}_1}{\bar{y}_2-\bar{y}_1}]=\Clower$.
By \eqref{xbarybar:e4}, $\bar{x}_2\ge\Clower\bar{y}_2$, which implies that $\frac{\bar{x}_2-\Clower\bar{y}_1}{\bar{y}_2-\bar{y}_1}\ge\Clower$.
By \eqref{xbarybar:e9}, $\frac{(\bar{x}_2-\bar{x}_1)+V\bar{y}_1}{\bar{y}_2-\bar{y}_1}\ge\Clower$.
Thus, $\min\{\frac{\bar{x}_2-\Clower\bar{y}_1}{\bar{y}_2-\bar{y}_1},\frac{(\bar{x}_2-\bar{x}_1)+V\bar{y}_1}{\bar{y}_2-\bar{y}_1},\Vupper\}\ge\Clower$; that is, $\rho_4\ge\Clower$.
Hence, $\rho_1,\rho_2,\rho_3,\rho_4\ge\Clower$.
It is easy to verify that equation \eqref{rhoconvcomb} holds and that $(\rho_3,0,1,0),(0,0,0,0)\in\Pe_2$.
Therefore, it suffices to show that $(\rho_1,\rho_2,1,1),(0,\rho_4,0,1)\in\Pe_2$.

To show that $(\rho_1,\rho_2,1,1)\in\Pe_2$, we first note that $(\rho_1,\rho_2,1,1)$ satisfies \eqref{eqn:P2-a} (because $\rho_1,\rho_2\ge\Clower$).
By \eqref{xbarybar:e3}, $\rho_1\le\Cupper$.
Note that
\begin{align*}
\rho_2
&=\frac{1}{\bar{y}_1}\bigg[\bar{x}_2-(\bar{y}_2-\bar{y}_1)\min\bigg\{\frac{\bar{x}_2-\Clower\bar{y}_1}{\bar{y}_2-\bar{y}_1},\frac{(\bar{x}_2-\bar{x}_1)+V\bar{y}_1}{\bar{y}_2-\bar{y}_1},\Vupper\bigg\}\bigg]\allowdisplaybreaks\\
&=\frac{1}{\bar{y}_1}\big[\bar{x}_2-\min\big\{\bar{x}_2-\Clower\bar{y}_1,(\bar{x}_2-\bar{x}_1)+V\bar{y}_1,(\bar{y}_2-\bar{y}_1)\Vupper\big\}\big]\allowdisplaybreaks\\
&=\frac{1}{\bar{y}_1}\max\big\{\Clower\bar{y}_1,\bar{x}_1-V\bar{y}_1,\bar{x}_2-(\bar{y}_2-\bar{y}_1)\Vupper\big\}\allowdisplaybreaks\\
&\le\frac{1}{\bar{y}_1}\max\big\{\Cupper\bar{y}_1,\bar{x}_1,\bar{x}_2-(\bar{y}_2-\bar{y}_1)\Vupper\big\}\allowdisplaybreaks\\
&=\max\bigg\{\Cupper,\frac{\bar{x}_1}{\bar{y}_1},\frac{\bar{x}_2-(\bar{y}_2-\bar{y}_1)\Vupper}{\bar{y}_1}\bigg\}\\
&\le\Cupper,
\end{align*}
where the last inequality follows from \eqref{xbarybar:e3} and \eqref{xbarybar:e6}.
Hence, $(\rho_1,\rho_2,1,1)$ satisfies \eqref{eqn:P2-b}.
Note that
\begin{align*}
\rho_2-\rho_1
&=\frac{1}{\bar{y}_1}\bigg[(\bar{x}_2-\bar{x}_1)-(\bar{y}_2-\bar{y}_1)\min\bigg\{\frac{\bar{x}_2-\Clower\bar{y}_1}{\bar{y}_2-\bar{y}_1},\frac{(\bar{x}_2-\bar{x}_1)+V\bar{y}_1}{\bar{y}_2-\bar{y}_1},\Vupper\bigg\}\bigg]\allowdisplaybreaks\\
&=\frac{1}{\bar{y}_1}\bigg[(\bar{x}_2-\bar{x}_1)-\min\bigg\{\bar{x}_2-\Clower\bar{y}_1,(\bar{x}_2-\bar{x}_1)+V\bar{y}_1,(\bar{y}_2-\bar{y}_1)\Vupper\bigg\}\bigg]\allowdisplaybreaks\\
&=\frac{1}{\bar{y}_1}\max\big\{\Clower\bar{y}_1-\bar{x}_1,-V\bar{y}_1,(\bar{x}_2-\bar{x}_1)-(\bar{y}_2-\bar{y}_1)\Vupper\big\}\\
&=\max\bigg\{\Clower-\frac{\bar{x}_1}{\bar{y}_1},-V,\frac{(\bar{x}_2-\bar{x}_1)-(\bar{y}_2-\bar{y}_1)\Vupper}{\bar{y}_1}\bigg\}.
\end{align*}
By \eqref{xbarybar:e3}, $\Clower-\frac{\bar{x}_1}{\bar{y}_1}\le 0$.
By \eqref{xbarybar:e8}, $\frac{(\bar{x}_2-\bar{x}_1)-(\bar{y}_2-\bar{y}_1)\Vupper}{\bar{y}_1}\le V$.
Thus, $\rho_2-\rho_1\le V$.
Hence, $(\rho_1,\rho_2,1,1)$ satisfies \eqref{eqn:P2-c}.
Because
$$\rho_1-\rho_2
=\frac{1}{\bar{y}_1}\big[-(\bar{x}_2-\bar{x}_1)+(\bar{y}_2-\bar{y}_1)\rho_4\big]
\le\frac{1}{\bar{y}_1}\bigg[-(\bar{x}_2-\bar{x}_1)+(\bar{y}_2-\bar{y}_1)\cdot\frac{(\bar{x}_2-\bar{x}_1)+V\bar{y}_1}{\bar{y}_2-\bar{y}_1}\bigg]
=V,$$
$(\rho_1,\rho_2,1,1)$ satisfies \eqref{eqn:P2-d}.
Therefore, $(\rho_1,\rho_2,1,1)\in\Pe_2$.

To show that $(0,\rho_4,0,1)\in\Pe_2$, we note that $(0,\rho_4,0,1)$ satisfies \eqref{eqn:P2-a} and \eqref{eqn:P2-d} (because $\rho_4\ge\Clower$).
Because $\rho_4\le\Vupper\le\Cupper$, $(0,\rho_4,0,1)$ satisfies \eqref{eqn:P2-b} and \eqref{eqn:P2-c}.
Therefore, $(0,\rho_4,0,1)\in\Pe_2$.

Case~4: $0<\bar{y}_2<\bar{y}_1$.
In this case, $\lambda_1=\bar{y}_2$, $\lambda_2=\bar{y}_1-\bar{y}_2$, $\lambda_3=0$, and $\lambda_4=1-\bar{y}_1$.
We set
\begin{align*}
&\rho_1=\frac{1}{\bar{y}_2}[\bar{x}_1-(\bar{y}_1-\bar{y}_2)\rho_3];\allowdisplaybreaks\\
&\rho_2=\frac{\bar{x}_2}{\bar{y}_2};\allowdisplaybreaks\\
&\rho_3=\min\bigg\{\frac{\bar{x}_1-\Clower\bar{y}_2}{\bar{y}_1-\bar{y}_2},\frac{(\bar{x}_1-\bar{x}_2)+V\bar{y}_2}{\bar{y}_1-\bar{y}_2},\Vupper\bigg\};\allowdisplaybreaks\\
&\rho_4=\Clower.
\end{align*}
Note that $\rho_3\le\frac{\bar{x}_1-\Clower\bar{y}_2}{\bar{y}_1-\bar{y}_2}$,
which implies that $\rho_1\ge\frac{1}{\bar{y}_2}[\bar{x}_1-(\bar{y}_1-\bar{y}_2)\frac{\bar{x}_1-\Clower\bar{y}_2}{\bar{y}_1-\bar{y}_2}]=\Clower$.
By \eqref{xbarybar:e4}, $\rho_2\ge\Clower$.
By \eqref{xbarybar:e3}, $\bar{x}_1\ge\Clower\bar{y}_1$, which implies that $\frac{\bar{x}_1-\Clower\bar{y}_2}{\bar{y}_1-\bar{y}_2}\ge\Clower$.
By \eqref{xbarybar:e7}, $\frac{(\bar{x}_1-\bar{x}_2)+V\bar{y}_2}{\bar{y}_1-\bar{y}_2}\ge\Clower$.
Thus, $\min\{\frac{\bar{x}_1-\Clower\bar{y}_2}{\bar{y}_1-\bar{y}_2},\frac{(\bar{x}_1-\bar{x}_2)+V\bar{y}_2}{\bar{y}_1-\bar{y}_2},\Vupper\}\ge\Clower$; that is, $\rho_3\ge\Clower$.
Hence, $\rho_1,\rho_2,\rho_3,\rho_4\ge\Clower$.
It is easy to verify that equation \eqref{rhoconvcomb} holds and that $(0,\rho_4,0,1),(0,0,0,0)\in\Pe_2$.
Therefore, it suffices to show that $(\rho_1,\rho_2,1,1),(\rho_3,0,1,0)\in\Pe_2$.

To show that $(\rho_1,\rho_2,1,1)\in\Pe_2$, we first note that $(\rho_1,\rho_2,1,1)$ satisfies \eqref{eqn:P2-a} (because $\rho_1,\rho_2\ge\Clower$).
Note that
\begin{align*}
\rho_1
&=\frac{1}{\bar{y}_2}\bigg[\bar{x}_1-(\bar{y}_1-\bar{y}_2)\min\bigg\{\frac{\bar{x}_1-\Clower\bar{y}_2}{\bar{y}_1-\bar{y}_2},\frac{(\bar{x}_1-\bar{x}_2)+V\bar{y}_2}{\bar{y}_1-\bar{y}_2},\Vupper\bigg\}\bigg]\allowdisplaybreaks\\
&=\frac{1}{\bar{y}_2}\big[\bar{x}_1-\min\big\{\bar{x}_1-\Clower\bar{y}_2,(\bar{x}_1-\bar{x}_2)+V\bar{y}_2,(\bar{y}_1-\bar{y}_2)\Vupper\big\}\big]\allowdisplaybreaks\\
&=\frac{1}{\bar{y}_2}\max\big\{\Clower\bar{y}_2,\bar{x}_2-V\bar{y}_2,\bar{x}_1-(\bar{y}_1-\bar{y}_2)\Vupper\big\}\allowdisplaybreaks\\
&\le\frac{1}{\bar{y}_2}\max\big\{\Cupper\bar{y}_2,\bar{x}_2,\bar{x}_1-(\bar{y}_1-\bar{y}_2)\Vupper\big\}\allowdisplaybreaks\\
&=\max\bigg\{\Cupper,\frac{\bar{x}_2}{\bar{y}_2},\frac{\bar{x}_1-(\bar{y}_1-\bar{y}_2)\Vupper}{\bar{y}_2}\bigg\}\\
&\le\Cupper,
\end{align*}
where the last inequality follows from \eqref{xbarybar:e4} and \eqref{xbarybar:e5}.
By \eqref{xbarybar:e4}, $\rho_2\le\Cupper$.
Hence, $(\rho_1,\rho_2,1,1)$ satisfies \eqref{eqn:P2-b}.
Because
$$\rho_2-\rho_1
=\frac{1}{\bar{y}_2}\big[-(\bar{x}_1-\bar{x}_2)+(\bar{y}_1-\bar{y}_2)\rho_3\big]
\le\frac{1}{\bar{y}_2}\bigg[-(\bar{x}_1-\bar{x}_2)+(\bar{y}_1-\bar{y}_2)\cdot\frac{(\bar{x}_1-\bar{x}_2)+V\bar{y}_2}{\bar{y}_1-\bar{y}_2}\bigg]
=V,$$
$(\rho_1,\rho_2,1,1)$ satisfies \eqref{eqn:P2-c}.
Note that
\begin{align*}
\rho_1-\rho_2
&=\frac{1}{\bar{y}_2}\bigg[(\bar{x}_1-\bar{x}_2)-(\bar{y}_1-\bar{y}_2)\min\bigg\{\frac{\bar{x}_1-\Clower\bar{y}_2}{\bar{y}_1-\bar{y}_2},\frac{(\bar{x}_1-\bar{x}_2)+V\bar{y}_2}{\bar{y}_1-\bar{y}_2},\Vupper\bigg\}\bigg]\allowdisplaybreaks\\
&=\frac{1}{\bar{y}_2}\bigg[(\bar{x}_1-\bar{x}_2)-\min\bigg\{\bar{x}_1-\Clower\bar{y}_2,(\bar{x}_1-\bar{x}_2)+V\bar{y}_2,(\bar{y}_1-\bar{y}_2)\Vupper\bigg\}\bigg]\allowdisplaybreaks\\
&=\frac{1}{\bar{y}_2}\max\big\{\Clower\bar{y}_2-\bar{x}_2,-V\bar{y}_2,(\bar{x}_1-\bar{x}_2)-(\bar{y}_1-\bar{y}_2)\Vupper\big\}\\
&=\max\bigg\{\Clower-\frac{\bar{x}_2}{\bar{y}_2},-V,\frac{(\bar{x}_1-\bar{x}_2)-(\bar{y}_1-\bar{y}_2)\Vupper}{\bar{y}_2}\bigg\}.
\end{align*}
By \eqref{xbarybar:e4}, $\Clower-\frac{\bar{x}_2}{\bar{y}_2}\le 0$.
By \eqref{xbarybar:e10}, $\frac{(\bar{x}_1-\bar{x}_2)-(\bar{y}_1-\bar{y}_2)\Vupper}{\bar{y}_2}\le V$.
Thus, $\rho_1-\rho_2\le V$.
Hence, $(\rho_1,\rho_2,1,1)$ satisfies \eqref{eqn:P2-d}.
Therefore, $(\rho_1,\rho_2,1,1)\in\Pe_2$.

To show that $(\rho_3,0,1,0)\in\Pe_2$, we note that $(\rho_3,0,1,0)$ satisfies \eqref{eqn:P2-a} and \eqref{eqn:P2-c} (because $\rho_3\ge\Clower$).
Because $\rho_3\le\Vupper\le\Cupper$,  $(\rho_3,0,1,0)$ satisfies \eqref{eqn:P2-b} and \eqref{eqn:P2-d}.
Therefore, $(\rho_3,0,1,0)\in\Pe_2$.

Case~5: $0<\bar{y}_1=\bar{y}_2$.
In this case, $\lambda_1=\bar{y}_1=\bar{y}_2$, $\lambda_2=0$, $\lambda_3=0$, and $\lambda_4=1-\bar{y}_1=1-\bar{y}_2$.
We set
$$\rho_1=\frac{\bar{x}_1}{\bar{y}_1};\
\rho_2=\frac{\bar{x}_2}{\bar{y}_2};\
\rho_3=\Clower;\
\rho_4=\Clower.$$
Clearly, $\rho_1,\rho_2,\rho_3,\rho_4\ge 0$.
It is easy to verify that equation \eqref{rhoconvcomb} holds and that $(\rho_3,0,1,0),(0,\rho_4,0,1),(0,0,0,0)\in\Pe_2$.
Therefore, it suffices to show that $(\rho_1,\rho_2,1,1)\in\Pe_2$.
By \eqref{xbarybar:e3} and \eqref{xbarybar:e4}, $\rho_1,\rho_2\ge\Clower$ and $\rho_1,\rho_2\le\Cupper$.
Thus, $(\rho_1,\rho_2,1,1)$ satisfies \eqref{eqn:P2-a} and \eqref{eqn:P2-b}.
By \eqref{xbarybar:e7}, $\bar{x}_2-\bar{x}_1\le V\bar{y}_1$, which implies that $\rho_2-\rho_1\le V$.
Thus, $(\rho_1,\rho_2,1,1)$ satisfies \eqref{eqn:P2-c}.
By \eqref{xbarybar:e9}, $\bar{x}_1-\bar{x}_2\le V\bar{y}_1$, which implies that $\rho_1-\rho_2\le V$.
Thus, $(\rho_1,\rho_2,1,1)$ satisfies \eqref{eqn:P2-d}.
Therefore, $(\rho_1,\rho_2,1,1)\in\Pe_2$.

Combining Cases 1--5, we conclude that there exist $\rho_1,\rho_2,\rho_3,\rho_4,\lambda_1,\lambda_2,\lambda_3,\lambda_4\ge 0$ that satisfy equation \eqref{rhoconvcomb}.
Hence, $\Qe_2\subseteq\convPtwo$.

Note that \cite{damci2016polyhedral} have also derived the two-period convex hulls for the two-period ramp-up and ramp-down polytopes \textit{separately}.
In contrast, our Theorem~\ref{the:convPtwo} provides the convex hull for the two-period UC polytope containing both the ramp-up and ramp-down polytopes.
\Halmos
\vskip20pt


\section{Supplement to Section~\ref{sec:multi-period}}
\subsection{Proof of Proposition \ref{prop-T:x_t-ub-1}}\label{apx:A-prop-1}

\noindent{\bf Proposition \ref{prop-T:x_t-ub-1}.} {\it
Consider any $\Se\subseteq[0, \min\{L-1, T-2, \lfloor (\Cupper - \Vupper)/V \rfloor\}]_{\Z}$.
For any $t\in[1,T]_\Z$ such that $t\ge s+2$ for all $s\in\Se$, the inequality
\begin{equation}
x_t \leq \Cupper y_t - \sum_{s\in\Se} (\Cupper - \Vupper - sV) (y_{t-s} - y_{t-s-1})\tag{\ref{eqn-T:x_t-ub-1-1}}
\end{equation}
is valid and facet-defining for $\convP$.
For any $t\in[1,T]_\Z$ such that $t\le T-s-1$ for all $s\in\Se$, the inequality
\begin{equation}
x_t \leq \Cupper y_t - \sum_{s\in\Se} (\Cupper - \Vupper - sV) (y_{t+s} - y_{t+s+1})\tag{\ref{eqn-T:x_t-ub-1-2}}
\end{equation}
is valid and facet-defining for $\convP$.}
\vskip8pt

\noindent{\bf Proof.}
We first prove that inequality \eqref{eqn-T:x_t-ub-1-1} is valid and facet-defining for $\convP$.
Note that the proof of facet-defining of \eqref{eqn-T:x_t-ub-1-1} here can also be used to prove the facet-defining of \eqref{eqn-T:x_t-ub-1-1} in Proposition~\ref{prop-T:x_t-ub-2}.
For notational convenience, we define $s_{\max}=\max\{s:s\in\Se\}$ if $\Se\ne\emptyset$, and $s_{\max}=-1$ if $\Se=\emptyset$.

Consider any $t\in[s_{\max}+2,T]_\Z$ (i.e., $t\in[1,T]_\Z$ such that $t\ge s+2$ for all $s\in\Se$).
To prove that linear inequalities \eqref{eqn-T:x_t-ub-1-1} is valid for $\convP$, it suffices to show that they are valid for $\Pe$. 
Consider any element $(\BFx,\BFy)$ of $\Pe$. 
We show that $(\BFx,\BFy)$ satisfies \eqref{eqn-T:x_t-ub-1-1}.

Case~1: $y_t=0$.
By Lemma \ref{lem:lookbackward}(i), $y_{t-j}-y_{t-j-1}\le0$ for all $j\in[0,\min\{t-2,L-1\}]_\Z$.
Because $\Se\subseteq[0,L-1]_{\Z}$ and $s_{\max}\le t-2$, we have $\Se\subseteq[0,\min\{t-2,L-1\}]_{\Z}$.
Thus, $y_{t-s}-y_{t-s-1}\le0$ for all $s\in\Se$.
Because $\Se\subseteq[0,\lfloor(\Cupper-\Vupper)/V\rfloor]_\Z$, we have $\Cupper-\Vupper-sV\ge 0$ for all $s\in\Se$.
Thus, the right-hand side of inequality \eqref{eqn-T:x_t-ub-1-1} is nonnegative.
Because $y_t=0$, by \eqref{eqn:p-upper-bound}, $x_t=0$.
Therefore, in this case, $(\BFx,\BFy)$ satisfies \eqref{eqn-T:x_t-ub-1-1}.

Case~2: $y_t=1$ and $y_{t-s'}-y_{t-s'-1}=1$ for some $s'\in\Se$. 
By Lemma \ref{lem:lookbackward}(ii), there exists at most one $j\in[0,\min\{t-2,L\}]_{\Z}$ such that $y_{t-j}-y_{t-j-1}=1$.
Because $\Se\subseteq[0,L]_{\Z}$ and $s_{\max}\le t-2$, we have $\Se\subseteq[0,\min\{t-2,L\}]_{\Z}$.
This implies that $y_{t-s}-y_{t-s-1}\le0$, for all for all $s\in\Se\setminus\{s'\}$.
For any $s\in\Se$, because $s\le\lfloor(\Cupper-\Vupper)/V\rfloor$, we have $\Cupper-\Vupper-sV\ge0$.
Thus, for any $s\in\Se\setminus\{s'\}$, $(\Cupper-\Vupper-sV)(y_{t-s}-y_{t-s-1})$ is either zero or negative.
Hence, $\sum_{s\in\Se}(\Cupper-\Vupper-sV)(y_{t-s}-y_{t-s-1})\le\Cupper-\Vupper-s'V$.
Thus, the right-hand side of inequality \eqref{eqn-T:x_t-ub-1-1} is at least $s'V+\Vupper$.
By \eqref{eqn:p-ramp-up}, 
$\sum_{\tau=t-s'}^{t}(x_{\tau}-x_{\tau-1})\le\sum_{\tau=t-s'}^{t}Vy_{\tau-1}+\sum_{\tau=t-s_j}^{t}\Vupper(1-y_{\tau-1})$, 
which implies that $x_t-x_{t-s'-1}\le s'V+\Vupper$. 
Because $y_{t-s'}-y_{t-s'-1}=1$, we have $y_{t-s'-1}=0$.
By \eqref{eqn:p-upper-bound}, $x_{t-s'-1}=0$. 
Hence, $x_t\le s'V+\Vupper$.
Therefore, in this case, $(\BFx,\BFy)$ satisfies \eqref{eqn-T:x_t-ub-1-1}.

Case~3: $y_t=1$ and $y_{t-s}-y_{t-s-1}\ne 1$ for all $s\in\Se$.
In this case, $y_{t-s}-y_{t-s-1}\le 0$ for all $s\in\Se$.
For any $s\in\Se$, because $s\le\lfloor(\Cupper-\Vupper)/V\rfloor$, we have $\Cupper-\Vupper-sV\ge 0$.
Thus, the right-hand side of inequality \eqref{eqn-T:x_t-ub-1-1} is at least $\Cupper$. 
By \eqref{eqn:p-upper-bound}, $x_t\le\Cupper$.
Therefore, in this case, $(\BFx,\BFy)$ satisfies \eqref{eqn-T:x_t-ub-1-1}.

Summarizing Cases 1--3, we conclude that $(\BFx,\BFy)$ satisfies \eqref{eqn-T:x_t-ub-1-1}. 
Hence, \eqref{eqn-T:x_t-ub-1-1} is valid for $\convP$.

Consider any $t\in[s_{\max}+2,T]_\Z$.
To prove that inequality \eqref{eqn-T:x_t-ub-1-1} is facet-defining for $\convP$, it suffices to show that there exist $2T$ affinely independent points in $\convP$ that satisfy \eqref{eqn-T:x_t-ub-1-1} at equality.
Because $\BFzero\in\convP$ and $\BFzero$ satisfies \eqref{eqn-T:x_t-ub-1-1} at equality, it suffices to create the remaining $2T-1$ nonzero linearly independent points.
We denote these $2T-1$ points as $(\bar{\BFx}^r,\bar{\BFy}^r)$ for $r\in[1,T]_{\Z}\setminus\{t\}$ and $(\hat{\BFx}^r,\hat{\BFy}^r)$ for $r\in[1,T]_{\Z}$, and denote the $q$th component of $\bar{\BFx}^r$, $\bar{\BFy}^r$, $\hat{\BFx}^r$, and $\hat{\BFy}^r$ as $\bar{x}^r_q$, $\bar{y}^r_q$, $\hat{x}^r_q$, and $\hat{y}^r_q$, respectively.
Let $\epsilon=\Vupper-\Clower>0$.
We divide these $2T-1$ points into the following \rred{four} groups:

\begin{enumerate}[label={(A\arabic*)},ref={(A\arabic*)}]
\item\label{points:prop1-eq1-A1} For each $r\in[1,T]_\Z\setminus\{t\}$, we create a point $(\bar{\BFx}^r,\bar{\BFy}^r)$ as follows:
\begin{equation*}
\bar{x}^r_q=\left\{
   \begin{array}{ll}
   \Cupper,          &\hbox{for $q\in [1,T]_\Z\setminus\{r\}$};\\
   \Cupper-\epsilon, &\hbox{for $q=r$};
   \end{array}\right.
\end{equation*}
and $\bar{y}^r_q=1$ for all $q\in [1,T]_\Z$.
It is easy to verify that $(\bar{\BFx}^r,\bar{\BFy}^r)$ satisfies \eqref{eqn:p-minup}--\eqref{eqn:p-ramp-down}.
Thus, $(\bar{\BFx}^r,\bar{\BFy}^r)\in\convP$.
It is also easy to verify that $(\bar{\BFx}^r,\bar{\BFy}^r)$ satisfies \eqref{eqn-T:x_t-ub-1-1} at equality.

\item\label{points:prop1-eq1-A2} For each \rred{$r\in[1,t-1]_\Z$}, we create a point $(\hat{\BFx}^r,\hat{\BFy}^r)$ as follows:
If $t-r-1\notin\Se$, then
\begin{equation*}
\hat{x}^r_q=\left\{
   \begin{array}{ll}
   \Clower, &\hbox{for $q\in[1,r]_\Z$};\\
   0,       &\hbox{for $q\in[r+1,T]_\Z$};
   \end{array}\right.
\end{equation*}
and
\begin{equation*}
\hat{y}^r_q=\left\{
   \begin{array}{ll}
   1, &\hbox{for $q\in[1,r]_\Z$};\\
   0, &\hbox{for $q\in[r+1,T]_\Z$}.
   \end{array}\right.
\end{equation*}
If $t-r-1\in\Se$, then
\begin{equation*}
\hat{x}^r_q=\left\{
   \begin{array}{ll}
   0, &\hbox{for $q\in[1,r]_\Z$};\\
   \Vupper+(q-r-1)V, &\hbox{for $q\in[r+1,t]_\Z$};\\
   \Vupper+(t-r-1)V, &\hbox{for $q\in[t+1,T]_\Z$};
   \end{array}\right.
\end{equation*}
and
\begin{equation*}
\hat{y}^r_q=\left\{
   \begin{array}{ll}
   0, &\hbox{for $q\in[1,r]_\Z$};\\
   1, &\hbox{for $q\in[r+1,T]_\Z$}.
   \end{array}\right.
\end{equation*}
We first consider the case where $t-r-1\notin\Se$.
In this case, it is easy to verify that $(\hat{\BFx}^r,\hat{\BFy}^r)$ satisfies \eqref{eqn:p-minup}--\eqref{eqn:p-ramp-down} and is therefore in $\convP$.
Note that in this case $\hat{x}^r_t=\hat{y}^r_t=0$, and $t-s-1\ne r$ for all $s\in\Se$, which implies that $\hat{y}^r_{t-s}-\hat{y}^r_{t-s-1}=0$ for all $s\in\Se$.
Thus, $(\hat{\BFx}^r,\hat{\BFy}^r)$ satisfies \eqref{eqn-T:x_t-ub-1-1} at equality.
Next, we consider the case where $t-r-1\in\Se$.
In this case, it is easy to verify that $(\hat{\BFx}^r,\hat{\BFy}^r)$ satisfies \eqref{eqn:p-minup} and \eqref{eqn:p-mindn}.
For each $q\in[1,r]_\Z$, we have $\hat{x}^r_q=\hat{y}^r_q=0$.
For each $q\in[r+1,T]_\Z$, because $t-r-1\in\Se$, we have $t-r-1\le\lfloor(\Cupper-\Vupper)/V\rfloor$, which implies that $\Vupper+(t-r-1)V\le\Cupper$, 
which in turn implies that $\Clower\le\hat{x}^r_q\le\Cupper$.
Hence, $(\hat{\BFx}^r,\hat{\BFy}^r)$ satisfies \eqref{eqn:p-lower-bound} and \eqref{eqn:p-upper-bound}.
Note that $\hat{x}^r_q-\hat{x}^r_{q-1}=0$ when $q\in[2,r]_\Z$, $\hat{x}^r_q-\hat{x}^r_{q-1}=\Vupper$ when $q=r+1$, and $0\le\hat{x}^r_q-\hat{x}^r_{q-1}\le V$ when $q\in[r+2,T]_\Z$.
Thus, $-V\hat{y}^r_q-\Vupper(1-\hat{y}^r_q)\le\hat{x}^r_q-\hat{x}^r_{q-1}\le V\hat{y}^r_{q-1}+\Vupper(1-\hat{y}^r_{q-1})$ for all $q\in[2,T]_\Z$.
Hence, $(\hat{\BFx}^r,\hat{\BFy}^r)$ satisfies \eqref{eqn:p-ramp-up} and \eqref{eqn:p-ramp-down}.
Therefore, $(\hat{\BFx}^r,\hat{\BFy}^r)\in\convP$.
Note that in this case $\hat{x}^r_t=\Vupper+(t-r-1)V$, $\hat{y}^r_t=1$, $\hat{y}^r_{t-s}-\hat{y}^r_{t-s-1}=1$ when $s=t-r-1$, and $\hat{y}^{\rred{r}}_{t-s}-\hat{y}^r_{t-s-1}=0$ when $s\neq t-r-1$.
Thus, $(\hat{\BFx}^r,\hat{\BFy}^r)$ satisfies \eqref{eqn-T:x_t-ub-1-1} at equality.

\item\label{points:prop1-eq1-A3} We create a point $(\hat{\BFx}^t,\hat{\BFy}^t)$ by setting $\hat{x}^t_q=\Cupper$ and $\hat{y}^t_q=1$ for $q\in[1,T]_\Z$.
It is easy to verify that $(\hat{\BFx}^t,\hat{\BFy}^t)$ satisfies \eqref{eqn:p-minup}--\eqref{eqn:p-ramp-down}.
Thus, $(\hat{\BFx}^t,\hat{\BFy}^t)\in\convP$.
It is also easy to verify that $(\hat{\BFx}^t,\hat{\BFy}^t)$ satisfies \eqref{eqn-T:x_t-ub-1-1} at equality.

\item\label{points:prop1-eq1-A4} For each $r\in[t+1,T]_\Z$, we create a point $(\hat{\BFx}^r,\hat{\BFy}^r)$ as follows:
\begin{equation*}
\hat{x}^r_q=\left\{
   \begin{array}{ll}
   0,       &\hbox{for $q\in[1,r-1]_\Z$};\\
   \Clower, &\hbox{for $q\in[r,T]_\Z$};
   \end{array}\right.
\end{equation*}
and
\begin{equation*}
\hat{y}^r_q=\left\{
   \begin{array}{ll}
   0, &\hbox{for $q\in[1,r-1]_\Z$};\\
   1, &\hbox{for $q\in[r,T]_\Z$}.
   \end{array}\right.
\end{equation*}
It is easy to verify that $(\rred{\hat{\BFx}^r,\hat{\BFy}^r})$ satisfies \eqref{eqn:p-minup}--\eqref{eqn:p-ramp-down}.
Thus, $(\rred{\hat{\BFx}^r,\hat{\BFy}^r})\in\convP$.
It is also easy to verify that $(\rred{\hat{\BFx}^r,\hat{\BFy}^r})$ satisfies \eqref{eqn-T:x_t-ub-1-1} at equality.
\end{enumerate}

Table~\ref{tab:eqn-T:x_t-ub-1-1-facet-matrix-1} shows a matrix with $2T-1$ rows, where each row represents a point created by this process.
This matrix can be transformed into the matrix in Table~\ref{tab:eqn-T:x_t-ub-1-1-facet-matrix-2} via the following Gaussian elimination process:

\begin{enumerate}[label=(\roman*)]
\item For each $r\in[1,T]_\Z\setminus\{t\}$, the point with index $r$ in group (B1), denoted $(\underline{\bar{\BFx}}^r,\underline{\bar{\BFy}}^r)$, 
is obtained by setting $(\underline{\bar{\BFx}}^r,\underline{\bar{\BFy}}^r)=(\bar{\BFx}^r,\bar{\BFy}^r)-(\hat{\BFx}^t,\hat{\BFy}^t)$. 
Here, $(\bar{\BFx}^r,\bar{\BFy}^r)$ is the point with index $r$ in group \ref{points:prop1-eq1-A1}, and $(\hat{\BFx}^t,\hat{\BFy}^t)$ is the point in group \ref{points:prop1-eq1-A3}.
\item For each \rred{$r\in[1,t-1]_\Z$}, the point with index $r$ in group \rred{(B2)}, denoted $(\underline{\hat{\BFx}}^r,\underline{\hat{\BFy}}^r)$, 
is obtained by setting $(\underline{\hat{\BFx}}^r,\underline{\hat{\BFy}}^r)=(\hat{\BFx}^r,\hat{\BFy}^r)$ if $t-r-1\notin\Se$, 
and setting $(\underline{\hat{\BFx}}^r,\underline{\hat{\BFy}}^r)=\rred{(\hat{\BFx}^t,\hat{\BFy}^t)-(\hat{\BFx}^r,\hat{\BFy}^r)}$ if $t-r-1\in\Se$.
Here, $(\hat{\BFx}^r, \hat{\BFy}^r)$ is the point with index $r$ in group \ref{points:prop1-eq1-A2}, and $(\hat{\BFx}^t,\hat{\BFy}^t)$ is the point in group \ref{points:prop1-eq1-A3}.
\item The point in group \rred{(B3)}, denoted $(\underline{\hat{\BFx}}^t,\underline{\hat{\BFy}}^t)$, 
is obtained by setting $(\underline{\hat{\BFx}}^t,\underline{\hat{\BFy}}^t)=(\hat{\BFx}^t,\hat{\BFy}^t)-(\hat{\BFx}^{t+1},\hat{\BFy}^{t+1})$.
Here, $(\hat{\BFx}^t,\hat{\BFy}^t)$ is the point in group \ref{points:prop1-eq1-A3}, and $(\hat{\BFx}^{t+1},\hat{\BFy}^{t+1})$ is the point with index $t+1$ in group \ref{points:prop1-eq1-A4}.
\item For each $r\in[t+1,T]_\Z$, the point with index $r$ in group \rred{(B4)}, denoted $(\underline{\hat{\BFx}}^r,\underline{\hat{\BFy}}^r)$, 
is obtained by setting $(\underline{\hat{\BFx}}^r,\underline{\hat{\BFy}}^r)=(\hat{\BFx}^r,\hat{\BFy}^r)-(\hat{\BFx}^{r+1},\hat{\BFy}^{r+1})$ if $r\ne T$, 
and setting $(\underline{\hat{\BFx}}^r,\underline{\hat{\BFy}}^r)=(\hat{\BFx}^r,\hat{\BFy}^r)$ if $r=T$.
Here, $(\hat{\BFx}^r,\hat{\BFy}^r)$ and $(\hat{\BFx}^{r+1},\hat{\BFy}^{r+1})$ are the points with indices $r$ and $r+1$, respectively, in group \ref{points:prop1-eq1-A4}.
\end{enumerate}

The matrix shown in Table~\ref{tab:eqn-T:x_t-ub-1-1-facet-matrix-2} is lower triangular; 
that is, the position of the last nonzero component of a row of the matrix is greater than the position of the last nonzero component of the previous row.
This implies that the $2T-1$ points in groups \ref{points:prop1-eq1-A1}--\ref{points:prop1-eq1-A4} are linearly independent.
Therefore, inequality \eqref{eqn-T:x_t-ub-1-1} is facet-defining for $\convP$.

\afterpage{
\begin{landscape}
\begin{table}
    \renewcommand{\arraystretch}{2}
    \centering
    \caption{A matrix with the rows representing $2T-1$ points in $\convP$ that satisfy inequality \eqref{eqn-T:x_t-ub-1-1} at equality}
    \rule{0pt}{4ex}
    \setlength\tabcolsep{6.5pt}
    \scriptsize
    \begin{tabular}{|c|c|c|*{7}{c}|*{7}{c}|}
        \hline
         \multirow{2}{*}{Group} & \multirow{2}{*}{Point} & \multirow{2}{*}{Index $r$} & \multicolumn{7}{c|}{$\BFx$} & \multicolumn{7}{c|}{$\BFy$} \\
         \cline{4-9}\cline{10-17}
         &&& $1$ & $\cdots$ & $t-1$ & $t$ & $t+1$ & $\cdots$ & $T$ 
           & $\ \ 1\ \ $ & $\cdots$ & $t-1$ & $t$ & $t+1$ & $\cdots$ & $\ \ T\ \ $ \\
         \hline
         \multirow{6}{*}{(A1)} & \multirow{6}{*}{$(\bar{\BFx}^r,\bar{\BFy}^r)$} 
         & $1$ 
         & $\Cupper-\epsilon$ & $\cdots$ & $\Cupper$ & $\Cupper$ & $\Cupper$ & $\cdots$ & $\Cupper$
         & $1$ & $\cdots$ & $1$ & $1$ & $1$ & $\cdots$ & $1$ \\ 
         && $\vdots$ 
         & $\vdots$ & \rotatebox{0}{$\ddots$} & $\vdots$ & $\vdots$ & $\vdots$ && $\vdots$ 
         & $\vdots$ && $\vdots$ & $\vdots$ & $\vdots$ && $\vdots$ \\        
         && $t-1$
         & $\Cupper$ & $\cdots$ & $\Cupper-\epsilon$ & $\Cupper$ & $\Cupper$ & $\cdots$ & $\Cupper$
         & $1$ & $\cdots$ & $1$ & $1$ & $1$ & $\cdots$ & $1$ \\  
         && $t+1$
         & $\Cupper$ & $\cdots$ & $\Cupper$ & $\Cupper$ & $\Cupper-\epsilon$ & $\cdots$ & $\Cupper$ 
         & $1$ & $\cdots$ & $1$ & $1$ & $1$ & $\cdots$ & $1$ \\
         && $\vdots$ 
         & $\vdots$ && $\vdots$ & $\vdots$ & $\vdots$ & \rotatebox{0}{$\ddots$} & $\vdots$ 
         & $\vdots$ && $\vdots$ & $\vdots$ & $\vdots$ && $\vdots$ \\        
         && $T$
         & $\Cupper$ & $\cdots$ & $\Cupper$ & $\Cupper$ & $\Cupper$ & $\cdots$ & $\Cupper-\epsilon$ 
         & $1$ & $\cdots$ & $1$ & $1$ & $1$ & $\cdots$ & $1$ \\ 
         
         \hline
         \multirow{3}{*}{\rred{(A2)}} &\multirow{7}{*}{$(\hat{\BFx}^r,\hat{\BFy}^r)$}
         & $1$ & \multicolumn{7}{c|}{\multirow{3}{*}{(See Note \ref{tab:eqn-T:x_t-ub-1-1-facet-matrix-1}-1)}} & \multicolumn{7}{c|}{\multirow{3}{*}{(See Note \ref{tab:eqn-T:x_t-ub-1-1-facet-matrix-1}-1)}}\\
         && $\rred{\vdots}$ & \multicolumn{7}{c|}{} & \multicolumn{7}{c|}{}\\    
         && $t-1$ & \multicolumn{7}{c|}{} & \multicolumn{7}{c|}{}\\
         
         \cline{1-1} \cline{3-17}
          \rred{(A3)} &
         & $t$
         & $\Cupper$ & $\cdots$ & $\Cupper$ & $\Cupper$ & $\Cupper$ & $\cdots$ & $\Cupper$
         & $1$ & $\cdots$ & $1$ & $1$ & $1$ & $\cdots$ & $1$ \\
         
         \cline{1-1} \cline{3-17}
         \multirow{3}{*}{\rred{(A4)}} &  
         & $t+1$
         & $0$ & $\cdots$ & $0$ & $0$ & $\Clower$ & $\cdots$ & $\Clower$
         & $0$ & $\cdots$ & $0$ & $0$ & $1$ & $\cdots$ & $1$ \\
         && $\vdots$ 
         & $\vdots$ && $\vdots$ & $\vdots$ & $\vdots$ & \rotatebox{0}{$\ddots$} & $\vdots$ 
         & $\vdots$ && $\vdots$ & $\vdots$ & $\vdots$ & \rotatebox{0}{$\ddots$} & $\vdots$ \\                 
         && $T$
         & $0$ & $\cdots$ & $0$ & $0$ & $0$ & $\cdots$ & $\Clower$
         & $0$ & $\cdots$ & $0$ & $0$ & $0$ & $\cdots$ & $1$ \\
        \hline
        \multicolumn{17}{l}{Note \ref{tab:eqn-T:x_t-ub-1-1-facet-matrix-1}-1: For $r\in\rred{[1,t-1]_\Z}$, the $\BFx$ and $\BFy$ vectors in group \rred{(A2)} are given as follows:}\\
        \multicolumn{17}{l}{
        $\hat{\BFx}^r=(\underbrace{\Clower,\ldots,\Clower}_{r\ {\rm terms}},\underbrace{0,\ldots,0}_{T-r\ {\rm terms}\!\!\!\!\!})$ and $\hat{\BFy}^r=(\underbrace{1,\ldots,1}_{r\ {\rm terms}}, \underbrace{0,\ldots,0}_{T-r\ {\rm terms}\!\!\!\!\!})$ if $t-r-1\notin\Se$;}\\
        \multicolumn{17}{l}{
        $\hat{\BFx}^r=(\underbrace{0,\ldots,0}_{r\ {\rm terms}}, \underbrace{\Vupper,\Vupper+V,\Vupper+2V,\ldots,\Vupper+(t-r-1)V}_{t-r\ {\rm terms}}, \underbrace{\Vupper+(t-r-1)V,\ldots,\Vupper+(t-r-1)V}_{T-t\ {\rm terms}\!\!\!\!\!})$
        and
        $\hat{\BFy}^r=(\underbrace{0,\ldots,0}_{r\ {\rm terms}},\underbrace{1,\ldots,1}_{T-r\ {\rm terms}\!\!\!\!\!})$
        if $t-r-1\in\Se$.}
    \end{tabular}
    \label{tab:eqn-T:x_t-ub-1-1-facet-matrix-1}
\end{table}
\end{landscape}

\begin{landscape}
\begin{table}
    \renewcommand{\arraystretch}{2}
    \centering
    \caption{Lower triangular matrix obtained from Table~\ref{tab:eqn-T:x_t-ub-1-1-facet-matrix-1} via Gaussian elimination}
    \rule{0pt}{4ex}
    \setlength\tabcolsep{6.5pt}
    \scriptsize
    \begin{tabular}{|c|c|c|*{7}{c}|*{7}{c}|}
        \hline
         \multirow{2}{*}{Group} & \multirow{2}{*}{Point} & \multirow{2}{*}{Index $r$} & \multicolumn{7}{c|}{$\BFx$} & \multicolumn{7}{c|}{$\BFy$} \\
         \cline{4-9}\cline{10-17}
         &&& $1$ & $\cdots$ & $t-1$ & $\ t\ $ & $t+1$ & $\cdots$ & $T$ 
           & $\ \ 1\ \ $ & $\cdots$ & $t-1$ & $\ t\ $ & $t+1$ & $\cdots$ & $\ \ T\ \ $ \\
         \hline
         \multirow{6}{*}{(B1)} & \multirow{6}{*}{$(\underline{\bar{\BFx}}^r,\underline{\bar{\BFy}}^r)$}
         & $1$ 
         & $-\epsilon$ & $\cdots$ & $0$ & $0$ & $0$ & $\cdots$ & $0$
         & $0$ & $\cdots$ & $0$ & $0$ & $0$ & $\cdots$ & $0$ \\ 
         && $\vdots$ 
         & $\vdots$ & \rotatebox{0}{$\ddots$} & $\vdots$ & $\vdots$ & $\vdots$ && $\vdots$ 
         & $\vdots$ && $\vdots$ & $\vdots$ & $\vdots$ && $\vdots$ \\        
         && $t-1$
         & $0$ & $\cdots$ & $-\epsilon$ & $0$ & $0$ & $\cdots$ & $0$
         & $0$ & $\cdots$ & $0$ & $0$ & $0$ & $\cdots$ & $0$ \\
         && $t+1$
         & $0$ & $\cdots$ & $0$ & $0$ & $-\epsilon$ & $\cdots$ & $0$ 
         & $0$ & $\cdots$ & $0$ & $0$ & $0$ & $\cdots$ & $0$ \\
         && $\vdots$ 
         & $\vdots$ && $\vdots$ & $\vdots$ & $\vdots$ & \rotatebox{0}{$\ddots$} & $\vdots$ 
         & $\vdots$ && $\vdots$ & $\vdots$ & $\vdots$ && $\vdots$ \\        
         && $T$
         & $0$ & $\cdots$ & $0$ & $0$ & $0$ & $\cdots$ & $-\epsilon$ 
         & $0$ & $\cdots$ & $0$ & $0$ & $0$ & $\cdots$ & $0$ \\

         \hline
         \multirow{3}{*}{\rred{(B2)}} & \multirow{7}{*}{$(\underline{\hat{\BFx}}^r,\underline{\hat{\BFy}}^r)$}
         & $1$
         &\multicolumn{7}{c|}{}
         & $\rred 1$ & $\rred \cdots$ & $\rred 0$ & $\rred 0$ & $\rred 0$ & $\rred \cdots$ & $\rred 0$ \\
         && $\vdots$ 
         & \multicolumn{7}{c|}{(Omitted)} 
         & $\rred \vdots$ & \rotatebox{0}{$\rred \ddots$} & $\rred \vdots$ & $\rred \vdots$ & $\rred \vdots$ && $\rred \vdots$ \\
         && $t-1$
         & \multicolumn{7}{c|}{}
         & $\rred 1$ & $\rred \cdots$ & $\rred 1$ & $\rred 0$ & $\rred 0$ & $\rred \cdots$ & $\rred 0$ \\
         
         \cline{1-1} \cline{3-17}
         \rred{(B3)} &  
         & $t$
         & \multicolumn{7}{c|}{(Omitted)}
         & $1$ & $\cdots$ & $1$ & $1$ & $0$ & $\cdots$ & $0$ \\
         \cline{1-1} \cline{3-17}
         
         \multirow{3}{*}{\rred{(B4)}} & 
         & $t+1$
         &\multicolumn{7}{c|}{}
         & $0$ & $\cdots$ & $0$ & $0$ & $1$ & $\cdots$ & $0$ \\
         && $\vdots$ 
         & \multicolumn{7}{c|}{(Omitted)} 
         & $\vdots$ && $\vdots$ & $\vdots$ & $\vdots$ & \rotatebox{0}{$\ddots$} & $\vdots$ \\ 
         && $T$
         & \multicolumn{7}{c|}{}
         & $0$ & $\cdots$ & $0$ & $0$ & $0$ & $\cdots$ & $1$ \\
        \hline
    \end{tabular}
    \label{tab:eqn-T:x_t-ub-1-1-facet-matrix-2}
\end{table}
\end{landscape}
}

Next, we show that inequality \eqref{eqn-T:x_t-ub-1-2} is valid and facet-defining for $\convP$.
Note that this proof can also be used to prove the validity and facet-defining of inequality \eqref{eqn-T:x_t-ub-1-2} in Proposition~\ref{prop-T:x_t-ub-2}.
Denote $x'_t=x_{T-t+1}$ and $y'_t=y_{T-t+1}$ for $t\in[1,T]_\Z$.
Because inequality \eqref{eqn-T:x_t-ub-1-1} is valid and facet-defining for $\convP$ for any $t\in[s_{\max}+2,T]_\Z$, the inequality
$$x'_{T-t+1}\le\Cupper y'_{T-t+1}-\sum_{s\in\Se}(\Cupper-\Vupper-sV)(y'_{T-t+s+1}-y'_{T-t+s+2})$$
is valid and facet-defining for $\convPprime$ for any $t\in[s_{\max}+2,T]_\Z$.
Let $t'=T-t+1$.
Then, the inequality
$$x'_{t'}\le\Cupper y'_{t'}-\sum_{s\in\Se}(\Cupper-\Vupper-sV)(y'_{t'+s}-y'_{t'+s+1})$$
is valid and facet-defining for $\convPprime$ for any $t'\in[1,T-s_{\max}-1]_\Z$.
Hence, by Lemma~\ref{lem:Pprime}, inequality \eqref{eqn-T:x_t-ub-1-2} is valid and facet-defining for $\convP$ for any $t\in[1,T-s_{\max}-1]_\Z$.
\Halmos

\subsection{Proof of Proposition \ref{prop-T:x_t-ub-2}}

\noindent{\bf Proposition \ref{prop-T:x_t-ub-2}.} {\it
Consider any integers $\alpha$, $\beta$, and $s_{\max}$ such that
(a)~$L\le s_{\max}\le\min\{T-2,\lfloor(\Cupper-\Vupper)/V\rfloor\}$,
(b)~$0\le\alpha<\beta\le s_{\max}$, and
(c)~$\beta=\alpha+1$ or $s_{\max}\le L+\alpha$.
Let $\Se=[0,\alpha]_{\Z}\cup[\beta,s_{\max}]_{\Z}$.
For any $t\in[s_{\max}+2,T]_{\Z}$, inequality \eqref{eqn-T:x_t-ub-1-1} is valid and facet-defining for $\convP$.
For any $t\in[1,T-s_{\max}-1]_{\Z}$, inequality \eqref{eqn-T:x_t-ub-1-2} is valid and facet-defining for $\convP$.}
\vskip8pt

\noindent{\bf Proof.}
Consider any $t\in[s_{\max}+2,T]_\Z$.
To prove that the linear inequality \eqref{eqn-T:x_t-ub-1-1} is valid for $\convP$ when $\Se=[0,\alpha]_\Z\cup[\beta,s_{\max}]_\Z$, it suffices to show that \eqref{eqn-T:x_t-ub-1-1} is valid for $\Pe$ when $\Se=[0,\alpha]_\Z\cup[\beta,s_{\max}]_\Z$.
Consider any element $(\BFx,\BFy)$ of $\Pe$.
We show that $(\BFx,\BFy)$ satisfies \eqref{eqn-T:x_t-ub-1-1} when $\Se=[0,\alpha]_\Z\cup[\beta,s_{\max}]_\Z$.
We divide the analysis into four cases.

Case~1: $y_t=0$ and $y_{t-s}-y_{t-s-1}\le0$ for all $s\in\Se$.
Because $\Se\subseteq[0,\lfloor(\Cupper-\Vupper)/V\rfloor]_\Z$, we have $\Cupper-\Vupper-sV\ge0$ for all $s\in\Se$.
Thus, in this case, the right-hand side of inequality \eqref{eqn-T:x_t-ub-1-1} is \rred{nonnegative}.
Because $y_t=0$, by \eqref{eqn:p-upper-bound}, $x_t=0$.
Therefore, in this case, $(\BFx,\BFy)$ satisfies \eqref{eqn-T:x_t-ub-1-1}.

Case~2: $y_t=0$ and $y_{t-s}-y_{t-s-1}>0$ for some $s\in\Se$.
Let $\tilde{\Se}=\{\sigma\in\Se:y_{t-\sigma}-y_{t-\sigma-1}>0\}$ and $v=|\tilde{\Se}|$.
Then, $v\ge1$.
Denote $\tilde{\Se}=\{\sigma_1,\sigma_2,\ldots,\sigma_v\}$, where $\sigma_1<\sigma_2<\cdots<\sigma_v$.
Note that $y_{t-\sigma_j-1}=0$ and $y_{t-\sigma_j}=1$ for $j=1,\ldots,v$.
Denote $\sigma_0=-1$.
Then for each $j=1,\ldots,v$, there exists $\sigma'_j\in[\sigma_{j-1}+1,\sigma_j-1]_{\Z}$ such that $y_{t-\sigma'_j-1}=1$ and $y_{t-\sigma'_j}=0$.
Thus,
$$0\le\sigma'_1<\sigma_1<\sigma'_2<\sigma_2\cdots<\sigma'_v<\sigma_v\le s_{\max}.$$
Because $y_{t-\sigma_v}-y_{t-\sigma_v-1}=1$ and $t-\sigma_v\in[2,T]_\Z$, by \eqref{eqn:p-minup}, $y_k=1$ for all $k\in[t-\sigma_v,\min\{T,t-\sigma_v+L-1\}]_{\Z}$, which implies that $t-\sigma'_j\ge t-\sigma_v+L$ for $j=1,\ldots,v$.
Hence, for $j=1,\ldots,v$, we have $\sigma'_j\le \sigma_v-L$, which implies that 
\begin{equation}\label{eq:Prop3_eqA}
\sigma'_j\le s_{\max}-L.
\end{equation}
If $\beta=\alpha+1$, then $\Se=[0,s_{\max}]_{\Z}$, which implies that $\sigma'_j\in\Se$ for $j=1,\ldots,v$.
If $\beta\neq\alpha+1$, then condition (c) of Proposition \ref{prop-T:x_t-ub-2} implies that $s_{\max}\le L+\alpha$, which, by \eqref{eq:Prop3_eqA}, implies that $\sigma'_j\le\alpha$ for $j=1,\ldots,v$.
Thus, in both cases, $\sigma'_j\in\Se$ for $j=1,\ldots,v$.
Because $y_t=0$, by \eqref{eqn:p-upper-bound}, $x_t=0$.
Hence, the left-hand side of inequality \eqref{eqn-T:x_t-ub-1-1} is $0$.
Because $\Se\subseteq[0,\lfloor(\Cupper-\Vupper)/V\rfloor]_\Z$, we have $\Cupper-\Vupper-sV\ge0$ for all $s\in\Se$.
Note that $\{\sigma'_1,\ldots,\sigma'_v\}\subseteq\Se\setminus\tilde{\Se}$ and $y_{t-s}-y_{t-s-1}\le0$ for all $s\in\Se\setminus\tilde{\Se}$.
Thus, $\sum_{s\in\Se\setminus\tilde{\Se}}(\Cupper-\Vupper-sV)(y_{t-s}-y_{t-s-1})\le\sum_{j=1}^v(\Cupper-\Vupper-\sigma'_jV)(y_{t-\sigma'_j}-y_{t-\sigma'_j-1})$.
Hence, the right-hand side of inequality \eqref{eqn-T:x_t-ub-1-1} is
\begin{align*}
    &\Cupper y_t-\sum_{s\in\Se}(\Cupper-\Vupper-sV)(y_{t-s}-y_{t-s-1})\\
    &= -\sum_{s\in\tilde{\Se}}(\Cupper-\Vupper-sV)(y_{t-s}-y_{t-s-1})-\sum_{s\in\Se\setminus\tilde{\Se}}(\Cupper-\Vupper-sV)(y_{t-s}-y_{t-s-1})\\
    &\ge -\sum_{j=1}^v(\Cupper-\Vupper-\sigma_jV)(y_{t-\sigma_j}-y_{t-\sigma_j-1})-\sum_{j=1}^v(\Cupper-\Vupper-\sigma'_jV)(y_{t-\sigma_j'}-y_{t-\sigma'_j-1})\\
    &= -\sum_{j=1}^v(\Cupper-\Vupper-\sigma_jV)+\sum_{j=1}^v(\Cupper-\Vupper-\sigma'_jV)\\
    &=\sum_{j=1}^v(\sigma_j-\sigma'_j)V\\
    &>0.
\end{align*}
Therefore, in this case, $(\BFx,\BFy)$ satisfies \eqref{eqn-T:x_t-ub-1-1}.

Case~3: $y_t=1$ and $y_{t-s}-y_{t-s-1}\le0$ for all $s\in\Se$.
Because $\Se\subseteq[0,\lfloor(\Cupper-\Vupper)/V\rfloor]_{\Z}$, we have $\Cupper-\Vupper-sV\ge0$ for all $s\in\Se$.
Thus, in this case, the right-hand side of inequality \eqref{eqn-T:x_t-ub-1-1} is at least $\Cupper$.
By \eqref{eqn:p-upper-bound}, $x_t\le\Cupper$.
Therefore, in this case, $(\BFx,\BFy)$ satisfies \eqref{eqn-T:x_t-ub-1-1}.

Case~4: $y_t=1$ and $y_{t-s}-y_{t-s-1}>0$ for some $s\in\Se$.
Let $\tilde{\Se}=\{\sigma\in\Se:y_{t-\sigma}-y_{t-\sigma-1}>0\}$ and $v=|\tilde{\Se}|$.
Then, $v\ge1$.
Denote $\tilde{\Se}=\{\sigma_1,\sigma_2,\ldots,\sigma_v\}$, where $\sigma_1<\sigma_2<\cdots<\sigma_v$.
Note that $y_{t-\sigma_j-1}=0$ and $y_{t-\sigma_j}=1$ for $j=1,\ldots,v$.
Then, for each $j=2,\ldots,v$, there exists $\sigma'_j\in[\sigma_{j-1}+1,\sigma_j-1]_{\Z}$ such that $y_{t-\sigma'_j-1}=1$ and $y_{t-\sigma'_j}=0$.
Thus,
$$0\le\sigma_1<\sigma'_2<\sigma_2<\cdots<\sigma'_v<\sigma_v\le s_{\max}.$$
In addition, $y_k=1$ for all $k\in[t-\sigma_1,t]_{\Z}$.
Because $y_{t-\sigma_v}-y_{t-\sigma_v-1}=1$ and $t-\sigma_v\in[2,T]_\Z$, by \eqref{eqn:p-minup}, $y_k=1$ for all $k\in[t-\sigma_v,\min\{T,t-\sigma_v+L-1\}]_{\Z}$, which implies that $t-\sigma'_j\ge t-\sigma_v+L$ for $j=2,\ldots,v$.
Hence, for $j=2,\ldots,v$, we have $\sigma'_j\le\sigma_v-L$, which implies that
\begin{equation}\label{eq:Prop3_eqB}
    \sigma'_j\le s_{\max}-L.
\end{equation}
If $\beta=\alpha+1$, then $\Se=[0,s_{\max}]_{\Z}$, which implies that $\sigma'_j\in\Se$ for $j=2,\ldots,v$.
If $\beta\neq\alpha+1$, then condition (c) of Proposition \ref{prop-T:x_t-ub-2} implies that $s_{\max}\le L+\alpha$, which, by \eqref{eq:Prop3_eqB} implies that $\sigma'_j\le\alpha$ for all $j=2,\ldots,v$.
Thus, in both cases, $\sigma'_j\in\Se$ for $j=2,\ldots,v$.
By \eqref{eqn:p-ramp-up},
$$\sum_{\tau=t-\sigma_1}^t(x_{\tau}-x_{\tau-1})\le\sum_{\tau=t-\sigma_1}^t Vy_{\tau-1}+\sum_{\tau=t-\sigma_1}^t\Vupper(1-y_{\tau-1}),$$
which implies that
$$x_t-x_{t-\sigma_1-1}\le\sum_{\tau=t-\sigma_1}^tVy_{\tau-1}+\sum_{\tau=t-\sigma_1}^t\Vupper(1-y_{\tau-1})=\sigma_1V+\Vupper.$$
Because $y_{t-\sigma_1-1}=0$, by \eqref{eqn:p-upper-bound}, $x_{t-\sigma_1-1}=0$.
Hence, $x_t\le\sigma_1V+\Vupper$;
that is, the left-hand side of inequality \eqref{eqn-T:x_t-ub-1-1} is at most $\sigma_1V+\Vupper$.
Because $\Se\subseteq[0,\lfloor(\Cupper-\Vupper)/V\rfloor]_\Z$, we have $\Cupper-\Vupper-sV\ge0$ for all $s\in\Se$.
Note that $\{\sigma'_2,\ldots,\sigma'_v\}\subseteq\Se\setminus\tilde{\Se}$ and $y_{t-s}-y_{t-s-1}\le0$ for all $s\in\Se\setminus\tilde{\Se}$.
Thus, $\sum_{s\in\Se\setminus\tilde{\Se}}(\Cupper-\Vupper-sV)(y_{t-s}-y_{t-s-1})\le\sum_{j=2}^v(\Cupper-\Vupper-\sigma'_jV)(y_{t-\sigma'_j}-y_{t-\sigma'_j-1})$.
Hence, the right-hand side of inequality \eqref{eqn-T:x_t-ub-1-1} is 
\begin{align*}
    &\Cupper y_t-\sum_{s\in\Se}(\Cupper-\Vupper-sV)(y_{t-s}-y_{t-s-1})\\
    &=\Cupper-\sum_{s\in\tilde{\Se}}(\Cupper-\Vupper-sV)(y_{t-s}-y_{t-s-1})-\sum_{s\in\Se\setminus\tilde{\Se}}(\Cupper-\Vupper-sV)(y_{t-s}-y_{t-s-1})\\
    &\ge\Cupper-\sum_{j=1}^v(\Cupper-\Vupper-\sigma_jV)(y_{t-\sigma_j}-y_{t-\sigma_j-1})-\sum_{j=2}^v(\Cupper-\Vupper-\sigma'_jV)(y_{t-\sigma'_j}-y_{t-\sigma'_j-1})\\
    &=\Cupper-\sum_{j=1}^v(\Cupper-\Vupper-\sigma_jV)+\sum_{j=2}^v(\Cupper-\Vupper-\sigma'_jV)\\
    &=\sigma_1V+\Vupper+\sum_{j=2}^v(\sigma_j-\sigma'_j)V\\
    &\ge\sigma_1V+\Vupper.
\end{align*}
Therefore, in this case, $(\BFx,\BFy)$ satisfies \eqref{eqn-T:x_t-ub-1-1}.

Summarizing Cases 1--4, we conclude that $(\BFx,\BFy)$ satisfies \eqref{eqn-T:x_t-ub-1-1}. 
Hence, \eqref{eqn-T:x_t-ub-1-1} is valid for $\convP$ when $\Se=[0,\alpha]_\Z\cup[\beta,s_{\max}]_\Z$.

It is easy to verify that the proof of facet-defining of inequality \eqref{eqn-T:x_t-ub-1-1} in the proof of Proposition~\ref{prop-T:x_t-ub-1} remains valid when $\Se=[0,\alpha]_\Z\cup[\beta,s_{\max}]_\Z$.
Therefore, inequality \eqref{eqn-T:x_t-ub-1-1} is facet-defining for $\convP$ under the conditions stated in Proposition \ref{prop-T:x_t-ub-2}.

It is also easy to verify that the proof of validity and facet-defining of inequality \eqref{eqn-T:x_t-ub-1-2} in the proof of Proposition~\ref{prop-T:x_t-ub-1} remains valid when $\Se=[0,\alpha]_\Z\cup[\beta,s_{\max}]_\Z$.
Therefore, inequality \eqref{eqn-T:x_t-ub-1-2} is valid and facet-defining for $\convP$ under the conditions stated in Proposition \ref{prop-T:x_t-ub-2}.
\Halmos


\subsection{Proof of Proposition \ref{prop-T:x_t-ub-3}}

\noindent{\bf Proposition \ref{prop-T:x_t-ub-3}.} {\it
Consider any set $\Se\subseteq[0,\min\{L-1,T-3,\lfloor(\Cupper-\Vupper)/V\rfloor\}]_{\Z}$ and any real number $\eta$ such that $0\le\eta\le\min\{L-1,(\Cupper-\Vupper)/V\}$.
For any $t\in[1,T-1]_\Z$ such that $t\ge s+2$ for all $s\in\Se$, the inequality
\begin{equation}
    x_t \leq (\Cupper - \eta V) y_t + \eta V y_{t+1} - \sum_{s\in\Se} (\Cupper - \Vupper - sV) (y_{t-s} - y_{t-s-1})\tag{\ref{eqn-T:x_t-ub-3-1}}
\end{equation}
is valid for $\convP$.
For any $t\in[2,T]_\Z$ such that $t\le T-s-1$ for all $s\in\Se$, the inequality
\begin{equation}
    x_t \leq (\Cupper - \eta V) y_t + \eta V y_{t-1} - \sum_{s\in\Se} (\Cupper - \Vupper - sV) (y_{t+s} - y_{t+s+1})\tag{\ref{eqn-T:x_t-ub-3-2}}
\end{equation}
is valid for $\convP$.
Furthermore, inequalities \eqref{eqn-T:x_t-ub-3-1} and \eqref{eqn-T:x_t-ub-3-2} are facet-defining for $\convP$ when $\eta\in\{0, (\Cupper-\Vupper)/V\}$ or $\eta=L-1\in\Se$.}
\vskip8pt

\noindent{\bf Proof.}
We first prove that inequality \eqref{eqn-T:x_t-ub-3-1} is valid and facet-defining for $\convP$. 
Note that the proof of facet-defining of \eqref{eqn-T:x_t-ub-3-1} here can also be used to prove the facet-defining of \eqref{eqn-T:x_t-ub-3-1} in Proposition~\ref{prop-T:x_t-ub-4}.
For notational convenience, we define $s_{\max}=\max\{s:s\in\Se\}$ if $\Se\ne\emptyset$, and $s_{\max}=-1$ if $\Se=\emptyset$.

Consider any $t\in[s_{\max}+2,T-1]_\Z$ (i.e., $t\in[1,T-1]_\Z$ such that $t\ge s+2$ for all $s\in\Se$).
To prove that the linear inequality \eqref{eqn-T:x_t-ub-3-1} is valid for $\convP$, it suffices to show that it is valid for $\Pe$.
Consider any element $(\BFx,\BFy)$ of $\Pe$.
We show that $(\BFx,\BFy)$ satisfies \eqref{eqn-T:x_t-ub-3-1}.
We divide the analysis into three cases.

Case~1: $y_t=0$.
By Lemma \ref{lem:lookbackward}(i), $y_{t-j}-y_{t-j-1}\le0$ for all $j\in[0,\min\{t-2,L-1\}]_{\Z}$.
Because $\Se\subseteq[0,L-1]_\Z$ and $s_{\max}\le t-2$, we have $\Se\subseteq[0,\min\{t-2,L-1\}]_{\Z}$.
Thus, $y_{t-s}-y_{t-s-1}\le0$ for all $s\in\Se$.
Because $\Se\subseteq[0,\lfloor(\Cupper-\Vupper)/V\rfloor]_\Z$, we have $\Cupper-\Vupper-sV\ge0$ for all $s\in\Se$.
Thus, the right-hand side of \eqref{eqn-T:x_t-ub-3-1} is at least $\eta Vy_{t+1}\ge 0$.
Because $y_t=0$, by \eqref{eqn:p-upper-bound}, $x_t=0$.
Therefore, in this case, $(\BFx,\BFy)$ satisfies \eqref{eqn-T:x_t-ub-3-1}.

Case~2: $y_t=1$ and $y_{t-s'}-y_{t-s'-1}=1$ for some $s'\in\Se$.
In this case, $y_{t-s'}=1$ and $y_{t-s'-1}=0$.
Because $s'\le s_{\max}\le t-2$, we have $t-s'\in[2,T]_\Z$.
By \eqref{eqn:p-minup}, $y_k=1$ for all $k\in[t-s',\min\{T,t-s'+L-1\}]_\Z$.
Because $\Se\subseteq[0,L-1]_\Z$, we have $s'\le L-1$, or equivalently $t-s'+L-1\ge t$, and thus $y_{t-s}=1$ for all $s\in[0,s']_\Z$,
which implies that $y_{t-s}-y_{t-s-1}=0$ for all $s\in[0,s'-1]_\Z$.
Because $s'\le t-2$, either $s'=t-2$ or $s'\le t-3$.
If $s'=t-2$, then it does not exist any $s\in\Se$ such that $s>s'$.
If $s'\le t-3$, then $t-s'-1\in[2,T]_\Z$, and by Lemma~\ref{lem:lookbackward}(i), $y_{t-s'-j-1}-y_{t-s'-j-2}\le 0$ for all $j\in[0,\min\{t-s'-3,L-1\}]_\Z$, 
which implies that $y_{t-s}-y_{t-s-1}\le 0$ for all $s\in[s'+1,\min\{t-2,L+s'\}]_\Z$,
which in turn implies that $y_{t-s}-y_{t-s-1}\le 0$ for all $s\in\Se$ such that $s>s'$.
Hence, $y_{t-s}-y_{t-s-1}\le 0$ for all $s\in\Se\setminus\{s'\}$.
Because $\Se\subseteq[0,\lfloor(\Cupper-\Vupper)/V\rfloor]_\Z$, we have $\Cupper-\Vupper-sV\ge0$ for all $s\in\Se$.
Thus,
\begin{equation}\label{eq1:apx:prop-T:x_t-ub-3}
-\sum_{s\in\Se}(\Cupper-\Vupper-sV)(y_{t-s}-y_{t-s-1})\ge-(\Cupper-\Vupper-s'V).
\end{equation}
Because $y_t=1$, by \eqref{eq1:apx:prop-T:x_t-ub-3}, the right-hand side of \eqref{eqn-T:x_t-ub-3-1} is at least $s'V+\Vupper$ when $y_{t+1}=1$ and is at least $\Vupper+(s'-\eta)V$ when $y_{t+1}=0$.
By \eqref{eqn:p-upper-bound}, $x_{t-s'-1}=0$.
By \eqref{eqn:p-ramp-up}, $\sum_{\tau=t-s'}^t(x_{\tau}-x_{\tau-1})\le\sum_{\tau=t-s'}^tVy_{\tau-1}+\sum_{\tau=t-s'}^t\Vupper(1-y_{\tau-1})$,
which implies that $x_t\le s'V+\Vupper$.
If $y_{t+1}=0$, then by \eqref{eqn:p-upper-bound} and \eqref{eqn:p-ramp-down}, $x_{t+1}=0$ and $x_t-x_{t+1}\le Vy_{t+1}+\Vupper(1-y_{t+1})$,
implying that $x_t\le\Vupper$.
In addition, if $y_{t+1}=0$, then because $y_k=1$ for all $k\in[t-s',\min\{T,t-s'+L-1\}]_\Z$, we have $t+1\ge t-s'+L$, 
which implies that $s'\ge L-1\ge\eta$.
Thus, if $y_{t+1}=0$, then $x_t\le\Vupper+(s'-\eta)V$.
Hence, $x_t$ is at most $s'V+\Vupper$, and it is at most $\Vupper+(s'-\eta)V$ when $y_{t+1}=0$.
Therefore, in this case, $(\BFx,\BFy)$ satisfies \eqref{eqn-T:x_t-ub-3-1}.

Case~3: $y_t=1$ and $y_{t-s}-y_{t-s-1}\neq1$ for all $s\in\Se$.
In this case, $y_{t-s}-y_{t-s-1}\le0$ for all $s\in\Se$.
Because $\Se\subseteq[0,\lfloor(\Cupper-\Vupper)/V\rfloor]_{\Z}$, we have $\Cupper-\Vupper-sV\ge0$ for all $s\in\Se$.
Thus, $\sum_{s\in\Se}(\Cupper-\Vupper-sV)(y_{t-s}-y_{t-s-1})\le0$.
Because $y_t=1$, the right-hand side of \eqref{eqn-T:x_t-ub-3-1} is at least $\Cupper$ when $y_{t+1}=1$ and is at least $\Cupper-\eta V\ge\Vupper$ when $y_{t+1}=0$ (as $\eta\le(\Cupper-\Vupper)/V$).
By \eqref{eqn:p-upper-bound}, $x_t\le\Cupper$.
If $y_{t+1}=0$, then by \eqref{eqn:p-upper-bound} and \eqref{eqn:p-ramp-down}, $x_{t+1}=0$ and $x_t-x_{t+1}\le Vy_{t+1}+\Vupper(1-y_{t+1})$, 
which imply that $x_t\le\Vupper$.
Hence, $x_t$ is at most $\Cupper$, and it is at most $\Vupper$ when $y_{t+1}=0$.
Therefore, in this case, $(\BFx,\BFy)$ satisfies \eqref{eqn-T:x_t-ub-3-1}.

Summarizing Cases 1--3, we conclude that $(\BFx,\BFy)$ satisfies \eqref{eqn-T:x_t-ub-3-1}.
Hence, \eqref{eqn-T:x_t-ub-3-1} is valid for $\convP$.

We first show that $\convP$ is full dimensional.
As the $\convP$ contains $2T$ decision variables, to show $\dimP=2T$, we need to find $2T+1$ affinely independent points in $\convP$.
Because $\BFzero\in\convP$, it suffices to create the remaining $2T$ nonzero linearly independent points.
We denote these $2T$ points as $(\bar{\BFx}^r, \bar{\BFy}^r)$ and $(\hat{\BFx}^r, \hat{\BFy}^r)$, and denote the $q$th component of $\bar{\BFx}^r$, $\bar{\BFy}^r$, $\hat{\BFx}^r$, and $\hat{\BFy}^r$ as $\bar{x}^r_q, \bar{y}^r_q, \hat{x}^r_q$ and $\hat{y}^r_q$, respectively.
Let $\epsilon=\Vupper- \Clower>0$.
\begin{enumerate}[label={(A\arabic*)}]
\item For each $r\in[1,T]_{\Z}$, we create a point $(\bar{\BFx}^r, \bar{\BFy}^r)$ as follows:
\begin{equation*}
    \bar{x}_q^r = \left\{\begin{array}{ll}
    \Clower + \epsilon, & q \in [1, r]_{\Z}; \\
    0, & q \in [r+1, T]_{\Z};
    \end{array} \right.
\end{equation*}
and
\begin{equation*}
\bar{y}_q^r = \left\{ \begin{array}{ll}
    1, & q \in [1,r]_{\Z}; \\
    0, & q \in [r+1, T]_{\Z}.
    \end{array} \right.
\end{equation*}
It is easy to observe that $(\bar{\BFx}^r, \bar{\BFy}^r)$ satisfies \eqref{eqn:p-minup}--\eqref{eqn:p-ramp-down} and thus $(\bar{\BFx}^r, \bar{\BFy}^r)\in\convP$ for $r\in[1,T]_{\Z}$.
\item For each $r\in[1,T]_{\Z}$, we create a point $(\hat{\BFx}^r, \hat{\BFy}^r)$ as follows:
\begin{equation*}
    \hat{x}_q^r = \left\{\begin{array}{ll}
    \Clower, & q \in [1, r]_{\Z}; \\
    0, & q \in [r+1, T]_{\Z};
    \end{array} \right.
\end{equation*}
and
\begin{equation*}
    \hat{y}_q^r = \left\{ \begin{array}{ll}
    1, & q \in [1,r]_{\Z}; \\
    0, & q \in [r+1, T]_{\Z}.
    \end{array} \right.
\end{equation*}
It is easy to observe that $(\hat{\BFx}^r, \hat{\BFy}^r)$ satisfies \eqref{eqn:p-minup}--\eqref{eqn:p-ramp-down} and thus $(\hat{\BFx}^r, \hat{\BFy}^r)\in\convP$ for $r\in[1,T]_{\Z}$.
\end{enumerate}
It is also easy to observe that the above $2T$ points are linearly independent.
Therefore, the dimension of $\convP$ is $2T$, i.e., $\convP$ is full dimensional.
Consider any $t\in[s_{\max}+2,T-1]_\Z$.
To prove that inequality \eqref{eqn-T:x_t-ub-3-1} is facet-defining for $\convP$ when $\eta\in\{0,(\Cupper-\Vupper)/V\}$ or $\eta=L-1\in\Se$, 
it suffices to show that there exist $2T$ affinely independent points in $\convP$ that satisfy \eqref{eqn-T:x_t-ub-3-1} at equality when $\eta\in\{0,(\Cupper-\Vupper)/V\}$ or $\eta=L-1\in\Se$.
When $\eta=0$, inequalities \eqref{eqn-T:x_t-ub-3-1} become inequality \eqref{eqn-T:x_t-ub-1-1}, and by Proposition~\ref{prop-T:x_t-ub-1}, it is facet-defining for $\convP$.
Hence, in the following, we only consider the case where $\eta=(\Cupper-\Vupper)/V$ or $\eta=L-1\in\Se$.
Because $\BFzero\in\convP$ and $\BFzero$ satisfies \eqref{eqn-T:x_t-ub-3-1} at equality, 
it suffices to create the remaining $2T-1$ nonzero linearly independent points.
We denote these $2T-1$ points as $(\bar{\BFx}^r,\bar{\BFy}^r)$ for $r\in[1,T]_{\Z}\setminus\{t\}$ and $(\hat{\BFx}^r,\hat{\BFy}^r)$ for $r\in[1,T]_{\Z}$. 

We divide these $2T-1$ points into the following \rred{five} groups:
\begin{enumerate}[label={(A\arabic*)},ref={(A\arabic*)}]
\item\label{points:prop4-eq1-A1}
For each $r\in[1, T]_{\Z}\setminus\{t\}$, we create the same point $(\bar{\BFx}^r,\bar{\BFy}^r)$ as in group \ref{points:prop1-eq1-A1} in the proof of Proposition~\ref{prop-T:x_t-ub-1}.
Thus, $(\bar{\BFx}^r,\bar{\BFy}^r)\in\convP$.
It is easy to verify that $(\bar{\BFx}^r,\bar{\BFy}^r)$ satisfies \eqref{eqn-T:x_t-ub-3-1} at equality.

\item\label{points:prop4-eq1-A2}
For each \rred{$r\in[1,t-1]_{\Z}$}, we create the same point $(\hat{\BFx}^r,\hat{\BFy}^r)$ as in group \ref{points:prop1-eq1-A2} in the proof of Proposition~\ref{prop-T:x_t-ub-1}.
Thus, $(\hat{\BFx}^r,\hat{\BFy}^r)\in\convP$.
Consider the case where $t-r-1\notin\Se$.
In this case, $\hat{x}_t^r=\hat{y}_t^r=\hat{y}_{t+1}^r=0$.
In addition, $t-s-1\neq r$ for all $s\in\Se$, which implies that $\hat{y}^r_{t-s}-\hat{y}^r_{t-s-1}=0$ for all $s\in\Se$.
Hence, $(\hat{\BFx}^r,\hat{\BFy}^r)$ satisfies \eqref{eqn-T:x_t-ub-3-1} at equality.
Next, consider the case where $t-r-1\in\Se$.
In this case, $\hat{x}^r_t=\Vupper+(t-r-1)V$ and $\hat{y}^r_t=\hat{y}^r_{t+1}=1$.
In addition, $\hat{y}^r_{t-s}-\hat{y}^r_{t-s-1}=1$ when $s=t-r-1$, and $\hat{y}^r_{t-s}-\hat{y}^r_{t-s-1}=0$ when $s\neq t-r-1$.
Hence, $(\hat{\BFx}^r,\hat{\BFy}^r)$ satisfies \eqref{eqn-T:x_t-ub-3-1} at equality.

\item\label{points:prop4-eq1-A3}
We create a point $(\hat{\BFx}^t,\hat{\BFy}^t)$ as follows:
If $\eta=(\Cupper-\Vupper)/V$, then
\begin{equation*}
    \hat{x}_q^t=\left\{
    \begin{array}{ll}
    \Vupper, &\hbox{for $q\in[1,t]_{\Z}$};\\
    0, &\hbox{for $q\in[t+1,T]_{\Z}$};
    \end{array}\right.
\end{equation*}
and
\begin{equation*}
    \hat{y}_q^t=\left\{
    \begin{array}{ll}
    1, &\hbox{for $q\in[1,t]_{\Z}$}; \\
    0, &\hbox{for $q\in[t+1,T]_{\Z}$}.
    \end{array}\right.
\end{equation*}
If $\eta=L-1\in\Se$, then
\begin{equation*}
    \hat{x}_q^t=\left\{
    \begin{array}{ll}
    \Vupper, &\hbox{for $q\in[t-L+1,t]_\Z$};\\
    0, &\hbox{for $q\in[1,t-L]_{\Z}\cup[t+1,T]_\Z$};
    \end{array}\right.
\end{equation*}
and
\begin{equation*}
    \hat{y}_q^t=\left\{
    \begin{array}{ll}
    1, &\hbox{$q\in[t-L+1,t]_\Z$}; \\
    0, &\hbox{$q\in[1,t-L]_\Z\cup[t+1,T]_{\Z}$}.
    \end{array}\right.
\end{equation*}
We first consider the case where $\eta=(\Cupper-\Vupper)/V$.
It is easy to verify that $(\hat{\BFx}^t,\hat{\BFy}^t)$ satisfies \eqref{eqn:p-minup}--\eqref{eqn:p-ramp-down}.
Thus, $(\hat{\BFx}^t,\hat{\BFy}^t)\in\convP$.
In this case, $\hat{x}^t_t=\Vupper$, $\hat{y}^t_t=1$, and $\hat{y}^t_{t+1}=0$, and $\hat{y}^t_{t-s}-\hat{y}^t_{t-s-1}=0$ for all $s\in\Se$.
Hence, $(\hat{\BFx}^t,\hat{\BFy}^t)$ satisfies \eqref{eqn-T:x_t-ub-3-1} at equality.
Next, we consider the case where $\eta=L-1\in\Se$.
In this case, for any $q\in[2,T]_\Z$, $\hat{y}^t_q-\hat{y}^t_{q-1}\le 0$ if $q\ne t-L+1$, while $\hat{y}^t_q-\hat{y}^t_{q-1}=1$ and $\hat{y}^t_k=1$ for all $k\in[q,\min\{T,q+L-1\}]_\Z$ if $q=t-L+1$.
Thus, $(\hat{\BFx}^t,\hat{\BFy}^t)$ satisfies \eqref{eqn:p-minup}.
For any $q\in[2,T]_\Z$, $\hat{y}^t_{q-1}-\hat{y}^t_q\le 0$ if $q\ne t+1$, while $\hat{y}^t_{q-1}-\hat{y}^t_q=1$ and $\hat{y}^t_k=0$ for all $k\in[q,\min\{T,q+\ell-1\}]_\Z$ if $q=t+1$.
Thus, $(\hat{\BFx}^t,\hat{\BFy}^t)$ satisfies \eqref{eqn:p-mindn}.
It is easy to verify that $(\hat{\BFx}^t,\hat{\BFy}^t)$ satisfies \eqref{eqn:p-lower-bound}--\eqref{eqn:p-ramp-down}.
Thus, $(\hat{\BFx}^t,\hat{\BFy}^t)\in\convP$.
Note that $\hat{x}^t_t=\Vupper$, $\hat{y}^t_t=1$, $\hat{y}^t_{t+1}=0$, $\hat{y}^t_{t-s}-\hat{y}^t_{t-s-1}=0$ for all $s\in\Se\setminus\{L-1\}$, 
$\hat{y}^t_{t-L+1}-\hat{y}^t_{t-L}=1$, and $(\Cupper-\eta V)-(\Cupper-\Vupper-(L-1)V)=\Vupper$.
Thus, $(\hat{\BFx}^t,\hat{\BFy}^t)$ satisfies \eqref{eqn-T:x_t-ub-3-1} at equality.

\item\label{points:prop4-eq1-A4}
We create a point $(\hat{\BFx}^{t+1},\hat{\BFy}^{t+1})$ by setting $\hat{x}^{t+1}_q=\Cupper$ and $\hat{y}^{t+1}_q=1$ for $q\in[1,T]_\Z$.
It is easy to verify that $(\hat{\BFx}^{t+1},\hat{\BFy}^{t+1})$ satisfies \eqref{eqn:p-minup}--\eqref{eqn:p-ramp-down}.
Thus, $(\hat{\BFx}^{t+1},\hat{\BFy}^{t+1})\in\convP$.
It is also easy to verify that $(\hat{\BFx}^{t+1},\hat{\BFy}^{t+1})$ satisfies \eqref{eqn-T:x_t-ub-3-1} at equality.

\item\label{points:prop4-eq1-A5}
For each $r\in[t+2,T]_\Z$, we create the same point $(\hat{\BFx}^r,\hat{\BFy}^r)$ as in group \ref{points:prop1-eq1-A4} in the proof of Proposition \ref{prop-T:x_t-ub-1}.
Thus, $(\hat{\BFx}^r,\hat{\BFy}^r)\in\convP$.
It is easy to verify that $(\hat{\BFx}^r,\hat{\BFy}^r)$ satisfies \eqref{eqn-T:x_t-ub-3-1} at equality.
\end{enumerate}

Table \ref{tab:eqn-T:x_t-ub-3-1-facet-matrix-1} shows a matrix with $2T-1$ rows, where each row represents a point created by this process.
This matrix can be transformed into the matrix in Table \ref{tab:eqn-T:x_t-ub-3-1-facet-matrix-2} via the following Gaussian elimination process:

\begin{enumerate}[label=(\roman*)]
    \item For each $r\in[1,T]_{\Z}\setminus\{t\}$, the point with index $r$ in group (B1), denoted $(\underline{\bar{\BFx}}^r, \underline{\bar{\BFy}}^r)$, is obtained by setting $(\underline{\bar{\BFx}}^r, \underline{\bar{\BFy}}^r)=(\bar{\BFx}^r,\bar{\BFy}^r)-(\hat{\BFx}^{t+1},\hat{\BFy}^{t+1})$.
    Here, $(\bar{\BFx}^r,\bar{\BFy}^r)$ is the point with index $r$ in group \ref{points:prop4-eq1-A1}, and $(\hat{\BFx}^{t+1},\hat{\BFy}^{t+1})$ is the point in group \ref{points:prop4-eq1-A5}.

    \item For each \rred{$r\in[1,t-1]_{\Z}$}, the point with index $r$ in group \rred{(B2)}, denoted $(\underline{\hat{\BFx}}^r, \underline{\hat{\BFy}}^r)$, is obtained by setting $(\underline{\hat{\BFx}}^r, \underline{\hat{\BFy}}^r)=(\hat{\BFx}^r,\hat{\BFy}^r)$ if $t-r-1\notin\Se$, and setting $(\underline{\hat{\BFx}}^r, \underline{\hat{\BFy}}^r)=\rred{(\hat{\BFx}^{t+1},\hat{\BFy}^{t+1})-(\hat{\BFx}^r,\hat{\BFy}^r)}$ if $t-r-1\in\Se$.
    Here, $(\hat{\BFx}^r,\hat{\BFy}^r)$ is the point with index $r$ in group \ref{points:prop4-eq1-A2}, and $(\hat{\BFx}^{t+1},\hat{\BFy}^{t+1})$ is the point in group \ref{points:prop4-eq1-A4}.

    \item The point in group \rred{(B3)}, denoted $(\underline{\hat{\BFx}}^t,\underline{\hat{\BFy}}^t)$, is obtained by setting $(\underline{\hat{\BFx}}^t,\underline{\hat{\BFy}}^t)=(\hat{\BFx}^t,\hat{\BFy}^t)$.
    Here, $(\hat{\BFx}^t,\hat{\BFy}^t)$ is the point in group \ref{points:prop4-eq1-A3}.

    \item The point in group \rred{(B4)}, denoted $(\underline{\hat{\BFx}}^{t+1},\underline{\hat{\BFy}}^{t+1})$, is obtained by setting $(\underline{\hat{\BFx}}^{t+1},\underline{\hat{\BFy}}^{t+1})=(\hat{\BFx}^{t+1},\hat{\BFy}^{t+1})-(\hat{\BFx}^{t+2},\hat{\BFy}^{t+2})$.
    Here, $(\hat{\BFx}^{t+1},\hat{\BFy}^{t+1})$ is the point in group \ref{points:prop4-eq1-A4}, and $(\hat{\BFx}^{t+2},\hat{\BFy}^{t+2})$ is the point with index $t+2$ in group \ref{points:prop4-eq1-A5}.

    \item For each $r\in[t+2,T]_{\Z}$, the point with index $r$ in group \rred{(B5)}, denoted $(\underline{\hat{\BFx}}^r, \underline{\hat{\BFy}}^r)$, is obtained by setting $(\underline{\hat{\BFx}}^r, \underline{\hat{\BFy}}^r)=(\hat{\BFx}^r,\hat{\BFy}^r)-(\hat{\BFx}^{r+1},\hat{\BFy}^{r+1})$ if $r\neq T$, and setting $(\underline{\hat{\BFx}}^r, \underline{\hat{\BFy}}^r)=(\hat{\BFx}^r,\hat{\BFy}^r)$ if $r=T$.
    Here, $(\hat{\BFx}^r,\hat{\BFy}^r)$ and $(\hat{\BFx}^{r+1},\hat{\BFy}^{r+1})$ are the points with indices $r$ and $r+1$, respectively, in group \ref{points:prop4-eq1-A5}.
\end{enumerate}

    

The matrix shown in Table \ref{tab:eqn-T:x_t-ub-3-1-facet-matrix-2} is lower triangular; that is, the position of the last nonzero component of a row of the matrix is greater than the position of the last nonzero component of the previous row.
This implies that the $2T-1$ points in groups \ref{points:prop4-eq1-A1}--\ref{points:prop4-eq1-A5} are linearly independent.
Therefore, inequality \eqref{eqn-T:x_t-ub-3-1} is facet-defining for $\convP$.

Next, we show that inequality \eqref{eqn-T:x_t-ub-3-2} is valid for $\convP$ and is facet-defining for $\convP$ when $\eta\in\{0,(\Cupper-\Vupper)/V\}$ or $\eta=L-1\in\Se$.
Note that this proof can also be used to prove the validity and facet-defining of inequality \eqref{eqn-T:x_t-ub-3-2} in Proposition~\ref{prop-T:x_t-ub-4}.
Denote $x'_t=x_{T-t+1}$ and $y'_t=y_{T-t+1}$ for $t\in[1,T]_\Z$.
Because inequality \eqref{eqn-T:x_t-ub-3-1} is valid for $\convP$ and is facet-defining for $\convP$ when $\eta\in\{0,(\Cupper-\Vupper)/V\}$ or $\eta=L-1\in\Se$ for any $t\in[s_{\max}+2,T-1]_\Z$,
the inequality
$$x'_{T-t+1}\le(\Cupper-\eta V)y'_{T-t+1}+\eta Vy'_{T-t}-\sum_{s\in\Se}(\Cupper-\Vupper-sV)(y'_{T-t+s+1}-y'_{T-t+s+2})$$
is valid for $\convPprime$ and is facet-defining for $\convPprime$ when $\eta\in\{0,(\Cupper-\Vupper)/V\}$ or $\eta=L-1\in\Se$ for any $t\in[s_{\max}+2,T-1]_\Z$.
Let $t'=T-t+1$.
Then, the inequality
$$x'_{t'}\le(\Cupper-\eta V)y'_{t'}+\eta Vy'_{t'-1}-\sum_{s\in\Se}(\Cupper-\Vupper-sV)(y'_{t'+s}-y'_{t'+s+1})$$
is valid for $\convPprime$ and is facet-defining for $\convPprime$ when $\eta\in\{0,(\Cupper-\Vupper)/V\}$ or $\eta=L-1\in\Se$ for any $t'\in[2,T-s_{\max}-1]_\Z$.
Hence, by Lemma~\ref{lem:Pprime}, inequality \eqref{eqn-T:x_t-ub-3-2} is valid for $\convP$ and is facet-defining for $\convP$ when $\eta\in\{0,(\Cupper-\Vupper)/V\}$ or $\eta=L-1\in\Se$ for any $t\in[2,T-s_{\max}-1]_\Z$.\Halmos

\afterpage{
\begin{landscape}
\begin{table}
    \renewcommand{\arraystretch}{2}
    \centering
    \caption{A matrix with the rows representing $2T-1$ points in $\convP$ that satisfy inequality \eqref{eqn-T:x_t-ub-3-1} at equality}
    \rule{0pt}{0ex}
    \setlength\tabcolsep{6.5pt}
    \scriptsize
    \begin{tabular}{|c|c|c|*{8}{c}|*{8}{c}|}
        \hline
         \multirow{2}{*}{Group} & \multirow{2}{*}{Point} & \multirow{2}{*}{Index $r$} & \multicolumn{8}{c|}{$\BFx$} & \multicolumn{8}{c|}{$\BFy$} \\
         \cline{4-11}\cline{12-19}
         &&& $1$ & $\cdots$ & $t-1$ & $t$ & $t+1$ & $t+2$ & $\cdots$ & $T$ 
         & $\ \ 1\ \ $ & $\cdots$ & $t-1$ & $t$ & $t+1$ & $t+2$ & $\cdots$ & $\ \ T\ \ $ \\
         \hline
         \multirow{7}{*}{(A1)} & \multirow{7}{*}{$(\bar{\BFx}^r,\bar{\BFy}^r)$}
         & $1$ 
         & $\Cupper-\epsilon$ & $\cdots$ & $\Cupper$ & $\Cupper$ & $\Cupper$ & $\Cupper$ & $\cdots$ & $\Cupper$
         & $1$ & $\cdots$ & $1$ & $1$ & $1$ & $1$ & $\cdots$ & $1$ \\ 
         && $\vdots$ 
         & $\vdots$ & \rotatebox{0}{$\ddots$} & $\vdots$ & $\vdots$ & $\vdots$ && $\vdots$ & $\vdots$ & $\vdots$ && $\vdots$ & $\vdots$ & $\vdots$ & $\vdots$ && $\vdots$ \\        
         && $t-1$
         & $\Cupper$ & $\cdots$ & $\Cupper-\epsilon$ & $\Cupper$ & $\Cupper$ & $\Cupper$ & $\cdots$ & $\Cupper$
         & $1$ & $\cdots$ & $1$ & $1$ & $1$ & $1$ & $\cdots$ & $1$ \\  
         && $t+1$
         & $\Cupper$ & $\cdots$ & $\Cupper$ & $\Cupper$ & $\Cupper-\epsilon$ & $\Cupper$ & $\cdots$ & $\Cupper$ 
         & $1$ & $\cdots$ & $1$ & $1$ & $1$ & $1$ & $\cdots$ & $1$ \\
         && $t+2$
         & $\Cupper$ & $\cdots$ & $\Cupper$ & $\Cupper$ & $\Cupper$ & $\Cupper-\epsilon$ & $\cdots$ & $\Cupper$ 
         & $1$ & $\cdots$ & $1$ & $1$ & $1$ & $1$ & $\cdots$ & $1$ \\
         && $\vdots$ 
         & $\vdots$ && $\vdots$ & $\vdots$ & $\vdots$ & $\vdots$ & \rotatebox{0}{$\ddots$} & $\vdots$
         & $\vdots$ && $\vdots$ & $\vdots$ & $\vdots$ & $\vdots$ && $\vdots$ \\        
         && $T$
         & $\Cupper$ & $\cdots$ & $\Cupper$ & $\Cupper$ & $\Cupper$ & $\Cupper$ & $\cdots$ & $\Cupper-\epsilon$ 
         & $1$ & $\cdots$ & $1$ & $1$ & $1$ & $1$ & $\cdots$ & $1$ \\ 
         
         \hline
         \multirow{3}{*}{\rred{(A2)}} & \multirow{8}{*}{$(\hat{\BFx}^r,\hat{\BFy}^r)$}
         & $1$ & \multicolumn{8}{c|}{\multirow{3}{*}{(See Note \rred{\ref{tab:eqn-T:x_t-ub-3-1-facet-matrix-1}-1)}}} & \multicolumn{8}{c|}{\multirow{3}{*}{(See Note \rred{\ref{tab:eqn-T:x_t-ub-3-1-facet-matrix-1}-1)}}}\\
         && $\vdots$ & \multicolumn{8}{c|}{} & \multicolumn{8}{c|}{}\\
         && $t-1$ & \multicolumn{8}{c|}{} & \multicolumn{8}{c|}{}\\
         
         \cline{1-1} \cline{3-19}
         \rred{(A3)} && $t$ & \multicolumn{8}{c|}{(See Note \rred{\ref{tab:eqn-T:x_t-ub-3-1-facet-matrix-1}-2)}} & \multicolumn{8}{c|}{(See Note \rred{\ref{tab:eqn-T:x_t-ub-3-1-facet-matrix-1}-2)}}\\

         \cline{1-1} \cline{3-19}
         \rred{(A4)} && $t+1$
         & $\Cupper$ & $\cdots$ & $\Cupper$ & $\Cupper$ & $\Cupper$ & $\Cupper$ & $\cdots$ & $\Cupper$
         & $1$ & $\cdots$ & $1$ & $1$ & $1$ & $1$ & $\cdots$ & $1$ \\
         
         \cline{1-1} \cline{3-19}
         \multirow{3}{*}{\rred{(A5)}} &  
         & $t+2$
         & $0$ & $\cdots$ & $0$ & $0$ & $0$ & $\Clower$ & $\cdots$ & $\Clower$
         & $0$ & $\cdots$ & $0$ & $0$ & $0$ & $1$ & $\cdots$ & $1$ \\
         && $\vdots$ 
         & $\vdots$ && $\vdots$ & $\vdots$ & $\vdots$ & $\vdots$ & \rotatebox{0}{$\ddots$} & $\vdots$ 
         & $\vdots$ && $\vdots$ & $\vdots$ & $\vdots$ & $\vdots$ & \rotatebox{0}{$\ddots$} & $\vdots$ \\                 
         && $T$
         & $0$ & $\cdots$ & $0$ & $0$ & $0$ & $0$ & $\cdots$ & $\Clower$
         & $0$ & $\cdots$ & $0$ & $0$ & $0$ & $0$ & $\cdots$ & $1$ \\
        \hline
        
        \multicolumn{19}{l}{Note \rred{\ref{tab:eqn-T:x_t-ub-3-1-facet-matrix-1}-1}: For $r\in\rred{[1,t-1]_\Z}$, the $\BFx$ and $\BFy$ vectors in group \rred{(A2)} are given as follows:}\\
        \multicolumn{19}{l}{
        $\hat{\BFx}^r=(\underbrace{\Clower,\ldots,\Clower}_{r\ {\rm terms}},\underbrace{0,\ldots,0}_{T-r\ {\rm terms}\!\!\!\!\!})$
        and
        $\hat{\BFy}^r=(\underbrace{1,\ldots,1}_{r\ {\rm terms}},\underbrace{0,\ldots,0}_{T-r\ {\rm terms}\!\!\!\!\!})$
        if $t-r-1\notin\Se$;}\\
        \multicolumn{19}{l}{
        $\hat{\BFx}^r=(\underbrace{0,\ldots,0}_{r\ {\rm terms}},\underbrace{\Vupper,\Vupper+V,\Vupper+2V,\ldots,\Vupper+(t-r-1)V}_{t-r\ {\rm terms}},\underbrace{\Vupper+(t-r-1)V,\ldots,\Vupper+(t-r-1)V}_{T-t\ {\rm terms}\!\!\!\!\!})$
        and
        $\hat{\BFy}^r=(\underbrace{0,\ldots,0}_{r\ {\rm terms}},\underbrace{1,\ldots,1}_{T-r\ {\rm terms}\!\!\!\!\!})$
        if $t-r-1\in\Se$.}\\
        \multicolumn{19}{l}{Note \rred{\ref{tab:eqn-T:x_t-ub-3-1-facet-matrix-1}-2}: The $\BFx$ and $\BFy$ vectors in group \rred{(A3)} are given as follows:}\\
        \multicolumn{19}{l}{
        $\hat{\BFx}^t=(\underbrace{\Vupper,\ldots,\Vupper}_{t\ {\rm terms}},\underbrace{0,\ldots,0}_{T-t\ {\rm terms}\!\!\!\!\!})$
        and
        $\hat{\BFy}^t=(\underbrace{1,\ldots,1}_{t\ {\rm terms}},\underbrace{0,\ldots,0}_{T-t\ {\rm terms}\!\!\!\!\!})$
        if $\eta=(\Cupper-\Vupper)/V$;
        $\hat{\BFx}^t=(\underbrace{0,\ldots,0}_{t-L\ {\rm terms}},\underbrace{\Vupper,\ldots,\Vupper}_{L\ {\rm terms}},\underbrace{0,\ldots,0}_{T-t\ {\rm terms}\!\!\!\!\!})$
        and
        $\hat{\BFy}^t=(\underbrace{0,\ldots,0}_{t-L\ {\rm terms}},\underbrace{1,\ldots,1}_{L\ {\rm terms}},\underbrace{0,\ldots,0}_{T-t\ {\rm terms}\!\!\!\!\!})$
        if $\eta=L-1\in\Se$.}
    \end{tabular}
    \label{tab:eqn-T:x_t-ub-3-1-facet-matrix-1}
\end{table}
\end{landscape}

\begin{landscape}
\begin{table}
    \renewcommand{\arraystretch}{2}
    \centering
    \caption{Lower triangular matrix obtained from Table~\ref{tab:eqn-T:x_t-ub-3-1-facet-matrix-1} via Gaussian elimination}
    \rule{0pt}{4ex}
    \setlength\tabcolsep{6.5pt}
    \scriptsize
    \begin{tabular}{|c|c|c|*{8}{c}|*{8}{c}|}
        \hline
         \multirow{2}{*}{Group} & \multirow{2}{*}{Point} & \multirow{2}{*}{Index $r$} & \multicolumn{8}{c|}{$\BFx$} & \multicolumn{8}{c|}{$\BFy$} \\
         \cline{4-11}\cline{12-19}
         &&& $\ \ 1\ \ $ & $\cdots$ & $t-1$ & $t$ & $t+1$ & $t+2$ & $\cdots$ & $\ \ T\ \ $ 
         & $\ \ 1\ \ $ & $\cdots$ & $t-1$ & $t$ & $t+1$ & $t+2$ & $\cdots$ & $\ \ T\ \ $ \\
         \hline
         \multirow{7}{*}{(B1)} & \multirow{7}{*}{$(\underline{\bar{\BFx}}^r,\underline{\bar{\BFy}}^r)$}
         & $1$ 
         & $-\epsilon$ & $\cdots$ & $0$ & $0$ & $0$ & $0$ & $\cdots$ & $0$
         & $0$ & $\cdots$ & $0$ & $0$ & $0$ & $0$ & $\cdots$ & $0$ \\ 
         && $\vdots$ 
         & $\vdots$ & \rotatebox{0}{$\ddots$} & $\vdots$ & $\vdots$ & $\vdots$ & $\vdots$ && $\vdots$
         & $\vdots$ && $\vdots$ & $\vdots$ & $\vdots$ & $\vdots$ && $\vdots$ \\        
         && $t-1$
         & $0$ & $\cdots$ & $-\epsilon$ & $0$ & $0$ & $0$ & $\cdots$ & $0$
         & $0$ & $\cdots$ & $0$ & $0$ & $0$ & $0$ & $\cdots$ & $0$ \\  
         && $t+1$
         & $0$ & $\cdots$ & $0$ & $0$ & $-\epsilon$ & $0$ & $\cdots$ & $0$ 
         & $0$ & $\cdots$ & $0$ & $0$ & $0$ & $0$ & $\cdots$ & $0$ \\
         && $t+2$
         & $0$ & $\cdots$ & $0$ & $0$ & $0$ & $-\epsilon$ & $\cdots$ & $0$ 
         & $0$ & $\cdots$ & $0$ & $0$ & $0$ & $0$ & $\cdots$ & $0$ \\
         && $\vdots$ 
         & $\vdots$ && $\vdots$ & $\vdots$ & $\vdots$ & $\vdots$ & \rotatebox{0}{$\ddots$} & $\vdots$
         & $\vdots$ && $\vdots$ & $\vdots$ & $\vdots$ & $\vdots$ && $\vdots$ \\        
         && $T$
         & $0$ & $\cdots$ & $0$ & $0$ & $0$ & $0$ & $\cdots$ & $-\epsilon$ 
         & $0$ & $\cdots$ & $0$ & $0$ & $0$ & $0$ & $\cdots$ & $0$ \\ 
         
         \hline
         \multirow{3}{*}{\rred{(B2)}} & \multirow{8}{*}{$(\underline{\hat{\BFx}}^r,\underline{\hat{\BFy}}^r)$}
         & $1$
         & \multicolumn{8}{c|}{}
         & $\rred 1$ & $\rred \cdots$ & $\rred 0$ & $\rred 0$ & $\rred 0$ & $\rred 0$ & $\rred \cdots$ & $\rred 0$ \\
         && $\rred \vdots$ 
         & \multicolumn{8}{c|}{(Omitted)} 
         & $\rred \vdots$ & \rotatebox{0}{$\rred \ddots$} & $\rred \vdots$ & $\rred \vdots$ & $\rred \vdots$ & $\rred \vdots$ && $\rred \vdots$ \\                 
         && $t-1$
         & \multicolumn{8}{c|}{}
         & $\rred 1$ & $\rred \cdots$ & $\rred 1$ & $\rred 0$ & $\rred 0$ & $\rred 0$ & $\rred \cdots$ & $\rred 0$ \\
         
         \cline{1-1} \cline{3-19}
         \rred{(B3)} && $t$ & \multicolumn{8}{c|}{(Omitted)} & \multicolumn{8}{c|}{(See Note \rred{\ref{tab:eqn-T:x_t-ub-3-1-facet-matrix-2}-1)}}\\
         
         \cline{1-1} \cline{3-19}
         \rred{(B4)} && $t+1$
         & \multicolumn{8}{c|}{(Omitted)}
         & $1$ & $\cdots$ & $1$ & $1$ & $1$ & $0$ & $\cdots$ & $0$ \\
         
         \cline{1-1} \cline{3-19}
         \multirow{3}{*}{\rred{(B5)}} &  
         & $t+2$
         & \multicolumn{8}{c|}{}
         & $0$ & $\cdots$ & $0$ & $0$ & $0$ & $1$ & $\cdots$ & $0$ \\
         && $\vdots$ 
         & \multicolumn{8}{c|}{(Omitted)} 
         & $\vdots$ && $\vdots$ & $\vdots$ & $\vdots$ & $\vdots$ & \rotatebox{0}{$\ddots$} & $\vdots$ \\                 
         && $T$
         & \multicolumn{8}{c|}{}
         & $0$ & $\cdots$ & $0$ & $0$ & $0$ & $0$ & $\cdots$ & $1$ \\
         
        \hline
        \multicolumn{19}{l}{Note \rred{\ref{tab:eqn-T:x_t-ub-3-1-facet-matrix-2}-1}: The $\BFy$ vector in group \rred{(B3)} is given as follows:}\\
        \multicolumn{19}{l}{
        $\underline{\hat{\BFy}}^t=(\underbrace{1,\ldots,1}_{t\ {\rm terms}},\underbrace{0,\ldots,0}_{T-t\ {\rm terms}})$
        if $\eta=(\Cupper-\Vupper)/V$;
        $\underline{\hat{\BFy}}^t=(\underbrace{0,\ldots,0}_{t-L\ {\rm terms}},\underbrace{1,\ldots,1}_{L\ {\rm terms}},\underbrace{0,\ldots,0}_{T-t\ {\rm terms}})$
        if $\eta=L-1\in\Se$.}
    \end{tabular}
    \label{tab:eqn-T:x_t-ub-3-1-facet-matrix-2}
\end{table}
\end{landscape}
}

\subsection{Proof of Proposition \ref{prop-T:x_t-ub-4}}

\noindent{\bf Proposition \ref{prop-T:x_t-ub-4}.} {\it
Consider any real number $\eta$ such that $0\le\eta\le\min\{L-1,(\Cupper-\Vupper)/V\}$ and any integers $\alpha$, $\beta$, and $s_{\max}$ such that
(a)~$L\le s_{\max}\le\min\{T-3,\lfloor(\Cupper-\Vupper)/V\rfloor\}$,
(b)~$0\le\alpha<\beta\le s_{\max}$, and
(c)~$\beta=\alpha+1$ or $s_{\max}\le L+\alpha$.
Let $\Se=[0,\alpha]_{\Z}\cup[\beta,s_{\max}]_{\Z}$.
For any $t\in[s_{\max}+2,T-1]_\Z$, inequality \eqref{eqn-T:x_t-ub-3-1} is valid for $\convP$.
For any $t\in[2,T-s_{\max}-1]_\Z$, inequality \eqref{eqn-T:x_t-ub-3-2} is valid for $\convP$.
Furthermore, \eqref{eqn-T:x_t-ub-3-1} and \eqref{eqn-T:x_t-ub-3-2} are facet-defining for $\convP$ when $\eta\in\{0,(\Cupper-\Vupper)/V\}$ or $\eta=L-1\in\Se$.}
\vskip8pt

\noindent{\bf Proof.}
Consider any $t\in[s_{\max}+2,T-1]_\Z$.
To prove that the linear inequality \eqref{eqn-T:x_t-ub-3-1} is valid for $\convP$ when $\Se=[0,\alpha]_\Z\cup[\beta,s_{\max}]_\Z$, 
it suffices to show that \eqref{eqn-T:x_t-ub-3-1} is valid for $\Pe$ when $\Se=[0,\alpha]_\Z\cup[\beta,s_{\max}]_\Z$.
Consider any element $(\BFx,\BFy)$ of $\Pe$.
We show that $(\BFx,\BFy)$ satisfies \eqref{eqn-T:x_t-ub-3-1} when $\Se=[0,\alpha]_\Z\cup[\beta,s_{\max}]_\Z$.
We divide the analysis into four cases.

Case~1: $y_t=0$ and $y_{t-s}-y_{t-s-1}\le0$ for all $s\in\Se$.
Because $\Se\subseteq[0,\lfloor(\Cupper-\Vupper)/V\rfloor]_\Z$, we have $\Cupper-\Vupper-sV\ge0$ for all $s\in\Se$.
Thus, in this case, the right-hand side of inequality \eqref{eqn-T:x_t-ub-3-1} is nonnegative.
Because $y_t=0$, by \eqref{eqn:p-upper-bound}, $x_t=0$.
Therefore, $(\BFx,\BFy)$ satisfies \eqref{eqn-T:x_t-ub-3-1}.

Case~2: $y_t=0$ and $y_{t-s}-y_{t-s-1}>0$ for some $s\in\Se$.
Let $\tilde{\Se}=\{\sigma\in\Se:y_{t-\sigma}-y_{t-\sigma-1}>0\}$ and $v=|\tilde{\Se}|$.
Then, $v\ge1$.
Denote $\tilde{\Se}=\{\sigma_1,\sigma_2,\ldots,\sigma_v\}$, where $\sigma_1<\sigma_2<\cdots<\sigma_v$.
Note that $y_{t-\sigma_j-1}=0$ and $y_{t-\sigma_j}=1$ for $j=1,\ldots,v$.
Denote $\sigma_0=-1$.
Then, for each $j=1,\ldots,v$, there exists $\sigma'_j\in[\sigma_{j-1}+1,\sigma_j-1]_{\Z}$ such that $y_{t-\sigma'_j-1}=1$ and $y_{t-\sigma'_j}=0$.
Thus,
$$0\le\sigma'_1<\sigma_1<\sigma'_2<\sigma_2<\cdots<\sigma'_v<\sigma_v\le s_{\max}.$$
Because $y_{t-\sigma_v}-y_{t-\sigma_v-1}=1$ and $t-\sigma_v\in[2,T]_\Z$, by \eqref{eqn:p-minup}, $y_k=1$ for all $k\in[t-\sigma_v,\min\{T,t-\sigma_v+L-1\}]_{\Z}$, 
which implies that $t-\sigma'_j\ge t-\sigma_v+L$ for $j=1,\ldots,v$.
Hence, for $j=1,\ldots,v$, we have $\sigma'_j\le\sigma_v-L$, which implies that
\begin{equation}\label{eq:Prop7_eqA}
    \sigma'_j\le s_{\max}-L.
\end{equation}
If $\beta=\alpha+1$, then $\Se=[0,s_{\max}]_{\Z}$, which implies that $\sigma'_j\in\Se$ for $j=1,\ldots,v$.
If $\beta\neq\alpha+1$, then condition (c) of Proposition \ref{prop-T:x_t-ub-4} implies that $s_{\max}\le L+\alpha$, which, by \eqref{eq:Prop7_eqA}, implies that $\sigma'_j\le\alpha$ for $j=1,\ldots,v$.
Thus, in both cases, $\sigma'_j\in\Se$ for $j=1,\ldots,v$.
Because $y_t=0$, by \eqref{eqn:p-upper-bound}, $x_t=0$.
Hence, the left-hand side of inequality \eqref{eqn-T:x_t-ub-3-1} is $0$.
Because $\Se\subseteq[0,\lfloor(\Cupper-\Vupper)/V\rfloor]_\Z$, we have $\Cupper-\Vupper-sV\ge0$ for all $s\in\Se$.
Note that $\{\sigma'_1,\ldots,\sigma'_v\}\subseteq\Se\setminus\tilde{\Se}$ and $y_{t-s}-y_{t-s-1}\le0$ for all $s\in\Se\setminus\tilde{\Se}$.
Thus, $\sum_{s\in\Se\setminus\tilde{\Se}}(\Cupper-\Vupper-sV)(y_{t-s}-y_{t-s-1})\le\sum_{j=1}^v(\Cupper-\Vupper-\sigma'_jV)(y_{t-\sigma'_j}-y_{t-\sigma'_j-1})$.
Hence, the right-hand side of inequality \eqref{eqn-T:x_t-ub-3-1} is
\begin{align*}
    &(\Cupper-\eta V)y_t+\eta Vy_{t+1}-\sum_{s\in\Se}(\Cupper-\Vupper-sV)(y_{t-s}-y_{t-s-1})\\
    &= \eta Vy_{t+1}-\sum_{s\in\tilde{\Se}}(\Cupper-\Vupper-sV)(y_{t-s}-y_{t-s-1})-\sum_{s\in\Se\setminus\tilde{\Se}}(\Cupper-\Vupper-sV)(y_{t-s}-y_{t-s-1})\\
    &\ge \eta Vy_{t+1}-\sum_{j=1}^v(\Cupper-\Vupper-\sigma_jV)(y_{t-\sigma_j}-y_{t-\sigma_j-1})-\sum_{j=1}^v(\Cupper-\Vupper-\sigma'_jV)(y_{t-\sigma'_j}-y_{t-\sigma'_j-1})\\
    &= \eta Vy_{t+1}-\sum_{j=1}^v(\Cupper-\Vupper-\sigma_jV)+\sum_{j=1}^v(\Cupper-\Vupper-\sigma'_jV)\\
    &= \eta Vy_{t+1}+\sum_{j=1}^v(\sigma_j-\sigma'_j)V >0.
\end{align*}
Therefore, in this case, $(\BFx,\BFy)$ satisfies \eqref{eqn-T:x_t-ub-3-1}.

Case~3: $y_t=1$ and $y_{t-s}-y_{t-s-1}\le 0$ for all $s\in\Se$.
Because $\Se\subseteq[0,\lfloor(\Cupper-\Vupper)/V\rfloor]_\Z$, we have $\Cupper-\Vupper-sV\ge 0$ for all $s\in\Se$.
Thus, $\sum_{s\in\Se}(\Cupper-\Vupper-sV)(y_{t-s}-y_{t-s-1})\le0$.
Because $y_t=1$, the right-hand side of \eqref{eqn-T:x_t-ub-3-1} is at least $\Cupper$ when $y_{t+1}=1$ and is at least $\Cupper-\eta V\ge\Vupper$ when $y_{t+1}=0$ (as $\eta\le(\Cupper-\Vupper)/V$).
By \eqref{eqn:p-upper-bound}, $x_t\le\Cupper$.
If $y_{t+1}=0$, then by \eqref{eqn:p-upper-bound} and \eqref{eqn:p-ramp-down}, $x_{t+1}=0$ and $x_t-x_{t+1}\le Vy_{t+1}+\Vupper(1-y_{t+1})$, 
which imply that $x_t\le\Vupper$.
Hence, $x_t$ is at most $\Cupper$, and it is at most $\Vupper$ when $y_{t+1}=0$.
Therefore, in this case, $(\BFx,\BFy)$ satisfies \eqref{eqn-T:x_t-ub-3-1}.

Case~4: $y_t=1$ and $y_{t-s}-y_{t-s-1}>0$ for some $s\in\Se$.
If $y_{t+1}=1$, then inequality \eqref{eqn-T:x_t-ub-3-1} becomes inequality \eqref{eqn-T:x_t-ub-1-1}, and by Proposition~\ref{prop-T:x_t-ub-1}, $(\BFx,\BFy)$ satisfies the inequality.
In the following, we consider the case where $y_{t+1}=0$.
Let $\tilde{\Se}=\{\sigma\in\Se:y_{t-\sigma}-y_{t-\sigma-1}>0\}$ and $v=|\tilde{\Se}|$.
Then, $v\ge1$.
Denote $\tilde{\Se}=\{\sigma_1,\sigma_2,\ldots,\sigma_v\}$, where $\sigma_1<\sigma_2<\cdots<\sigma_v$.
Note that $y_{t-\sigma_j-1}=0$ and $y_{t-\sigma_j}=1$ for $j=1,\ldots,v$.
Then, for each $j=2,\ldots,v$, there exists $\sigma'_j\in[\sigma_{j-1}+1,\sigma_j-1]_{\Z}$ such that $y_{t-\sigma'_j-1}=1$ and $y_{t-\sigma'_j}=0$.
Thus,
$$0\le\sigma_1<\sigma'_2<\sigma_2<\cdots<\sigma'_v<\sigma_v\le s_{\max}.$$
In addition, $y_k=1$ for all $k\in[t-\sigma_1,t]_{\Z}$.
Because $y_{t-\sigma_v}-y_{t-\sigma_v-1}=1$ and $t-\sigma_v\in[2,T]_\Z$, by \eqref{eqn:p-minup}, $y_k=1$ for all $k\in[t-\sigma_v,\min\{T,t-\sigma_v+L-1\}]_{\Z}$, which implies that $t-\sigma'_j\ge t-\sigma_v+L$ for $j=2,\ldots,v$.
Hence, for $j=2,\ldots,v$, we have $\sigma'_j\le\sigma_v-L$, which implies that
\begin{equation}\label{eq:Prop7_eqB}
    \sigma'_j\le s_{\max}-L.
\end{equation}
If $\beta=\alpha+1$, then $\Se=[0,s_{\max}]_{\Z}$, which implies that $\sigma'_j\in\Se$ for $j=2,\ldots,v$.
If $\beta\neq\alpha+1$, then condition (c) of Proposition \ref{prop-T:x_t-ub-4} implies that $s_{\max}\le L+\alpha$, which, by \eqref{eq:Prop7_eqB}, implies that $\sigma'_j\le\alpha$ for $j=2,\ldots,v$.
Thus, in both cases, $\sigma'_j\in\Se$ for $j=2,\ldots,v$.
Because $y_{t+1}=0$, by \eqref{eqn:p-upper-bound} and \eqref{eqn:p-ramp-down}, $x_{t+1}=0$ and $x_t-x_{t+1}\le Vy_{t+1}+\Vupper(1-y_{t+1})$,
which imply that $x_t\le\Vupper$; that is, the left-hand side of inequality \eqref{eqn-T:x_t-ub-3-1} is at most $\Vupper$.
Because $y_{t-\sigma_1}-y_{t-\sigma_1-1}=1$ and $t-\sigma_1\in[2,T]_\Z$, by \eqref{eqn:p-minup}, $y_k=1$ for all $k\in[t-\sigma_1,\min\{T,t-\sigma_1+L-1\}]_{\Z}$.
Because $y_{t+1}=0$, this implies that $t+1\ge t-\sigma_1+L$, or equivalently, $L-1\le\sigma_1$.
Because $\eta\le L-1$, we have
\begin{equation}\label{eq:prop7:eta-and-sigma-A}
\eta\le\sigma_1.
\end{equation}
Because $\Se\subseteq[0,\lfloor(\Cupper-\Vupper)/V\rfloor]_\Z$, we have $\Cupper-\Vupper-sV\ge0$ for all $s\in\Se$.
Note that $\{\sigma'_2,\ldots,\sigma'_v\}\subseteq\Se\setminus\tilde{\Se}$ and $y_{t-s}-y_{t-s-1}\le0$ for all $s\in\Se\setminus\tilde{\Se}$.
Thus, $\sum_{s\in\Se\setminus\tilde{\Se}}(\Cupper-\Vupper-sV)(y_{t-s}-y_{t-s-1})\le\sum_{j=2}^v(\Cupper-\Vupper-\sigma'_jV)(y_{t-\sigma'_j}-y_{t-\sigma'_j-1})$.
Hence, the right-hand side of inequality \eqref{eqn-T:x_t-ub-3-1} is
\begin{align*}
    &(\Cupper-\eta V)y_t+\eta Vy_{t+1}-\sum_{s\in\Se}(\Cupper-\Vupper-sV)(y_{t-s}-y_{t-s-1})\\
    &=(\Cupper-\eta V)-\sum_{s\in\tilde{\Se}}(\Cupper-\Vupper-sV)(y_{t-s}-y_{t-s-1})-\sum_{s\in\Se\setminus\tilde{\Se}}(\Cupper-\Vupper-sV)(y_{t-s}-y_{t-s-1})\\
    &\ge(\Cupper-\eta V)-\sum_{j=1}^v(\Cupper-\Vupper-\sigma_jV)(y_{t-\sigma_j}-y_{t-\sigma_j-1})-\sum_{j=2}^v(\Cupper-\Vupper-\sigma'_jV)(y_{t-\sigma'_j}-y_{t-\sigma'_j-1})\\
    &=(\Cupper-\eta V)-\sum_{j=1}^v(\Cupper-\Vupper-\sigma_jV)+\sum_{j=2}^v(\Cupper-\Vupper-\sigma'_jV)\\
    &=\Vupper+(\sigma_1-\eta)V+\sum_{j=2}^v(\sigma_j-\sigma'_j)V\\
    &\ge\Vupper+(\sigma_1-\eta)V\\
    &\ge\Vupper,
\end{align*}
where the last inequality follows from \eqref{eq:prop7:eta-and-sigma-A}.
Therefore, in this case, $(\BFx,\BFy)$ satisfies \eqref{eqn-T:x_t-ub-3-1}.

Summarizing Cases 1--4, we conclude that $(\BFx,\BFy)$ satisfies \eqref{eqn-T:x_t-ub-3-1}.
Hence, \eqref{eqn-T:x_t-ub-3-1} is valid for $\convP$ when $\Se=[0,\alpha]_\Z\cup[\beta,s_{\max}]_\Z$.

It is easy to verify that the proof of facet-defining of inequality \eqref{eqn-T:x_t-ub-3-1} in the proof of Proposition~\ref{prop-T:x_t-ub-3} remains valid when $\Se=[0,\alpha]_\Z\cup[\beta,s_{\max}]_\Z$.
Therefore, inequality \eqref{eqn-T:x_t-ub-3-1} is facet-defining for $\convP$ under the conditions stated in Proposition~\ref{prop-T:x_t-ub-4}.

It is also easy to verify that the proof of validity and facet-defining of inequality \eqref{eqn-T:x_t-ub-3-2} in the proof of Proposition~\ref{prop-T:x_t-ub-3} remains valid when $\Se=[0,\alpha]_\Z\cup[\beta,s_{\max}]_\Z$.
Therefore, inequality \eqref{eqn-T:x_t-ub-3-2} is valid and facet-defining for $\convP$ under the conditions stated in Proposition \ref{prop-T:x_t-ub-4}.
\Halmos

\subsection{Proof of Proposition \ref{prop-T:x_t-ub-5}}

\noindent{\bf Proposition \ref{prop-T:x_t-ub-5}.} {\it
Consider any $\Se\subseteq[1,\min\{L,T-2,\lfloor(\Cupper-\Vupper)/V\rfloor\}]_{\Z}$ and any real number $\eta$ such that $0\le\eta\le\min\{L,(\Cupper-\Vupper)/V\}$.
For any $t\in[2,T]_\Z$ such that $t\ge s+2$ for all $s\in\Se$, the inequality
\begin{equation}
    x_t\le(\Vupper+\eta V)y_t+(\Cupper-\Vupper-\eta V)y_{t-1}-\sum_{s\in\Se}(\Cupper-\Vupper-sV)(y_{t-s}-y_{t-s-1})\tag{\ref{eqn-T:x_t-ub-5-1}}
\end{equation}
is valid for $\convP$.
For any $t\in[1,T-1]_\Z$ such that $t\le T-s-1$ for all $s\in\Se$, the inequality
\begin{equation}
    x_t\le(\Vupper+\eta V)y_t+(\Cupper-\Vupper-\eta V)y_{t+1}-\sum_{s\in\Se}(\Cupper-\Vupper-sV)(y_{t+s}-y_{t+s+1})\tag{\ref{eqn-T:x_t-ub-5-2}}
\end{equation}
is valid for $\convP$.
Furthermore, inequalities \eqref{eqn-T:x_t-ub-5-1} and \eqref{eqn-T:x_t-ub-5-2} are facet-defining for $\convP$ when $\eta\in\{0,(\Cupper-\Vupper)/V\}$ or $\eta=L\in\Se$.}
\vskip8pt

\noindent{\bf Proof.}
We first prove that inequality \eqref{eqn-T:x_t-ub-5-1} is valid and facet-defining for $\convP$.
Note that the proof of facet-defining of \eqref{eqn-T:x_t-ub-5-1} here can also be used to prove the facet-defining of \eqref{eqn-T:x_t-ub-5-1} in Proposition~\ref{prop-T:x_t-ub-6}.
For notational convenience, we define $s_{\max}=\max\{s:s\in\Se\}$ if $\Se\ne\emptyset$, and $s_{\max}=0$ if $\Se=\emptyset$.

Consider any $t\in[s_{\max}+2,T]_\Z$ (i.e., $t\in[2,T]_\Z$ such that $t\ge s+2$ for all $s\in\Se$).
To prove that the linear inequality \eqref{eqn-T:x_t-ub-5-1} is valid for $\convP$, it suffices to show that it is valid for $\Pe$.
Consider any element $(\BFx,\BFy)$ of $\Pe$.
We show that $(\BFx,\BFy)$ satisfies \eqref{eqn-T:x_t-ub-5-1}.
We divide the analysis into three cases.

Case~1: $y_t=0$.
In this case, by \eqref{eqn:p-upper-bound}, $x_t=0$.
Thus, the left-hand side of \eqref{eqn-T:x_t-ub-5-1} and the first term on the right-hand side of \eqref{eqn-T:x_t-ub-5-1} are $0$.
Because $y_t=0$, by Lemma \ref{lem:lookbackward}(i), $y_{t-j}-y_{t-j-1}\le0$ for all $j\in[0,\min\{t-2,L-1\}]_{\Z}$.
Because $s_{\max}\le t-2$, we have $\Se\subseteq[0,\min\{t-2,L\}]_{\Z}$.
Thus, $y_{t-s}-y_{t-s-1}\le0$ for all $s\in\Se\setminus\{L\}$.
Because $\eta\le(\Cupper-\Vupper)/V$, the coefficient ``$\Cupper-\Vupper-\eta V$" on the right-hand side of \eqref{eqn-T:x_t-ub-5-1} is nonnegative.
Because $\Se\subseteq[1,\lfloor(\Cupper-\Vupper)/V\rfloor]_{\Z}$, for any $s\in\Se$, the coefficient ``$\Cupper-\Vupper-sV$" on the right-hand side of \eqref{eqn-T:x_t-ub-5-1} is also nonnegative.
Hence, if $s_{\max}\le L-1$ or $y_{t-L}-y_{t-L-1}\le0$, then the right-hand side of \eqref{eqn-T:x_t-ub-5-1} is nonnegative.
Now, consider the situation where $s_{\max}=L$ and $y_{t-L}-y_{t-L-1}>0$.
Then, $y_{t-L}=1$ and $y_{t-L-1}=0$.
By \eqref{eqn:p-minup}, $y_{t-1}=1$.
Thus, the right-hand side of \eqref{eqn-T:x_t-ub-5-1} is at least $(\Cupper-\Vupper-\eta V)y_{t-1}-(\Cupper-\Vupper-LV)(y_{t-L}-y_{t-L-1})=(\Cupper-\Vupper-\eta V)-(\Cupper-\Vupper-LV)=(L-\eta)V\ge0$.
Therefore, in this case, $(\BFx,\BFy)$ satisfies \eqref{eqn-T:x_t-ub-5-1}.

Case~2: $y_t=1$ and $y_{t-s'}-y_{t-s'-1}=1$ for some $s'\in\Se$.
In this case, $y_{t-s'}=1$ and $y_{t-s'-1}=0$.
Because $s_{\max}\le t-2$, we have $s'\le t-2$.
If $s'=t-2$, then it does not exist any $s\in\Se$ such that $s>s'$.
If $s'\le t-3$, then $t-s'-1\in[2,T]_{\Z}$, and by Lemma \ref{lem:lookbackward}(i), $y_{t-s'-j-1}-y_{t-s'-j-2}\le0$ for all $j\in[0,\min\{t-s'-3,L-1\}]_{\Z}$, which implies that $y_{t-s}-y_{t-s-1}\le0$ for all $s\in[s'+1,\min\{t-2,L+s'\}]_{\Z}$.
Thus, $y_{t-s}-y_{t-s-1}\le0$ for all $s\in\Se$ such that $s>s'$.
Because $y_{t-s'}-y_{t-s'-1}=1$ and $t-s'\in[2,T]_{\Z}$, by \eqref{eqn:p-minup}, $y_k=1$ for all $k\in[t-s',\min\{T,t-s'+L-1\}]_{\Z}$.
This implies that $y_{t-s}-y_{t-s-1}=0$ for all $s\in[1,s'-1]_\Z$ (as $s'\le L$).
This in turn implies that $y_{t-s}-y_{t-s-1}=0$ for all $s\in\Se$ such that $s<s'$.
Hence, $y_{t-s}-y_{t-s-1}\le0$ for all $s\in\Se\setminus\{s'\}$.
Because $\Se\subseteq[1,\lfloor(\Cupper-\Vupper)/V\rfloor]_{\Z}$, we have $\Cupper-\Vupper-sV\ge0$ for all $s\in\Se$.
Thus,
\begin{equation}\label{eq1:apx:prop-T:x_t-ub-5}
-\sum_{s\in\Se}(\Cupper-\Vupper-sV)(y_{t-s}-y_{t-s-1})\ge -(\Cupper-\Vupper-s'V).
\end{equation}
Note that $t-1\in[t-s',\min\{T,t-s'+L-1\}]_{\Z}$.
Hence, $y_{t-1}=1$.
Because $y_t=1$ and $y_{t-1}=1$, by \eqref{eq1:apx:prop-T:x_t-ub-5}, the right-hand side of inequality \eqref{eqn-T:x_t-ub-5-1} is at least $s'V+\Vupper$.
By \eqref{eqn:p-ramp-up}, $\sum_{\tau=t-s'}^t(x_{\tau}-x_{\tau-1})\le\sum_{\tau=t-s'}^tVy_{\tau-1}+\sum_{\tau=t-s'}^t\Vupper(1-y_{\tau-1})$, which implies that $x_t-x_{t-s'-1}\le s'V+\Vupper$.
Because $y_{t-s'-1}=0$, we have $x_{t-s'-1}=0$.
Thus, $x_t\le s'V+\Vupper$.
Therefore, in this case, $(\BFx,\BFy)$ satisfies \eqref{eqn-T:x_t-ub-5-1}.

Case~3: $y_t=1$ and $y_{t-s}-y_{t-s-1}\neq1$ for all $s\in\Se$.
In this case, $y_{t-s}-y_{t-s-1}\le0$ for all $s\in\Se$.
Because $\Se\subseteq[1,\lfloor(\Cupper-\Vupper)/V\rfloor]_{\Z}$, we have $\Cupper-\Vupper-sV\ge0$ for all $s\in\Se$.
Thus, $\sum_{s\in\Se}(\Cupper-\Vupper-sV)(y_{t-s}-y_{t-s-1})\le0$.
The right-hand side of \eqref{eqn-T:x_t-ub-5-1} is at least $\Vupper+\eta V$ when $y_{t-1}=0$, and is at least $\Cupper$ when $y_{t-1}=1$.
If $y_{t-1}=0$, then by \eqref{eqn:p-upper-bound} and \eqref{eqn:p-ramp-up}, $x_{t-1}=0$ and $x_t-x_{t-1}\le\Vupper$, which imply that $x_t\le\Vupper$, and hence, $x_t$ is less than or equal to the right-hand side of \eqref{eqn-T:x_t-ub-5-1}.
If $y_{t-1}=1$, then by \eqref{eqn:p-upper-bound}, $x_t\le\Cupper$, and hence, $x_t$ is less than or equal to the right-hand side of \eqref{eqn-T:x_t-ub-5-1}.
Therefore, in this case, $(\BFx,\BFy)$ satisfies \eqref{eqn-T:x_t-ub-5-1}.

Summarizing Cases 1--3, we conclude that $(\BFx,\BFy)$ satisfies \eqref{eqn-T:x_t-ub-5-1}.
Hence, \eqref{eqn-T:x_t-ub-5-1} is valid for $\convP$.

Consider any $t\in[s_{\max}+2,T]_\Z$.
To prove that inequality \eqref{eqn-T:x_t-ub-5-1} is facet-defining for $\convP$ when $\eta\in\{0,(\Cupper-\Vupper)/V\}$ or $\eta=L\in\Se$, 
it suffices to show that there exist $2T$ affinely independent points in $\convP$ that satisfy \eqref{eqn-T:x_t-ub-5-1} at equality when $\eta\in\{0,(\Cupper-\Vupper)/V\}$ or $\eta=L\in\Se$.
Because $\BFzero\in\convP$ and $\BFzero$ satisfies \eqref{eqn-T:x_t-ub-5-1} at equality, it suffices to create the remaining $2T-1$ nonzero linearly independent points.
We denote these $2T-1$ points as $(\bar{\BFx}^r,\bar{\BFy}^r)$ for $r\in[1,T]_{\Z}\setminus\{t\}$ and $(\hat{\BFx}^r,\hat{\BFy}^r)$ for $r\in[1,T]_{\Z}$, 
and denote the $q$th component of $\bar{\BFx}^r$, $\bar{\BFy}^r$, $\hat{\BFx}^r$, and $\hat{\BFy}^r$ as $\bar{x}^r_q$, $\bar{y}^r_q$, $\hat{x}^r_q$, and $\hat{y}^r_q$, respectively.
Let $\epsilon=\Vupper-\Clower>0$.
We divide these $2T-1$ points into the following \rred{five} groups:
\begin{enumerate}[label=(A\arabic*)]
\item\label{points:prop9-eq1-A1} For each $r\in[1,T]_\Z\setminus\{t\}$, we create the same point $(\bar{\BFx}^r,\bar{\BFy}^r)$ as in group \ref{points:prop1-eq1-A1} in the proof of Proposition \ref{prop-T:x_t-ub-1}.
Thus, $(\bar{\BFx}^r,\bar{\BFy}^r)\in\convP$.
It is easy to verify that $(\bar{\BFx}^r,\bar{\BFy}^r)$ satisfies \eqref{eqn-T:x_t-ub-5-1} at equality.

\item\label{points:prop9-eq1-A2} For each \rred{$r\in[1,t-2]_{\Z}$}, we create the same point $(\hat{\BFx}^r,\hat{\BFy}^r)$ as in group \ref{points:prop1-eq1-A2} in the proof of Proposition \ref{prop-T:x_t-ub-1}.
Thus, $(\hat{\BFx}^r,\hat{\BFy}^r)\in\convP$.
Consider the case where $t-r-1\notin\Se$.
In this case, $\hat{x}_t^r=\hat{y}^r_t=\hat{y}^r_{t-1}=0$.
In addition, $t-s-1\neq r$ for all $s\in\Se$, which implies that $\hat{y}^r_{t-s}-\hat{y}^r_{t-s-1}=0$ for all $s\in\Se$.
Hence, $(\hat{\BFx}^r,\hat{\BFy}^r)$ satisfies \eqref{eqn-T:x_t-ub-5-1} at equality.
Next, consider the case where $t-r-1\in\Se$.
In this case, $\hat{x}^r_t=\Vupper+(t-r-1)V$ and $\hat{y}^r_t=\hat{y}^r_{t-1}=1$.
In addition, $\hat{y}^r_{t-s}-\hat{y}^r_{t-s-1}=1$ when $s=t-r-1$, and $\hat{y}^r_{t-s}-\hat{y}^r_{t-s-1}=0$ when $s\neq t-r-1$.
Hence, $(\hat{\BFx}^r,\hat{\BFy}^r)$ satisfies \eqref{eqn-T:x_t-ub-5-1} at equality.

\item\label{points:prop9-eq1-A3} We create a point $(\hat{\BFx}^{t-1},\hat{\BFy}^{t-1})$ as follows:
If $\eta=0$, then
\begin{equation*}
\rred{\hat{x}^{t-1}_q}=\left\{
   \begin{array}{ll}
   0, &\hbox{for $q\in[0,t-1]_{\Z}$;}\\
   \Vupper,&\hbox{for $q\in[t,T]_{\Z}$;}
   \end{array}\right.
\end{equation*}
and
\begin{equation*}
    \rred{\hat{y}^{t-1}_q}=\left\{
    \begin{array}{ll}
    0,&\hbox{for $q\in[0,t-1]_{\Z}$;}\\
    1, &\hbox{for $q\in[t,T]_{\Z}$.}
    \end{array}\right.
\end{equation*}
If $\eta=(\Cupper-\Vupper)/V$, then
\begin{equation*}
\rred{\hat{x}^{t-1}_q}=\left\{
    \begin{array}{ll}
    \Clower, &\hbox{for $q\in[1,t-1]_{\Z}$;}\\
    0, &\hbox{for $q\in[t,T]_{\Z}$;}
    \end{array}\right.
\end{equation*}
and
\begin{equation*}
\rred{\hat{y}^{t-1}_q}=\left\{
    \begin{array}{ll}
    1, &\hbox{for $q\in[1,t-1]_{\Z}$;}\\
    0, &\hbox{for $q\in[t,T]_{\Z}$.}
    \end{array}\right.
\end{equation*}
If $\eta=L\in\Se$, then
\begin{equation*}
\rred{\hat{x}^{t-1}_q}=\left\{
    \begin{array}{ll}
    0, &\hbox{for $q\in[1,t-L-1]_\Z\cup[t,T]_\Z$;}\\
    \Vupper, &\hbox{for $q\in[t-L,t-1]_\Z$;}
    \end{array}\right.
\end{equation*}
and
\begin{equation*}
\rred{\hat{y}^{t-1}_q}=\left\{
    \begin{array}{ll}
    0, &\hbox{for $q\in[1,t-L-1]_\Z\cup[t,T]_\Z$;}\\
    1, &\hbox{for $q\in[t-L,t-1]_\Z$.}
    \end{array}\right.
\end{equation*}
We first consider the case where $\eta=0$.
It is easy to verify that $(\hat{\BFx}^{t-1},\hat{\BFy}^{t-1})$ satisfies \eqref{eqn:p-minup}--\eqref{eqn:p-ramp-down}.
Thus, $(\hat{\BFx}^{t-1},\hat{\BFy}^{t-1})\in\convP$.
In this case, $\hat{x}_t^{t-1}=\Vupper$, $\hat{y}^{t-1}_t=1$, $\hat{y}_{t-1}^{t-1}=0$, and $\hat{y}^{t-1}_{t-s}-\hat{y}^{t-1}_{t-s-1}=0$ for all $s\in\Se$.
Hence, $(\hat{\BFx}^{t-1},\hat{\BFy}^{t-1})$ satisfies \eqref{eqn-T:x_t-ub-5-1} at equality.
Next, we consider the case where $\eta=(\Cupper-\Vupper)/V$.
It is easy to verify that $(\hat{\BFx}^{t-1},\hat{\BFy}^{t-1})$ satisfies \eqref{eqn:p-minup}--\eqref{eqn:p-ramp-down}.
Thus, $(\hat{\BFx}^{t-1},\hat{\BFy}^{t-1})\in\convP$.
In this case, $\hat{x}^{t-1}_t=\hat{y}^{t-1}_t=0$, $\Cupper-\Vupper-\eta V=0$, and $\hat{y}_{t-s}^{t-1}=\hat{y}^{t-1}_{t-s-1}=0$ for all $s\in\Se$.
Hence, $(\hat{\BFx}^{t-1},\hat{\BFy}^{t-1})$ satisfies \eqref{eqn-T:x_t-ub-5-1} at equality.
Next, we consider the case where $\eta=L\in\Se$.
In this case, for any $q\in[2,T]_{\Z}$, $\hat{y}^{t-1}_q-\hat{y}^{t-1}_{q-1}\le0$ if $q\ne t-L$, while $\hat{y}^{t-1}_q-\hat{y}^{t-1}_{q-1}=1$ and $\hat{y}^{t-1}_k=1$ for all $k\in[q,\min\{T,q+L-1\}]_{\Z}$ if $q=t-L$.
Thus, $(\hat{\BFx}^{t-1},\hat{\BFy}^{t-1})$ satisfies \eqref{eqn:p-minup}.
For any $q\in[2,T]_{\Z}$, $\hat{y}^{t-1}_{q-1}-\hat{y}^{t-1}_q\le0$ if $q\neq t$, while $\hat{y}^{t-1}_{q-1}-\hat{y}^{t-1}_q=1$ and $\hat{y}^{t-1}_q=0$ for all $k\in[q,\min\{T,q+\ell-1\}]_{\Z}$ if $q=t$.
Thus, $(\hat{\BFx}^{t-1},\hat{\BFy}^{t-1})$ satisfies \eqref{eqn:p-mindn}.
It is easy to verify that $(\hat{\BFx}^{t-1},\hat{\BFy}^{t-1})$ satisfies \eqref{eqn:p-lower-bound}--\eqref{eqn:p-ramp-down}.
Hence, $(\hat{\BFx}^{t-1},\hat{\BFy}^{t-1})\in\convP$.
Note that $\hat{x}^{t-1}_t=\hat{y}^{t-1}_t=0$, $\hat{y}^{t-1}_{t-1}=1$, $\hat{y}_{t-s}^{t-1}-\hat{y}_{t-s-1}^{t-1}=0$ for all $s\in\Se\setminus\{L\}$, $\hat{y}_{t-L}^{t-1}-\hat{y}_{t-L-1}^{t-1}=1$, and $\Cupper-\Vupper-\eta V=\Cupper-\Vupper-LV$.
Thus, $(\hat{\BFx}^{t-1},\hat{\BFy}^{t-1})$ satisfies \eqref{eqn-T:x_t-ub-5-1} at equality.
 
\item\label{points:prop9-eq1-A4} We create the same point $(\hat{\BFx}^t,\hat{\BFy}^t)$ as in group \ref{points:prop1-eq1-A3} in the proof of Proposition \ref{prop-T:x_t-ub-1}.
Thus, $(\hat{\BFx}^t,\hat{\BFy}^t)\in\convP$.
It is easy to verify that $(\hat{\BFx}^t,\hat{\BFy}^t)$ satisfies \eqref{eqn-T:x_t-ub-5-1} at equality.

\item\label{points:prop9-eq1-A5} For each $r\in[t+1,T]_{\Z}$, we create the same point $(\hat{\BFx}^r,\hat{\BFy}^r)$ as in group \ref{points:prop1-eq1-A4} in the proof of Proposition \ref{prop-T:x_t-ub-1}.
Thus, $(\hat{\BFx}^r,\hat{\BFy}^r)\in\convP$.
It is easy to verify that $(\hat{\BFx}^r,\hat{\BFy}^r)$ satisfies \eqref{eqn-T:x_t-ub-5-1} at equality.
\end{enumerate}

Table \ref{tab:eqn-T:x_t-ub-5-1-facet-matrix-1} shows a matrix with $2T-1$ rows, where each row represents a point created by this process.
This matrix can be transformed into the matrix in Table \ref{tab:eqn-T:x_t-ub-5-1-facet-matrix-2} via the following Gaussian elimination process:

\begin{enumerate}[label=(\roman*)]
    \item For each $r\in[1,T]_{\Z}\setminus\{t\}$, the point with index $r$ in group (B1), denoted $(\underline{\bar{\BFx}}^r,\underline{\bar{\BFy}}^r)$, is obtained by setting $(\underline{\bar{\BFx}}^r,\underline{\bar{\BFy}}^r)=(\bar{\BFx}^r,\bar{\BFy}^r)-(\hat{\BFx}^t,\hat{\BFy}^t)$.
    Here, $(\bar{\BFx}^r,\bar{\BFy}^r)$ is the point with index $r$ in group \ref{points:prop9-eq1-A1}, and $(\hat{\BFx}^t,\hat{\BFy}^t)$ is the point in group \ref{points:prop9-eq1-A4}.

    \item For each \rred{$r\in[1,t-2]_{\Z}$}, the point with index $r$ in group \rred{(B2)}, denoted $(\underline{\hat{\BFx}}^r,\underline{\hat{\BFy}}^r)$, is obtained by setting $(\underline{\hat{\BFx}}^r,\underline{\hat{\BFy}}^r)=(\hat{\BFx}^r,\hat{\BFy}^r)$ if $t-r-1\notin\Se$, and setting $(\underline{\hat{\BFx}}^r,\underline{\hat{\BFy}}^r)=\rred{(\hat{\BFx}^t,\hat{\BFy}^t)-(\hat{\BFx}^r,\hat{\BFy}^r)}$ if $t-r-1\in\Se$.
    Here, $(\hat{\BFx}^r,\hat{\BFy}^r)$ is the point with index $r$ in group \ref{points:prop9-eq1-A2}, and $(\hat{\BFx}^t,\hat{\BFy}^t)$ is the point in group \ref{points:prop9-eq1-A4}.

    \item The point in group \rred{(B3)}, denoted $(\underline{\hat{\BFx}}^{t-1},\underline{\hat{\BFy}}^{t-1})$, is obtained by setting $(\underline{\hat{\BFx}}^{t-1},\underline{\hat{\BFy}}^{t-1})=(\hat{\BFx}^{t-1},\hat{\BFy}^{t-1})-(\hat{\BFx}^t,\hat{\BFy}^t)$ if $\eta=0$, and setting $(\underline{\hat{\BFx}}^{t-1},\underline{\hat{\BFy}}^{t-1})=(\hat{\BFx}^{t-1},\hat{\BFy}^{t-1})$ if $\eta=(\Cupper-\Vupper)/V$ or $\eta=L\in\Se$.
    Here, $(\hat{\BFx}^{t-1},\hat{\BFy}^{t-1})$ is the point in group \ref{points:prop9-eq1-A3}, and $(\hat{\BFx}^t,\hat{\BFy}^t)$ is the point in group \ref{points:prop9-eq1-A4}.
    
    \item The point in group \rred{(B4)}, denoted $(\underline{\hat{\BFx}}^t,\underline{\hat{\BFy}}^t)$, is obtained by setting $(\underline{\hat{\BFx}}^t,\underline{\hat{\BFy}}^t)=(\hat{\BFx}^t,\hat{\BFy}^t)-(\hat{\BFx}^{t+1},\hat{\BFy}^{t+1})$.
    Here, $(\hat{\BFx}^t,\hat{\BFy}^t)$ is the point in group \ref{points:prop9-eq1-A4}, and $(\hat{\BFx}^{t+1},\hat{\BFy}^{t+1})$ is the point with index $t+1$ in group \ref{points:prop9-eq1-A5}.
    
    \item For each $r\in[t+1,T]_{\Z}$, the point with index $r$ \rred{in group (B5)}, denoted $(\underline{\hat{\BFx}}^r,\underline{\hat{\BFy}}^r)$, is obtained by setting $(\underline{\hat{\BFx}}^r,\underline{\hat{\BFy}}^r)=(\hat{\BFx}^r,\hat{\BFy}^r)-(\hat{\BFx}^{r+1},\hat{\BFy}^{r+1})$ if $r\neq T$, and setting $(\underline{\hat{\BFx}}^r,\underline{\hat{\BFy}}^r)=(\hat{\BFx}^r,\hat{\BFy}^r)$ if $r=T$.
    Here, $(\hat{\BFx}^r,\hat{\BFy}^r)$ and $(\hat{\BFx}^{r+1},\hat{\BFy}^{r+1})$ are the points with indices $r$ and $r+1$, respectively, in group \ref{points:prop9-eq1-A5}.
\end{enumerate}

\afterpage{
\begin{landscape}
\begin{table}
    \renewcommand{\arraystretch}{2}
    \centering
    \caption{A matrix with the rows representing $2T-1$ points in $\convP$ that satisfy inequality \eqref{eqn-T:x_t-ub-5-1} at equality}
    \rule{0pt}{0ex}
    \setlength\tabcolsep{7pt}
    \scriptsize
    \begin{tabular}{|c|c|c|*{8}{c}|*{8}{c}|}
        \hline
         \multirow{2}{*}{Group} & \multirow{2}{*}{Point} & \multirow{2}{*}{Index $r$} & \multicolumn{8}{c|}{$\BFx$} & \multicolumn{8}{c|}{$\BFy$} \\
         \cline{4-11}\cline{12-19}
         &&& $1$ & $\cdots$ & $t-2$ & $t-1$ & $t$ & $t+1$ & $\cdots$ & $T$ 
         & $\ \ 1\ \ $ & $\cdots$ & $t-2$ & $t-1$ & $t$ & $t+1$ & $\cdots$ & $\ \ T\ \ $ \\
         \hline
         \multirow{7}{*}{(A1)} & \multirow{7}{*}{$(\bar{\BFx}^r,\bar{\BFy}^r)$}
         & $1$ 
         & $\Cupper-\epsilon$ & $\cdots$ & $\Cupper$ & $\Cupper$ & $\Cupper$ & $\Cupper$ & $\cdots$ & $\Cupper$
         & $1$ & $\cdots$ & $1$ & $1$ & $1$ & $1$ & $\cdots$ & $1$ \\ 
         && $\vdots$ 
         & $\vdots$ & \rotatebox{0}{$\ddots$} & $\vdots$ & $\vdots$ & $\vdots$ & $\vdots$ && $\vdots$
         & $\vdots$ && $\vdots$ & $\vdots$ & $\vdots$ & $\vdots$ && $\vdots$ \\        
         && $t-2$
         & $\Cupper$ & $\cdots$ & $\Cupper-\epsilon$ & $\Cupper$ & $\Cupper$ & $\Cupper$ & $\cdots$ & $\Cupper$
         & $1$ & $\cdots$ & $1$ & $1$ & $1$ & $1$ & $\cdots$ & $1$ \\  
         && $t-1$
         & $\Cupper$ & $\cdots$ & $\Cupper$ & $\Cupper-\epsilon$ & $\Cupper$ & $\Cupper$ & $\cdots$ & $\Cupper$ 
         & $1$ & $\cdots$ & $1$ & $1$ & $1$ & $1$ & $\cdots$ & $1$ \\
         && $t+1$
         & $\Cupper$ & $\cdots$ & $\Cupper$ & $\Cupper$ & $\Cupper$ & $\Cupper-\epsilon$ & $\cdots$ & $\Cupper$ 
         & $1$ & $\cdots$ & $1$ & $1$ & $1$ & $1$ & $\cdots$ & $1$ \\
         && $\vdots$ 
         & $\vdots$ && $\vdots$ & $\vdots$ & $\vdots$ & $\vdots$ & \rotatebox{0}{$\ddots$} & $\vdots$
         & $\vdots$ && $\vdots$ & $\vdots$ & $\vdots$ & $\vdots$ && $\vdots$ \\        
         && $T$
         & $\Cupper$ & $\cdots$ & $\Cupper$ & $\Cupper$ & $\Cupper$ & $\Cupper$ & $\cdots$ & $\Cupper-\epsilon$ 
         & $1$ & $\cdots$ & $1$ & $1$ & $1$ & $1$ & $\cdots$ & $1$ \\ 
         
         \hline
         \multirow{3}{*}{\rred{(A2)}} & \multirow{8}{*}{$(\hat{\BFx}^r,\hat{\BFy}^r)$}
         & $1$ & \multicolumn{8}{c|}{\multirow{3}{*}{(See Note \rred{\ref{tab:eqn-T:x_t-ub-5-1-facet-matrix-1}-1)}}} & \multicolumn{8}{c|}{\multirow{3}{*}{(See Note \rred{\ref{tab:eqn-T:x_t-ub-5-1-facet-matrix-1}-1)}}}\\
         && $\vdots$ & \multicolumn{8}{c|}{} & \multicolumn{8}{c|}{}\\
         && $t-2$ & \multicolumn{8}{c|}{} & \multicolumn{8}{c|}{}\\
         
         \cline{1-1} \cline{3-19}
         \rred{(A3)} && $t-1$ & \multicolumn{8}{c|}{(See Note \rred{\ref{tab:eqn-T:x_t-ub-5-1-facet-matrix-1}-2)}} & \multicolumn{8}{c|}{(See Note \rred{\ref{tab:eqn-T:x_t-ub-5-1-facet-matrix-1}-2)}}\\
         
         \cline{1-1} \cline{3-19}
         \rred{(A4)} && $t$
         & $\Cupper$ & $\cdots$ & $\Cupper$ & $\Cupper$ & $\Cupper$ & $\Cupper$ & $\cdots$ & $\Cupper$
         & $1$ & $\cdots$ & $1$ & $1$ & $1$ & $1$ & $\cdots$ & $1$ \\
         
         \cline{1-1} \cline{3-19}
         \multirow{3}{*}{\rred{(A5)}} &  
         & $t+1$
         & $0$ & $\cdots$ & $0$ & $0$ & $0$ & $\Clower$ & $\cdots$ & $\Clower$
         & $0$ & $\cdots$ & $0$ & $0$ & $0$ & $1$ & $\cdots$ & $1$ \\
         && $\vdots$ 
         & $\vdots$ && $\vdots$ & $\vdots$ & $\vdots$ & $\vdots$ & \rotatebox{0}{$\ddots$} & $\vdots$ 
         & $\vdots$ && $\vdots$ & $\vdots$ & $\vdots$ & $\vdots$ & \rotatebox{0}{$\ddots$} & $\vdots$ \\                 
         && $T$
         & $0$ & $\cdots$ & $0$ & $0$ & $0$ & $0$ & $\cdots$ & $\Clower$
         & $0$ & $\cdots$ & $0$ & $0$ & $0$ & $0$ & $\cdots$ & $1$ \\
        \hline
        
        \multicolumn{19}{l}{Note \rred{\ref{tab:eqn-T:x_t-ub-5-1-facet-matrix-1}-1}:
        For $r\in\rred{[1,t-2]_\Z}$, the $\BFx$ and $\BFy$ vectors in group \rred{(A2)} are given as follows:  
        $\hat{\BFx}^r=(\underbrace{\Clower,\ldots,\Clower}_{r\ {\rm terms}},\underbrace{0,\ldots,0}_{T-r\ {\rm terms}\!\!\!\!\!})$
        and
        $\hat{\BFy}^r=(\underbrace{1,\ldots,1}_{r\ {\rm terms}},\underbrace{0,\ldots,0}_{T-r\ {\rm terms}\!\!\!\!\!})$
        if $t-r-1\notin\Se$;}\\
        \multicolumn{19}{l}{
        $\hat{\BFx}^r=(\underbrace{0,\ldots,0}_{r\ {\rm terms}},\underbrace{\Vupper,\Vupper+V,\Vupper+2V,\ldots,\Vupper+(t-r-1)V}_{t-r\ {\rm terms}},\underbrace{\Vupper+(t-r-1)V,\ldots,\Vupper+(t-r-1)V}_{T-t\ {\rm terms}\!\!\!\!\!})$
        and
        $\hat{\BFy}^r=(\underbrace{0,\ldots,0}_{r\ {\rm terms}},\underbrace{1,\ldots,1}_{T-r\ {\rm terms}\!\!\!\!\!})$
        if $t-r-1\in\Se$.}\\
        \multicolumn{19}{l}{Note \rred{\ref{tab:eqn-T:x_t-ub-5-1-facet-matrix-1}-2}: The $\BFx$ and $\BFy$ vectors in group \rred{(A3)} are given as follows:
        $\hat{\BFx}^{t-1}=(\underbrace{0,\ldots,0}_{\!\!\!\!\!t\!-\!1\ {\rm terms}},\underbrace{\Vupper,\ldots,\Vupper}_{T\!-\!t\!+\!1\ {\rm terms}\!\!\!\!\!\!\!\!})$
        and
        $\hat{\BFy}^{t-1}=(\underbrace{0,\ldots,0}_{\!\!\!\!\!\!\!\!t\!-\!1\ {\rm terms}},\underbrace{1,\ldots,1}_{T\!-\!t\!+\!1\ {\rm terms}\!\!\!\!\!\!\!\!\!\!\!\!})$
        if $\eta=0$;}\\
        \multicolumn{19}{l}{
        $\hat{\BFx}^{t-1}=(\underbrace{\Clower,\ldots,\Clower}_{\!\!\!\!\!t\!-\!1\ {\rm terms}},\underbrace{0,\ldots,0}_{T\!-\!t\!+\!1\ {\rm terms}\!\!\!\!\!\!\!\!})$
        and
        $\hat{\BFy}^{t-1}=(\underbrace{1,\ldots,1}_{\!\!\!\!\!\!t\!-\!1\ {\rm terms}},\underbrace{0,\ldots,0}_{T\!-\!t\!+\!1\ {\rm terms}\!\!\!\!\!\!\!\!})$
        if $\eta=(\Cupper-\Vupper)/V$;
        $\hat{\BFx}^{t-1}=(\underbrace{0,\ldots,0}_{\!\!\!\!\!t\!-\!L\!-\!1\ {\rm terms}},\underbrace{\Vupper,\ldots,\Vupper}_{L\ {\rm terms}},\underbrace{0,\ldots,0}_{T\!-\!t\!+\!1\ {\rm terms}\!\!\!\!\!\!\!\!})$
        and
        $\hat{\BFy}^{t-1}=(\underbrace{0,\ldots,0}_{\!\!\!\!\!t\!-\!L\!-\!1\ {\rm terms}},\underbrace{1,\ldots,1}_{L\ {\rm terms}},\underbrace{0,\ldots,0}_{T\!-\!t\!+\!1\ {\rm terms}\!\!\!\!\!\!\!\!})$
        if $\eta=L\in\Se$.}
    \end{tabular}
    \label{tab:eqn-T:x_t-ub-5-1-facet-matrix-1}
\end{table}
\end{landscape}

\begin{landscape}
\begin{table}
    \renewcommand{\arraystretch}{2}
    \centering
    \caption{Lower triangular matrix obtained from Table \ref{tab:eqn-T:x_t-ub-5-1-facet-matrix-1} via Gaussian elimination}
    \rule{0pt}{4ex}
    \setlength\tabcolsep{6.5pt}
    \scriptsize
    \begin{tabular}{|c|c|c|*{8}{c}|*{8}{c}|}
        \hline
         \multirow{2}{*}{Group} & \multirow{2}{*}{Point} & \multirow{2}{*}{Index $r$} & \multicolumn{8}{c|}{$\BFx$} & \multicolumn{8}{c|}{$\BFy$} \\
         \cline{4-11}\cline{12-19}
         &&& $\ \ 1\ \ $ & $\cdots$ & $t-2$ & $t-1$ & $t$ & $t+1$ & $\cdots$ & $\ \ T\ \ $ 
         & $\ \ 1\ \ $ & $\cdots$ & $t\!-\!2$ & $\ t\!-\!1\ $ & $\ t\ $ & $\ t\!+\!1\ $ & $\cdots$ & $\ \ T\ \ $ \\
         \hline
         \multirow{7}{*}{(B1)} & \multirow{7}{*}{$(\underline{\bar{\BFx}}^r,\underline{\bar{\BFy}}^r)$}
         & $1$ 
         & $-\epsilon$ & $\cdots$ & $0$ & $0$ & $0$ & $0$ & $\cdots$ & $0$
         & $0$ & $\cdots$ & $0$ & $0$ & $0$ & $0$ & $\cdots$ & $0$ \\ 
         && $\vdots$ 
         & $\vdots$ & \rotatebox{0}{$\ddots$} & $\vdots$ & $\vdots$ & $\vdots$ & $\vdots$ && $\vdots$
         & $\vdots$ && $\vdots$ & $\vdots$ & $\vdots$ & $\vdots$ && $\vdots$ \\        
         && $t-2$
         & $0$ & $\cdots$ & $-\epsilon$ & $0$ & $0$ & $0$ & $\cdots$ & $0$
         & $0$ & $\cdots$ & $0$ & $0$ & $0$ & $0$ & $\cdots$ & $0$ \\  
         && $t-1$
         & $0$ & $\cdots$ & $0$ & $-\epsilon$ & $0$ & $0$ & $\cdots$ & $0$ 
         & $0$ & $\cdots$ & $0$ & $0$ & $0$ & $0$ & $\cdots$ & $0$ \\
         && $t+1$
         & $0$ & $\cdots$ & $0$ & $0$ & $0$ & $-\epsilon$ & $\cdots$ & $0$ 
         & $0$ & $\cdots$ & $0$ & $0$ & $0$ & $0$ & $\cdots$ & $0$ \\
         && $\vdots$ 
         & $\vdots$ && $\vdots$ & $\vdots$ & $\vdots$ & $\vdots$ & \rotatebox{0}{$\ddots$} & $\vdots$
         & $\vdots$ && $\vdots$ & $\vdots$ & $\vdots$ & $\vdots$ && $\vdots$ \\        
         && $T$
         & $0$ & $\cdots$ & $0$ & $0$ & $0$ & $0$ & $\cdots$ & $-\epsilon$ 
         & $0$ & $\cdots$ & $0$ & $0$ & $0$ & $0$ & $\cdots$ & $0$ \\ 
         
         \hline
         \multirow{3}{*}{\rred{(B2)}} &  \multirow{8}{*}{\rred{$(\underline{\hat{\BFx}}^r,\underline{\hat{\BFy}}^r)$}}
         & $1$
         & \multicolumn{8}{c|}{}
         & $\rred 1$ & $\rred \cdots$ & $\rred 0$ & $\rred 0$ & $\rred 0$ & $\rred 0$ & $\rred \cdots$ & $\rred 0$ \\
         && $\vdots$ 
         & \multicolumn{8}{c|}{(Omitted)} 
         & $\rred \vdots$ & \rotatebox{0}{$\rred \ddots$} & $\rred \vdots$ & $\rred \vdots$ & $\rred \vdots$ & $\rred \vdots$ && $\rred \vdots$ \\                 
         && $t-2$
         & \multicolumn{8}{c|}{}
         & $\rred 1$ & $\rred \cdots$ & $\rred 1$ & $\rred 0$ & $\rred 0$ & $\rred 0$ & $\rred \cdots$ & $\rred 0$ \\
         
         \cline{1-1} \cline{3-19}
         \rred{(B3)} && $t-1$ & \multicolumn{8}{c|}{(Omitted)} & \multicolumn{8}{c|}{(See Note \rred{\ref{tab:eqn-T:x_t-ub-5-1-facet-matrix-2}-1)}}\\
         
         \cline{1-1} \cline{3-19}
         \rred{(B4)} && $t$
         & \multicolumn{8}{c|}{(Omitted)}
         & $1$ & $\cdots$ & $1$ & $1$ & $1$ & $0$ & $\cdots$ & $0$ \\
         
         \cline{1-1} \cline{3-19}
         \multirow{3}{*}{\rred{(B5)}} &  
         & $t+1$
         & \multicolumn{8}{c|}{}
         & $0$ & $\cdots$ & $0$ & $0$ & $0$ & $1$ & $\cdots$ & $0$ \\
         && $\vdots$ 
         & \multicolumn{8}{c|}{(Omitted)} 
         & $\vdots$ && $\vdots$ & $\vdots$ & $\vdots$ & $\vdots$ & \rotatebox{0}{$\ddots$} & $\vdots$ \\                 
         && $T$
         & \multicolumn{8}{c|}{}
         & $0$ & $\cdots$ & $0$ & $0$ & $0$ & $0$ & $\cdots$ & $1$ \\
        \hline
        
        \multicolumn{19}{l}{Note \rred{\ref{tab:eqn-T:x_t-ub-5-1-facet-matrix-2}-1}: The $\BFy$ vector in group \rred{(B3)} is given as follows:}\\
        \multicolumn{19}{l}{
        $\underline{\hat{\BFy}}^{t-1}=(\underbrace{-1,\ldots,-1}_{t-1\ {\rm terms}},\underbrace{0,\ldots,0}_{T\!-\!t\!+\!1\ {\rm terms}\!\!\!\!\!\!\!\!})$
        if $\eta=0$;
        $\underline{\hat{\BFy}}^{t-1}=(\underbrace{1,\ldots,1}_{\!\!\!t\!-\!1\ {\rm terms}},\underbrace{0,\ldots,0}_{T\!-\!t\!+\!1\ {\rm terms}\!\!\!\!\!\!\!\!})$
        if $\eta=(\Cupper-\Vupper)/V$;
        $\underline{\hat{\BFy}}^{t-1}=(\underbrace{0,\ldots,0}_{\!\!\!\!\!t\!-\!L\!-\!1\ {\rm terms}},\underbrace{1,\ldots,1}_{L\ {\rm terms}},\underbrace{0,\ldots,0}_{T\!-\!t\!+\!1\ {\rm terms}\!\!\!\!\!\!\!\!})$
        if $\eta=L\in\Se$.}
    \end{tabular}
    \label{tab:eqn-T:x_t-ub-5-1-facet-matrix-2}
\end{table}
\end{landscape}
}
    

The matrix shown in Table \ref{tab:eqn-T:x_t-ub-5-1-facet-matrix-2} is lower triangular; that is, the position of the last nonzero component of a row of the matrix is greater than the position of the last nonzero component of the previous row.
This implies that the $2T-1$ points in groups \ref{points:prop9-eq1-A1}--\ref{points:prop9-eq1-A5} are linearly independent.
Therefore, inequality \eqref{eqn-T:x_t-ub-5-1} is facet-defining for $\convP$.

Next, we show that inequality \eqref{eqn-T:x_t-ub-5-2} is valid for $\convP$ and is facet-defining for $\convP$ when $\eta\in\{0,(\Cupper-\Vupper)/V\}$ or $\eta=L\in\Se$.
Note that this proof can also be used to prove the validity and facet-defining of inequality \eqref{eqn-T:x_t-ub-5-2} in Proposition~\ref{prop-T:x_t-ub-6}.
Denote $x'_t=x_{T-t+1}$ and $y'_t=y_{T-t+1}$ for $t\in[1,T]_\Z$.
Because inequality \eqref{eqn-T:x_t-ub-5-1} is valid for $\convP$ and is facet-defining for $\convP$ when $\eta\in\{0,(\Cupper-\Vupper)/V\}$ or $\eta=L\in\Se$ for any $t\in[s_{\max}+2,T]_\Z$, the inequality
$$x'_{T-t+1}\le(\Vupper+\eta V)y'_{T-t+1}+(\Cupper-\Vupper-\eta V)y'_{T-t+2}-\sum_{s\in\Se}(\Cupper-\Vupper-sV)(y'_{T-t+s+1}-y'_{T-t+s+2})$$
is valid for $\convPprime$ and is facet-defining for $\convPprime$ when $\eta\in\{0,(\Cupper-\Vupper)/V\}$ or $\eta=L\in\Se$ for any $t\in[s_{\max}+2,T]_\Z$.
Let $t'=T-t+1$.
Then, the inequality
$$x'_{t'}\le(\Vupper+\eta V)y'_{t'}+(\Cupper-\Vupper-\eta V)y'_{t'+1}-\sum_{s\in\Se}(\Cupper-\Vupper-sV)(y'_{t'+s}-y'_{t'+s+1})$$
is valid for $\convPprime$ and is facet-defining for $\convPprime$ when $\eta\in\{0,(\Cupper-\Vupper)/V\}$ or $\eta=L\in\Se$ for any $t'\in[1,T-s_{\max}-1]_\Z$.
Hence, by Lemma~\ref{lem:Pprime}, inequality \eqref{eqn-T:x_t-ub-5-2} is valid for $\convP$ and is facet-defining for $\convP$ when $\eta\in\{0,(\Cupper-\Vupper)/V\}$ or $\eta=L\in\Se$ for any $t\in[1,T-s_{\max}-1]_\Z$.\Halmos

\subsection{Proof of Proposition \ref{prop-T:x_t-ub-6}}

\noindent{\bf Proposition \ref{prop-T:x_t-ub-6}.} {\it
Consider any integers $\alpha$, $\beta$, and $s_{\max}$ such that
(a)~$L+1\le s_{\max}\le\min\{T-2,\lfloor(\Cupper-\Vupper)/V\rfloor\}$,
(b)~$1\le\alpha<\beta\le s_{\max}$, and
(c)~$\beta=\alpha+1$ or $s_{\max}\le L+\alpha$.
Let $\Se=[1,\alpha]_{\Z}\cup[\beta,s_{\max}]_{\Z}$.
For any $t\in[s_{\max}+2,T]_{\Z}$, inequality \eqref{eqn-T:x_t-ub-5-1} is valid for $\convP$.
For any $t\in[1,T-s_{\max}-1]_{\Z}$, inequality \eqref{eqn-T:x_t-ub-5-2} is valid for $\convP$.
Furthermore, \eqref{eqn-T:x_t-ub-5-1} and \eqref{eqn-T:x_t-ub-5-2} are facet-defining for $\convP$ when $\eta\in\{0,(\Cupper-\Vupper)/V\}$ or $\eta=L\in\Se$.}
\vskip8pt

\noindent{\bf Proof.}
Consider any $t\in[s_{\max}+2,T]_\Z$.
To prove that the linear inequality \eqref{eqn-T:x_t-ub-5-1} is valid for $\convP$ when $\Se=[1,\alpha]_\Z\cup[\beta,s_{\max}]_\Z$, 
it suffices to show that \eqref{eqn-T:x_t-ub-5-1} is valid for $\Pe$ when $\Se=[1,\alpha]_\Z\cup[\beta,s_{\max}]_\Z$.
Consider any element $(\BFx,\BFy)$ of $\Pe$.
We show that $(\BFx,\BFy)$ satisfies \eqref{eqn-T:x_t-ub-5-1} when $\Se=[1,\alpha]_\Z\cup[\beta,s_{\max}]_\Z$.
We divide the analysis into four cases.

Case~1: $y_t=0$ and $y_{t-s}-y_{t-s-1}\le0$ for all $s\in\Se$.
Because $\Se\subseteq[1,\lfloor(\Cupper-\Vupper)/V\rfloor]_{\Z}$, we have $\Cupper-\Vupper-sV\ge0$ for all $s\in\Se$.
Because $\eta\le(\Cupper-\Vupper)/V$, we have $\Cupper-\Vupper-\eta V\ge0$.
Thus, in this case, the right-hand side of inequality \eqref{eqn-T:x_t-ub-5-1} is nonnegative.
Because $y_t=0$, by \eqref{eqn:p-upper-bound}, $x_t=0$.
Therefore, in this case, $(\BFx,\BFy)$ satisfies \eqref{eqn-T:x_t-ub-5-1}.

Case~2: $y_t=0$ and $y_{t-s}-y_{t-s-1}>0$ for some $s\in\Se$.
Let $\tilde{\Se}=\{\sigma\in\Se:y_{t-\sigma}-y_{t-\sigma-1}>0\}$ and $v=|\tilde{\Se}|$.
Then, $v\ge1$.
Denote $\tilde{\Se}=\{\sigma_1,\sigma_2,\ldots,\sigma_v\}$, where $\sigma_1<\sigma_2<\cdots<\sigma_v$.
Note that $y_{t-\sigma_j-1}=0$ and $y_{t-\sigma_j}=1$ for $j=1,\ldots,v$.
Denote $\sigma_0=-1$.
Then, for each $j=1,\ldots,v$, there exists $\sigma'_j\in[\sigma_{j-1}+1,\sigma_j-1]_{\Z}$ such that $y_{t-\sigma'_j-1}=1$ and $y_{t-\sigma'_j}=0$.
Thus,
$$0\le\sigma'_1<\sigma_1<\sigma'_2<\sigma_2<\cdots<\sigma'_v<\sigma_v\le s_{\max}.$$
Because $y_{t-\sigma_v}-y_{t-\sigma_v-1}=1$ and $t-\sigma_v\in[2,T]_{\Z}$, by \eqref{eqn:p-minup}, $y_k=1$ for all $k\in[t-\sigma_v,\min\{T,t-\sigma_v+L-1\}]_{\Z}$, which implies that $t-\sigma'_j\ge t-\sigma_v+L$ for $j=1,\ldots,v$.
Hence, for $j=1,\ldots,v$, we have $\sigma'_j\le\sigma_v-L$, which implies that
\begin{equation}\label{eq:Prop6_eqA}
    \sigma'_j\le s_{\max}-L.
\end{equation}
If $\beta=\alpha+1$, then $\Se=[1,s_{\max}]_{\Z}$, which implies that $\sigma'_j\in\Se$ for $j=2,\ldots,v$.
If $\beta\neq\alpha+1$, then condition (c) of Proposition \ref{prop-T:x_t-ub-6} implies that $s_{\max}\le L+\alpha$, which, by \eqref{eq:Prop6_eqA}, implies that $\sigma'_j\le\alpha$ for $j=1,\ldots,v$.
Because $\sigma'_2>\sigma_1\ge1$, we have $\sigma'_j\in\Se$ for $j=2,\ldots,v$.
Thus, in both cases, $\sigma'_j\in\Se$ for $j=2,\ldots,v$.
Because $y_{t-\sigma_1}-y_{t-\sigma_1-1}=1$ and $t-\sigma_1\in[2,T]_{\Z}$, by \eqref{eqn:p-minup}, $y_k=1$ for all $k\in[t-\sigma_1,\min\{T,t-\sigma_1+L-1\}]_{\Z}$.
Because $y_t=0$, , this implies that $t\ge t-\sigma_1+L$, or equivalent, $\sigma_1\ge L$.
Because $\eta\le L$, we have
\begin{equation}\label{eq:Prop6_eqB}
  \eta\le\sigma_1.
\end{equation}
Because $y_t=0$, by \eqref{eqn:p-upper-bound}, $x_t=0$.
Hence, the left-hand side of inequality \eqref{eqn-T:x_t-ub-5-1} is $0$.
Because $s_{\max}\le\lfloor(\Cupper-\Vupper)/V\rfloor$, we have $\Cupper-\Vupper-sV\ge0$ for all $s\in\Se$.
Note that if $y_{t-1}=0$, then $\sigma'_1\ge1$ and $\sigma'_1\in\Se$; if $y_{t-1}=1$, then $\sigma'_1=0$ and $\sigma'_1\notin\Se$.
Note that $\{\sigma'_2,\ldots,\sigma'_v\}\subseteq\Se\setminus\tilde{\Se}$ and $y_{t-s}-y_{t-s-1}\le0$ for all $s\in\Se\setminus\tilde{\Se}$.
Thus, $\sum_{s\in\Se\setminus\tilde{\Se}}(\Cupper-\Vupper-sV)(y_{t-s}-y_{t-s-1})\le\sum_{j=2}^v(\Cupper-\Vupper-\sigma'_jV)(y_{t-\sigma'_j}-y_{t-\sigma'_j-1})+(\Cupper-\Vupper-\sigma_1'V)(y_{t-\sigma'_1}-y_{t-\sigma'_1-1})(1-y_{t-1})$.
Hence, the right-hand side of inequality \eqref{eqn-T:x_t-ub-5-1} is
\begin{align*}
    &(\Vupper+\eta V)y_t+(\Cupper-\Vupper-\eta V)y_{t-1}-\sum_{s\in\Se}(\Cupper-\Vupper-sV)(y_{t-s}-y_{t-s-1})\\
    &= (\Cupper-\Vupper-\eta V)y_{t-1}-\sum_{s\in\tilde{\Se}}(\Cupper-\Vupper-sV)(y_{t-s}-y_{t-s-1})-\sum_{s\in\Se\setminus\tilde{\Se}}(\Cupper-\Vupper-sV)(y_{t-s}-y_{t-s-1})\\
    &\ge (\Cupper-\Vupper-\eta V)y_{t-1}-\sum_{j=1}^v(\Cupper-\Vupper-\sigma_jV)(y_{t-\sigma_j}-y_{t-\sigma_j-1})-\sum_{j=2}^v(\Cupper-\Vupper-\sigma'_jV)(y_{t-\sigma'_j}-y_{t-\sigma'_j-1})\\
    &\quad -(\Cupper-\Vupper-\sigma'_1V)(y_{t-\sigma'_1}-y_{t-\sigma'_1-1})(1-y_{t-1})\\
    &= (\Cupper-\Vupper-\eta V)y_{t-1}-(\Cupper-\Vupper-\sigma'_1V)y_{t-1}-(\Cupper-\Vupper-\sigma_1V)+(\Cupper-\Vupper-\sigma'_1V)\\
    &\quad -\sum_{j=2}^v(\Cupper-\Vupper-\sigma_jV)+\sum_{j=2}^v(\Cupper-\Vupper-\sigma'_jV)\\
    &=(\sigma'_1-\eta)Vy_{t-1}+(\sigma_1-\sigma'_1)V+\sum_{j=2}^v(\sigma_j-\sigma'_j)V\\
    & \ge (\sigma'_1-\eta)Vy_{t-1}+(\sigma_1-\sigma'_1)V\\
    &\ge 0.
\end{align*}
where the last inequality follows from $y_{t-1}\in\{0,1\}$, $\sigma_1>\sigma'_1$, and \eqref{eq:Prop6_eqB}.
Therefore, in this case, $(\BFx,\BFy)$ satisfies \eqref{eqn-T:x_t-ub-5-1}.

Case~3: $y_t=1$ and $y_{t-s}-y_{t-s-1}\le0$ for all $s\in\Se$.
Because $\Se\subseteq[0,\lfloor(\Cupper-\Vupper)/V\rfloor]_{\Z}$, we have $\Cupper-\Vupper-sV\ge0$ for all $s\in\Se$.
Thus, $\sum_{s\in\Se}(\Cupper-\Vupper-sV)(y_{t-s}-y_{t-s-1})\le0$.
Because $y_t=1$, the right-hand side of inequality \eqref{eqn-T:x_t-ub-5-1} is at least $\Cupper$ when $y_{t-1}=1$ and is at least  $\Vupper+\eta V\ge\Vupper$ when $y_{t-1}=0$ (as $\eta\ge0$).
By \eqref{eqn:p-upper-bound}, $x_t\le\Cupper$.
If $y_{t-1}=0$, then by \eqref{eqn:p-upper-bound} and \eqref{eqn:p-ramp-up}, $x_{t-1}=0$ and $x_t-x_{t-1}\le Vy_{t-1}+\Vupper(1-y_{t-1})$, which imply that $x_t\le\Vupper$.
Hence, $x_t$ is at most $\Cupper$, and is at most $\Vupper$ when $y_{t-1}=0$.
Therefore, in this case, $(\BFx,\BFy)$ satisfies \eqref{eqn-T:x_t-ub-5-1}.

Case~4: $y_t=1$ and $y_{t-s}-y_{t-s-1}>0$ for some $s\in\Se$. 
Let $\tilde{\Se}=\{\sigma\in\Se:y_{t-\sigma}-y_{t-\sigma-1}>0\}$ and $v=|\tilde{\Se}|$.
Then, $v\ge1$.
Denote $\tilde{\Se}=\{\sigma_1,\sigma_2,\ldots,\sigma_v\}$, where $\sigma_1<\sigma_2<\cdots<\sigma_v$.
Note that $y_{t-\sigma_j-1}=0$ and $y_{t-\sigma_j}=1$ for $j=1,\ldots,v$.
Then, for each $j=2,\ldots,v$, there exists $\sigma'_j\in[\sigma_{j-1}+1,\sigma_j-1]_{\Z}$ such that $y_{t-\sigma'_j-1}=1$ and $y_{t-\sigma'_j}=0$.
Thus,
$$1\le\sigma_1<\sigma'_2<\sigma_2<\cdots<\sigma'_v<\sigma_v\le s_{\max}.$$
Because $y_{t-\sigma_v}-y_{t-\sigma_v-1}=1$ and $t-\sigma_v\in[2,T]_{\Z}$, by \eqref{eqn:p-minup}, $y_k=1$ for all $k\in[t-\sigma_v,\min\{T,t-\sigma_v+L-1\}]_{\Z}$, which implies that $t-\sigma'_j\ge t-\sigma_v+L$ for $j=2,\ldots,v$.
Hence, for $j=2,\ldots,v$, we have $\sigma'_j\le\sigma_v-L$, which implies that
\begin{equation}\label{eq:Prop6_eqC}
    \sigma'_j\le s_{\max}-L.
\end{equation}
If $\beta=\alpha+1$, then $\Se=[1,s_{\max}]_{\Z}$, which implies that $\sigma'_j\in\Se$ for $j=2,\ldots,v$.
If $\beta\neq\alpha+1$, then condition (c) of Proposition \ref{prop-T:x_t-ub-6} implies that $s_{\max}\le L+\alpha$, which, by \eqref{eq:Prop6_eqC}, implies that $1<\sigma_j'\le\alpha$ for $j=2,\ldots,v$.
Thus, in both cases, $\sigma'_j\in\Se$ for $j=2,\ldots,v$.
If $y_{t-1}=0$, by \eqref{eqn:p-upper-bound} and \eqref{eqn:p-ramp-up}, then $x_{t-1}=0$ and $x_t-x_{t-1}\le Vy_{t-1}+\Vupper(1-y_{t-1})=\Vupper$, which implies that $x_t\le\Vupper$.
In addition, there exists $\sigma'_1\in[1,\sigma_1-1]_{\Z}$ such that $y_{t-\sigma'_1}=0$, $y_{t-\sigma'_1-1}=1$, and $\sigma'_1\in\Se$.
If $y_{t-1}=1$, then we have $y_k=1$ for all $k\in[t-\sigma_1, t]_{Z}$.
Because $y_{t-\sigma_1-1}=0$, by \eqref{eqn:p-upper-bound} and \eqref{eqn:p-ramp-up}, we have $x_{t-\sigma_1-1}=0$ and
$$\sum_{\tau=t-\sigma_1}^t(x_{\tau}-x_{\tau-1})\le\sum_{\tau=t-\sigma_1}^t Vy_{\tau-1}+\sum_{\tau=t-\sigma_1}^t \Vupper(1-y_{\tau-1}),$$
which implies that
$$x_t-x_{t-\sigma_1-1}\le\sum_{\tau=t-\sigma_1}^t Vy_{\tau-1}+\sum_{\tau=t-\sigma_1}^t \Vupper(1-y_{\tau-1})=\sigma_1V+\Vupper.$$
Thus, the left-hand side of inequality \eqref{eqn-T:x_t-ub-5-1} is at most $\sigma_1V+\Vupper$.
Because $\Se\subseteq[0,\lfloor(\Cupper-\Vupper)/V\rfloor]_{\Z}$, we have $\Cupper-\Vupper-sV\ge0$ for all $s\in\Se$.
Note that $\{\sigma'_2,\ldots,\sigma'_v\}\subseteq\Se\setminus\tilde{\Se}$.
Thus, $\sum_{s\in\Se\setminus\tilde{\Se}}(\Cupper-\Vupper-sV)(y_{t-s}-y_{t-s-1})\le\sum_{j=2}^v(\Cupper-\Vupper-\sigma'_jV)(y_{t-\sigma'_j}-y_{t-\sigma'_j-1})$.
Hence, when $y_{t-1}=0$, the right-hand side of inequality \eqref{eqn-T:x_t-ub-5-1} is
\begin{align*}
  &(\Vupper+\eta V)y_t+(\Cupper-\Vupper-\eta V)y_{t-1}-\sum_{s\in\Se}(\Cupper-\Vupper-sV)(y_{t-s}-y_{t-s-1})\\
  &=\Vupper+\eta V-\sum_{s\in\tilde{\Se}}(\Cupper-\Vupper-sV)(y_{t-s}-y_{t-s-1})-\sum_{s\in\Se\setminus\tilde{\Se}}(\Cupper-\Vupper-sV)(y_{t-s}-y_{t-s-1})\\
  &\ge \Vupper+\eta V-\sum_{j=1}^v(\Cupper-\Vupper-\sigma_jV)(y_{t-\sigma_j}-y_{t-\sigma_j-1})-\sum_{j=1}^v(\Cupper-\Vupper-\sigma'_jV)(y_{t-\sigma'_j}-y_{t-\sigma'_j-1})\\
  &=\Vupper+\eta V+\sum_{j=1}^v(\sigma_j-\sigma'_j)V\\
  &\ge\Vupper+\eta V\\
  &\ge \Vupper.
\end{align*}
When $y_{\rred{t-1}}=1$, the right-hand side of inequality \eqref{eqn-T:x_t-ub-5-1} is
\begin{align*}
    &(\Vupper+\eta V)y_t+(\Cupper-\Vupper-\eta V)y_{t-1}-\sum_{s\in\Se}(\Cupper-\Vupper-sV)(y_{t-s}-y_{t-s-1})\\
    &=\Cupper-\sum_{s\in\tilde{\Se}}(\Cupper-\Vupper-sV)(y_{t-s}-y_{t-s-1})-\sum_{s\in\Se\setminus\tilde{\Se}}(\Cupper-\Vupper-sV)(y_{t-s}-y_{t-s-1})\\
    &\ge \Cupper-\sum_{j=1}^v(\Cupper-\Vupper-\sigma_jV)(y_{t-\sigma_j}-y_{t-\sigma_j-1})-\sum_{j=2}^v(\Cupper-\Vupper-\sigma'_jV)(y_{t-\sigma'_j}-y_{t-\sigma'_j-1})\\
    &=\Cupper-(\Cupper-\Vupper-\sigma_1V)-\sum_{j=2}^v(\Cupper-\Vupper-\sigma_jV)+\sum_{j=2}^v(\Cupper-\Vupper-\sigma'_jV)\\
    &=\sigma_1V+\Vupper+\sum_{j=2}^v(\sigma_j-\sigma'_j)V\\
    &\ge\sigma_1V+\Vupper.
\end{align*}
Therefore, in this case, $(\BFx,\BFy)$ satisfies \eqref{eqn-T:x_t-ub-5-1}.

Summarizing Cases 1--4, we conclude that $(\BFx,\BFy)$ satisfies \eqref{eqn-T:x_t-ub-5-1}.
Hence, \eqref{eqn-T:x_t-ub-5-1} is valid for $\convP$.

It is easy to verify that the proof of facet-defining of inequality \eqref{eqn-T:x_t-ub-5-1} in the proof of Proposition~\ref{prop-T:x_t-ub-5} remains valid when $\Se=[1,\alpha]_\Z\cup[\beta,s_{\max}]_\Z$.
Therefore, inequality \eqref{eqn-T:x_t-ub-5-1} is facet-defining under the conditions stated in Proposition \ref{prop-T:x_t-ub-6}.

It is also easy to verify that the proof of validity and facet-defining of inequality \eqref{eqn-T:x_t-ub-5-2} in the proof of Proposition~\ref{prop-T:x_t-ub-5} remains valid when $\Se=[1,\alpha]_\Z\cup[\beta,s_{\max}]_\Z$.
Therefore, inequality \eqref{eqn-T:x_t-ub-5-2} is valid and facet-defining for $\convP$ under the conditions stated in Proposition \ref{prop-T:x_t-ub-6}.
\Halmos


\subsection{Proof of Proposition \ref{prop:separation-A}}

\noindent{\bf Proposition \ref{prop:separation-A}.} {\it
For any given point $(\BFx,\BFy)\in\R_{+}^{2T}$, \tred{the} most violated inequalities 
\eqref{eqn-T:x_t-ub-1-1}--\eqref{eqn-T:x_t-ub-1-2}, \eqref{eqn-T:x_t-ub-3-1}--\eqref{eqn-T:x_t-ub-3-2}, and \eqref{eqn-T:x_t-ub-5-1}--\eqref{eqn-T:x_t-ub-5-2}
in Propositions \ref{prop-T:x_t-ub-1}, \ref{prop-T:x_t-ub-3}, and \ref{prop-T:x_t-ub-5}, respectively,
can be determined in $O(T)$ time if such violated inequalities exist.}
\vskip8pt

\noindent{\bf Proof.}
Let $\hat{\eta}$ and $a_1,\ldots,a_6$ be any real numbers such that $\hat{\eta}\ge 0$.
Let $\check{s}$ and $\hat{s}$ be any integers such that $0\le\check{s}\le\hat{s}\le\min\{T-2,\lfloor(\Cupper-\Vupper)/V\rfloor\}$.
Let $\check{t}=1$ if $a_1=a_2=0$, and let $\check{t}=2$ otherwise.
Let $\hat{t}=T$ if $a_5=a_6=0$, and let $\hat{t}=T-1$ otherwise.

(i)~Consider the following family of inequalities:
\begin{equation}\label{eq:A-inequality1}
x_t\le(a_1+a_2\eta)y_{t-1}+(a_3+a_4\eta)y_t+(a_5+a_6\eta)y_{t+1}-\sum_{s\in\Se}(\Cupper-\Vupper-sV)(y_{t-s}-y_{t-s-1}),
\end{equation}
where 
$\eta\in[0,\hat{\eta}]$,  
$\Se\subseteq[\check{s},\hat{s}]_\Z$,
$t\in[\check{t},\hat{t}]_\Z$, and
$t\ge s+2$ for all $s\in\Se$.
Note that 
inequality family \eqref{eqn-T:x_t-ub-1-1} in Proposition~\ref{prop-T:x_t-ub-1},
inequality family \eqref{eqn-T:x_t-ub-3-1} in Proposition~\ref{prop-T:x_t-ub-3}, and
inequality family \eqref{eqn-T:x_t-ub-5-1} in Proposition~\ref{prop-T:x_t-ub-5}
are special cases of this inequality family.
Consider any given point $(\BFx,\BFy)\in\R_{+}^{2T}$.
We show that the set $\Se$, the real number $\eta$, and the integer $t$ corresponding to a most violated inequality \eqref{eq:A-inequality1} can be determined in $O(T)$ time.

For any integer $t\le T$, let 
$$\theta(t)=\sum_{\tau=2}^t\max\{y_\tau-y_{\tau-1},0\}.$$
Then, for any $t\in[2,T]_\Z$,
\begin{equation}\label{eq:A-ysummation1}
\sum_{\substack{s\in[\check{s},\hat{s}]_\Z\\ t-s\ge2}}\max\{y_{t-s}-y_{t-s-1},0\}=\theta(t-\check{s})-\theta(t-\hat{s}-1)
\end{equation}
and
\begin{equation}\label{eq:A-ysummation2}
\sum_{\substack{s\in[\check{s}+1,\hat{s}+1]_\Z\\ t-s\ge2}}\max\{y_{t-s}-y_{t-s-1},0\}=\theta(t-\check{s}-1)-\theta(t-\hat{s}-2).
\end{equation}
Furthermore, for any $t\in[2,T]_\Z$,
\begin{align*}
&\sum_{\substack{s\in[\check{s},\hat{s}]_\Z\\ t-s\ge2}}s\max\{y_{t-s}-y_{t-s-1},0\}-\sum_{\substack{s\in[\check{s}+1,\hat{s}+1]_\Z\\ t-s\ge2}}s\max\{y_{t-s}-y_{t-s-1},0\}\\
&\qquad=\left\{
\begin{array}{ll}
\check{s}\max\{y_{t-\check{s}}-y_{t-\check{s}-1},0\}-(\hat{s}+1)\max\{y_{t-\hat{s}-1}-y_{t-\hat{s}-2},0\},&\hbox{if $2\le t-\hat{s}-1$;}\\[2pt]
\check{s}\max\{y_{t-\check{s}}-y_{t-\check{s}-1},0\},&\hbox{if $t-\hat{s}-1<2\le t-\check{s}$;}\\[2pt]
0,&\hbox{if $t-\check{s}<2$.}
\end{array}
\right.
\end{align*}
Note that
\begin{align*}
    \theta(t-\check{s})-\theta(t-\check{s}-1)=\left\{
    \begin{array}{ll}
      \max\{y_{t-\check{s}}-y_{t-\check{s}-1},0\}, &\hbox{if $t-\check{s}\ge 2$;}\\[2pt]
      0,&\hbox{if $t-\check{s}<2$.}
    \end{array}\right. 
\end{align*}
and
\begin{align*}
    \theta(t-\hat{s}-1)-\theta(t-\hat{s}-2)=\left\{
    \begin{array}{ll}
      \max\{y_{t-\hat{s}-1}-y_{t-\hat{s}-2},0\}, &\hbox{if $t-\hat{s}-1\ge 2$;}\\[2pt]
      0,&\hbox{if $t-\hat{s}-1<2$.}
    \end{array}\right. 
\end{align*}
Hence, for any $t\in[2,T]_\Z$,
\begin{align}\label{eq:A-ysummation3}
&\sum_{\substack{s\in[\check{s},\hat{s}]_\Z\\ t-s\ge2}}s\max\{y_{t-s}-y_{t-s-1},0\}-\sum_{\substack{s\in[\check{s}+1,\hat{s}+1]_\Z\\ t-s\ge2}}s\max\{y_{t-s}-y_{t-s-1},0\}\nonumber\\
&\qquad\qquad=\check{s}[\theta(t-\check{s})-\theta(t-\check{s}-1)]-(\hat{s}+1)[\theta(t-\hat{s}-1)-\theta(t-\hat{s}-2)].
\end{align}
For any $\eta\in[0,\hat{\eta}]$, $\Se\subseteq[\check{s},\hat{s}]_\Z$, and $t\in[\check{t},\hat{t}]_\Z$ such that $t\ge s+2$ $\forall s\in\Se$, let
$$\tilde{v}(\eta,\Se,t)=x_t-(a_1+a_2\eta)y_{t-1}-(a_3+a_4\eta)y_t-(a_5+a_6\eta)y_{t+1}+\sum_{s\in\Se}(\Cupper-\Vupper-sV)(y_{t-s}-y_{t-s-1}).$$
If $\tilde{v}(\eta,\Se,t)>0$, then $\tilde{v}(\eta,\Se,t)$ is the amount of violation of inequality \eqref{eq:A-inequality1}.
If $\tilde{v}(\eta,\Se,t)\le0$, then there is no violation of inequality \eqref{eq:A-inequality1}.
For any $\eta\in[0,\hat{\eta}]$ and $t\in[\check{t},\hat{t}]_\Z$, let
$$v(\eta,t)=\max_{\Se\subseteq[\check{s},\min\{\hat{s},t-2\}]_\Z}\{\tilde{v}(\eta,\Se,t)\}.$$
If $v(\eta,t)>0$, then $v(\eta,t)$ is the largest possible violation of inequality \eqref{eq:A-inequality1} for this combination of $\eta$ and $t$.
If $v(\eta,t)\le 0$, then the largest possible violation of inequality \eqref{eq:A-inequality1} is zero for this combination of $\eta$ and $t$.

Note that $\Cupper-\Vupper-sV\ge0$ for any $s\in[\check{s},\hat{s}]_\Z$.
Thus, for any given $\eta\in[0,\hat{\eta}]$ and $t\in[\check{t},\hat{t}]_\Z$, 
$\tilde{v}(\eta,\Se,t)$ is maximized when $\Se$ contains all $s\in[\check{s},\min\{\hat{s},t-2\}]_\Z$ such that $y_{t-s}-y_{t-s-1}>0$ (if any).
If it does not exist any $s\in[\check{s},\min\{\hat{s},t-2\}]_{\Z}$ such that $y_{t-s}-y_{t-s-1}>0$, 
then $\tilde{v}(\eta,\Se,t)$ is maximized when $\Se=\emptyset$, and $v(\eta,t)=x_t-(a_1+a_2\eta)y_{t-1}-(a_3+a_4\eta)y_t-(a_5+a_6\eta)y_{t+1}$.
Hence, for any $\eta\in[0,\hat{\eta}]$ and $t\in[\check{t},\hat{t}]_\Z$,
\begin{align*}
v(\eta,t)=\ &x_t-(a_1+a_2\eta)y_{t-1}-(a_3+a_4\eta)y_t-(a_5+a_6\eta)y_{t+1}\\
            &+(\Cupper-\Vupper)\sum_{\substack{s\in[\check{s},\hat{s}]_\Z\\ t-s\ge 2}}\max\{y_{t-s}-y_{t-s-1},0\}
             -V\sum_{\substack{s\in[\check{s},\hat{s}]_\Z\\ t-s\ge 2}}s\max\{y_{t-s}-y_{t-s-1},0\}.
\end{align*}
When $t=\check{t}$, we have
\begin{align*}
v(\eta,\check{t})=\left\{
  \begin{array}{ll}
    x_1-(a_3+a_4\eta)y_1-(a_5+a_6\eta)y_2, &\hbox{ if $\check{t}=1$;}\\[2pt]
    x_2-(a_1+a_2\eta)y_1-(a_3+a_4\eta)y_2-(a_5+a_6\eta)y_3\\[2pt]
    \qquad\qquad\qquad\qquad\quad+(\Cupper-\Vupper)\max\{y_2-y_1,0\},&\hbox{ if $\check{t}=2$ and $\check{s}=0$;}\\[2pt]
    x_2-(a_1+a_2\eta)y_1-(a_3+a_4\eta)y_2-(a_5+a_6\eta)y_3, &\hbox{ otherwise.}
  \end{array}\right.
\end{align*}
For any $\eta\in[0,\hat{\eta}]$ and $t\in[\check{t}+1,\hat{t}]_\Z$,
\begin{align*}
v(\eta,t)-v(\eta,t-1)=\ &(x_t-x_{t-1})\\
                        &-(a_1+a_2\eta)(y_{t-1}-y_{t-2})-(a_3+a_4\eta)(y_t-y_{t-1})-(a_5+a_6\eta)(y_{t+1}-y_t)\\
                        &+(\Cupper-\Vupper)\left[\sum_{\substack{s\in[\check{s},\hat{s}]_\Z\\ t-s\ge2}}\max\{y_{t-s}-y_{t-s-1},0\}
                         -\sum_{\substack{s\in[\check{s}+1,\hat{s}+1]_\Z\\ t-s\ge2}}\max\{y_{t-s}-y_{t-s-1},0\}\right]\\
                        &-V\left[\sum_{\substack{s\in[\check{s},\hat{s}]_\Z\\ t-s\ge2}}s\max\{y_{t-s}-y_{t-s-1},0\}
                         -\sum_{\substack{s\in[\check{s}+1,\hat{s}+1]_\Z\\ t-s\ge2}}(s-1)\max\{y_{t-s}-y_{t-s-1},0\}\right],
\end{align*}
which implies that
\begin{align*}
v(\eta,t)=\ &v(\eta,t-1)+(x_t-x_{t-1})\\
            &-(a_1+a_2\eta)(y_{t-1}-y_{t-2})-(a_3+a_4\eta)(y_t-y_{t-1})-(a_5+a_6\eta)(y_{t+1}-y_t)\\
            &+(\Cupper-\Vupper)\left[\sum_{\substack{s\in[\check{s},\hat{s}]_\Z\\ t-s\ge2}}\max\{y_{t-s}-y_{t-s-1},0\}
             -\sum_{\substack{s\in[\check{s}+1,\hat{s}+1]_\Z\\ t-s\ge2}}\max\{y_{t-s}-y_{t-s-1},0\}\right]\allowdisplaybreaks\\
            &-V\sum_{\substack{s\in[\check{s}+1,\hat{s}+1]_\Z\\ t-s\ge2}}\max\{y_{t-s}-y_{t-s-1},0\}\\
            &-V\left[\sum_{\substack{s\in[\check{s},\hat{s}]_\Z\\ t-s\ge2}}s\max\{y_{t-s}-y_{t-s-1},0\}
             -\sum_{\substack{s\in[\check{s}+1,\hat{s}+1]_\Z\\ t-s\ge2}}s\max\{y_{t-s}-y_{t-s-1},0\}\right].
\end{align*}
Thus, by \eqref{eq:A-ysummation1}, \eqref{eq:A-ysummation2}, and \eqref{eq:A-ysummation3},
\begin{align}\label{eq:A-recursion1}
v(\eta,t)=\ &v(\eta,t-1)+(x_t-x_{t-1})\nonumber\\[-2pt]
            &-(a_1+a_2\eta)(y_{t-1}-y_{t-2})-(a_3+a_4\eta)(y_t-y_{t-1})-(a_5+a_6\eta)(y_{t+1}-y_t)\nonumber\\[-2pt]
            &+(\Cupper-\Vupper)[\theta(t-\check{s})-\theta(t-\hat{s}-1)-\theta(t-\check{s}-1)+\theta(t-\hat{s}-2)]\nonumber\\[-2pt]
            &-V[\check{s}\theta(t-\check{s})-(\check{s}-1)\theta(t-\check{s}-1)-(\hat{s}+1)\theta(t-\hat{s}-1)+\hat{s}\theta(t-\hat{s}-2)]
\end{align}
for any $\eta\in[0,\hat{\eta}]$ and $t\in[\check{t}+1,\hat{t}]_\Z$.
Note that $\tilde{v}(\eta,\Se,t)$ is linear in $\eta$.
Thus, for any given $t$, $v(\eta,t)$ is maximized when $\eta=0$ or $\eta=\hat{\eta}$.
That is, the largest possible value of $v(\eta,t)$ is equal to $v(0,t)$ if $a_2y_{t-1}+a_4y_t+a_6y_{t+1}\ge 0$, 
and the largest possible value of $v(\eta,t)$ is equal to $v(\hat{\eta},t)$ if $a_2y_{t-1}+a_4y_t+a_6y_{t+1}<0$.
Hence, to determine the $\eta$ and $t$ values corresponding to the largest violation of inequality \eqref{eq:A-inequality1},
it suffices to determine $v(0,\check{t}),v(0,\check{t}+1),\ldots,v(0,\hat{t})$ and $v(\hat{\eta},\check{t}),v(\hat{\eta},\check{t}+1),\ldots,v(\hat{\eta},\hat{t})$.
Algorithm~\ref{SeparationAlg-A1} performs this computation.

\begin{algorithm}[t]
\caption{Determination of \rred{a} most violated inequality \eqref{eq:A-inequality1} for any given $(\BFx,\BFy)\in\R_{+}^{2T}$}\label{SeparationAlg-A1}
\begin{algorithmic}[1]
  \small
  \State $\theta(t)\leftarrow 0\ \forall\ t\in[-\hat{s},1]_\Z$
  \For {$t=2,\ldots,\hat{t}$}
    \State $\theta(t)\leftarrow\theta(t-1)+\max\{y_t-y_{t-1},0\}$
  \EndFor
  \For {$\eta=0,\hat{\eta}$}
    \If {$\check{t}=1$}
      \State $v(\eta,\check{t})\leftarrow x_1-(a_3+a_4\eta)y_1-(a_5+a_6\eta)y_2$
    \ElsIf {$\check{s}=0$}
      \State $v(\eta,\check{t})\leftarrow x_2-(a_1+a_2\eta)y_1-(a_3+a_4\eta)y_2-(a_5+a_6\eta)y_3+(\Cupper-\Vupper)\max\{y_2-y_1,0\}$
    \Else
      \State $v(\eta,\check{t})\leftarrow x_2-(a_1+a_2\eta)y_1-(a_3+a_4\eta)y_2-(a_5+a_6\eta)y_3$
    \EndIf
    \For {$t=\check{t}+1,\ldots,\hat{t}$}
      \State $v(\eta,t)\leftarrow v(\eta,t-1)+(x_t-x_{t-1})\vspace{-2pt}\hfill\break
             \hbox{\hskip70pt}-(a_1+a_2\eta)(y_{t-1}-y_{t-2})-(a_3+a_4\eta)(y_t-y_{t-1})-(a_5+a_6\eta)(y_{t+1}-y_t)\vspace{-2pt}\hfill\break
             \hbox{\hskip70pt}+(\Cupper-\Vupper)[\theta(t-\check{s})-\theta(t-\hat{s}-1)-\theta(t-\check{s}-1)+\theta(t-\hat{s}-2)]\vspace{-2pt}\hfill\break
             \hbox{\hskip70pt}-V[\check{s}\theta(t-\check{s})-(\check{s}-1)\theta(t-\check{s}-1)-(\hat{s}+1)\theta(t-\hat{s}-1)+\hat{s}\theta(t-\hat{s}-2)]$
    \EndFor
  \EndFor
  \State $(\eta^*,t^*)\leftarrow\argmax_{(\eta,t)\in\{0,\hat{\eta}\}\times[\check{t},\hat{t}]_\Z}\{v(\eta,t)\}$
  \State $\Se^*\leftarrow\emptyset$
  \For {$s=\check{s},\ldots,\min\{\hat{s},t^*-2\}$}
    \State {\bf if} $y_{t^*-s}-y_{t^*-s-1}>0$ {\bf then} $\Se^*\leftarrow\Se^*\cup \{s\}$
  \EndFor
\end{algorithmic}
\end{algorithm}

In Algorithm~\ref{SeparationAlg-A1},
step~1 sets $\theta(t)$ to zero when $t\le1$.
Steps 2--4 determine the $\theta(t)$ values recursively for $t=2,3,\ldots,\hat{t}$.
These steps require $O(T)$ time.
Steps 5--16 consider the case $\eta=0$ and the case $\eta=\hat{\eta}$.
For each of these two $\eta$ values, these steps first determine $v(\eta,\check{t})$, and then determine
$v(\eta,\check{t}+1),v(\eta,\check{t}+2),\ldots,v(\eta,\hat{t})$ recursively using equation \eqref{eq:A-recursion1}.
These steps require $O(T)$ time.
Steps 17--21 identify \rred{a} most violated inequality \eqref{eq:A-inequality1} by setting the $\eta$ and $t$ values to $(\eta^*,t^*)=\argmax_{(\eta,t)\in\{0,\hat{\eta}\}\times[\check{t},\hat{t}]_\Z}\{v(\eta,t)\}$
and setting $\Se$ equal to the set of $s$ values such that $s\in[\check{s},\min\{\hat{s},t^*-2\}]_\Z$ and $y_{t^*-s}-y_{t^*-s-1}>0$.
These steps also require $O(T)$ time.
Therefore, the total computational time of Algorithm~\ref{SeparationAlg-A1} is $O(T)$.

(ii) Consider the following family of inequalities:
\begin{equation}\label{eq:A-inequality2}
x_t\le(a_1+a_2\eta)y_{t-1}+(a_3+a_4\eta)y_t+(a_5+a_6\eta)y_{t+1}-\sum_{s\in\Se}(\Cupper-\Vupper-sV)(y_{t+s}-y_{t+s+1}),
\end{equation}
where 
$\eta\in[0,\hat{\eta}]$,  
$\Se\subseteq[\check{s},\hat{s}]_\Z$,
$t\in[\check{t},\hat{t}]_\Z$, and
$t\ge s+2$ for all $s\in\Se$.
Note that 
inequality family \eqref{eqn-T:x_t-ub-1-2} in Proposition~\ref{prop-T:x_t-ub-1},
inequality family \eqref{eqn-T:x_t-ub-3-2} in Proposition~\ref{prop-T:x_t-ub-3}, and
inequality family \eqref{eqn-T:x_t-ub-5-2} in Proposition~\ref{prop-T:x_t-ub-5}
are special cases of this inequality family.
Consider any given point $(\BFx,\BFy)\in\R_{+}^{2T}$.
Let $x'_t=x_{T-t+1}$ and $y'_t=y_{T-t+1}$ for $t\in[1,T]_\Z$.
Inequality \eqref{eq:A-inequality2} becomes
$$x'_{T-t+1}\le(a_1+a_2\eta)y'_{T-t+2}+(a_3+a_4\eta)y'_{T-t+1}+(a_5+a_6\eta)y'_{T-t}-\sum_{s\in\Se}(\Cupper-\Vupper-sV)(y'_{T-t-s+1}-y'_{T-t-s}).$$
Letting $t'=T-t+1$, this inequality becomes
\begin{equation}\label{eq:A-inequality2prime}
x'_{t'}\le(a_5+a_6\eta)y'_{t'-1}+(a_3+a_4\eta)y'_{t'}+(a_1+a_2\eta)y'_{t'+1}-\sum_{s\in\Se}(\Cupper-\Vupper-sV)(y'_{t'-s}-y'_{t'-s-1}).
\end{equation}
Let $\check{t}'=1$ if $a_5=a_6=0$, and let $\check{t}'=2$ otherwise.
Let $\hat{t}'=T$ if $a_1=a_2=0$, and let $\hat{t}'=T-1$ otherwise.
From the analysis in part~(i), the set $\Se\subseteq[\check{s},\hat{s}]$, the real number $\eta\in[0,\hat{\eta}]_\Z$, and the integer $t'\in[\check{t}',\hat{t}']_\Z$ with $t'\ge s+2$ $\forall s\in\Se$ corresponding to a most violated inequality \eqref{eq:A-inequality2prime} can be obtained in $O(T)$ time using Algorithm~\ref{SeparationAlg-A1}.
Hence, the set $\Se\subseteq[\check{s},\hat{s}]$, the real number $\eta\in[0,\hat{\eta}]_\Z$, and the integer $t\in[\check{t},\hat{t}]_\Z$ with $t\le T-s-1$ $\forall s\in\Se$ corresponding to a most violated inequality \eqref{eq:A-inequality2} can be obtained in $O(T)$ time.

Summarizing (i) and (ii), we conclude that for any given point $(\BFx,\BFy)\in\R_{+}^{2T}$, \tred{the} most violated inequalities 
\eqref{eqn-T:x_t-ub-1-1}--\eqref{eqn-T:x_t-ub-1-2}, \eqref{eqn-T:x_t-ub-3-1}--\eqref{eqn-T:x_t-ub-3-2}, and \eqref{eqn-T:x_t-ub-5-1}--\eqref{eqn-T:x_t-ub-5-2}
in Propositions \ref{prop-T:x_t-ub-1}, \ref{prop-T:x_t-ub-3}, and \ref{prop-T:x_t-ub-5}, respectively,
can be determined in $O(T)$ time if such violated inequalities exist.
\Halmos


\subsection{Proof of Proposition \ref{prop:separation-B}}

\noindent{\bf Proposition \ref{prop:separation-B}.} {\it
For any given point $(\BFx,\BFy)\in\R_{+}^{2T}$, \tred{the} most violated inequalities 
\eqref{eqn-T:x_t-ub-1-1}--\eqref{eqn-T:x_t-ub-1-2}, \eqref{eqn-T:x_t-ub-3-1}--\eqref{eqn-T:x_t-ub-3-2}, and \eqref{eqn-T:x_t-ub-5-1}--\eqref{eqn-T:x_t-ub-5-2}
in Propositions \ref{prop-T:x_t-ub-2}, \ref{prop-T:x_t-ub-4}, and \ref{prop-T:x_t-ub-6}, respectively,
can be determined in $O(T^3)$ time if such violated inequalities exist.}
\vskip8pt

\noindent{\bf Proof.}
Let $\hat{\eta}$ and $a_1,\ldots,a_6$ be any real numbers such that $\hat{\eta}\ge 0$.
Let $\check{s}$, $\check{s}_{\max}$, and $\hat{s}_{\max}$ be any integers such that $0\le\check{s}\le\check{s}_{\max}\le\hat{s}_{\max}\le\min\{T-2,\lfloor(\Cupper-\Vupper)/V\rfloor\}$.
Let $\check{t}=1$ if $a_1=a_2=0$, and let $\check{t}=2$ otherwise.
Let $\hat{t}=T$ if $a_5=a_6=0$, and let $\hat{t}=T-1$ otherwise.

(i)~Consider the following family of inequalities:
\begin{equation}\label{eq:B-inequality1}
x_t\le(a_1+a_2\eta)y_{t-1}+(a_3+a_4\eta)y_t+(a_5+a_6\eta)y_{t+1}-\sum_{s\in\Se}(\Cupper-\Vupper-sV)(y_{t-s}-y_{t-s-1}),
\end{equation}
where 
$\eta\in[0,\hat{\eta}]$,
$\Se=[\check{s},\alpha]_\Z\cup[\beta,s_{\max}]_\Z$, 
$t\in[s_{\max}+2,\hat{t}]_\Z$, and
$\alpha$, $\beta$, and $s_{\max}$ are integers such that 
(a)~$\check{s}_{\max}\le s_{\max}\le\hat{s}_{\max}$,
(b)~$\check{s}\le\alpha<\beta\le s_{\max}$, and
(c)~$\beta=\alpha+1$ or $s_{\max}\le L+\alpha$.
Note that 
inequality family \eqref{eqn-T:x_t-ub-1-1} in Proposition~\ref{prop-T:x_t-ub-2}, 
inequality family \eqref{eqn-T:x_t-ub-3-1} in Proposition~\ref{prop-T:x_t-ub-4}, and 
inequality family \eqref{eqn-T:x_t-ub-5-1} in Proposition~\ref{prop-T:x_t-ub-6} 
are special cases of this inequality family.
Consider any given point $(\BFx,\BFy)\in\R_{+}^{2T}$.
We show that the integers $\alpha$, $\beta$, $s_{\max}$, $t$ and the real number $\eta$ corresponding to a most violated inequality 
can be determined in $O(T^3)$ time.

For any $\eta\in[0,\hat{\eta}]_\Z$, $s_{\max}\in[\check{s}_{\max},\hat{s}_{\max}]_\Z$, $\Se\subseteq[\check{s},s_{\max}]_\Z$, $t\in[s_{\max}+2,\hat{t}]_\Z$, let
$$\tilde{v}(\eta,\Se,t)=x_t-(a_1+a_2\eta)y_{t-1}-(a_3+a_4\eta)y_t-(a_5+a_6\eta)y_{t+1}+\sum_{s\in\Se}(\Cupper-\Vupper-sV)(y_{t-s}-y_{t-s-1}).$$
If $\tilde{v}(\eta,\Se,t)>0$, then $\tilde{v}(\eta,\Se,t)$ is the amount of violation of inequality \eqref{eq:B-inequality1}.
If $\tilde{v}(\eta,\Se,t)\le 0$, then there is no violation of inequality \eqref{eq:B-inequality1}.
Note that $\tilde{v}(\eta,\Se,t)$ is linear in $\eta$.
Thus, for any given $\Se$ and $t$, the function $\tilde{v}(\eta,\Se,t)$ is maximized at $\eta=0$ if $a_2y_{t-1}+a_4y_t+a_6y_{t+1}\ge 0$, 
and is maximized at $\eta=\hat{\eta}$ if $a_2y_{t-1}+a_4y_t+a_6y_{t+1}<0$.
For any $s_{\max}\in[\check{s}_{\max},\hat{s}_{\max}]_\Z$, $t\in[s_{\max}+2,\hat{t}]_\Z$, and $i\in[\check{s},s_{\max}]_\Z$, let
\begin{align*}
v_1(s_{\max},t,i)=\left\{\begin{array}{ll}
  \tilde{v}(0,[\check{s},i]_\Z,t), &\hbox{if $a_2y_{t-1}+a_4y_t+a_6y_{t+1}\ge 0$;}\\[2pt]
  \tilde{v}(\hat{\eta},[\check{s},i]_\Z,t), &\hbox{if $a_2y_{t-1}+a_4y_t+a_6y_{t+1}<0$;}
\end{array}\right.
\end{align*}
that is,
\begin{align*}
v_1(s_{\max},t,i)=\left\{\begin{array}{ll}
  x_t-a_1y_{t-1}-a_3y_t-a_5y_{t+1}\\[2pt]
  \qquad\qquad\qquad\ +\sum_{s=\check{s}}^i(\Cupper-\Vupper-sV)(y_{t-s}-y_{t-s-1}), &\hbox{ if $a_2y_{t-1}+a_4y_t+a_6y_{t+1}\ge 0$;}\\[2pt]
  x_t-(a_1+a_2\hat{\eta})y_{t-1}-(a_3+a_4\hat{\eta})y_t-(a_5+a_6\hat{\eta})y_{t+1}\\[2pt]
  \qquad\qquad\qquad\ +\sum_{s=\check{s}}^i(\Cupper-\Vupper-sV)(y_{t-s}-y_{t-s-1}), &\hbox{ if $a_2y_{t-1}+a_4y_t+a_6y_{t+1}<0$.}
\end{array}\right.
\end{align*}
For any $s_{\max}\in[\check{s}_{\max},\hat{s}_{\max}]_\Z$, $t\in[s_{\max}+2,\hat{t}]_\Z$, and $j\in[\check{s}+1,s_{\max}]_\Z$, let
$$v_2(s_{\max},t,j)=\sum_{s=j}^{s_{\max}}(\Cupper-\Vupper-sV)(y_{t-s}-y_{t-s-1}).$$
Note that
\begin{align*}
v_1(s_{\max},t,i)\!=\!\left\{\begin{array}{ll}
  v_1(s_{\max},t,i-1)+(\Cupper-\Vupper-iV)(y_{t-i}-y_{t-i-1}), &\hbox{if $i\ge\check{s}+1$;}\\[2pt]
  x_t-a_1y_{t-1}-a_3y_t-a_5y_{t+1}\\[2pt]
  \qquad\qquad\qquad\quad\ \ +(\Cupper-\Vupper-\check{s}V)(y_{t-\check{s}}-y_{t-\check{s}-1}), &\hbox{if $i=\check{s}$ and $a_2y_{t-1}+a_4y_t+a_6y_{t+1}\ge 0$;}\\[2pt]
  x_t-(a_1+a_2\hat{\eta})y_{t-1}-(a_3+a_4\hat{\eta})y_t\\[2pt]
  \,\ -(a_5+a_6\hat{\eta})y_{t+1}+(\Cupper-\Vupper-\check{s}V)(y_{t-\check{s}}-y_{t-\check{s}-1}), &\hbox{if $i=\check{s}$ and $a_2y_{t-1}+a_4y_t+a_6y_{t+1}<0$.}
\end{array}\right.
\end{align*}
Thus, for each $s_{\max}$ and $t$, the values of $v_1(s_{\max},t,\check{s}),v_1(s_{\max},t,\check{s}+1),\ldots,v_1(s_{\max},t,s_{\max})$ can be determined recursively in $O(T)$ time.
This implies that the $v_1(s_{\max},t,i)$ values (for all $s_{\max}$, $t$, and $i$) can be determined in $O(T^3)$ time.
Similarly, the $v_2(s_{\max},t,j)$ values (for all $s_{\max}$, $t$, and $j$) can be determined in $O(T^3)$ time.
For each $s_{\max}$, $t$, and $j$, let
$$\hat{v}_2(s_{\max},t,j)=\underset{\beta\in[j,s_{\max}]_\Z}{\max}\{v_2(s_{\max},t,\beta)\}$$
and
$$\hat{\beta}(s_{\max},t,j)=\underset{\beta\in[j,s_{\max}]_\Z}{\argmax}\{v_2(s_{\max},t,\beta)\}.$$
Note that for each $s_{\max}$ and $t$, the values of 
$\hat{v}_2(s_{\max},t,\check{s}+1),\hat{v}_2(s_{\max},t,\check{s}+2),\ldots,\hat{v}_2(s_{\max},t,s_{\max})$
and 
$\hat{\beta}_2(s_{\max},t,\check{s}+1),\hat{\beta}_2(s_{\max},t,\check{s}+2),\ldots,\hat{\beta}_2(s_{\max},t,s_{\max})$
can be determined in $O(T)$ time by setting
\begin{align*}
\hat{v}_2(s_{\max},t,j)=\left\{
\begin{array}{ll}
  \max\{\hat{v}_2(s_{\max},t,j+1),v_2(s_{\max},t,j)\}, &\hbox{if $j\le s_{\max}-1$;}\\[2pt]
  v_2(s_{\max},t,s_{\max}), &\hbox{if $j=s_{\max}$;}
\end{array}\right.
\end{align*}
and
\begin{align*}
\hat{\beta}(s_{\max},t,j)=\left\{
\begin{array}{ll}
  \hat{\beta}(s_{\max},t,j+1), &\hbox{if $\hat{v}_2(s_{\max},t,j+1)\ge v_2(s_{\max},t,j)$;}\\[2pt]
  j, &\hbox{if $\hat{v}_2(s_{\max},t,j+1)<v_2(s_{\max},t,j)$.}
\end{array}\right.
\end{align*}
This implies that the $\hat{v}_2(s_{\max},t,j)$ and $\hat{\beta}(s_{\max},t,j)$ values (for all $s_{\max}$, $t$, and $j$) can be determined in $O(T^3)$ time.
Note that the condition ``$\beta=\alpha+1$ or $s_{\max}\le L+\alpha$" in Proposition~\ref{prop:separation-B} implies that ``$\Se=[\check{s},s_{\max}]_\Z$" or ``$s_{\max}\le L+\alpha$."
For any $s_{\max}\in[\check{s}_{\max},\hat{s}_{\max}]_\Z$ and $t\in[s_{\max}+2,\hat{t}]_\Z$, 
if $\Se=[\check{s},s_{\max}]_\Z$, then the largest possible amount of violation of inequality \eqref{eq:B-inequality1} is equal to $v_1(s_{\max},t,s_{\max})$.
For any $s_{\max}\in[\check{s}_{\max},\hat{s}_{\max}]_\Z$, $t\in[s_{\max}+2,\hat{t}]_\Z$, and $\alpha\in[\check{s},s_{\max}-1]_\Z$, 
if $s_{\max}\le L+\alpha$, then the largest possible amount of violation of inequality \eqref{eq:B-inequality1} is equal to 
$v_1(s_{\max},t,\alpha)+v_2(s_{\max},t,\hat{\beta}(s_{\max},t,\alpha+1))=v_1(s_{\max},t,\alpha)+\hat{v}_2(s_{\max},t,\alpha+1)$.

To determine the most violated inequality \eqref{eq:B-inequality1} that satisfies the conditions in Proposition~\ref{prop:separation-B}, 
we first determine all $v_1(s_{\max},t,i)$, $v_2(s_{\max},t,j)$, $\hat{v}_2(s_{\max},t,j)$, and $\hat{\beta}(s_{\max},t,j)$ values, which requires $O(T^3)$ time.
Next, we search for the $s_{\max}$ and $t$ values such that $v_1(s_{\max},t,s_{\max})$ is the largest possible.
This requires $O(T^2)$ time.
Let $s_{\max}^*$ and $t^*$ be the $s_{\max}$ and $t$ values obtained, and let $\Se^*=[\check{s},s^*_{\max}]_\Z$.
Let $\eta^*=0$ if $a_2y_{t^*-1}+a_4y_{t^*}+a_6y_{t^*+1}\ge 0$, and $\eta^*=\hat{\eta}$ otherwise.
Next, we search for the $s_{\max}$, $t$, and $\alpha$ values, where $\alpha\in[s_{\max}-L,s_{\max}-1]_\Z$, such that $v_1(s_{\max},t,\alpha)+\hat{v}_2(s_{\max},t,\alpha+1)$ is the largest possible.
This requires $O(T^3)$ time.
Let $s^{**}_{\max}$, $t^{**}$, and $\alpha^{**}$ be the $s_{\max}$, $t$, and $\alpha$ values obtained, 
and let $\Se^{**}=[\check{s},\alpha^{**}]_\Z\cup[\beta^{**},s^{**}_{\max}]_\Z$, where $\beta^{**}=\hat{\beta}(s^{**}_{\max},t^{**},\alpha^{**}+1)$.
Let $\eta^{**}=0$ if $a_2y_{t^{**}-1}+a_4y_{t^{**}}+a_6y_{t^{**}+1}\ge 0$, and $\eta^{**}=\hat{\eta}$ otherwise.
If $v_1(s^*_{\max},t^*,s^*_{\max})>v_1(s^{**}_{\max},t^{**},\alpha^{**})+\hat{v}_2(s_{\max}^{**},t^{**},\alpha^{**}+1)$, 
then \rred{a} most violated inequality \eqref{eq:B-inequality1} is obtained by setting $\Se=\Se^*$, $\eta=\eta^*$, and $t=t^*$.
Otherwise, it is obtained by setting $\Se=\Se^{**}$, $\eta=\eta^{**}$, and $t=t^{**}$.
The overall computational time of this process is $O(T^3)$.

(ii)~Consider the following family of inequalities:
\begin{equation}\label{eq:B-inequality2}
x_t\le(a_1+a_2\eta)y_{t-1}+(a_3+a_4\eta)y_t+(a_5+a_6\eta)y_{t+1}-\sum_{s\in\Se}(\Cupper-\Vupper-sV)(y_{t+s}-y_{t+s+1}),
\end{equation}
where 
$\eta\in[0,\hat{\eta}]$,
$\Se=[\check{s},\alpha]_\Z\cup[\beta,s_{\max}]_\Z$, 
$t\in[\check{t},T-s_{\max}-1]_\Z$, and
$\alpha$, $\beta$, and $s_{\max}$ are integers such that
(a)~$\check{s}_{\max}\le s_{\max}\le\hat{s}_{\max}$,
(b)~$\check{s}\le\alpha<\beta\le s_{\max}$, and
(c)~$\beta=\alpha+1$ or $s_{\max}\le L+\alpha$.
Note that 
inequality family \eqref{eqn-T:x_t-ub-1-2} in Proposition~\ref{prop-T:x_t-ub-2},
inequality family \eqref{eqn-T:x_t-ub-3-2} in Proposition~\ref{prop-T:x_t-ub-4}, and
inequality family \eqref{eqn-T:x_t-ub-5-2} in Proposition~\ref{prop-T:x_t-ub-6}
are special cases of this inequality family.
Consider any given point $(\BFx,\BFy)\in\R_{+}^{2T}$.
Let $x'_t=x_{T-t+1}$ and $y'_t=y_{T-t+1}$ for $t\in[1,T]_\Z$.
Inequality \eqref{eq:B-inequality2} becomes
$$x'_{T-t+1}\le(a_1+a_2\eta)y'_{T-t+2}+(a_3+a_4\eta)y'_{T-t+1}+(a_5+a_6\eta)y'_{T-t}-\sum_{s\in\Se}(\Cupper-\Vupper-sV)(y'_{T-t-s+1}-y'_{T-t-s}).$$
Letting $t'=T-t+1$, this inequality becomes
\begin{equation}\label{eq:B-inequality2prime}
x'_{t'}\le(a_5+a_6\eta)y'_{t'-1}+(a_3+a_4\eta)y'_{t'}+(a_1+a_2\eta)y'_{t'+1}-\sum_{s\in\Se}(\Cupper-\Vupper-sV)(y'_{t'-s}-y'_{t'-s-1}).
\end{equation}
Let $\hat{t}'=T$ if $a_1=a_2=0$, and let $\hat{t}'=T-1$ otherwise.
From the analysis in part~(i), the integers $\alpha$, $\beta$, $s_{\max}$, and $t'\in[s_{\max}+2,\hat{t}']_\Z$ corresponding to a most violated inequality \eqref{eq:B-inequality2prime} can be obtained in $O(T^3)$ time.
Hence, the integers $\alpha$, $\beta$, $s_{\max}$, and $t\in[\check{t},T-s_{\max}-1]_\Z$ corresponding to a most violated inequality \eqref{eq:B-inequality2} can be obtained in $O(T^3)$ time.

Summarizing (i) and (ii), we conclude that for any given point $(\BFx,\BFy)\in\R_{+}^{2T}$, \tred{the} most violated inequalities 
\eqref{eqn-T:x_t-ub-1-1}--\eqref{eqn-T:x_t-ub-1-2}, \eqref{eqn-T:x_t-ub-3-1}--\eqref{eqn-T:x_t-ub-3-2}, and \eqref{eqn-T:x_t-ub-5-1}--\eqref{eqn-T:x_t-ub-5-2}
in Propositions \ref{prop-T:x_t-ub-2}, \ref{prop-T:x_t-ub-4}, and \ref{prop-T:x_t-ub-6}, respectively,
can be determined in $O(T^3)$ time if such violated inequalities exist.
\Halmos

\subsection{Proof of Proposition \ref{prop-T:kth-ramp-2}}

\noindent{\bf Proposition \ref{prop-T:kth-ramp-2}.} {\it
Consider any $k\in[1,T-1]_\Z$ such that $\Cupper-\Clower-kV>0$, any $m\in[0,k-1]_\Z$, and any $\Se\subseteq[0,\min\{k-1,L-m-1\}]_\Z$.
For any $t\in[k+1,T-m]_\Z$, the inequality
\begin{align}
x_t-x_{t-k}\le (\Clower+(k-m)V)y_t+V\sum_{i=1}^m y_{t+i}-\Clower y_{t-k}-\sum_{s\in\Se}(\Clower+(k-s)V-\Vupper)(y_{t-s}-y_{t-s-1})\tag{\ref{eqn-T:kth-ramp-2-1}}
\end{align}
is valid for $\convP$.
For any $t\in[m+1,T-k]_\Z$, the inequality
\begin{align}
x_t-x_{t+k}\le (\Clower+(k-m)V)y_t+V\sum_{i=1}^m y_{t-i}-\Clower y_{t+k}-\sum_{s\in\Se}(\Clower+(k-s)V-\Vupper)(y_{t+s}-y_{t+s+1})\tag{\ref{eqn-T:kth-ramp-2-2}}
\end{align}
is valid for $\convP$.
Furthermore, \eqref{eqn-T:kth-ramp-2-1} and \eqref{eqn-T:kth-ramp-2-2} are facet-defining for $\convP$ when $m=0$ and $s\ge\min\{k-1,1\}$ for all $s\in\Se$.}
\vskip8pt

\noindent{\bf Proof.}
We first prove that inequality \eqref{eqn-T:kth-ramp-2-1} is valid and facet-defining for $\convP$.
For notational convenience, we define $s_{\max}=\max\{s:s\in\Se\}$ if $\Se\ne\emptyset$, and $s_{\max}=-1$ if $\Se=\emptyset$.
Consider any $t\in[k+1,T-m]_\Z$.
To prove that the linear inequality \eqref{eqn-T:kth-ramp-2-1} is valid for $\convP$, it suffices to show that it is valid for $\Pe$.
Consider any element $(\BFx,\BFy)$ of $\Pe$.
We show that $(\BFx,\BFy)$ satisfies \eqref{eqn-T:kth-ramp-2-1}.
We divide the analysis into three cases:

Case~1: $y_t=0$.
In this case, by \eqref{eqn:p-lower-bound} and \eqref{eqn:p-upper-bound}, $-x_{t-k}\le-\Clower y_{t-k}$ and $x_t=0$.
Thus, \rred{the} left-hand side of \eqref{eqn-T:kth-ramp-2-1} is at most $-\Clower y_{t-k}$ and the first term on the right-hand side of \eqref{eqn-T:kth-ramp-2-1} is $0$.
Because $y_t=0$ and $t\in[2,T]_{\Z}$, by Lemma \ref{lem:lookbackward}(i), we have $y_{t-j}-y_{t-j-1}\le0$ for all $j\in[0,\min\{t-2,L-1\}]_{\Z}$.
Because $\Se\subseteq[0,\min\{k-1,L-m-1\}]_{\Z}$, $m\ge0$ and, $t\ge k+1$, we have $\Se\subseteq[0,\min\{t-2,L-1\}]_{\Z}$.
Thus, $y_{t-s}-y_{t-s-1}\le0$ for all $s\in\Se$.
Because $\Se\subseteq[0,k-1]_{\Z}$ and $\Clower+V>\Vupper$, for any $s\in\Se$, the coefficient ``$\Clower+(k-s)V-\Vupper$" on the right-hand side of \eqref{eqn-T:kth-ramp-2-1} is positive.
Hence, the right-hand side of \eqref{eqn-T:kth-ramp-2-1} is at least $-\Clower y_{t-k}$.
Therefore, in this case, $(\BFx,\BFy)$ satisfies \eqref{eqn-T:kth-ramp-2-1}.

Case~2: $y_t=1$ and $y_{t-s'}-y_{t-s'-1}=1$ for some $s'\in\Se$.
In this case, $y_{t-s'}=1$ and $y_{t-s'-1}=0$.
Because $y_t=1$ and $t\in[2,T]_{\Z}$, by Lemma \ref{lem:lookbackward}(ii), there exists at most one $j\in[0,\min\{t-2,L\}]_{\Z}$ such that $y_{t-j}-y_{t-j-1}=1$.
Because $\Se\subseteq[0,\min\{k-1,L-m-1\}]_{\Z}$, $m\ge0$ and, $t\ge k+1$, we have $\Se\subseteq[0,\min\{t-2,L\}]_{\Z}$.
\rred{Thus,} $y_{t-s}-y_{t-s-1}\le0$ for all $s\in\Se\setminus\{s'\}$.
Because $y_{t-s'}-y_{t-s'-1}=1$ and $t-s'\in[2,T]_{\Z}$, by \eqref{eqn:p-minup}, we have $\rred{y_{\tau}}=1$ for all $\rred{\tau}\in[t-s',\min\{T,t-s'+L-1\}]_{\Z}$.
Because $\Se\subseteq[0, L-m-1]_{\Z}$, we have $t-s'+L-1\ge t+m$.
Thus, $y_{\tau}=1$ for all $\tau\in[t-s',t+m]_{\Z}$\rred{, which implies that $(\Clower+(k-m)V)y_t+V\sum_{i=1}^my_{t+i}=\Clower+kV$}.
Because $\Se\subseteq[0,k-1]_{\Z}$ and $\Clower+V>\Vupper$, the coefficient ``$\Clower+(k-s)V-\Vupper$" on the right-hand side of \eqref{eqn-T:kth-ramp-2-1} is positive for all $s\in\Se$.
Hence, the right-hand side of \rred{\eqref{eqn-T:kth-ramp-2-1}} is at least $\rred{\Clower+kV-\Clower y_{t-k}-(\Clower+(k-s')V-\Vupper)}=s'V+\Vupper-\Clower y_{t-k}$.
\rred{By \eqref{eqn:p-ramp-up}, $\sum_{\tau=t-s'}^{t}(x_{\tau}-x_{\tau-1})\le\sum_{\tau=t-s'}^{t}Vy_{\tau-1}+\sum_{\tau=t-s'}^{t}\Vupper(1-y_{\tau-1})$}, which implies that $x_t-x_{t-s'-1}\le s'V+\Vupper$.
Because $y_{t-s'-1}=0$, by \eqref{eqn:p-upper-bound}, \rred{$x_{t-s'-1}=0$.}
\rred{By \eqref{eqn:p-lower-bound}, $-x_{t-k}\ge-\Clower y_{t-k}$.}
\rred{Thus,} $x_t-x_{t-k}\le s'V+\Vupper-\Clower y_{t-k}$.
Therefore, in this case, $(\BFx,\BFy)$ satisfies \eqref{eqn-T:kth-ramp-2-1}.

Case~3: $y_t=1$ and $y_{t-s}-y_{t-s-1}\le0$ for all $s\in\Se$.
\rred{Because $\Se\subseteq[0,k-1]_{\Z}$ and $\Clower+V>\Vupper$, we have $\Clower+(k-s)V-\Vupper>0$ for each $s\in\Se$.
Hence, the term ``$-\sum_{s\in\Se}(\Clower+(k-s)V-\Vupper)(y_{t-s}-y_{t-s-1})$" on the right-hand side of inequality \eqref{eqn-T:kth-ramp-2-1} is nonnegative.
We divide our analysis into three subcases.}

\rred{Case~3.1: $y_{\tau}=0$ for some $\tau\in[t,t+m]_\Z$.
Let $t'=\min\{\tau\in[t,t+m]_\Z:y_{\tau}=0\}$.
Then, $y_{\tau}=1$ for all $\tau\in[t,t'-1]_\Z$.
Thus, the right-hand side of \eqref{eqn-T:kth-ramp-2-1} is at least $\Clower+(k-m)V+(t'-t-1)V-\Clower y_{t-k}$.
Because $k-m\ge 1$ and $\Clower+V>\Vupper$, the right-hand side of \eqref{eqn-T:kth-ramp-2-1} is at least $\Vupper+(t'-t-1)V-\Clower y_{t-k}$.
By \eqref{eqn:p-ramp-down}, $\sum_{\tau=t+1}^{t'}(x_{\tau-1}-x_{\tau})\le\sum_{\tau=t+1}^{t'}Vy_{\tau}+\sum_{\tau=t+1}^{t'}\Vupper(1-y_{\tau})$, 
which implies that $x_t-x_{t'}\le\Vupper+(t'-t-1)V$.
Because $y_{t'}=0$, by \eqref{eqn:p-upper-bound}, $x_{t'}=0$.
By \eqref{eqn:p-lower-bound}, $-x_{t-k}\le-\Clower y_{t-k}$.
Hence, $x_t-x_{t-k}\le\Vupper+(t'-t-1)V-\Clower y_{t-k}$.
Thus, the left-hand side of \eqref{eqn-T:kth-ramp-2-1} is less than or equal to the right-hand side.}

\rred{Case~3.2: $y_{\tau}=1$ for all $\tau\in[t,t+m]_\Z$ and $y_{\tau}=0$ for some $\tau\in[t-k,t-1]_\Z$.
In this case, the right-hand side of \eqref{eqn-T:kth-ramp-2-1} is at least $\Clower+kV-\Clower y_{t-k}$.
Let $t'=\max\{\tau\in[t-k,t-1]_\Z:y_{\tau}=0\}$.
Because $t'\ge t-k$ and $\Clower+V>\Vupper$, the right-hand side of \eqref{eqn-T:kth-ramp-2-1} is greater than $\Vupper+(t-t'-1)V-\Clower y_{t-k}$.
By \eqref{eqn:p-ramp-up}, $\sum_{\tau=t'+1}^t(x_{\tau}-x_{\tau-1})\le\sum_{\tau=t'+1}^tVy_{\tau-1}+\sum_{\tau=t'+1}^t\Vupper(1-y_{\tau-1})$, 
which implies that $x_t-x_{t'}\le\Vupper+(t-t'-1)V$.
Because $y_{t'}=0$, by \eqref{eqn:p-upper-bound}, $x_{t'}=0$.
By \eqref{eqn:p-lower-bound}, $-x_{t-k}\le-\Clower y_{t-k}$.
Hence, $x_t-x_{t-k}\le\Vupper+(t-t'-1)V-\Clower y_{t-k}$.
Thus, the left-hand side of \eqref{eqn-T:kth-ramp-2-1} is less than the right-hand side.}

\rred{Case~3.3: $y_{\tau}=1$ for all $\tau\in[t-k,t+m]_\Z$.
In this case, the right-hand side of \eqref{eqn-T:kth-ramp-2-1} is at least $kV$.
By \eqref{eqn:p-ramp-up}, $\sum_{\tau=t-k+1}^t(x_{\tau}-x_{\tau-1})\le\sum_{\tau=t-k+1}^tVy_{\tau-1}+\sum_{\tau=t-k+1}^t\Vupper(1-y_{\tau-1})$, 
which implies that $x_t-x_{t-k}\le kV$.
Thus, the left-hand side of \eqref{eqn-T:kth-ramp-2-1} is less than or equal to the right-hand side.}

\rred{In Cases 3.1--3.3, $(\BFx,\BFy)$ satisfies \eqref{eqn-T:kth-ramp-2-1}.
Summarizing Cases 1--3, we conclude that \eqref{eqn-T:kth-ramp-2-1} is valid for $\convP$.}

\rred{Consider any $t\in[k+1,T-m]_\Z$.
To prove that inequality \eqref{eqn-T:kth-ramp-2-1} is facet-defining for $\convP$ when $m=0$ and $s\ge\min\{k-1,1\}$ for all $s\in\Se$, 
it suffices to show that there exist $2T$ affinely independent points in $\convP$ that satisfy \eqref{eqn-T:kth-ramp-2-1} at equality when $m=0$ and $s\ge\min\{k-1,1\}$ for all $s\in\Se$.
Because $\BFzero\in\convP$ and $\BFzero$ satisfies \eqref{eqn-T:kth-ramp-2-1} at equality, it suffices to create the remaining $2T-1$ nonzero linearly independent points.
We denote these $2T-1$ points as $(\bar{\BFx}^r,\bar{\BFy}^r)$ for $r\in[1,T]_{\Z}\setminus\{t-k\}$ and $(\hat{\BFx}^r,\hat{\BFy}^r)$ for $r\in[1,T]_{\Z}$, 
and denote the $q$th component of $\bar{\BFx}^r$, $\bar{\BFy}^r$, $\hat{\BFx}^r$, and $\hat{\BFy}^r$ as $\bar{x}^r_q$, $\bar{y}^r_q$, $\hat{x}^r_q$, and $\hat{y}^r_q$, respectively.
Let $\epsilon=\min\{\Vupper-\Clower, \Cupper-\Clower-kV\}>0$.}
We divide these $2T-1$ points into the following eight groups:
\begin{enumerate}[label=(A\arabic*)]
    \item\label{points:prop14-eq1-A1} For each $r\in[1,t-k-1]_{\Z}$, we create a point $(\bar{\BFx}^r,\bar{\BFy}^r)$ as follows:
    \begin{equation*}
        \bar{x}^r_q=\left\{
        \begin{array}{ll}
        \Clower, &\hbox{for $q\in[1,r-1]_{\Z}$};\\
        \Clower+\epsilon, &\hbox{for $q=r$};\\
        0, &\hbox{for $q\in[r+1,T]$};
        \end{array}\right.
    \end{equation*}
    and
    \begin{equation*}
        \bar{y}^r_q=\left\{
        \begin{array}{ll}
        1, &\hbox{for $q\in[1,r]_{\Z}$};\\
        0, &\hbox{for $q\in[r+1,T]_{\Z}$}.
        \end{array}\right.
    \end{equation*}
    It is easy to verify that $(\bar{\BFx}^r,\bar{\BFy}^r)$ satisfies \eqref{eqn:p-minup}--\eqref{eqn:p-ramp-down}.
    Thus, $(\bar{\BFx}^r,\bar{\BFy}^r)\in\convP$.
    Note that $\bar{x}^r_t=\bar{x}^r_{t-k}=\bar{y}^r_t=\bar{y}^r_{t-k}=0$ and $m=0$.
    Because $t-s-1\ne r$ for all $s\in\Se$, we have $\bar{y}^r_{t-s}-\bar{y}^r_{t-s-1}=0$ for all $s\in\Se$.
    Hence, $(\bar{\BFx}^r,\bar{\BFy}^r)$ satisfies \eqref{eqn-T:kth-ramp-2-2} at equality.
    
    \item\label{points:prop14-eq1-A2} For each $r\in[t-k+1,t-1]_{\Z}$, we create a point $(\bar{\BFx}^r,\bar{\BFy}^r)$ as follows:
    \begin{equation*}
        \bar{x}^r_q=\left\{
        \begin{array}{ll}
        \Clower, &\hbox{for $q\in[1,t-1]_{\Z}\setminus\{r\}$};\\
        \Clower+\epsilon, &\hbox{for $q=r$};\\
        0, &\hbox{for $q\in[t,T]_{\Z}$};
        \end{array}\right.
    \end{equation*}
    and
    \begin{equation*}
        \bar{y}^r_q=\left\{
        \begin{array}{ll}
        1, &\hbox{for $q\in[1,t-1]_{\Z}$};\\
        0, &\hbox{for $q\in[t,T]_{\Z}$}.
        \end{array}\right.
    \end{equation*}
    It is easy to verify that $(\bar{\BFx}^r,\bar{\BFy}^r)$ satisfies \eqref{eqn:p-minup}--\eqref{eqn:p-ramp-down}.
    Thus, $(\bar{\BFx}^r,\bar{\BFy}^r)\in\convP$.
    Note that $\bar{x}^r_t=\bar{y}^r_t=0$, $\bar{x}^r_{t-k}=\Clower$, $\bar{y}^r_{t-k}=1$, and $m=0$.
    The existence of $r\in[t-k+1,t-1]_\Z$ implies that $k\ge 2$, which in turn implies that $s\ge 1$ for all $s\in\Se$.
    Hence, $\bar{y}^r_{t-s}-\bar{y}^r_{t-s-1}=0$ for all $s\in\Se$.
    Thus, $(\bar{\BFx}^r,\bar{\BFy}^r)$ satisfies \eqref{eqn-T:kth-ramp-2-2} at equality.
    
    \item\label{points:prop14-eq1-A3}
     We create a point $(\bar{\BFx}^t,\bar{\BFy}^t)$ as follows:
     \begin{equation*}
        \bar{x}^t_q=\left\{
        \begin{array}{ll}
        \Clower, &\hbox{for $q\in[1,t-k-1]_{\Z}$};\\
        \Clower+(q-t+k)V+\epsilon, &\hbox{for $q\in[t-k,t]_{\Z}$};\\
        \Clower+kV, &\hbox{for $q\in[t+1,T]_{\Z}$};
        \end{array}\right.
    \end{equation*}
    and $\bar{y}^t_q=1$ for all $q\in[1,T]_{\Z}$.
    It is easy to verify that $(\bar{\BFx}^t,\bar{\BFy}^t)$ satisfies \eqref{eqn:p-minup}--\eqref{eqn:p-upper-bound}.
    Note that $\bar{x}^t_q-\bar{x}^t_{q-1}=0$ when $q\in[2,t-k-1]_{\Z}$, $0<\bar{x}^t_q-\bar{x}^t_{q-1}\le V$ when $q\in[t-k,t]_{\Z}$, and $-\epsilon\le\bar{x}^t_q-\bar{x}^t_{q-1}\le0$ when $q\in[t+1,T]_{\Z}$.
    Thus, $-V\bar{y}^t_q-\Vupper(1-\bar{y}^t_q)\le\bar{x}^t_q-\bar{x}^t_{q-1}\le V\bar{y}^t_{q-1}+\Vupper(1-\bar{y}^t_{q-1})$ for all $q\in[2,T]_\Z$.
    Hence, $(\bar{\BFx}^t,\bar{\BFy}^t)$ satisfies \eqref{eqn:p-ramp-up} and \eqref{eqn:p-ramp-down}.
    Therefore, $(\bar{\BFx}^t,\bar{\BFy}^t)\in\convP$.
    Note that $\bar{x}^t_t=\Clower+kV+\epsilon$, $\bar{x}^t_{t-k}=\Clower+\epsilon$, $\bar{y}^t_t=\bar{y}^t_{t-k}=1$, $m=0$, and $\bar{y}^t_{t-s}-\bar{y}^t_{t-s-1}=0$ for all $s\in\Se$.
    Thus, $(\bar{\BFx}^t,\bar{\BFy}^t)$ satisfies \eqref{eqn-T:kth-ramp-2-1} at equality.
    
    \item\label{points:prop14-eq1-A4} For each $r\in[t+1,T]_\Z$, we create a point $(\bar{\BFx}^r,\bar{\BFy}^r)$ as follows:
    \begin{equation*}
        \bar{x}^r_q=\left\{\rred{
        \begin{array}{ll}
        0, &\hbox{for $q\in[1,r-1]_\Z$};\\
        \Clower+\epsilon, &\hbox{for $q=r$};\\
        \Clower, &\hbox{for $q\in[r+1,T]_\Z$};
        \end{array}}\right.
    \end{equation*}
    and 
    \begin{equation*}
        \bar{y}^r_q=\left\{\rred{
        \begin{array}{ll}
        0, &\hbox{for $q\in[1,r-1]_\Z$};\\
        1, &\hbox{for $q\in[r,T]_\Z$}.\\
        \end{array}}\right.
    \end{equation*}
    It is easy to verify that $(\bar{\BFx}^r,\bar{\BFy}^r)$ satisfies \eqref{eqn:p-minup}--\eqref{eqn:p-ramp-down}.
    Thus, $(\bar{\BFx}^r,\bar{\BFy}^r)\in\convP$.
    Note that $\bar{x}^r_t=\bar{x}^r_{t-k}=\bar{y}^r_t=\rred{\bar{y}^r_{t-k}}=0$, $m=0$, and $\bar{y}^r_{t-s}-\bar{y}^r_{t-s-1}=0$ for all $s\in\Se$.
    Hence, $(\bar{\BFx}^t,\bar{\BFy}^t)$ satisfies \eqref{eqn-T:kth-ramp-2-1} at equality.
    
    \item\label{points:prop14-eq1-A5} For each \rred{$r\in[1,t-1]_{\Z}$}, we create the same point $(\hat{\BFx}^r,\hat{\BFy}^r)$ as in group \ref{points:prop1-eq1-A2} in the proof of Proposition~\ref{prop-T:x_t-ub-1}.
    Thus, $(\hat{\BFx}^r,\hat{\BFy}^r)\in\convP$.
    To show that $(\hat{\BFx}^r,\hat{\BFy}^r)$ satisfies \eqref{eqn-T:kth-ramp-2-1} at equality, we first consider the case where $t-r-1\notin\Se$.
    In this case, $\hat{x}^r_t=\hat{y}^r_t=0$ and $m=0$. Because $t-k\le t-s_{\max}-1\le r$, we have $\hat{x}^r_{t-k}=\Clower$ and $\hat{y}^r_{t-k}=1$.
    Because $t-s-1\ne r$ for all $s\in\Se$, we have $\hat{y}^r_{t-s}-\hat{y}^r_{t-s-1}=0$ for all $s\in\Se$.
    Hence, $(\hat{\BFx}^r,\hat{\BFy}^r)$ satisfies \eqref{eqn-T:kth-ramp-2-1} at equality.
    Next, we consider the case where $t-r-1\in\Se$.
    In this case, $\hat{x}^r_t=\Vupper+(t-r-1)V$, $\hat{y}^r_t=1$, $\hat{x}^r_{t-k}=\hat{y}^r_{t-k}=0$, and $m=0$.
    In addition, $\hat{y}^r_{t-s}-\hat{y}^r_{t-s-1}=1$ when $s=t-r-1$, and $\hat{y}^r_{t-s}-\hat{y}^r_{t-s-1}=0$ when $s\ne t-r-1$.
    Hence, $(\hat{\BFx}^r,\hat{\BFy}^r)$ satisfies \eqref{eqn-T:kth-ramp-2-1} at equality.
    
    \item\label{points:prop14-eq1-A6} We create a point $(\hat{\BFx}^t,\hat{\BFy}^t)$ as follows: 
    \begin{equation*}
        \hat{\BFx}^t_q=\left\{
        \begin{array}{ll}
        \Clower, &\hbox{for $q\in[1,t-k-1]_\Z$};\\
        \Clower+(q-t+k)V, &\hbox{for $q\in[t-k,t]_\Z$};\\
        \Clower+kV, &\hbox{for $q\in[t+1,T]_\Z$};
        \end{array}\right.
    \end{equation*}
    and $\hat{y}^{t}_q=1$ for all $q\in[1,T]_\Z$.
    It is easy to verify that $(\hat{\BFx}^t,\hat{\BFy}^t)$ satisfies \eqref{eqn:p-minup}--\eqref{eqn:p-upper-bound}.
    Note that 
    $\hat{x}^t_q-\hat{x}^t_{q-1}=0$ when $q\in[2,t-k]_\Z$, 
    $\hat{x}^t_q-\hat{x}^t_{q-1}=V$ when $q\in[t-k+1,t]_\Z$, and 
    $\hat{x}^t_q-\hat{x}^t_{q-1}=0$ when $q\in[t+1,T]_\Z$.
    Thus, $-V\hat{y}^t_q-\Vupper(1-\hat{y}^t_q)\le\hat{x}^t_q-\hat{x}^t_{q-1}\le V\hat{y}^t_{q-1}+\Vupper(1-\hat{y}^t_{q-1})$ for all $q\in[2,T]_\Z$.
    Hence, $(\hat{\BFx}^t,\hat{\BFy}^t)$ satisfies \eqref{eqn:p-ramp-up} and \eqref{eqn:p-ramp-down}.
    Therefore, $(\hat{\BFx}^t,\hat{\BFy}^t)\in\convP$.
    Note that $\hat{x}^t_t=\Clower+kV$, $\hat{x}^t_{t-k}=\Clower$, $\hat{y}^t_t=\hat{y}^t_{t-k}=1$, $m=0$,
    and $\hat{y}^t_{t-s}-\hat{y}^t_{t-s-1}=0$ for all $s\in\Se$, 
    Thus, $(\hat{\BFx}^t,\hat{\BFy}^t)$ satisfies \eqref{eqn-T:kth-ramp-2-1} at equality.
    
    \item\label{points:prop14-eq1-A7} For each $r\in[t+1,T]_{\Z}$, we create the same point $(\hat{\BFx}^r,\hat{\BFy}^r)$ as in group \ref{points:prop1-eq1-A4} in the proof of Proposition~\ref{prop-T:x_t-ub-1}.
    Thus, $(\hat{\BFx}^r,\hat{\BFy}^r)\in\convP$.
    Note that $\hat{x}^r_t=\hat{x}^r_{t-k}=\hat{y}^r_t=\hat{y}^r_{t-k}=0$, $m=0$, and $\hat{y}^r_{t-s}-\hat{y}^r_{t-s-1}=0$ for all $s\in\Se$.
    Hence, $(\hat{\BFx}^r,\hat{\BFy}^r)$ satisfies \eqref{eqn-T:kth-ramp-2-1} at equality.
\end{enumerate}

Table \ref{tab:eqn-T:kth-ramp-2-1-facet-matrix-1} shows a matrix with $2T-1$ rows, where each row represents a point created by this process.
This matrix can be transformed into the matrix in Table \ref{tab:eqn-T:kth-ramp-2-1-facet-matrix-2} via the following Gaussian elimination process:

\begin{enumerate}[label=(\roman*)]
    \item For each $r\in[1,t-k-1]_{\Z}$, the point with index $r$ in group (B1), denoted $(\underline{\bar{\BFx}}^r,\underline{\bar{\BFy}}^r)$, is obtained by 
    setting $(\underline{\bar{\BFx}}^r,\underline{\bar{\BFy}}^r)=(\bar{\BFx}^r,\bar{\BFy}^r)-(\hat{\BFx}^r,\hat{\BFy}^r)$.
    Here, $(\bar{\BFx}^r,\bar{\BFy}^r)$ is the point with index $r$ in group \ref{points:prop14-eq1-A1}, and $(\hat{\BFx}^r,\hat{\BFy}^r)$ is the point with index $r$ in group \ref{points:prop14-eq1-A5}.
    \rred{Note that $t-r-1\notin\Se$ for all $r\in[1,t-k-1]_\Z$.
    Thus, when $r\le t-k-1$, the point with index $r$ in group \ref{points:prop14-eq1-A5} is given by $\hat{x}^r_q=\Clower$ and $\hat{y}^r_q=1$ for $q\in[1,r]_\Z$, and $\hat{x}^r_q=\hat{y}^r_q=0$ for $q\in[r+1,T]_\Z$.}
    
    \item For each $r\in[t-k+1,t-1]_{\Z}$, the point with index $r$ in group (B2), denoted $(\underline{\bar{\BFx}}^r,\underline{\bar{\BFy}}^r)$, is obtained by 
    setting $(\underline{\bar{\BFx}}^r,\underline{\bar{\BFy}}^r)=(\bar{\BFx}^r,\bar{\BFy}^r)-(\hat{\BFx}^{t-1},\hat{\BFy}^{t-1})$.
    Here, $(\bar{\BFx}^r,\bar{\BFy}^r)$ is the point in group \ref{points:prop14-eq1-A2}, and $(\hat{\BFx}^{t-1},\hat{\BFy}^{t-1})$ is the point with index $t-1$ in group \ref{points:prop14-eq1-A5}.
    Note that because $s\ge 1$ for all $s\in\Se$, the point with index $t-1$ in group \ref{points:prop14-eq1-A5} is given by 
    $\hat{x}^{t-1}_q=\Clower$ and $\hat{y}^{t-1}_q=1$ for $q\in[1,t-1]_\Z$, and $\hat{x}^{t-1}_q=\hat{y}^{t-1}_q=0$ for $q\in[t,T]_\Z$.
    
    \item The point in group (B3), denoted $(\underline{\bar{\BFx}}^t,\underline{\bar{\BFy}}^t)$, is obtained by 
    setting $(\underline{\bar{\BFx}}^t,\underline{\bar{\BFy}}^t)=(\bar{\BFx}^t,\bar{\BFy}^t)-(\hat{\BFx}^t,\hat{\BFy}^t)$.
    Here, $(\bar{\BFx}^t,\bar{\BFy}^t)$ is the point in group \ref{points:prop14-eq1-A3}, and $(\hat{\BFx}^t,\hat{\BFy}^t)$ is the point in group \ref{points:prop14-eq1-A6}.

    \item For each $r\in[t+1,T]_{\Z}$, the point with index $r$ in group (B4), denoted $(\underline{\bar{\BFx}}^r,\underline{\bar{\BFy}}^r)$, is obtained by 
    setting $(\underline{\bar{\BFx}}^r,\underline{\bar{\BFy}}^r)=(\bar{\BFx}^r,\bar{\BFy}^r)-(\hat{\BFx}^r,\hat{\BFy}^r)$.
    Here, $(\bar{\BFx}^r,\bar{\BFy}^r)$ is the point with index $r$ in group \ref{points:prop14-eq1-A4}, and $(\hat{\BFx}^r,\hat{\BFy}^r)$ is the point in group \ref{points:prop14-eq1-A7}.
    
    \item For each \rred{$r\in[1,t-1]_{\Z}$}, the point with index $r$ in group \rred{(B5)}, denoted $(\underline{\hat{\BFx}}^r,\underline{\hat{\BFy}}^r)$, is obtained by setting $(\underline{\hat{\BFx}}^r,\underline{\hat{\BFy}}^r)=(\hat{\BFx}^r,\hat{\BFy}^r)$ if $t-r-1\notin\Se$, and setting $(\underline{\hat{\BFx}}^r,\underline{\hat{\BFy}}^r)=\rred{(\hat{\BFx}^t,\hat{\BFy}^t)-(\hat{\BFx}^r,\hat{\BFy}^r)}$ if $t-r-1\in\Se$.
    Here, $(\hat{\BFx}^r,\hat{\BFy}^r)$ is the point with index $r$ in group \ref{points:prop14-eq1-A5}, and $(\hat{\BFx}^t,\hat{\BFy}^t)$ is the point in group \ref{points:prop14-eq1-A6}.

    \item The point in group \rred{(B6)}, denoted $(\underline{\hat{\BFx}}^t,\underline{\hat{\BFy}}^t)$, is obtained by setting $(\underline{\hat{\BFx}}^t,\underline{\hat{\BFy}}^t)=(\hat{\BFx}^t,\hat{\BFy}^t)-(\hat{\BFx}^{t+1},\hat{\BFy}^{t+1})$.
    Here, $(\hat{\BFx}^t,\hat{\BFy}^t)$ is the point in group \ref{points:prop14-eq1-A6}, and $(\hat{\BFx}^{t+1},\hat{\BFy}^{t+1})$ is the point with index $t+1$ in group \ref{points:prop14-eq1-A7}.
    
    \item For each $r\in[t+1,T]_{\Z}$, the point with index $r$ in group \rred{(B7)}, denoted $(\underline{\hat{\BFx}}^r,\underline{\hat{\BFy}}^r)$, is obtained by setting $(\underline{\hat{\BFx}}^r,\underline{\hat{\BFy}}^r)=(\hat{\BFx}^r,\hat{\BFy}^r)-(\hat{\BFx}^{r+1},\hat{\BFy}^{r+1})$ if $r\neq T$, and setting $(\underline{\hat{\BFx}}^r,\underline{\hat{\BFy}}^r)=(\hat{\BFx}^r,\hat{\BFy}^r)$ if $r=T$.
    Here, $(\hat{\BFx}^r,\hat{\BFy}^r)$ and $(\hat{\BFx}^{r+1},\hat{\BFy}^{r+1})$ are the points with indices $r$ and $r+1$, respectively, in group \ref{points:prop14-eq1-A7}.
\end{enumerate}

\afterpage{
\begin{landscape}
\begin{table}
    \renewcommand{\arraystretch}{2}
    \centering
    \caption{A matrix with the rows representing $2T-1$ linearly independent points in $\convP$ satisfying inequality \eqref{eqn-T:kth-ramp-2-1} at equality}
    \rule{0pt}{0ex}
    \setlength\tabcolsep{4pt}
    \sscriptsize
        \begin{tabular}{|c|c|c|*{11}{c}|*{11}{c}|}
        \hline
         \multirow{2}{*}{Group} & \multirow{2}{*}{Point} & \multirow{2}{*}{Index $r$} & \multicolumn{11}{c|}{$\BFx$} & \multicolumn{11}{c|}{$\BFy$} \\
         \cline{4-14}\cline{15-25}
         &&& $1$ & $\cdots$ & $t\!-\!k\!-\!1$ & $t\!-\!k$ & $t\!-\!k\!+\!1$ & $\cdots$ & $t\!-\!1$ & $t$ & $t\!+\!1$ & $\cdots$ & $T$ 
           & $1$ & $\cdots$ & $t\!-\!k\!-\!1$ & $t\!-\!k$ & $t\!-\!k\!+\!1$ & $\cdots$ & $t\!-\!1$ & $t$ & $t\!+\!1$ & $\cdots$ & $T$ \\
         
         \hline
         \multirow{3}{*}{(A1)} & \multirow{10}{*}{$(\bar{\BFx}^r,\bar{\BFy}^r)$}
         & $1$ 
         & $\Clower\!+\!\epsilon$ & $\cdots$ & $0$ & $0$ & $0$ & $\cdots$ & $0$ & $0$ & $0$ & $\cdots$ & $0$
         & $1$ & $\cdots$ & $0$ & $0$ & $0$ & $\cdots$ & $0$ & $0$ & $0$ & $\cdots$ & $0$\\
         && $\vdots$
         & $\vdots$ & \rotatebox{0}{$\ddots$} & $\vdots$ & $\vdots$ & $\vdots$ && $\vdots$ & $\vdots$ & $\vdots$ && $\vdots$
         & $\vdots$ & \rotatebox{0}{$\ddots$} & $\vdots$ & $\vdots$ & $\vdots$ && $\vdots$ & $\vdots$ & $\vdots$ && $\vdots$\\
         && $t\!-\!k\!-\!1$ 
         & $\Clower$ & $\cdots$ & $\Clower\!+\!\epsilon$ & $0$ & $0$ & $\cdots$ & $0$ & $0$ & $0$ & $\cdots$ & $0$
         & $1$ & $\cdots$ & $1$ & $0$ & $0$ & $\cdots$ & $0$ & $0$ & $0$ & $\cdots$ & $0$\\
         
         \cline{1-1} \cline{3-25}
         \multirow{3}{*}{(A2)} &
         & $t\!-\!k\!+\!1$ 
         & $\Clower$ & $\cdots$ & $\Clower$ & $\Clower$ & $\Clower\!+\!\epsilon$ & $\cdots$ & $\Clower$ & $0$ & $0$ & $\cdots$ & $0$
         & $1$ & $\cdots$ & $1$ & $1$ & $1$ & $\cdots$ & $1$ & $0$ & $0$ & $\cdots$ & $0$\\
         && $\vdots$
         & $\vdots$ && $\vdots$ & $\vdots$ & $\vdots$ & \rotatebox{0}{$\ddots$} & $\vdots$ & $\vdots$ & $\vdots$ && $\vdots$
         & $\vdots$ && $\vdots$ & $\vdots$ & $\vdots$ && $\vdots$ & $\vdots$ & $\vdots$ && $\vdots$\\
         && $t\!-\!1$ 
         & $\Clower$ & $\cdots$ & $\Clower$ & $\Clower$ & $\Clower$ & $\cdots$ & $\Clower\!+\!\epsilon$ & $0$ & $0$ & $\cdots$ & $0$
         & $1$ & $\cdots$ & $1$ & $1$ & $1$ & $\cdots$ & $1$ & $0$ & $0$ & $\cdots$ & $0$\\
         
         \cline{1-1} \cline{3-25}
         (A3) &
         & $t$
         & $\Clower$ & $\cdots$ & $\Clower$ & $\Clower\!+\!\epsilon$ & $\Clower\!+\!V\!+\!\epsilon$ & $\cdots$ & $\Clower\!+\!(k\!-\!1)V\!+\!\epsilon$ & $\Clower\!+\!kV\!+\!\epsilon$ & $\Clower\!+\!kV$ & $\cdots$ & $\Clower\!+\!kV$
         & $1$ & $\cdots$ & $1$ & $1$ & $1$ & $\cdots$ & $1$ & $1$ & $1$ & $\cdots$ & $1$\\
         
         \cline{1-1} \cline{3-25}
         \multirow{3}{*}{(A4)} &
         &$t\!+\!1$
         & $0$ & $\cdots$ & $0$ & $0$ & $0$ & $\cdots$ & $0$ & $0$ & $\Clower\!+\!\epsilon$ & $\cdots$ & $\Clower$
         & $0$ & $\cdots$ & $0$ & $0$ & $0$ & $\cdots$ & $0$ & $0$ & $1$ & $\cdots$ & $1$\\
         && $\vdots$
         & $\vdots$ && $\vdots$ & $\vdots$ & $\vdots$ && $\vdots$ & $\vdots$ & $\vdots$ & \rotatebox{0}{$\ddots$} & $\vdots$
         & $\vdots$ && $\vdots$ & $\vdots$ & $\vdots$ && $\vdots$ & $\vdots$ & $\vdots$ & \rotatebox{0}{$\rred \ddots$} & $\vdots$\\
         && $T$
         & $0$ & $\cdots$ & $0$ & $0$ & $0$ & $\cdots$ & $0$ & $0$ & $0$ & $\cdots$ & $\Clower\!+\!\epsilon$
         & $0$ & $\cdots$ & $0$ & $0$ & $0$ & $\cdots$ & $0$ & $0$ & $0$ & $\cdots$ & $1$\\
         
         \hline
         \multirow{3}{*}{\rred{(A5)}} & \multirow{7}{*}{$(\hat{\BFx}^r,\hat{\BFy}^r)$}
         & $1$ & \multicolumn{11}{c|}{\multirow{3}{*}{(See Note \rred{\ref{tab:eqn-T:kth-ramp-2-1-facet-matrix-1}-1})}} & \multicolumn{11}{c|}{\multirow{3}{*}{(See Note \rred{\ref{tab:eqn-T:kth-ramp-2-1-facet-matrix-1}-1})}}\\
         && $\vdots$ & \multicolumn{11}{c|}{} & \multicolumn{11}{c|}{}\\
         && $t-1$ & \multicolumn{11}{c|}{} & \multicolumn{11}{c|}{}\\
         
         \cline{1-1} \cline{3-25}
         \rred{(A6)} &
         & $t$
         & $\Clower$ & $\cdots$ & $\Clower$ & $\Clower$ & $\Clower\!+\!V$ & $\cdots$ &
           $\Clower\!+\!(k\!-\!1)V$ & $\Clower\!+\!kV$ & $\Clower\!+\!kV$ & $\cdots$ & $\Clower\!+\!kV$
         & $1$ & $\cdots$ & $1$ & $1$ & $1$ & $\cdots$ & $1$ & $1$ & $1$ & $\cdots$ & $1$\\
         
         \cline{1-1} \cline{3-25}
         \multirow{3}{*}{\rred{(A7)}} &
         &$t\!+\!1$
         & $0$ & $\cdots$ & $0$ & $0$ & $0$ & $\cdots$ & $0$ & $0$ & $\Clower$ & $\cdots$ & $\Clower$
         & $0$ & $\cdots$ & $0$ & $0$ & $0$ & $\cdots$ & $0$ & $0$ & $1$ & $\cdots$ & $1$\\
         && $\vdots$
         & $\vdots$ && $\vdots$ & $\vdots$ & $\vdots$ && $\vdots$ & $\vdots$ & $\vdots$ & \rotatebox{0}{$\ddots$} & $\vdots$
         & $\vdots$ && $\vdots$ & $\vdots$ & $\vdots$ && $\vdots$ & $\vdots$ & $\vdots$ & \rotatebox{0}{$\ddots$} & $\vdots$\\
         && $T$
         & $0$ & $\cdots$ & $0$ & $0$ & $0$ & $\cdots$ & $0$ & $0$ & $0$ & $\cdots$ & $\Clower$
         & $0$ & $\cdots$ & $0$ & $0$ & $0$ & $\cdots$ & $0$ & $0$ & $0$ & $\cdots$ & $1$\\
         \hline
         
         \multicolumn{25}{l}{
         Note \rred{\ref{tab:eqn-T:kth-ramp-2-1-facet-matrix-1}-1}: For $r\in\rred{[1,t-1]_\Z}$, the $\BFx$ and $\BFy$ vectors in group \rred{(A5)} are given as follows:}\\
         \multicolumn{25}{l}{
         $\hat{\BFx}^r=(\underbrace{\Clower,\ldots,\Clower}_{r\ {\rm terms}},\underbrace{0,\ldots,0}_{T-r\ {\rm terms}\!\!\!\!\!})$
         and
         $\hat{\BFy}^r=(\underbrace{1,\ldots,1}_{r\ {\rm terms}},\underbrace{0,\ldots,0}_{T-r\ {\rm terms}\!\!\!\!\!})$
         if $t-r-1\notin\Se$;}\\
         \multicolumn{25}{l}{
         $\hat{\BFx}^r=(\underbrace{0,\ldots,0}_{r\ {\rm terms}},\underbrace{\Vupper,\Vupper+V,\Vupper+2V,\ldots,\Vupper+(t-r-1)V}_{t-r\ {\rm terms}},\underbrace{\Vupper+(t-r-1)V,\ldots,\Vupper+(t-r-1)V}_{T-t\ {\rm terms}\!\!\!\!\!})$
         and
         $\hat{\BFy}^r=(\underbrace{0,\ldots,0}_{r\ {\rm terms}},\underbrace{1,\ldots,1}_{T-r\ {\rm terms}\!\!\!\!\!})$
         if $t-r-1\in\Se$.
         }
    \end{tabular}
    \label{tab:eqn-T:kth-ramp-2-1-facet-matrix-1}
\end{table}
\end{landscape}


\begin{landscape}
\begin{table}
    \renewcommand{\arraystretch}{2}
    \centering
    \caption{Lower triangular matrix obtained from Table \ref{tab:eqn-T:kth-ramp-2-1-facet-matrix-1} via Gaussian elimination}
    \rule{0pt}{0ex}
    \setlength\tabcolsep{6.5pt}
    \sscriptsize
        \begin{tabular}{|c|c|c|*{11}{c}|*{7}{c}|}
        \hline
         \multirow{2}{*}{Group} & \multirow{2}{*}{Point} & \multirow{2}{*}{Index $r$} & \multicolumn{11}{c|}{$\BFx$} & \multicolumn{7}{c|}{$\BFy$} \\
         \cline{4-14}\cline{15-21}
         &&& $\ \ 1\ \ $ & $\cdots$ & $t\!-\!k\!-\!1$ & $t\!-\!k$ & $t\!-\!k\!+\!1$ & $\cdots$ & $t\!-\!1$ & $\ \ t\ \ $ & $t\!+\!1$ & $\cdots$ & $\ \ T\ \ $ 
           & $\ \ 1\ \ $ & $\cdots$ & $t-1$ & $\ \ t\ \ $ & $t+1$ & $\cdots$ & $\ \ T\ \ $ \\
         
         \hline
         \multirow{3}{*}{(B1)} & \multirow{10}{*}{$(\underline{\bar{\BFx}}^r,\underline{\bar{\BFy}}^r)$}
         & $1$ 
         & $\epsilon$ & $\cdots$ & $0$ & $0$ & $0$ & $\cdots$ & $0$ & $0$ & $0$ & $\cdots$ & $0$
         & $0$ & $\cdots$ & $0$ & $0$ & $0$ & $\cdots$ & $0$\\
         && $\vdots$
         & $\vdots$ & \rotatebox{0}{$\ddots$} & $\vdots$ & $\vdots$ & $\vdots$ && $\vdots$ & $\vdots$ & $\vdots$ && $\vdots$
         & $\vdots$ && $\vdots$ & $\vdots$ & $\vdots$ && $\vdots$\\
         && $t\!-\!k\!-\!1$ 
         & $0$ & $\cdots$ & $\epsilon$ & $0$ & $0$ & $\cdots$ & $0$ & $0$ & $0$ & $\cdots$ & $0$
         & $0$ & $\cdots$ & $0$ & $0$ & $0$ & $\cdots$ & $0$\\
         
         \cline{1-1} \cline{3-21}
         \multirow{3}{*}{(B2)} &
         & $t\!-\!k\!+\!1$ 
         & $0$ & $\cdots$ & $0$ & $0$ & $\epsilon$ & $\cdots$ & $0$ & $0$ & $0$ & $\cdots$ & $0$
         & $0$ & $\cdots$ & $0$ & $0$ & $0$ & $\cdots$ & $0$\\
         && $\vdots$
         & $\vdots$ && $\vdots$ & $\vdots$ & $\vdots$ & \rotatebox{0}{$\ddots$} & $\vdots$ & $\vdots$ & $\vdots$ && $\vdots$
         & $\vdots$ && $\vdots$ & $\vdots$ & $\vdots$ && $\vdots$\\
         && $t\!-\!2$ 
         & $0$ & $\cdots$ & $0$ & $0$ & $0$ & $\cdots$ & $\epsilon$ & $0$ & $0$ & $\cdots$ & $0$
         & $0$ & $\cdots$ & $0$ & $0$ & $0$ & $\cdots$ & $0$\\
         
         \cline{1-1} \cline{3-21}
         (B3) &
         & $t$
         & $0$ & $\cdots$ & $0$ & $\epsilon$ & $\epsilon$ & $\cdots$ & $\epsilon$ & $\epsilon$ & $0$ & $\cdots$ & $0$
         & $0$ & $\cdots$ & $0$ & $0$ & $0$ & $\cdots$ & $0$\\
         
         \cline{1-1} \cline{3-21}
         \multirow{3}{*}{(B4)} &
         &$t\!+\!1$
         & $0$ & $\cdots$ & $0$ & $0$ & $0$ & $\cdots$ & $0$ & $0$ & $\epsilon$ & $\cdots$ & $0$
         & $0$ & $\cdots$ & $0$ & $0$ & $0$ & $\cdots$ & $0$\\
         && $\vdots$
         & $\vdots$ && $\vdots$ & $\vdots$ & $\vdots$ && $\vdots$ & $\vdots$ & $\vdots$ & \rotatebox{0}{$\ddots$} & $\vdots$
         & $\vdots$ && $\vdots$ & $\vdots$ & $\vdots$ && $\vdots$\\
         && $T$
         & $0$ & $\cdots$ & $0$ & $0$ & $0$ & $\cdots$ & $0$ & $0$ & $0$ & $\cdots$ & $\epsilon$
         & $0$ & $\cdots$ & $0$ & $0$ & $0$ & $\cdots$ & $0$\\
         
         \hline
         \multirow{3}{*}{\rred{(B5)}} & \multirow{7}{*}{\rred{$(\underline{\hat{\BFx}}^r,\underline{\hat{\BFy}}^r)$}}
         & $1$
         & \multicolumn{11}{c|}{}
         & $\rred 1$ & $\rred \cdots$ & $\rred 0$ & $\rred 0$ & $\rred 0$ & $\rred \cdots$ & $\rred 0$\\
         && $\vdots$
         & \multicolumn{11}{c|}{(Omitted)}
         & $\rred \vdots$ & \rotatebox{0}{$\rred \ddots$} & $\rred \vdots$ & $\rred \vdots$ & $\rred \vdots$ && $\rred \vdots$\\
         && $t-1$
         & \multicolumn{11}{c|}{}
         & $\rred 1$ & $\rred \cdots$ & $\rred 1$ & $\rred 0$ & $\rred 0$ & $\rred \cdots$ & $\rred 0$\\
         
         \cline{1-1} \cline{3-21}
         \rred{(B6)} &
         & $t$
         & \multicolumn{11}{c|}{(Omitted)}
         & $1$ & $\cdots$ & $1$ & $1$ & $0$ & $\cdots$ & $0$\\
         
         \cline{1-1} \cline{3-21}
         \multirow{3}{*}{\rred{(B7)}} &
         &$t\!+\!1$
         & \multicolumn{11}{c|}{}
         & $0$ & $\cdots$ & $0$ & $0$ & $1$ & $\cdots$ & $0$\\
         && $\vdots$
         & \multicolumn{11}{c|}{(Omitted)}
         & $\vdots$ && $\vdots$ & $\vdots$ & $\vdots$ & \rotatebox{0}{$\ddots$} & $\vdots$\\
         && $T$
         & \multicolumn{11}{c|}{}
         & $0$ & $\cdots$ & $0$ & $0$ & $0$ & $\cdots$ & $1$\\
         \hline
    \end{tabular}
    \label{tab:eqn-T:kth-ramp-2-1-facet-matrix-2}
\end{table}
\end{landscape}
}
    

The matrix shown in Table \ref{tab:eqn-T:kth-ramp-2-1-facet-matrix-2} is lower triangular; that is, the position of the last nonzero component of a row of the matrix is greater than the position of the last nonzero component of the previous row.
This implies that these $2T-1$ points in groups \ref{points:prop14-eq1-A1}--\ref{points:prop14-eq1-A7} are linearly independent.
Therefore, inequality \eqref{eqn-T:kth-ramp-2-1} is facet-defining for $\convP$.

Next, we show that inequality \eqref{eqn-T:kth-ramp-2-2} is valid for $\convP$ and is facet-defining for $\convP$ when $m=0$ and $s\ge\min\{k-1,1\}$ $\forall s\in\Se$.
Denote $x'_t=x_{T-t+1}$ and $y'_t=y_{T-t+1}$ for $t\in[1,T]_\Z$.
Because inequality \eqref{eqn-T:kth-ramp-2-1} is valid for $\convP$ and is facet-defining for $\convP$ when $m=0$ and $s\ge\min\{k-1,1\}$ $\forall s\in\Se$ for any $t\in[k+1,T-m]_\Z$, the inequality
\begin{align*}
x'_{T-t+1}-x'_{T-t+k+1}\le\
& (\Clower+(k-m)V)y'_{T-t+1}+V\sum_{i=1}^my'_{T-t-i+1}-\Clower y'_{T-t+k+1}\\
& -\sum_{s\in\Se}(\Clower+(k-s)V-\Vupper)(y'_{T-t+s+1}-y'_{T-t+s+2})
\end{align*}
is valid for $\convPprime$ and is facet-defining for $\convPprime$ when $m=0$ and $s\ge\min\{k-1,1\}$ $\forall s\in\Se$ for any $t\in[k+1,T-m]_\Z$.
Let $t'=T-t+1$.
Then, the inequality
$$x'_{t'}-x'_{t'+k}\le(\Clower+(k-m)V)y'_{t'}+V\sum_{i=1}^my'_{t'-i}-\Clower y'_{t'+k}-\sum_{s\in\Se}(\Clower+(k-s)V-\Vupper)(y'_{t'+s}-y'_{t'+s+1})$$
is valid for $\convPprime$ and is facet-defining for $\convPprime$ when $m=0$ and $s\ge\min\{k-1,1\}$ $\forall s\in\Se$ for any $t'\in[m+1,T-k]_\Z$.
Hence, by Lemma~\ref{lem:Pprime}, inequality \eqref{eqn-T:kth-ramp-2-2} is valid for $\convP$ and is facet-defining for $\convP$ when $m=0$ and $s\ge\min\{k-1,1\}$ $\forall s\in\Se$ for any $t\in[m+1,T-k]_\Z$.\Halmos

\subsection{Proof of Proposition \ref{prop-T:kth-ramp-2-separation}}

\noindent{\bf Proposition \ref{prop-T:kth-ramp-2-separation}.} {\it
For any given point $(\BFx,\BFy)\in\R^{2T}_+$, \tred{the} most violated inequalities \eqref{eqn-T:kth-ramp-2-1} and \eqref{eqn-T:kth-ramp-2-2} can be determined in $O(T^3)$ time if such violated inequalities exist.}
\vskip8pt

\noindent{\bf Proof.}
We first consider inequality \eqref{eqn-T:kth-ramp-2-1}.
Consider any given $(\BFx,\BFy)\in\R^{2T}_{+}$.
For notational convenience, denote $\hat{k}=\max\{k\in[1,T-1]_\Z:\Cupper-\Clower-kV>0\}$, and denote $\hat{s}_{km}=\min\{k-1,L-m-1\}$ for any $k\in[1,\hat{k}]_\Z$ and $m\in[0,k-1]_\Z$.
For any $t\in[1,T]_\Z$, let
$$\theta(t)=\sum_{\tau=2}^t\max\{y_\tau-y_{\tau-1},0\}.$$
Then, for any $k\in[1,\hat{k}]_\Z$, $m\in[0,k-1]_\Z$, and $t\in[k+1,T-m]_\Z$,
\begin{equation}\label{eq:P10-ysummation1}
\sum_{s=1}^{\hat{s}_{km}}\max\{y_{t-s}-y_{t-s-1},0\}=\sum_{\tau=t-\hat{s}_{km}}^{t-1}\max\{y_\tau-y_{\tau-1},0\}=\theta(t-1)-\theta(t-\hat{s}_{km}-1).
\end{equation}
For any $k\in[1,\hat{k}]_\Z$, $m\in[0,k-1]_\Z$, $t\in[k+1,T-m]_\Z$, and $\Se\subseteq[0,\hat{s}_{km}]_\Z$, let
$$\tilde{v}_{km}(\Se,t)=x_t-x_{t-k}-(\Clower+(k-m)V)y_t-V\sum_{i=1}^m y_{t+i}+\Clower y_{t-k}+\sum_{s\in\Se}(\Clower+(k-s)V-\Vupper)(y_{t-s}-y_{t-s-1}).$$
If $\tilde{v}_{km}(\Se,t)>0$, then $\tilde{v}_{km}(\Se,t)$ is the amount of violation of inequality \eqref{eqn-T:kth-ramp-2-1}.
If $\tilde{v}_{km}(\Se,t)\le 0$, there is no violation of inequality \eqref{eqn-T:kth-ramp-2-1}.
For any $k\in[1,\hat{k}]_\Z$, $m\in[0,k-1]_\Z$, and $t\in[k+1,T-m]_\Z$, let
$$v_{km}(t)=\max_{\Se\subseteq[0,\hat{s}_{km}]_\Z}\{\tilde{v}_{km}(\Se,t)\}.$$
If $v_{km}(t)>0$, then $v_{km}(t)$ is the largest possible violation of inequality \eqref{eqn-T:kth-ramp-2-1} for this combination of $k$, $m$, and $t$.
If $v_{km}(t)\le 0$, the largest possible violation of inequality \eqref{eqn-T:kth-ramp-2-1} is zero for this combination of $k$, $m$, and $t$.
Because $\Clower+V>\Vupper$, we have $\Clower+(k-s)V-\Vupper>0$ for all $k\in[1,\hat{k}]_\Z$, $s\in[0,\hat{s}_{km}]_\Z$, and $m\in[0,k-1]_\Z$.
Thus, for any $k\in[1,\hat{k}]_\Z$, $m\in[0,k-1]_\Z$, and $t\in[k+1,T-m]_\Z$, 
$\tilde{v}_{km}(\Se,t)$ is maximized when $\Se$ contains all $s\in[0,\hat{s}_{km}]_\Z$ such that $y_{t-s}-y_{t-s-1}>0$ (if any).
If it does not exist any $s\in[0,\hat{s}]_\Z$ such that $y_{t-s}-y_{t-s-1}>0$, then $\tilde{v}_{km}(\Se,t)$ is maximized when $\Se=\emptyset$, 
and $v_{km}(t)=x_t-x_{t-k}-(\Clower+(k-m)V)y_t-V\sum_{i=1}^m y_{t+i}+\Clower y_{t-k}$.
Hence, for any $k\in[1,\hat{k}]_\Z$, $m\in[0,k-1]_\Z$, and $t\in[k+1,T-m]_\Z$, 
$$v_{km}(t)=x_t\!-\!x_{t-k}\!-\!(\Clower\!+\!(k\!-\!m)V)y_t\!-\!V\sum_{i=1}^m y_{t+i}\!+\!\Clower y_{t-k}\!+\!\sum_{s=0}^{\hat{s}_{km}}(\Clower\!+\!(k\!-\!s)V\!-\!\Vupper)\max\{y_{t-s}\!-\!y_{t-s-1},0\}.$$

Determining $\theta(t)$ for all $t\in[1,T]_\Z$ can be done recursively in $O(T)$ time by setting $\theta(1)=0$ and setting $\theta(t)=\theta(t-1)+\max\{y_t-y_{t-1},0\}$ for $t=2,\ldots,T$.
Clearly, for each $k\in[1,\hat{k}]_\Z$ and each $m\in[0,k-1]_\Z$, the value of $v_{km}(k+1)$ can be determined in $O(T)$ time.
For any $k\in[1,\hat{k}]_\Z$, $m\in[0,k-1]_\Z$, and $t\in[k+2,T-m]_\Z$,
\begin{align*}
v_{km}(t)-v_{km}(t-1)
&=(x_t-x_{t-1})-(x_{t-k}-x_{t-k-1})-(\Clower+(k-m)V)(y_t-y_{t-1})\\
&\quad -V\left[\sum_{i=1}^m y_{t+i}-\sum_{i=1}^m y_{t+i-1}\right]+\Clower(y_{t-k}-y_{t-k-1})\\
&\quad +(\Clower+kV-\Vupper)\left[\sum_{s=0}^{\hat{s}_{km}}\max\{y_{t-s}-y_{t-s-1},0\}-\sum_{s=0}^{\hat{s}_{km}}\max\{y_{t-s-1}-y_{t-s-2},0\}\right]\\
&\quad -V\left[\sum_{s=0}^{\hat{s}_{km}}s\max\{y_{t-s}-y_{t-s-1},0\}-\sum_{s=0}^{\hat{s}_{km}}s\max\{y_{t-s-1}-y_{t-s-2},0\}\right]\\
&=(x_t-x_{t-1})-(x_{t-k}-x_{t-k-1})-(\Clower+(k-m)V)(y_t-y_{t-1})\\
&\quad -V(y_{t+m}-y_t)+\Clower(y_{t-k}-y_{t-k-1})\\
&\quad +(\Clower+kV-\Vupper)\left[\max\{y_t-y_{t-1},0\}-\max\{y_{t-\hat{s}_{km}-1}-y_{t-\hat{s}_{km}-2},0\}\right]\\
&\quad -V\left[\sum_{s=1}^{\hat{s}_{km}}\max\{y_{t-s}-y_{t-s-1},0\}-\hat{s}_{km}\max\{y_{t-\hat{s}_{km}-1}-y_{t-\hat{s}_{km}-2},0\}\right].
\end{align*}
This, together with \eqref{eq:P10-ysummation1}, implies that
\begin{align*}
v_{km}(t)=\ &v_{km}(t-1)+(x_t-x_{t-1})-(x_{t-k}-x_{t-k-1})-(\Clower+(k-m)V)(y_t-y_{t-1})\\
&-V(y_{t+m}-y_t)+\Clower(y_{t-k}-y_{t-k-1})\\
&+(\Clower+kV-\Vupper)\left[\max\{y_t-y_{t-1},0\}-\max\{y_{t-\hat{s}_{km}-1}-y_{t-\hat{s}_{km}-2},0\}\right]\\
&-V\left[\theta(t-1)-\theta(t-\hat{s}_{km}-1)-\hat{s}_{km}\max\{y_{t-\hat{s}_{km}-1}-y_{t-\hat{s}_{km}-2},0\}\right].
\end{align*}
Thus, for each $k\in[1,\hat{k}]_\Z$ and $m\in[0,k-1]_\Z$, the values of $v_{km}(k+1),v_{km}(k+2),\ldots,v_{km}(T-m)$ can be determined recursively in $O(T)$ time.
Hence, the values of $k$, $m$, $t$ and the set $\Se$ corresponding to the largest possible violation of inequality \eqref{eqn-T:kth-ramp-2-1} can be obtained in $O(T^3)$ time.

Next, we consider inequality \eqref{eqn-T:kth-ramp-2-2}.
\rred{
Consider any given $(\BFx,\BFy)\in\R^{2T}_{+}$.
Let $x'_t=x_{T-t+1}$ and $y'_t=y_{T-t+1}$ for $t\in[1,T]_\Z$.
Inequality \eqref{eqn-T:kth-ramp-2-2} becomes
\begin{align*}
x'_{T-t+1}-x'_{T-t-k+1}\le\
&(\Clower+(k-m)V)y'_{T-t+1}+V\sum_{i=1}^m y'_{T-t+i+1}-\Clower y'_{T-t-k+1}\\
&-\sum_{s\in\Se}(\Clower+(k-s)V-\Vupper)(y'_{T-t-s+1}-y'_{T-t-s}).
\end{align*}
Letting $t'=T-t+1$, this inequality becomes
\begin{align}
x'_{t'}-x'_{t'-k}\le\
&(\Clower+(k-m)V)y'_{t'}+V\sum_{i=1}^m y'_{t'+i}-\Clower y'_{t'-k}\nonumber\\
&-\sum_{s\in\Se}(\Clower+(k-s)V-\Vupper)(y'_{t'-s}-y'_{t'-s-1}).\label{eqn-T:kth-ramp-2-2-prime}
\end{align}
Because the values of $k$, $m$, $t$ and the set $\Se$ corresponding to the largest possible violation of inequality \eqref{eqn-T:kth-ramp-2-1} can be obtained in $O(T^3)$ time, the values of $k$, $m$, $t'$ and the set $\Se$ corresponding to the largest possible violation of inequality \eqref{eqn-T:kth-ramp-2-2-prime} can be obtained in $O(T^3)$ time.
Hence, the values of $k$, $m$, $t$ and the set $\Se$ corresponding to the largest possible violation of inequality \eqref{eqn-T:kth-ramp-2-2} can be obtained in $O(T^3)$ time.
\Halmos
}

\subsection{Proof of Proposition \ref{prop-T:kth-ramp-3}}

\noindent{\bf Proposition \ref{prop-T:kth-ramp-3}.} {\it
Consider any $k\in[1,T-1]_\Z$ such that $\Cupper-\Clower-kV>0$, any $m\in[0,k-1]_\Z$, and any $\Se\subseteq[0,\min\{k-1,L-m-2\}]_\Z$.
For any $t\in[k+1,T-m-1]_\Z$, the inequality
\begin{align}
x_t-x_{t-k}&\le (\Clower+(k-m)V-\Vupper)y_{t+m+1}+V\sum_{i=1}^m y_{t+i}+\Vupper y_t-\Clower y_{t-k}\nonumber\\
           &\quad -\sum_{s\in\Se}(\Clower+(k-s)V-\Vupper)(y_{t-s}-y_{t-s-1})\tag{\ref{eqn-T:kth-ramp-3-1}}
\end{align}
is valid and facet-defining for $\convP$.
For any $t\in[m+2,T-k]_\Z$, the inequality
\begin{align}
x_t-x_{t+k}&\le (\Clower+(k-m)V-\Vupper)y_{t-m-1}+V\sum_{i=1}^m y_{t-i}+\Vupper y_t-\Clower y_{t+k}\nonumber\\
           &\quad -\sum_{s\in\Se}(\Clower+(k-s)V-\Vupper)(y_{t+s}-y_{t+s+1})\tag{\ref{eqn-T:kth-ramp-3-2}}
\end{align}
is valid and facet-defining for $\convP$.}
\vskip8pt

\noindent{\bf Proof.}
We first prove that inequality \eqref{eqn-T:kth-ramp-3-1} is valid and facet-defining for $\convP$.
For notational convenience, we define $s_{\max}=\max\{s:s\in\Se\}$ if $\Se\ne\emptyset$, and $s_{\max}=-1$ if $\Se=\emptyset$.
Consider any $t\in[k+1,T-m-1]_\Z$.
To prove that the linear inequality \eqref{eqn-T:kth-ramp-3-1} is valid for $\convP$, it suffices to show that it is valid for $\Pe$.
Consider any element $(\BFx,\BFy)$ of $\Pe$.
We show that $(\BFx,\BFy)$ satisfies \eqref{eqn-T:kth-ramp-3-1}.
We divide the analysis into three cases:

Case~1: $y_t=0$.
In this case, by \eqref{eqn:p-lower-bound} and \eqref{eqn:p-upper-bound}, $-x_{t-k}\le-\Clower y_{t-k}$ and $x_t=0$.
Thus, the left-hand side of \eqref{eqn-T:kth-ramp-3-1} is at most \rred{$-\Clower y_{t-k}$}.
Because $y_t=0$ and $t\in[2,T]_{\Z}$, by Lemma \ref{lem:lookbackward}(i), $y_{t-j}-y_{t-j-1}\le0$ for all $j\in[0,\min\{t-2,L-1\}]_{\Z}$.
Because $\Se\subseteq[0,\min\{k-1,L-m-2\}]_{\Z}$, $m\ge0$, and $t\ge k+1$, we have $\Se\subseteq[0,\min\{t-2,L-1\}]_{\Z}$.
Thus, $y_{t-s}-y_{t-s-1}\le0$ for all $s\in\Se$.
Because $m\le k-1$, $\Se\subseteq[0,k-1]_{\Z}$, and $\Clower+V>\Vupper$, the coefficients ``$\Clower+(k-m)V-\Vupper$" and ``$\Clower+(k-s)V-\Vupper$" on the right-hand side of \eqref{eqn-T:kth-ramp-3-1} are positive for any $s\in\Se$.
Thus, the right-hand side of \eqref{eqn-T:kth-ramp-3-1} is at least $-\Clower y_{t-k}$.
Therefore, in this case, $(\BFx,\BFy)$ satisfies \eqref{eqn-T:kth-ramp-3-1}.

Case~2: $y_t=1$ and $y_{t-s'}-y_{t-s'-1}=1$ for some $s'\in\Se$.
In this case, $y_{t-s'}=1$ and $y_{t-s'-1}=0$.
Because $y_t=1$ and $t\in[2,T]_{\Z}$, by Lemma \ref{lem:lookbackward}(ii), there exists at most one $j\in[0,\min\{t-2,L\}]_{\Z}$ such that $y_{t-j}-y_{t-j-1}=1$.
Because $\Se\subseteq[0,\min\{k-1,L-m-2\}]_{\Z}$, $m\ge0$, and $t\ge k+1$, we have $\Se\subseteq[0,\min\{t-2,L\}]_{\Z}$.
Thus, $y_{t-s}-y_{t-s-1}\le0$ for all $s\in\Se\setminus\{s'\}$.
Because $y_{t-s'}-y_{t-s'-1}=1$ and $t-s'\in[2,T]_{\Z}$, by \eqref{eqn:p-minup}, we have $y_{\tau}=1$ for all $\tau\in[t-s',\min\{T,t-s'+L-1\}]_{\Z}$.
Because $\Se\subseteq[0, L-m-2]_{\Z}$, we have $t-s'+L-1\ge t+m+1$.
Thus, $y_{\tau}=1$ for all $\tau\in[t-s',t+m+1]_{\Z}$\rred{, which implies that $(\Clower+(k-m)V-\Vupper)y_{t+m+1}+V\sum_{i=1}^m y_{t+i}+\Vupper y_t=\Clower+kV$}.
Because $\Se\subseteq[0,k-1]_{\Z}$ and $\Clower+V>\Vupper$, \rred{the} coefficient ``$\Clower+(k-s)V-\Vupper$" on the right-hand side of inequality \eqref{eqn-T:kth-ramp-3-1} is positive \rred{for all $s\in\Se$}. 
Hence, the right-hand side of \eqref{eqn-T:kth-ramp-3-1} is at least \rred{$\Clower+kV-\Clower y_{t-k}-(\Clower+(k-s')V-\Vupper)$}$=s'V+\Vupper-\Clower y_{t-k}$.
By \eqref{eqn:p-ramp-up}, $\sum_{\tau=t-s'}^t(x_{\tau}-x_{\tau-1})\le\sum_{\tau=t-s'}^tVy_{\tau-1}+\sum_{\tau=t-s'}^t\Vupper(1-y_{\tau-1})$, which implies that $x_t-x_{t-s'-1}\le s'V+\Vupper$.
Because $y_{t-s'-1}=0$, by \eqref{eqn:p-upper-bound}, $x_{t-s'-1}=0$.
\rred{By} \eqref{eqn:p-lower-bound}, $-x_{t-k}\le-\Clower y_{t-k}$.
Thus, $x_t-x_{t-k}\le s'V+\Vupper-\Clower y_{t-k}$.
Therefore, in this case, $(\BFx,\BFy)$ satisfies \eqref{eqn-T:kth-ramp-3-1}.

Case~3: $y_t=1$ and $y_{t-s}-y_{t-s-1}\le0$ for all $s\in\Se$.
Because $\Se\subseteq[0,k-1]_{\Z}$, $m\in[0, k-1]_{\Z}$, and $\Clower+V>\Vupper$, we have $\Clower+(k-m)V-\Vupper>0$ and $\Clower+(k-s)V-\Vupper>0$ for each $s\in\Se$.
Hence, the terms ``$(\Clower+(k-m)V-\Vupper)y_{t+m+1}$" and ``$-\sum_{s\in\Se}(\Clower+(k-s)V-\Vupper)(y_{t-s}-y_{t-s-1})$" on the right-hand side of inequality \eqref{eqn-T:kth-ramp-3-1} are nonnegative.
We divide our analysis into three subcases.

Case~3.1: $y_{\tau}=0$ for some $\tau\in[t,t+m+1]_\Z$.
Let $t'=\min\{\tau\in[t,t+m+1]_\Z:y_{\tau}=0\}$.
Then, $y_{\tau}=1$ for all $\tau\in[t,t'-1]_\Z$.
Thus, the right-hand side of \eqref{eqn-T:kth-ramp-3-1} is at least $(t'-t-1)V+\Vupper-\Clower y_{t-k}$.
By \eqref{eqn:p-ramp-down}, $\sum_{\tau=t+1}^{t'}(x_{\tau-1}-x_{\tau})\le\sum_{\tau=t+1}^{t'}Vy_{\tau}+\sum_{\tau=t+1}^{t'}\Vupper(1-y_{\tau})$, 
which implies that $x_t-x_{t'}\le(t'-t-1)V+\Vupper$.
Because $y_{t'}=0$, by \eqref{eqn:p-upper-bound}, $x_{t'}=0$.
By \eqref{eqn:p-lower-bound}, $-x_{t-k}\le-\Clower y_{t-k}$.
Hence, $x_t-x_{t-k}\le(t'-t-1)V+\Vupper-\Clower y_{t-k}$.
Thus, the left-hand side of \eqref{eqn-T:kth-ramp-3-1} is less than or equal to the right-hand side.

Case~3.2: $y_{\tau}=1$ for all $\tau\in[t,t+m+1]_\Z$ and $y_{\tau}=0$ for some $\tau\in[t-k,t-1]_\Z$.
In this case, the right-hand side of \eqref{eqn-T:kth-ramp-3-1} is at least $(\Clower+(k-m)V-\Vupper)+mV+\Vupper-\Clower y_{t-k}=\Clower+kV-\Clower y_{t-k}$.
Let $t'=\max\{\tau\in[t-k,t-1]_\Z:y_{\tau}=0\}$.
Because $t'\ge t-k$ and $\Clower+V>\Vupper$, the right-hand side of \eqref{eqn-T:kth-ramp-3-1} is greater than $\Vupper+(t-t'-1)V-\Clower y_{t-k}$.
By \eqref{eqn:p-ramp-up}, $\sum_{\tau=t'+1}^t(x_{\tau}-x_{\tau-1})\le\sum_{\tau=t'+1}^tVy_{\tau-1}+\sum_{\tau=t'+1}^t\Vupper(1-y_{\tau-1})$, 
which implies that $x_t-x_{t'}\le\Vupper+(t-t'-1)V$.
Because $y_{t'}=0$, by \eqref{eqn:p-upper-bound}, $x_{t'}=0$.
By \eqref{eqn:p-lower-bound}, $-x_{t-k}\le-\Clower y_{t-k}$.
Hence, $x_t-x_{t-k}\le\Vupper+(t-t'-1)V-\Clower y_{t-k}$.
Thus, the left-hand side of \eqref{eqn-T:kth-ramp-3-1} is less than the right-hand side.

Case~3.3: $y_{\tau}=1$ for all $\tau\in[t-k,t+m+1]_\Z$.
In this case, the right-hand side of \eqref{eqn-T:kth-ramp-3-1} is at least $kV$.
By \eqref{eqn:p-ramp-up}, $\sum_{\tau=t-k+1}^t(x_{\tau}-x_{\tau-1})\le\sum_{\tau=t-k+1}^tVy_{\tau-1}+\sum_{\tau=t-k+1}^t\Vupper(1-y_{\tau-1})$, 
which implies that $x_t-x_{t-k}\le kV$.
Thus, the left-hand side of \eqref{eqn-T:kth-ramp-3-1} is less than or equal to the right-hand side.

\rred{In Cases 3.1--3.3, $(\BFx,\BFy)$ satisfies \eqref{eqn-T:kth-ramp-3-1}.
Summarizing Cases 1--3, we conclude that \eqref{eqn-T:kth-ramp-3-1} is valid for $\convP$.}

\rred{Consider any $t\in[k+1,T-m-1]_\Z$.
To prove that inequality \eqref{eqn-T:kth-ramp-3-1} is facet-defining for $\convP$, 
it suffices to show that there exist $2T$ affinely independent points in $\convP$ that satisfy \eqref{eqn-T:kth-ramp-3-1} at equality.
Because $\BFzero\in\convP$ and $\BFzero$ satisfies \eqref{eqn-T:kth-ramp-3-1} at equality, it suffices to create the remaining $2T-1$ nonzero linearly independent points.
We denote these $2T-1$ points as $(\bar{\BFx}^r,\bar{\BFy}^r)$ for $r\in[1,T]_{\Z}\setminus\{t-k\}$ and $(\hat{\BFx}^r,\hat{\BFy}^r)$ for $r\in[1,T]_{\Z}$, 
and denote the $q$th component of $\bar{\BFx}^r$, $\bar{\BFy}^r$, $\hat{\BFx}^r$, and $\hat{\BFy}^r$ as $\bar{x}^r_q$, $\bar{y}^r_q$, $\hat{x}^r_q$, and $\hat{y}^r_q$, respectively.
Let $\epsilon=\min\{\Vupper-\Clower,\Cupper-\Clower-kV\}>0$.}

We divide these $2T-1$ points into the following eight groups.

\begin{enumerate}[label=(A\arabic*)]
    \item\label{points:prop16-eq1-A1} For each $r\in[1,t-1]_{\Z}\setminus\{t-k\}$, we create a point $(\bar{\BFx}^r,\bar{\BFy}^r)$ as follows:
    \begin{equation*}
        \bar{x}^r_q=\left\{
        \begin{array}{ll}
            \Clower, &\hbox{for $q\in[1,t-1]_{\Z}\setminus\{r\}$};\\
            \Clower+\epsilon, &\hbox{for $q=r$};\\
            \Vupper, &\hbox{for $q=t$};\\
            0, &\hbox{for $q\in[t+1,T]_{\Z}$};
        \end{array}\right.
    \end{equation*}
    and
    \begin{equation*}
        \bar{y}^r_q=\left\{
        \begin{array}{ll}
            1,&\hbox{for $q\in[1,t]_{\Z}$};\\
            0, &\hbox{for $q\in[t+1,T]_{\Z}$}.
        \end{array}\right.
    \end{equation*}
    It is easy to verify that $(\bar{\BFx}^r,\bar{\BFy}^r)$ satisfies \eqref{eqn:p-minup}--\eqref{eqn:p-ramp-down}.
    Thus, $(\bar{\BFx}^r,\bar{\BFy}^r)\in\convP$.
    Note that $\bar{x}^r_t=\Vupper$, $\bar{x}^r_{t-k}=\Clower$, $\bar{y}^r_t=\bar{y}^r_{t-k}=1$, $\bar{y}^r_{t+m+1}=0$, $\sum_{i=1}^m\bar{y}^r_{t+i}=0$, and $\bar{y}^r_{t-s}-\bar{y}^r_{t-s-1}=0$ for all $s\in\Se$.
    Hence, $(\bar{\BFx}^r,\bar{\BFy}^r)$ satisfies \eqref{eqn-T:kth-ramp-3-1} at equality.
    
    \item\label{points:prop16-eq1-A2} We create the same point $(\bar{\BFx}^t,\bar{\BFy}^t)$ as in group \ref{points:prop14-eq1-A3} in the proof of Proposition~\ref{prop-T:kth-ramp-2}. 
    Thus, $(\bar{\BFx}^t,\bar{\BFy}^t)\in\convP$.
    Note that $\bar{x}^t_t=\Clower+kV+\epsilon$, $\bar{x}^t_{t-k}=\Clower+\epsilon$, $\bar{y}^t_t=\bar{y}^t_{t-k}=\bar{y}^t_{t+m+1}=1$, $\sum_{i=1}^m \bar{y}^t_{t+i}=m$, and $\bar{y}^t_{t-s}-\bar{y}^t_{t-s-1}=0$ for all $s\in\Se$.
    Hence, $(\bar{\BFx}^t,\bar{\BFy}^t)$ satisfies \eqref{eqn-T:kth-ramp-3-1} at equality.
    
    \item\label{points:prop16-eq1-A3} For each $r\in[t+1,T]_{\Z}$, we create a point $(\bar{\BFx}^r,\bar{\BFy}^r)$ as follows:
    \begin{equation*}
        \bar{x}^r_q=\left\{
        \begin{array}{ll}
        \Clower, &\hbox{for $q\in[1,t-k-1]_{\Z}$};\\
        \Clower+(q-t+k)V, &\hbox{for $q\in[t-k,t]_{\Z}$};\\
        \Clower+kV, &\hbox{for $q\in[t+1,T]_{\Z}\setminus\{r\}$};\\
        \Clower+kV+\epsilon, &\hbox{for $q=r$};
        \end{array}\right.
    \end{equation*}
    and $\bar{y}^r_q=1$ for all $q\in[1,T]_{\Z}$.
    It is easy to verify that $(\bar{\BFx}^r,\bar{\BFy}^r)$ satisfies \eqref{eqn:p-minup}--\eqref{eqn:p-upper-bound}.
    Note that $\bar{x}^r_q-\bar{x}^r_{q-1}=0$ when $q\in[2,t-k]_{\Z}$, $\bar{x}^r_q-\bar{x}^r_{q-1}=V$ when $q\in[t-k+1,t]_{\Z}$, and $-\epsilon\le\bar{x}^r_q-\bar{x}^r_{q-1}\le\epsilon$ when $q\in[t+1,T]_{\Z}$.
    Thus, $-V\bar{y}^r_q-\Vupper(1-\bar{y}^r_q)\le\bar{x}^r_q-\bar{x}^r_{q-1}\le V\bar{y}^r_{q-1}+\Vupper(1-\bar{y}^r_{q-1})$ for all $q\in[2,T]_\Z$.
    Hence, $(\bar{\BFx}^r,\bar{\BFy}^r)$ satisfies \eqref{eqn:p-ramp-up} and \eqref{eqn:p-ramp-down}.
    Therefore, $(\bar{\BFx}^r,\bar{\BFy}^r)\in\convP$.
    Note that $\bar{x}^r_t=\Clower+kV$, $\bar{x}^r_{t-k}=\Clower$, $\bar{y}^r_t=\bar{y}^r_{t-k}=\bar{y}^r_{t+m+1}=1$, $\sum_{i=1}^m \bar{y}^r_{t+i}=m$, and $\bar{y}^r_{t-s}-\bar{y}^r_{t-s-1}=0$ for all $s\in\Se$.
    Hence, $(\bar{\BFx}^r,\bar{\BFy}^r)$ satisfies \eqref{eqn-T:kth-ramp-3-1} at equality.
    
    \item\label{points:prop16-eq1-A4} For each \rred{$r\in[1,t-1]_{\Z}$}, we create the same point $(\hat{\BFx}^r,\hat{\BFy}^r)$ as in group \ref{points:prop1-eq1-A2} in the proof of Proposition~\ref{prop-T:x_t-ub-1}.
    Thus, $(\hat{\BFx}^r,\hat{\BFy}^r)\in\convP$.
    To show that $(\hat{\BFx}^r,\hat{\BFy}^r)$ satisfies \eqref{eqn-T:kth-ramp-3-1} at equality, we first consider the case where $t-r-1\notin\Se$.
    In this case, $\hat{x}^r_q=\hat{y}^r_q=0$ for all $q\in[t,t+m+1]_\Z$.
    Because $t-k\le t-s_{\max}-1\le r$, we have $\hat{x}^r_{t-k}=\Clower$, and $\hat{y}^r_{t-k}=1$.
    Because $t-s-1\ne r$ for all $s\in\Se$, we have $\hat{y}^r_{t-s}-\hat{y}^r_{t-s-1}=0$ for all $s\in\Se$.
    Hence, $(\hat{\BFx}^r,\hat{\BFy}^r)$ satisfies \eqref{eqn-T:kth-ramp-3-1} at equality.
    Next, we consider the case where $t-r-1\in\Se$.
    In this case, $\hat{x}^r_t=\Vupper+(t-r-1)V$ and $\hat{y}^r_q=1$ for all $q\in[t,t+m+1]_{\Z}$.
    Because $t-k\le t-s_{\max}-1\le r$, we have $\hat{x}^r_{t-k}=\hat{y}^r_{t-k}=0$.
    In addition, $\hat{y}^r_{t-s}-\hat{y}^r_{t-s-1}=1$ when $s=t-r-1$, and $\hat{y}^r_{t-s}-\hat{y}^r_{t-s-1}=0$ when $s\neq t-r-1$.
    Hence, $(\hat{\BFx}^r,\hat{\BFy}^r)$ satisfies \eqref{eqn-T:kth-ramp-3-1} at equality.
    
    \item\label{points:prop16-eq1-A5} For each $r\in[t,t+m]_{\Z}$, we create a point $(\hat{\BFx}^r,\hat{\BFy}^r)$ as follows:
    \begin{equation*}
        \hat{x}^r_q=\left\{
        \begin{array}{ll}
            \Clower, &\hbox{for $q\in[1,2t-r-1]_{\Z}$};\\
            \Vupper+(q+r-2t)V, &\hbox{for $q\in[2t-r,t-1]_{\Z}$};\\
            \Vupper+(r-q)V, &\hbox{for $q\in[t,r]_{\Z}$};\\
            0, &\hbox{for $q\in[r+1,T]_{\Z}$};
        \end{array}\right.
    \end{equation*}
    and
    \begin{equation*}
        \hat{y}^r_q=\left\{
        \begin{array}{ll}
            1, &\hbox{for $q\in[1,r]_{\Z}$};\\
            0, &\hbox{for $q\in[r+1,T]_{\Z}$}.
        \end{array}\right.
    \end{equation*}
    It is easy to verify that $(\hat{\BFx}^r,\hat{\BFy}^r)$ satisfies \eqref{eqn:p-minup}--\eqref{eqn:p-upper-bound}.
    Note that $-V\le\hat{x}^r_q-\hat{x}^r_{q-1}\le V$ when $q\in[2,r]_\Z$, $\hat{x}^r_q-\hat{x}^r_{q-1}=-\Vupper$ when $q=r+1$, and $\hat{x}^r_q-\hat{x}^r_{q-1}=0$ when $q\in[r+2,T]_\Z$.
    Thus, $-V\hat{y}^r_q-\Vupper(1-\hat{y}^r_q)\le\hat{x}^r_q-\hat{x}^r_{q-1}\le V\hat{y}^r_{q-1}+\Vupper(1-\hat{y}^r_{q-1})$ for all $q\in[2,T]_{\Z}$.
    Hence, $(\hat{\BFx}^r,\hat{\BFy}^r)$ satisfies \eqref{eqn:p-ramp-up} \rred{and} \eqref{eqn:p-ramp-down}.
    Therefore, $(\hat{\BFx}^r,\hat{\BFy}^r)\in\convP$.
    Note that $\hat{x}^r_t=\Vupper+(r-t)V$ and $\hat{y}^r_t=1$.
    Note also that $\hat{y}^r_{t+m+1}=0$ and $V\sum_{i=1}^m y^r_{t+i}=(r-t)V$.
    Because $t-k\le t-m-1\le 2t-r-1$, we have $\hat{x}^r_{t-k}=\Clower$ and $\hat{y}^r_{t-k}=1$.
    For any $s\in\Se$, because $t-s\le t\le r$, we have $\hat{y}^r_{t-s}-\hat{y}^r_{t-s-1}=0$.
    Thus, $(\hat{\BFx}^r,\hat{\BFy}^r)$ satisfies \eqref{eqn-T:kth-ramp-3-1} at equality.
    
    \item\label{points:prop16-eq1-A6} We create a point $(\hat{\BFx}^{t+m+1},\hat{\BFy}^{t+m+1})$ as follows:
    \begin{equation*}
        \hat{\BFx}^{t+m+1}_q=\left\{
        \begin{array}{ll}
        \Clower, &\hbox{for $q\in[1,t-k-1]_\Z$};\\
        \Clower+(q-t+k)V, &\hbox{for $q\in[t-k,t]_\Z$};\\
        \Clower+kV, &\hbox{for $q\in[t+1,T]_\Z$};
        \end{array}\right.
    \end{equation*}
    and $\hat{y}^{t+m+1}_q=1$ for all $q\in[1,T]_\Z$.
    It is easy to verify that $(\hat{\BFx}^{t+m+1},\hat{\BFy}^{t+m+1})$ satisfies \eqref{eqn:p-minup}--\eqref{eqn:p-upper-bound}.
    Note that 
    $\hat{x}^{t+m+1}_q-\hat{x}^{t+m+1}_{q-1}=0$ when $q\in[2,t-k]_\Z$, 
    $\hat{x}^{t+m+1}_q-\hat{x}^{t+m+1}_{q-1}=V$ when $q\in[t-k+1,t]_\Z$, and 
    $\hat{x}^{t+m+1}_q-\hat{x}^{t+m+1}_{q-1}=0$ when $q\in[t+1,T]_\Z$.
    Thus, $-V\hat{y}^{t+m+1}_q-\Vupper(1-\hat{y}^{t+m+1}_q)\le\hat{x}^{t+m+1}_q-\hat{x}^{t+m+1}_{q-1}\le V\hat{y}^{t+m+1}_{q-1}+\Vupper(1-\hat{y}^{t+m+1}_{q-1})$ for all $q\in[2,T]_\Z$.
    Hence, $(\hat{\BFx}^{t+m+1},\hat{\BFy}^{t+m+1})$ satisfies \eqref{eqn:p-ramp-up} and \eqref{eqn:p-ramp-down}.
    Therefore, $(\hat{\BFx}^{t+m+1},\hat{\BFy}^{t+m+1})\in\convP$.
    Note that $\hat{x}^{t+m+1}_t=\Clower+kV$, $\hat{x}^{t+m+1}_{t-k}=\Clower$, $\hat{y}^{t+m+1}_{t+m+1}=\hat{y}^{t+m+1}_t=\hat{y}^{t+m+1}_{t-k}=1$, $\sum_{i=1}^m y^{t+m+1}_{t+i}=m$,
    and $\hat{y}^{t+m+1}_{t-s}-\hat{y}^{t+m+1}_{t-s-1}=0$ for all $s\in\Se$, 
    Thus, $(\hat{\BFx}^{t+m+1},\hat{\BFy}^{t+m+1})$ satisfies \eqref{eqn-T:kth-ramp-3-1} at equality.
    
    \item\label{points:prop16-eq1-A7} For each $r\in[t+m+2,T]_{\Z}$, we create a point $(\hat{\BFx}^r,\hat{\BFy}^r)$ as follows:
    \begin{equation*}
        \hat{x}^r_q=\left\{
        \begin{array}{ll}
        0, &\hbox{for $q\in[1,r-1]_{\Z}$};\\
        \Clower, &\hbox{for $q\in[r,T]_{\Z}$};
        \end{array}\right.
    \end{equation*}
    and
    \begin{equation*}
        \hat{y}^r_q=\left\{
        \begin{array}{ll}
        0, &\hbox{for $q\in[1,r-1]_{\Z}$};\\
        1, &\hbox{for $q\in[r,T]_{\Z}$}.
        \end{array}\right.
    \end{equation*}
    It is easy to verify that $(\hat{\BFx}^r,\hat{\BFy}^r)$ satisfies \eqref{eqn:p-minup}--\eqref{eqn:p-ramp-down}.
    Thus, $(\hat{\BFx}^r,\hat{\BFy}^r)\in\convP$.
    Note that $\hat{x}^r_t=\hat{x}^r_{t-k}=0$, $\hat{y}^r_t=\hat{y}^r_{t-k}=\hat{y}^r_{t+m+1}=0$, $\sum_{i=1}^m y^r_{t+i}=0$,
    and $\hat{y}^r_{t-s}-\hat{y}^r_{t-s-1}=0$ for all $s\in\Se$, 
    Hence, $(\hat{\BFx}^r,\hat{\BFy}^r)$ satisfies \eqref{eqn-T:kth-ramp-3-1} at equality.
\end{enumerate}

Table \ref{tab:eqn-T:kth-ramp-3-1-facet-matrix-1} shows a matrix with $2T-1$ rows, where each row represents a point created by this process.
This matrix can be transformed into the matrix in Table \ref{tab:eqn-T:kth-ramp-3-1-facet-matrix-2} via the following Gaussian elimination process:

\begin{enumerate}[label=(\roman*)]
    \item For each $r\in[1,t-1]_{\Z}\setminus\{t-k\}$, the point with index $r$ in group (B1), denoted $(\underline{\bar{\BFx}}^r,\underline{\bar{\BFy}}^r)$, 
    is obtained by setting $(\underline{\bar{\BFx}}^r,\underline{\bar{\BFy}}^r)=(\bar{\BFx}^r,\bar{\BFy}^r)-(\hat{\BFx}^t,\hat{\BFy}^t)$.
    Here, $(\bar{\BFx}^r,\bar{\BFy}^r)$ is the point with index $r$ in group \ref{points:prop16-eq1-A1}, and $(\hat{\BFx}^t,\hat{\BFy}^t)$ is the point with index $t$ in group \ref{points:prop16-eq1-A5}.
    
    \item The point in group (B2), denoted $(\underline{\bar{\BFx}}^t,\underline{\bar{\BFy}}^t)$, 
    is obtained by setting $(\underline{\bar{\BFx}}^t,\underline{\bar{\BFy}}^t)=(\bar{\BFx}^t,\bar{\BFy}^t)-(\hat{\BFx}^{t+m+1},\hat{\BFy}^{t+m+1})$.
    Here, $(\bar{\BFx}^t,\bar{\BFy}^t)$ is the point in group \ref{points:prop16-eq1-A2}, and $(\hat{\BFx}^{t+m+1},\hat{\BFy}^{t+m+1})$ is the point in group \ref{points:prop16-eq1-A6}.

    \item For each $r\in[t+1,T]_{\Z}$, the point with index $r$ in group (B3), denoted $(\underline{\bar{\BFx}}^r,\underline{\bar{\BFy}}^r)$, 
    is obtained by setting $(\underline{\bar{\BFx}}^r,\underline{\bar{\BFy}}^r)=(\bar{\BFx}^r,\bar{\BFy}^r)-(\hat{\BFx}^{t+m+1},\hat{\BFy}^{t+m+1})$.
    Here, $(\bar{\BFx}^r,\bar{\BFy}^r)$ is the point with index $r$ in group \ref{points:prop16-eq1-A3}, and $(\hat{\BFx}^{t+m+1},\hat{\BFy}^{t+m+1})$ is the point in group \ref{points:prop16-eq1-A6}.
    
    \item For each \rred{$r\in[1,t-1]_{\Z}$}, the point with index $r$ in group \rred{(B4)}, denoted $(\underline{\hat{\BFx}}^r,\underline{\hat{\BFy}}^r)$, 
    is obtained by setting $(\underline{\hat{\BFx}}^r,\underline{\hat{\BFy}}^r)=(\hat{\BFx}^r,\hat{\BFy}^r)$ if $t-r-1\notin\Se$, 
    and setting $(\underline{\hat{\BFx}}^r,\underline{\hat{\BFy}}^r)=\rred{(\hat{\BFx}^{t+m+1},\hat{\BFy}^{t+m+1})-(\hat{\BFx}^r,\hat{\BFy}^r)}$ if $t-r-1\in\Se$.
    Here, $(\hat{\BFx}^r,\hat{\BFy}^r)$ is the point with index $r$ in group \ref{points:prop16-eq1-A4}, and $(\hat{\BFx}^{t+m+1},\hat{\BFy}^{t+m+1})$ is the point in group \ref{points:prop16-eq1-A6}.
    
    \item For each $r\in[t,t+m]_{\Z}$, the point with index $r$ in group \rred{(B5)}, denoted
    $(\underline{\hat{\BFx}}^r,\underline{\hat{\BFy}}^r)$, is obtained by setting $(\underline{\hat{\BFx}}^r,\underline{\hat{\BFy}}^r)=(\hat{\BFx}^r,\hat{\BFy}^r)$.
    Here, $(\hat{\BFx}^r,\hat{\BFy}^r)$ is the point with index $r$ in group \ref{points:prop16-eq1-A5}.
    
    \item The point in group \rred{(B6)}, denoted $(\underline{\hat{\BFx}}^{t+m+1},\underline{\hat{\BFy}}^{t+m+1})$, 
    is obtained by setting $(\underline{\hat{\BFx}}^{t+m+1},\underline{\hat{\BFy}}^{t+m+1})=(\hat{\BFx}^{t+m+1},\hat{\BFy}^{t+m+1})-(\hat{\BFx}^{t+m+2},\hat{\BFy}^{t+m+2})$.
    Here, $(\hat{\BFx}^{t+m+1},\hat{\BFy}^{t+m+1})$ is the point in group \ref{points:prop16-eq1-A6}, and $(\hat{\BFx}^{t+m+2},\hat{\BFy}^{t+m+2})$ is the point with index $t+m+2$ in group \ref{points:prop16-eq1-A7}.
    
    \item For each $r\in[t+m+2,T]_{\Z}$, the point with index $r$ in group \rred{(B7)}, denoted $(\underline{\hat{\BFx}}^r,\underline{\hat{\BFy}}^r)$, 
    is obtained by setting $(\underline{\hat{\BFx}}^r,\underline{\hat{\BFy}}^r)=(\hat{\BFx}^r,\hat{\BFy}^r)-(\hat{\BFx}^{r+1},\hat{\BFy}^{r+1})$ if $r\neq T$, 
    and setting $(\underline{\hat{\BFx}}^r,\underline{\hat{\BFy}}^r)=(\hat{\BFx}^r,\hat{\BFy}^r)$ if $r=T$.
    Here, $(\hat{\BFx}^r,\hat{\BFy}^r)$ and $(\hat{\BFx}^{r+1},\hat{\BFy}^{r+1})$ are the points with indices $r$ and $r+1$, respectively, in group \ref{points:prop16-eq1-A7}.
\end{enumerate}

\afterpage{
\begin{landscape}
\begin{table}
    \renewcommand{\arraystretch}{1.75}
    \centering
    \caption{A matrix with the rows representing $2T-1$ linearly independent points in $\convP$ satisfying inequality \eqref{eqn-T:kth-ramp-3-1} at equality}
    \rule{0pt}{0ex}
    \setlength\tabcolsep{2.5pt}
    \tiny
        \begin{tabular}{|c|c|c|*{15}{c}|*{11}{c}|}
        \hline
        \multirow{2}{*}{Group} & \multirow{2}{*}{Point} & \multirow{2}{*}{Index $r$} & \multicolumn{15}{c|}{$\BFx$} & \multicolumn{11}{c|}{$\BFy$} \\
        \cline{4-18}\cline{19-29}
        &&& $1$ & $\cdots$ & $t\!-\!k\!-\!1$ & $t\!-\!k$ & $t\!-\!k\!+\!1$ & $\cdots$ & $t\!-\!1$ & $t$ & $t\!+\!1$ & $\cdots$ & $t\!+\!m$ & $t\!+\!m\!+\!1$ & $t\!+\!m\!+\!2$ & $\cdots$ & $T$ 
        & $\ 1\ $ & $\cdots$ & $t\!-\!1$ & $t$ & $t\!+\!1$ & $\cdots$ & $t\!+\!m$ & $t\!+\!m\!+\!1$ & $t\!+\!m\!+\!2$ & $\cdots$ & $\ T\ $\\
         
        \hline
        \multirow{6}{*}{(A1)} & \multirow{14}{*}{$(\bar{\BFx}^r,\bar{\BFy}^r)$}
        & $1$ 
        & $\Clower\!+\!\epsilon$ & $\cdots$ & $\Clower$ & $\Clower$ & $\Clower$ & $\cdots$ & $\Clower$ & $\Vupper$ & $0$ & $\cdots$ & $0$ & $0$ & $0$ & $\cdots$ & $0$
        & $1$ & $\cdots$ & $1$ & $1$ & $0$ & $\cdots$ & $0$ & $0$ & $0$ & $\cdots$ & $0$\\
                  
        && $\vdots$
        & $\vdots$ & \rotatebox{0}{$\ddots$} & $\vdots$ & $\vdots$ & $\vdots$ && $\vdots$ & $\vdots$ & $\vdots$ && $\vdots$ & $\vdots$ & $\vdots$ && $\vdots$
        & $\vdots$ && $\vdots$ & $\vdots$ & $\vdots$ && $\vdots$ & $\vdots$ & $\vdots$ && $\vdots$\\
         
        && $t\!-\!k\!-\!1$ 
        & $\Clower$ & $\cdots$ & $\Clower\!+\!\epsilon$ & $\Clower$ & $\Clower$ & $\cdots$ & $\Clower$ & $\Vupper$ & $0$ & $\cdots$ & $0$ & $0$ & $0$ & $\cdots$ & $0$
        & $1$ & $\cdots$ & $1$ & $1$ & $0$ & $\cdots$ & $0$ & $0$ & $0$ & $\cdots$ & $0$\\
         
        && $t\!-\!k\!+\!1$ 
        & $\Clower$ & $\cdots$ & $\Clower$ & $\Clower$ & $\Clower\!+\!\epsilon$ & $\cdots$ & $\Clower$ & $\Vupper$ & $0$ & $\cdots$ & $0$ & $0$ & $0$ & $\cdots$ & $0$
        & $1$ & $\cdots$ & $1$ & $1$ & $0$ & $\cdots$ & $0$ & $0$ & $0$ & $\cdots$ & $0$\\
         
        && $\vdots$
        & $\vdots$ && $\vdots$ & $\vdots$ & $\vdots$ & \rotatebox{0}{$\ddots$} & $\vdots$ & $\vdots$ & $\vdots$ && $\vdots$ & $\vdots$ & $\vdots$ && $\vdots$
        & $\vdots$ && $\vdots$ & $\vdots$ & $\vdots$ && $\vdots$ & $\vdots$ & $\vdots$ && $\vdots$\\
         
        && $t\!-\!1$ 
        & $\Clower$ & $\cdots$ & $\Clower$ & $\Clower$ & $\Clower$ & $\cdots$ & $\Clower\!+\!\epsilon$ & $\Vupper$ & $0$ & $\cdots$ & $0$ & $0$ & $0$ & $\cdots$ & $0$
        & $1$ & $\cdots$ & $1$ & $1$ & $0$ & $\cdots$ & $0$ & $0$ & $0$ & $\cdots$ & $0$\\
         
        \cline{1-1} \cline{3-29}
        (A2) &
        &$t$
        & $\Clower$ & $\cdots$ & $\Clower$ & $\Clower\!+\!\epsilon$ & $\Clower\!+\!V\!+\!\epsilon$ & $\cdots$ & 
          $\Clower\!+\!(k\!-\!1)V\!+\!\epsilon$ & $\Clower\!+\!kV\!+\!\epsilon$ & $\Clower\!+\!kV$ & $\cdots$ & $\Clower\!+\!kV$ & $\Clower\!+\!kV$ & $\Clower\!+\!kV$ & $\cdots$ & $\Clower\!+\!kV$
        & $1$ & $\cdots$ & $1$ & $1$ & $1$ & $\cdots$ & $1$ & $1$ & $1$ & $\cdots$ & $1$\\
         
        \cline{1-1} \cline{3-29}
        \multirow{7}{*}{(A3)} &
        &$t\!+\!1$
        & $\Clower$ & $\cdots$ & $\Clower$ & $\Clower$ & $\Clower\!+\!V$ & $\cdots$ & $\Clower\!+\!(k\!-\!1)V$ & $\Clower\!+\!kV$ & $\Clower\!+\!kV\!+\!\epsilon$ & $\cdots$ & $\Clower\!+\!kV$ &
          $\Clower\!+\!kV$ & $\Clower\!+\!kV$ & $\cdots$ & $\Clower\!+\!kV$
        & $1$ & $\cdots$ & $1$ & $1$ & $1$ & $\cdots$ & $1$ & $1$ & $1$ & $\cdots$ & $1$\\
         
        && $\vdots$
        & $\vdots$ && $\vdots$ & $\vdots$ & $\vdots$ && $\vdots$ & $\vdots$ & $\vdots$ & \rotatebox{0}{$\ddots$} & $\vdots$ & $\vdots$ & $\vdots$ && $\vdots$
        & $\vdots$ && $\vdots$ & $\vdots$ & $\vdots$ && $\vdots$ & $\vdots$ & $\vdots$ && $\vdots$\\
         
        && $t\!+\!m$
        & $\Clower$ & $\cdots$ & $\Clower$ & $\Clower$ & $\Clower\!+\!V$ & $\cdots$ & $\Clower\!+\!(k\!-\!1)V$ & $\Clower\!+\!kV$ & $\Clower\!+\!kV$ & $\cdots$ & $\Clower\!+\!kV\!+\!\epsilon$ &
          $\Clower\!+\!kV$ & $\Clower\!+\!kV$ & $\cdots$ & $\Clower\!+\!kV$
        & $1$ & $\cdots$ & $1$ & $1$ & $1$ & $\cdots$ & $1$ & $1$ & $1$ & $\cdots$ & $1$\\
         
        && $t\!+\!m\!+\!1$
        & $\Clower$ & $\cdots$ & $\Clower$ & $\Clower$ & $\Clower\!+\!V$ & $\cdots$ & $\Clower\!+\!(k\!-\!1)V$ & $\Clower\!+\!kV$ & $\Clower\!+\!kV$ & $\cdots$ & $\Clower\!+\!kV$ &
          $\Clower\!+\!kV\!+\!\epsilon$ & $\Clower\!+\!kV$ & $\cdots$ & $\Clower\!+\!kV$
        & $1$ & $\cdots$ & $1$ & $1$ & $1$ & $\cdots$ & $1$ & $1$ & $1$ & $\cdots$ & $1$\\
         
        && $t\!+\!m\!+\!2$
        & $\Clower$ & $\cdots$ & $\Clower$ & $\Clower$ & $\Clower\!+\!V$ & $\cdots$ & $\Clower\!+\!(k\!-\!1)V$ & $\Clower\!+\!kV$ & $\Clower\!+\!kV$ & $\cdots$ & $\Clower\!+\!kV$ & 
          $\Clower\!+\!kV$ & $\Clower\!+\!kV\!+\!\epsilon$ & $\cdots$ & $\Clower\!+\!kV$
        & $1$ & $\cdots$ & $1$ & $1$ & $1$ & $\cdots$ & $1$ & $1$ & $1$ & $\cdots$ & $1$\\
         
        && $\vdots$
        & $\vdots$ && $\vdots$ & $\vdots$ & $\vdots$ && $\vdots$ & $\vdots$ & $\vdots$ && $\vdots$ & $\vdots$ & $\vdots$ & \rotatebox{0}{$\ddots$} & $\vdots$
        & $\vdots$ && $\vdots$ & $\vdots$ & $\vdots$ && $\vdots$ & $\vdots$ & $\vdots$ && $\vdots$\\
         
         && $T$
         & $\Clower$ & $\cdots$ & $\Clower$ & $\Clower$ & $\Clower\!+\!V$ & $\cdots$ & $\Clower\!+\!(k\!-\!1)V$ & $\Clower\!+\!kV$ & $\Clower\!+\!kV$ & $\cdots$ & $\Clower\!+\!kV$ &
           $\Clower\!+\!kV$ & $\Clower\!+\!kV$ & $\cdots$ & $\Clower\!+\!kV\!+\!\epsilon$
         & $1$ & $\cdots$ & $1$ & $1$ & $1$ & $\cdots$ & $1$ & $1$ & $1$ & $\cdots$ & $1$\\
         
        \hline
        \multirow{3}{*}{\rred{(A4)}} & \multirow{11}{*}{$(\hat{\BFx}^r,\hat{\BFy}^r)$}
        & $1$ &\multicolumn{15}{c|}{\multirow{3}{*}{(See Note \rred{\ref{tab:eqn-T:kth-ramp-3-1-facet-matrix-1}-1})}} & \multicolumn{11}{c|}{\multirow{3}{*}{(See Note \rred{\ref{tab:eqn-T:kth-ramp-3-1-facet-matrix-1}-1})}}\\    
        && $\vdots$ &\multicolumn{15}{c|}{} & \multicolumn{11}{c|}{}\\
        && $t\!-\!1$ &\multicolumn{15}{c|}{} & \multicolumn{11}{c|}{}\\
         
        \cline{1-1} \cline{3-29}
        \multirow{4}{*}{\rred{(A5)}} &
        &$t$
        & $\Clower$ & $\cdots$ & $\Clower$ & $\Clower$ & $\Clower$ & $\cdots$ & $\Clower$ & $\Vupper$ & $0$ & $\cdots$ & $0$ & $0$ & $0$ & $\cdots$ & $0$
        & $1$ & $\cdots$ & $1$ & $1$ & $0$ & $\cdots$ & $0$ & $0$ & $0$ & $\cdots$ & $0$\\
         
        && $t\!+\!1$
        & $\Clower$ & $\cdots$ & $\Clower$ & $\Clower$ & $\Clower$ & $\cdots$ & $\Vupper$ & $\Vupper\!+\!V$ & $\Vupper$ & $\cdots$ & $0$ & $0$ & $0$ & $\cdots$ & $0$
        & $1$ & $\cdots$ & $1$ & $1$ & $1$ & $\cdots$ & $0$ & $0$ & $0$ & $\cdots$ & $0$\\
         
        && $\vdots$
        & $\vdots$ && $\vdots$ & $\vdots$ & $\vdots$ && $\vdots$ & $\vdots$ & $\vdots$ & \rotatebox{0}{$\ddots$} & $\vdots$ & $\vdots$ & $\vdots$ && $\vdots$
        & $\vdots$ && $\vdots$ & $\vdots$ & $\vdots$ & \rotatebox{0}{$\ddots$} & $\vdots$ & $\vdots$ & $\vdots$ && $\vdots$\\
        
        &&$t\!+\!m$
        & $\Clower$ & $\cdots$ & $\Clower$ & $\Clower$ & $\substack{\mbox{(See Note}\\\mbox{\rred{\ref{tab:eqn-T:kth-ramp-3-1-facet-matrix-1}-2})}}$ & $\cdots$ & $\Vupper\!+\!(m\!-\!1)V$ & $\Vupper+mV$ & $\Vupper\!+\!(m\!-\!1)V$ & $\cdots$ & $\Vupper$ & $0$ & $0$ & $\cdots$ & $0$
        & $1$ & $\cdots$ & $1$ & $1$ & $1$ & $\cdots$ & $1$ & $0$ & $0$ & $\cdots$ & $0$\\
        
        \cline{1-1} \cline{3-29}
        \rred{(A6)} & 
        &$t\!+\!m\!+\!1$
        & $\Clower$ & $\cdots$ & $\Clower$ & $\Clower$ & $\Clower\!+\!V$ & $\cdots$ & $\Clower\!+\!(k\!-\!1)V$ & $\Clower\!+\!kV$ & $\Clower\!+\!kV$ & $\cdots$ & $\Clower\!+\!kV$ &
          $\Clower\!+\!kV$ & $\Clower\!+\!kV$ & $\cdots$ & $\Clower\!+\!kV$
        & $1$ & $\cdots$ & $1$ & $1$ & $1$ & $\cdots$ & $1$ & $1$ & $1$ & $\cdots$ & $1$\\
         
        \cline{1-1} \cline{3-29}
        \multirow{3}{*}{\rred{(A7)}} & 
        & $t\!+\!m\!+\!2$
        & $0$ & $\cdots$ & $0$ & $0$ & $0$ & $\cdots$ & $0$ & $0$ & $0$ & $\cdots$ & $0$ & $0$ & $\Clower$ & $\cdots$ & $\Clower$
        & $0$ & $\cdots$ & $0$ & $0$ & $0$ & $\cdots$ & $0$ & $0$ & $1$ & $\cdots$ & $1$\\
        
        && $\vdots$
        & $\vdots$ && $\vdots$ & $\vdots$ & $\vdots$ && $\vdots$ & $\vdots$ & $\vdots$ && $\vdots$ & $\vdots$ & $\vdots$ & \rotatebox{0}{$\ddots$} & $\vdots$
        & $\vdots$ && $\vdots$ & $\vdots$ & $\vdots$ && $\vdots$ & $\vdots$ & $\vdots$ & \rotatebox{0}{$\ddots$} & $\vdots$\\
        
        && $T$
        & $0$ & $\cdots$ & $0$ & $0$ & $0$ & $\cdots$ & $0$ & $0$ & $0$ & $\cdots$ & $0$ & $0$ & $0$ & $\cdots$ & $\Clower$
        & $0$ & $\cdots$ & $0$ & $0$ & $0$ & $\cdots$ & $0$ & $0$ & $0$ & $\cdots$ & $1$\\
        \hline
        
        \multicolumn{29}{l}{Note \rred{\ref{tab:eqn-T:kth-ramp-3-1-facet-matrix-1}-1}: For $r\in\rred{[1,t-1]_\Z}$, the $\BFx$ and $\BFy$ vectors in group \rred{(A4)} are given as follows:}\\
        \multicolumn{29}{l}{
        $\hat{\BFx}^r=(\underbrace{\Clower,\ldots,\Clower}_{r\ {\rm terms}},\underbrace{0,\ldots,0}_{T-r\ {\rm terms}\!\!\!\!\!})$
        and
        $\hat{\BFy}^r=(\underbrace{1,\ldots,1}_{r\ {\rm terms}},\underbrace{0,\ldots,0}_{T-r\ {\rm terms}\!\!\!\!\!})$
        if $t-r-1\notin\Se$;}\\
        \multicolumn{29}{l}{
        $\hat{\BFx}^r=(\underbrace{0,\ldots,0}_{r\ {\rm terms}},\underbrace{\Vupper,\Vupper+V,\Vupper+2V,\ldots,\Vupper(r-t-1)V}_{t-r\ {\rm terms}},\underbrace{\Vupper+(t-r-1)V,\ldots,\Vupper+(t-r-1)V}_{T-t\ {\rm terms}\!\!\!\!\!})$
        and
        $\hat{\BFy}^r=(\underbrace{0,\ldots,0}_{r\ {\rm terms}},\underbrace{1,\ldots,1}_{T-r\ {\rm terms}\!\!\!\!\!})$
        if $t-r-1\in\Se$.}\\
        \multicolumn{29}{l}{Note \rred{\ref{tab:eqn-T:kth-ramp-3-1-facet-matrix-1}-2}: In group \rred{(A5)}, $\hat{x}^{t+m}_{t-k+1}=\Clower$ if $m<k-1$, and $\hat{x}^{t+m}_{t-k+1}=\Vupper$ if $m=k-1$.}
    \end{tabular}
    \label{tab:eqn-T:kth-ramp-3-1-facet-matrix-1}
\end{table}
\end{landscape}

\begin{landscape}
\begin{table}
    \renewcommand{\arraystretch}{2}
    \centering
    \caption{Lower triangular matrix obtained from Table \ref{tab:eqn-T:kth-ramp-3-1-facet-matrix-1} via Gaussian elimination}
    \rule{0pt}{0ex}
    \setlength\tabcolsep{3pt}
    \sscriptsize
        \begin{tabular}{|c|c|c|*{15}{c}|*{11}{c}|}
        \hline
        \multirow{2}{*}{Group} & \multirow{2}{*}{Point} & \multirow{2}{*}{Index $r$} & \multicolumn{15}{c|}{$\BFx$} & \multicolumn{11}{c|}{$\BFy$} \\
        \cline{4-18}\cline{19-29}
        &&& $\ \ 1\ \ $ & $\cdots$ & $t\!-\!k\!-\!1$ & $t\!-\!k$ & $t\!-\!k\!+\!1$ & $\cdots$ & $t\!-\!1$ & $t$ & $t\!+\!1$ & $\cdots$ & $t\!+\!m$ & $t\!+\!m\!+\!1$ & $t\!+\!m\!+\!2$ & $\cdots$ & $\ \ T\ \ $ 
        & $\ \ 1\ \ $ & $\cdots$ & $t\!-\!1$ & $t$ & $t\!+\!1$ & $\cdots$ & $t\!+\!m$ & $t\!+\!m\!+\!1$ & $t\!+\!m\!+\!2$ & $\cdots$ & $\ \ T\ \ $\\
         
        \hline
        \multirow{6}{*}{(B1)} & \multirow{14}{*}{$(\underline{\bar{\BFx}}^r,\underline{\bar{\BFy}}^r)$}
        & $1$ 
        & $\epsilon$ & $0$ & $\cdots$ & $0$ & $0$ & $\cdots$ & $0$ & $0$ & $0$ & $\cdots$ & $0$ & $0$ & $0$ & $\cdots$ & $0$
        & $0$ & $\cdots$ & $0$ & $0$ & $0$ & $\cdots$ & $0$ & $0$ & $0$ & $\cdots$ & $0$\\
         
        && $\vdots$
        & $\vdots$ & \rotatebox{0}{$\ddots$} & $\vdots$ & $\vdots$ & $\vdots$ && $\vdots$ & $\vdots$ & $\vdots$ && $\vdots$ & $\vdots$ & $\vdots$ && $\vdots$
        & $\vdots$ && $\vdots$ & $\vdots$ & $\vdots$ && $\vdots$ & $\vdots$ & $\vdots$ && $\vdots$\\
         
        && $t\!-\!k\!-\!1$ 
        & $0$ & $\cdots$ & $\epsilon$ & $0$ & $0$ & $\cdots$ & $0$ & $0$ & $0$ & $\cdots$ & $0$ & $0$ & $0$ & $\cdots$ & $0$
        & $0$ & $\cdots$ & $0$ & $0$ & $0$ & $\cdots$ & $0$ & $0$ & $0$ & $\cdots$ & $0$\\
         
        && $t\!-\!k\!+\!1$ 
        & $0$ & $\cdots$ & $0$ & $0$ & $\epsilon$ & $\cdots$ & $0$ & $0$ & $0$ & $\cdots$ & $0$ & $0$ & $0$ & $\cdots$ & $0$
        & $0$ & $\cdots$ & $0$ & $0$ & $0$ & $\cdots$ & $0$ & $0$ & $0$ & $\cdots$ & $0$\\
         
        && $\vdots$
        & $\vdots$ & $\vdots$ && $\vdots$ & $\vdots$ & \rotatebox{0}{$\ddots$} & $\vdots$ & $\vdots$ & $\vdots$ && $\vdots$ & $\vdots$ & $\vdots$ && $\vdots$
        & $\vdots$ && $\vdots$ & $\vdots$ & $\vdots$ && $\vdots$ & $\vdots$ & $\vdots$ && $\vdots$\\
         
        && $t\!-\!1$ 
        & $0$ & $\cdots$ & $0$ & $0$ & $0$ & $\cdots$ & $\epsilon$ & $0$ & $0$ & $\cdots$ & $0$ & $0$ & $0$ & $\cdots$ & $0$
        & $0$ & $\cdots$ & $0$ & $0$ & $0$ & $\cdots$ & $0$ & $0$ & $0$ & $\cdots$ & $0$\\
         
        \cline{1-1} \cline{3-29}
        (B2) &
        &$t$
        & $0$ & $0$ & $\cdots$ & $\epsilon$ & $\epsilon$ & $\cdots$ & $\epsilon$ & $\epsilon$ & $0$ & $\cdots$ & $0$ & $0$ & $0$ & $\cdots$ & $0$
        & $0$ & $\cdots$ & $0$ & $0$ & $0$ & $\cdots$ & $0$ & $0$ & $0$ & $\cdots$ & $0$\\
         
        \cline{1-1} \cline{3-29}
        \multirow{7}{*}{(B3)} &
        &$t\!+\!1$
        & $0$ & $0$ & $\cdots$ & $0$ & $0$ & $\cdots$ & $0$ & $0$ & $\epsilon$ & $\cdots$ & $0$ & $0$ & $0$ & $\cdots$ & $0$
        & $0$ & $\cdots$ & $0$ & $0$ & $0$ & $\cdots$ & $0$ & $0$ & $0$ & $\cdots$ & $0$\\
         
        && $\vdots$
        & $\vdots$ & $\vdots$ && $\vdots$ & $\vdots$ && $\vdots$ & $\vdots$ & $\vdots$ & \rotatebox{0}{$\ddots$} & $\vdots$ & $\vdots$ & $\vdots$ && $\vdots$
        & $\vdots$ && $\vdots$ & $\vdots$ & $\vdots$ && $\vdots$ & $\vdots$ & $\vdots$ && $\vdots$\\
         
        && $t\!+\!m$
        & $0$ & $0$ & $\cdots$ & $0$ & $0$ & $\cdots$ & $0$ & $0$ & $0$ & $\cdots$ & $\epsilon$ & $0$ & $0$ & $\cdots$ & $0$
        & $0$ & $\cdots$ & $0$ & $0$ & $0$ & $\cdots$ & $0$ & $0$ & $0$ & $\cdots$ & $0$\\
         
        && $t\!+\!m\!+\!1$
        & $0$ & $0$ & $\cdots$ & $0$ & $0$ & $\cdots$ & $0$ & $0$ & $0$ & $\cdots$ & $0$ & $\epsilon$ & $0$ & $\cdots$ & $0$
        & $0$ & $\cdots$ & $0$ & $0$ & $0$ & $\cdots$ & $0$ & $0$ & $0$ & $\cdots$ & $0$\\
         
        && $t\!+\!m\!+\!2$
        & $0$ & $0$ & $\cdots$ & $0$ & $0$ & $\cdots$ & $0$ & $0$ & $0$ & $\cdots$ & $0$ & $0$ & $\epsilon$ & $\cdots$ & $0$
        & $0$ & $\cdots$ & $0$ & $0$ & $0$ & $\cdots$ & $0$ & $0$ & $0$ & $\cdots$ & $0$\\
         
        && $\vdots$
        & $\vdots$ & $\vdots$ && $\vdots$ & $\vdots$ && $\vdots$ & $\vdots$ & $\vdots$ && $\vdots$ & $\vdots$ & $\vdots$ & \rotatebox{0}{$\ddots$} & $\vdots$
        & $\vdots$ && $\vdots$ & $\vdots$ & $\vdots$ && $\vdots$ & $\vdots$ & $\vdots$ && $\vdots$\\
         
         && $T$
         & $0$ & $0$ & $\cdots$ & $0$ & $0$ & $\cdots$ & $0$ & $0$ & $0$ & $\cdots$ & $0$ & $0$ & $0$ & $\cdots$ & $\epsilon$
         & $0$ & $\cdots$ & $0$ & $0$ & $0$ & $\cdots$ & $0$ & $0$ & $0$ & $\cdots$ & $0$\\
         
        \hline
        \multirow{3}{*}{\rred{(B4)}} & \multirow{11}{*}{$(\underline{\hat{\BFx}}^r,\underline{\hat{\BFy}}^r)$}
        & $1$
        & \multicolumn{15}{c|}{}
        & $\rred 1$ & $\rred \cdots$ & $\rred 0$ & $\rred 0$ & $\rred 0$ & $\rred \cdots$ & $\rred 0$ & $\rred 0$ & $\rred 0$  & $\rred \cdots$ & $\rred 0$\\
        && $\vdots$
        & \multicolumn{15}{c|}{(Omitted)}
        & $\rred \vdots$ & \rotatebox{0}{$\rred \ddots$} & $\rred \vdots$ & $\rred \vdots$ & $\rred \vdots$ && $\rred \vdots$ & $\rred \vdots$ & $\rred \vdots$ && $\rred \vdots$\\
        && $t\!-\!1$
        & \multicolumn{15}{c|}{}
        & $\rred 1$ & $\rred \cdots$ & $\rred 1$ & $\rred 0$ & $\rred 0$ & $\rred \cdots$ & $\rred 0$ & $\rred 0$ & $\rred 0$ & $\rred \cdots$ & $\rred 0$\\
         
        \cline{1-1} \cline{3-29}
        \multirow{4}{*}{\rred{(B5)}} &
        & $t$
        & \multicolumn{15}{c|}{}
        & $1$ & $\cdots$ & $1$ & $1$ & $0$ & $\cdots$ & $0$ & $0$ & $0$ & $\cdots$ & $0$\\
         
        && $t\!+\!1$
        & \multicolumn{15}{c|}{}
        & $1$ & $\cdots$ & $1$ & $1$ & $1$ & $\cdots$ & $0$ & $0$ & $0$ & $\cdots$ & $0$\\
         
        && $\vdots$
        & \multicolumn{15}{c|}{\raisebox{10pt}{(Omitted)}}
        & $\vdots$ && $\vdots$ & $\vdots$ & $\vdots$ & \rotatebox{0}{$\ddots$} & $\vdots$ & $\vdots$ & $\vdots$ && $\vdots$\\
        
        &&$t\!+\!m$
        & \multicolumn{15}{c|}{}
        & $1$ & $\cdots$ & $1$ & $1$ & $1$ & $\cdots$ & $1$ & $0$ & $0$ & $\cdots$ & $0$\\
        
        \cline{1-1} \cline{3-29}
        \rred{(B6)} & 
        &$t\!+\!m\!+\!1$
        & \multicolumn{15}{c|}{(Omitted)}
        & $1$ & $\cdots$ & $1$ & $1$ & $1$ & $\cdots$ & $1$ & $1$ & $0$ & $\cdots$ & $0$\\
         
        \cline{1-1} \cline{3-29}
        \multirow{3}{*}{\rred{(B7)}} & 
        & $t\!+\!m\!+\!2$
        & \multicolumn{15}{c|}{}
        & $0$ & $\cdots$ & $0$ & $0$ & $0$ & $\cdots$ & $0$ & $0$ & $1$ & $\cdots$ & $0$\\
        
        && $\vdots$
        & \multicolumn{15}{c|}{(Omitted)}
        & $\vdots$ && $\vdots$ & $\vdots$ & $\vdots$ && $\vdots$ & $\vdots$ & $\vdots$ & \rotatebox{0}{$\ddots$} & $\vdots$\\
        
        && $T$
        & \multicolumn{15}{c|}{}
        & $0$ & $\cdots$ & $0$ & $0$ & $0$ & $\cdots$ & $0$ & $0$ & $0$ & $\cdots$ & $1$\\
        \hline
    \end{tabular}
    \label{tab:eqn-T:kth-ramp-3-1-facet-matrix-2}
\end{table}
\end{landscape}
}

    

 The matrix shown in Table \ref{tab:eqn-T:kth-ramp-3-1-facet-matrix-2} is lower triangular; that is, the position of the last nonzero component of a row of the matrix is greater than the position of the last nonzero component of the previous row.
This implies that the $2T-1$ points in groups \ref{points:prop16-eq1-A1}--\ref{points:prop16-eq1-A7} are linearly independent.
Therefore, inequality \eqref{eqn-T:kth-ramp-3-1} is facet-defining for $\convP$.

\rred{Next, we show that inequality \eqref{eqn-T:kth-ramp-3-2} is valid and facet-defining for $\convP$.
Denote $x'_t=x_{T-t+1}$ and $y'_t=y_{T-t+1}$ for $t\in[1,T]_\Z$.
Because inequality \eqref{eqn-T:kth-ramp-3-1} is valid and facet-defining for $\convP$ for any $t\in[k+1,T-m-1]_\Z$, the inequality
\begin{align*}
x'_{T-t+1}-x'_{T-t+k+1}\le\
& (\Clower+(k-m)V-\Vupper)y'_{T-t-m}+V\sum_{i=1}^my'_{T-t-i+1}+\Vupper y'_{T-t+1}-\Clower y'_{T-t+k+1}\\
& -\sum_{s\in\Se}(\Clower+(k-s)V-\Vupper)(y'_{T-t+s+1}-y'_{T-t+s+2})
\end{align*}
is valid and facet-defining for $\convPprime$ for any $t\in[k+1,T-m-1]_\Z$.
Let $t'=T-t+1$.
Then, the inequality
$$x'_{t'}\!-\!x'_{t'+k}\!\le\!(\Clower\!+\!(k\!-\!m)V\!-\!\Vupper)y'_{t'-m-1}\!+\!V\!\sum_{i=1}^my'_{t'-i}\!+\!\Vupper y'_{t'}\!-\!\Clower y'_{t'+k}\!-\!\sum_{s\in\Se}\!(\Clower\!+\!(k\!-\!s)V\!-\!\Vupper)(y'_{t'+s}\!-\!y'_{t'+s+1})$$
is valid and facet-defining for $\convPprime$ for any $t'\in[m+2,T-k]_\Z$.
Hence, by Lemma~\ref{lem:Pprime}, inequality \eqref{eqn-T:kth-ramp-3-2} is valid and facet-defining for $\convP$ for any $t\in[m+2,T-k]_\Z$.\Halmos
}


\subsection{Proof of Proposition \ref{prop-T:kth-ramp-3-separation}}

\noindent{\bf Proposition \ref{prop-T:kth-ramp-3-separation}.} {\it
For any given point $(\BFx,\BFy)\in\R^{2T}_+$, \tred{the} most violated inequalities \eqref{eqn-T:kth-ramp-3-1} and \eqref{eqn-T:kth-ramp-3-2} can be determined in $O(T^3)$ time if such violated inequalities exist.}
\vskip8pt

\noindent{\bf Proof.}
We first consider inequality \eqref{eqn-T:kth-ramp-3-1}.
Consider any given $(\BFx,\BFy)\in\R^{2T}_{+}$.
For notational convenience, denote $\hat{k}=\max\{k\in[1,T-1]_\Z:\Cupper-\Clower-kV>0\}$, and denote $\hat{s}_{km}=\min\{k-1,L-m-2\}$ for any $k\in[1,\hat{k}]_\Z$ and $m\in[0,k-1]_\Z$.
For any $t\in[1,T]_\Z$, let
$$\theta(t)=\sum_{\tau=2}^t\max\{y_\tau-y_{\tau-1},0\}.$$
Then, for any $k\in[1,\hat{k}]_\Z$, $m\in[0,k-1]_\Z$, and $t\in[k+1,T-m-1]_\Z$,
\begin{equation}\label{eq:P12-ysummation1}
\sum_{s=1}^{\hat{s}_{km}}\max\{y_{t-s}-y_{t-s-1},0\}=\sum_{\tau=t-\hat{s}_{km}}^{t-1}\max\{y_\tau-y_{\tau-1},0\}=\theta(t-1)-\theta(t-\hat{s}_{km}-1).
\end{equation}
For any $k\in[1,\hat{k}]_\Z$, $m\in[0,k-1]_\Z$, $t\in[k+1,T-m-1]_\Z$, and $\Se\subseteq[0,\hat{s}_{km}]_\Z$, let
\begin{align*}
\tilde{v}_{km}(\Se,t)=\ &x_t-x_{t-k}-(\Clower+(k-m)V-\Vupper)y_{t+m+1}-V\sum_{i=1}^m y_{t+i}-\Vupper y_t+\Clower y_{t-k}\\[-2pt]
                        &+\sum_{s\in\Se}(\Clower+(k-s)V-\Vupper)(y_{t-s}-y_{t-s-1}).
\end{align*}
If $\tilde{v}_{km}(\Se,t)>0$, then $\tilde{v}_{km}(\Se,t)$ is the amount of violation of inequality \eqref{eqn-T:kth-ramp-3-1}.
If $\tilde{v}_{km}(\Se,t)\le 0$, there is no violation of inequality \eqref{eqn-T:kth-ramp-3-1}.
For any $k\in[1,\hat{k}]_\Z$, $m\in[0,k-1]_\Z$, and $t\in[k+1,T-m-1]_\Z$, let
$$v_{km}(t)=\max_{\Se\subseteq[0,\hat{s}_{km}]_\Z}\{\tilde{v}_{km}(\Se,t)\}.$$
If $v_{km}(t)>0$, then $v_{km}(t)$ is the largest possible violation of inequality \eqref{eqn-T:kth-ramp-3-1} for this combination of $k$, $m$, and $t$.
If $v_{km}(t)\le 0$, the largest possible violation of inequality \eqref{eqn-T:kth-ramp-3-1} is zero for this combination of $k$, $m$, and $t$.
Because $\Clower+V>\Vupper$, we have $\Clower+(k-s)V-\Vupper>0$ for all $k\in[1,\hat{k}]_\Z$, $s\in[0,\hat{s}_{km}]_\Z$, and $m\in[0,k-1]_\Z$.
Thus, for any $k\in[1,\hat{k}]_\Z$, $m\in[0,k-1]_\Z$, and $t\in[k+1,T-m-1]_\Z$, 
$\tilde{v}_{km}(\Se,t)$ is maximized when $\Se$ contains all $s\in[0,\hat{s}_{km}]_\Z$ such that $y_{t-s}-y_{t-s-1}>0$ (if any).
If it does not exist any $s\in[0,\hat{s}]_\Z$ such that $y_{t-s}-y_{t-s-1}>0$, then $\tilde{v}_{km}(\Se,t)$ is maximized when $\Se=\emptyset$, 
and $v_{km}(t)=x_t-x_{t-k}-(\Clower+(k-m)V-\Vupper)y_{t+m+1}-V\sum_{i=1}^m y_{t+i}-\Vupper y_t+\Clower y_{t-k}$.
Hence, for any $k\in[1,\hat{k}]_\Z$, $m\in[0,k-1]_\Z$, and $t\in[k+1,T-m-1]_\Z$, 
\begin{align*}
v_{km}(t)=\ &x_t-x_{t-k}-(\Clower+(k-m)V-\Vupper)y_{t+m+1}-V\sum_{i=1}^m y_{t+i}-\Vupper y_t+\Clower y_{t-k}\\[-2pt]
            &+\sum_{s=0}^{\hat{s}_{km}}(\Clower+(k-s)V-\Vupper)\max\{y_{t-s}-y_{t-s-1},0\}.
\end{align*}

Determining $\theta(t)$ for all $t\in[1,T]_\Z$ can be done recursively in $O(T)$ time by setting $\theta(1)=0$ and setting $\theta(t)=\theta(t-1)+\max\{y_t-y_{t-1},0\}$ for $t=2,\ldots,T$.
Clearly, for each $k\in[1,\hat{k}]_\Z$ and each $m\in[0,k-1]_\Z$, the value of $v_{km}(k+1)$ can be determined in $O(T)$ time.
For any $k\in[1,\hat{k}]_\Z$, $m\in[0,k-1]_\Z$, and $t\in[k+2,T-m-1]_\Z$,
\begin{align*}
v_{km}(t)-v_{km}(t-1)
&=(x_t-x_{t-1})-(x_{t-k}-x_{t-k-1})-(\Clower+(k-m)V-\Vupper)(y_{t+m+1}-y_{t+m})\\
&\quad -V\left[\sum_{i=1}^m y_{t+i}-\sum_{i=1}^m y_{t+i-1}\right]-\Vupper(y_t-y_{t-1})+\Clower(y_{t-k}-y_{t-k-1})\\
&\quad +(\Clower+kV-\Vupper)\left[\sum_{s=0}^{\hat{s}_{km}}\max\{y_{t-s}-y_{t-s-1},0\}-\sum_{s=0}^{\hat{s}_{km}}\max\{y_{t-s-1}-y_{t-s-2},0\}\right]\\
&\quad -V\left[\sum_{s=0}^{\hat{s}_{km}}s\max\{y_{t-s}-y_{t-s-1},0\}-\sum_{s=0}^{\hat{s}_{km}}s\max\{y_{t-s-1}-y_{t-s-2},0\}\right]\\
&=(x_t-x_{t-1})-(x_{t-k}-x_{t-k-1})-(\Clower+(k-m)V-\Vupper)(y_{t+m+1}-y_{t+m})\\
&\quad -V(y_{t+m}-y_t)-\Vupper(y_t-y_{t-1})+\Clower(y_{t-k}-y_{t-k-1})\\
&\quad +(\Clower+kV-\Vupper)\left[\max\{y_t-y_{t-1},0\}-\max\{y_{t-\hat{s}_{km}-1}-y_{t-\hat{s}_{km}-2},0\}\right]\\
&\quad -V\left[\sum_{s=1}^{\hat{s}_{km}}\max\{y_{t-s}-y_{t-s-1},0\}-\hat{s}_{km}\max\{y_{t-\hat{s}_{km}-1}-y_{t-\hat{s}_{km}-2},0\}\right].
\end{align*}
This, together with \eqref{eq:P12-ysummation1}, implies that
\begin{align*}
v_{km}(t)=\ &v_{km}(t-1)+(x_t-x_{t-1})-(x_{t-k}-x_{t-k-1})-(\Clower+(k-m)V-\Vupper)(y_{t+m+1}-y_{t+m})\\
&-V(y_{t+m}-y_t)-\Vupper(y_t-y_{t-1})+\Clower(y_{t-k}-y_{t-k-1})\\
&+(\Clower+kV-\Vupper)\left[\max\{y_t-y_{t-1},0\}-\max\{y_{t-\hat{s}_{km}-1}-y_{t-\hat{s}_{km}-2},0\}\right]\\
&-V\left[\theta(t-1)-\theta(t-\hat{s}_{km}-1)-\hat{s}_{km}\max\{y_{t-\hat{s}_{km}-1}-y_{t-\hat{s}_{km}-2},0\}\right].
\end{align*}
Thus, for each $k\in[1,\hat{k}]_\Z$ and $m\in[0,k-1]_\Z$, the values of $v_{km}(k+1),v_{km}(k+2),\ldots,v_{km}(T-m-1)$ can be determined recursively in $O(T)$ time.
Hence, the values of $k$, $m$, $t$ and the set $\Se$ corresponding to the largest possible violation of inequality \eqref{eqn-T:kth-ramp-3-1} can be obtained in $O(T^3)$ time.

Next, we consider inequality \eqref{eqn-T:kth-ramp-3-2}.
Consider any given $(\BFx,\BFy)\in\R^{2T}_{+}$.
Let $x'_t=x_{T-t+1}$ and $y'_t=y_{T-t+1}$ for $t\in[1,T]_\Z$.
Inequality \eqref{eqn-T:kth-ramp-3-2} becomes
\begin{align*}
x'_{T-t+1}-x'_{T-t-k+1}\le\
&(\Clower+(k-m)V-\Vupper)y'_{T-t+m+2}+V\sum_{i=1}^m y'_{T-t+i+1}+\Vupper y'_{T-t+1}-\Clower y'_{T-t-k+1}\\
&-\sum_{s\in\Se}(\Clower+(k-s)V-\Vupper)(y'_{T-t-s+1}-y'_{T-t-s}).
\end{align*}
Letting $t'=T-t+1$, this inequality becomes
\begin{align}
x'_{t'}-x'_{t'-k}\le\
&(\Clower+(k-m)V-\Vupper)y'_{t'+m+1}+V\sum_{i=1}^m y'_{t'+i}+\Vupper y'_{t'}-\Clower y'_{t'-k}\nonumber\\
&-\sum_{s\in\Se}(\Clower+(k-s)V-\Vupper)(y'_{t'-s}-y'_{t'-s-1}).\label{eqn-T:kth-ramp-3-2-prime}
\end{align}
Because the values of $k$, $m$, $t$ and the set $\Se$ corresponding to the largest possible violation of inequality \eqref{eqn-T:kth-ramp-3-1} can be obtained in $O(T^3)$ time, the values of $k$, $m$, $t'$ and the set $\Se$ corresponding to the largest possible violation of inequality \eqref{eqn-T:kth-ramp-3-2-prime} can be obtained in $O(T^3)$ time.
Hence, the values of $k$, $m$, $t$ and the set $\Se$ corresponding to the largest possible violation of inequality \eqref{eqn-T:kth-ramp-3-2} can be obtained in $O(T^3)$ time.
\Halmos


\section{Supplement to Section~\ref{sec:comp-exper}}\label{apx:B}
\subsection{Results of the Third Experiment Under Different Demand Settings}\label{apx:B-1}

For the third experiment, we consider two additional cases, namely a less congested demand setting and a more congested demand setting.
For the more congested demand setting, we obtain the new demand $\bar{d}_t^b$ at bus $b\in\Be$ for each period $t\in[1,T]_{\Z}$ by increasing the corresponding demand $d_t^b$ in the third experiment by $10\%$, while ensuring that $\bar{d}_t^b$ will not exceed the total generation capacity of these 54 thermal generators at the same time.
Thus, we set $\bar{d}_t^b = \min\{1.1 d_t^b, \sum_{g\in\Ge}\Cupper^g\}$. 
For the less congested demand setting, we obtain $\bar{d}_t^b$ by decreasing the demand $d_t^b$ by $10\%$; that is, we set $\bar{d}_t^b = 0.9 d_t^b$.

\afterpage{
\begin{table}[b]
  \renewcommand{\arraystretch}{1.2}
  \centering
  \caption{Performance of MIP Formulations in the Third Experiment Under a Less Congested Demand Setting}
  \fontsize{8}{12}\selectfont
  \setlength\tabcolsep{5pt}
  \begin{tabular}{|*{10}{c|}}
    \hline
    \multirow{2}{*}{Instance} & \multirow{2}{*}{\# var} & \multirow{2}{*}{\# bin var} & \multicolumn{2}{c|}{\# cstr} & \multicolumn{2}{c|}{CPU time [TGap]} 
    & \multicolumn{2}{c|}{\# nodes} & \# user cuts \\
    \cline{4-10}
       &&& \F1^+ & \F1^+-X & \F1^+     & \F1^+-X    & \F1^+  & \F1^+-X & \F1^+-X \\
    \hline
     1 & \multirow{20}{*}{6372} & \multirow{20}{*}{1296} & \multirow{20}{*}{36124} & \multirow{20}{*}{43576} & 307.8      & 212.0      & 19395  & 12293  & 88\\
     2 &&&&& 174.7      & 233.8      & 13023  & 10430  & 95 \\
     3 &&&&& 166.4      & 522.7      & 11248  & 13803  & 96 \\
     4 &&&&& 588.9      & 342.1      & 36443  & 17313  & 98 \\
     5 &&&&& 62.0       & 106.6      & 6775   & 2699   & 70 \\
     6 &&&&& **[0.11\%] & **[0.07\%] & 115728 & 103776 & 229 \\
     7 &&&&& 1116.6     & 848.6      & 43524  & 36756  & 126 \\
     8 &&&&& 1247.9     & 1985.2     & 63266  & 59530  & 217 \\
     9 &&&&& 821.9      & 1774.3     & 41451  & 56742  & 198 \\
    10 &&&&& 1881.1     & 2230.6     & 60109  & 69543  & 188 \\
    11 &&&&& 745.1      & 2793.1     & 40789  & 72982  & 182 \\
    12 &&&&& 87.8       & 226.9      & 9905   & 9968   & 93 \\
    13 &&&&& 35.5       & 66.7       & 2370   & 2702   & 45 \\
    14 &&&&& 1187.4     & 1043.2     & 49699  & 43317  & 143 \\
    15 &&&&& 272.5      & 818.3      & 21934  & 39631  & 173 \\
    16 &&&&& 306.6      & 206.8      & 21323  & 9596   & 99 \\
    17 &&&&& 200.1      & 614.2      & 18638  & 27580  & 128 \\
    18 &&&&& 82.3       & 543.2      & 9904   & 18321  & 88 \\
    19 &&&&& 70.6       & 119.5      & 6875   & 3618   & 73 \\
    20 &&&&& 81.8       & 148.5      & 10301  & 10550  & 66 \\
    \hline
  \end{tabular}
  \label{tab:Exp3-less-congested}
\end{table}

\begin{table}[ht]
  \renewcommand{\arraystretch}{1.2}
  \centering
  \caption{The Strength of LP Relaxations of MIP Formulations in the Third Experiment Under a Less Congested Demand Setting}
  \setlength\tabcolsep{5pt}
  \fontsize{8}{12}\selectfont
  \begin{tabular}{|c|c|*{10}{c}|}
    \hline
    \multirow{6}{*}{IGap} & Instance & 1 & 2 & 3 & 4 & 5 & 6 & 7 & 8 & 9 & 10\\
    \cline{2-12}
     & \F1^+ & 0.57\% & 0.59\% & 0.50\% & 0.63\% & 0.53\% & 0.58\% & 0.58\% & 0.64\% & 0.53\% & 0.58\% \\
     & \F1^+-X & 0.57\% & 0.59\% & 0.50\% & 0.63\% & 0.53\% & 0.58\% & 0.58\% & 0.63\% & 0.52\% & 0.57\% \\
    \cline{2-12}
     & Instance & 11 & 12 & 13 & 14 & 15 & 16 & 17 & 18 & 19 & 20\\
    \cline{2-12}
     & \F1^+ & 0.60\% & 0.57\% & 0.47\% & 0.58\% & 0.58\% & 0.59\% & 0.60\% & 0.59\% & 0.47\% & 0.46\% \\
     & \F1^+-X & 0.60\% & 0.57\% & 0.47\% & 0.58\% & 0.58\% & 0.59\% & 0.60\% & 0.59\% & 0.47\% & 0.46\% \\
    \hline
    \multirow{4}{*}{Pct. reduction} & Instance & 1 & 2 & 3 & 4 & 5 & 6 & 7 & 8 & 9 & 10\\
    \cline{2-12}
    & \F1^+-X & 0.14\% & 0\% & 0\% & 0\% & 0\% & 0\% & 0\% & 1.36\% & 1.82\% & 1.17\% \\
    \cline{2-12}
    & Instance & 11 & 12 & 13 & 14 & 15 & 16 & 17 & 18 & 19 & 20\\
    \cline{2-12}
    & \F1^+-X & 0\% & 0.17\% & 0\% & 0.04\% & 0.09\% & 0.13\% & 0\% & 0\% & 0.65\% & 0.11\%\\
    \hline
  \end{tabular}
  \label{tab:Exp3-less-congested-2}
\end{table}
}


\tred{Tables \ref{tab:Exp3-less-congested} and \ref{tab:Exp3-less-congested-2} present} the computational results for the less congested demand setting.
Under this setting, formulations \F1^+ and \F1^+-X exhibit similar overall performance.
Both formulations successfully solved 19 instances within the time limit.
The integrality gaps for \F1^+ and \F1^+-X are nearly identical across all instances.
This can be attributed to the decrease in demand for each time period, 
resulting in a smaller number of generators being started up, 
as well as a reduction in the number of ramp-ups and ramp-downs.
Therefore, few valid inequalities are applied in each solution process.
Consequently, the reduction in the integrality gap is small from \F1^+ to \F1^+-X.
For most instances, the solution time of \F1^+ is smaller than \F1^+-X, because valid inequalities \eqref{multi:eqn-q2:x1-ub}--\eqref{multi:eqn-q2:x1-x2-ub-2} are added as constraints in \F1^+-X, 
making the subproblems at each node of the branching process more difficult to solve.

\tred{Tables \ref{tab:Exp3-more-congested} and \ref{tab:Exp3-more-congested-2} present} the computational results for the more congested setting.
Under this setting, formulation \F1^+-X outperforms formulation \F1^+ in all instances.
The integrality gaps of the \F1^+-X formulation are much smaller than those of \F1^+, 
and the ``Pct.~reduction" is approximately 90\% for each instance.
Eight instances are solved to optimality by formulation \F1^+-X, 
and the TGaps for the remaining twelve instances are all within 0.05\%.
On the other hand, none of these instances are solved successfully by formulation \F1^+, 
and the TGaps are all above 0.26\%.
For most instances, formulation \F1^+-X explores less nodes than formulation \F1^+.
This is because as the demand for each time period increases, more generators are needed, 
and the number of ramp-ups and ramp-downs also increases.
As a result, more user cuts are used during the solution process.

\afterpage{
\begin{table}[t]
  \renewcommand{\arraystretch}{1.2}
  \centering
  \caption{Computational Results of the Third Experiment Under a More Congested Demand Setting}
  \fontsize{8}{12}\selectfont
  \setlength\tabcolsep{5pt}
  \begin{tabular}{|*{10}{c|}}
    \hline
    \multirow{2}{*}{Instance} & \multirow{2}{*}{\# var} & \multirow{2}{*}{\# bin var} & \multicolumn{2}{c|}{\# cstr} & \multicolumn{2}{c|}{CPU time [TGap]} 
    & \multicolumn{2}{c|}{\# nodes} & \# user cuts \\
    \cline{4-10}
       &&& \F1^+ & \F1^+-X & \F1^+     & \F1^+-X    & \F1^+  & \F1^+-X & \F1^+-X \\
    \hline
     1 & \multirow{20}{*}{6372} & \multirow{20}{*}{1296} & \multirow{20}{*}{36124} & \multirow{20}{*}{43576} & **[0.27\%] & 616.8      & 409894  & 124616  & 809\\
     2 &&&&& **[0.30\%] & 503.2      & 893931  & 62034   & 847 \\
     3 &&&&& **[0.42\%] & 471.4      & 49184   & 587395  & 581 \\
     4 &&&&& **[0.28\%] & 928.7     & 579540  & 87807   & 652 \\
     5 &&&&& **[0.33\%] & **[0.04\%] & 516315  & 860515  & 684 \\
     6 &&&&& **[0.31\%] & **[0.03\%] & 588982  & 1021155 & 764 \\
     7 &&&&& **[0.40\%] & 1195.5     & 602518  & 223840  & 669 \\
     8 &&&&& **[0.53\%] & **[0.05\%] & 634564  & 523103  & 1087 \\
     9 &&&&& **[0.51\%] & **[0.05\%] & 543776  & 289863  & 819 \\
    10 &&&&& **[0.53\%] & **[0.02\%] & 589285  & 434974  & 925 \\
    11 &&&&& **[0.39\%] & **[0.04\%] & 589812  & 432163  & 1088 \\
    12 &&&&& **[0.45\%] & **[0.03\%] & 641832  & 609698  & 932 \\
    13 &&&&& **[0.43\%] & 1714.3    & 610207  & 251633  & 897 \\
    14 &&&&& **[0.45\%] & **[0.03\%] & 603593  & 286153  & 827 \\
    15 &&&&& **[0.46\%] & **[0.04\%] & 834771  & 506532  & 797 \\
    16 &&&&& **[0.35\%] & **[0.04\%] & 617076  & 432809  & 734 \\
    17 &&&&& **[0.27\%] & **[0.02\%] & 838514  & 462097  & 778 \\
    18 &&&&& **[0.27\%] & 3221.3    & 603741  & 562787  & 892 \\
    19 &&&&& **[0.35\%] & 880.5      & 603631  & 310824  & 792 \\
    20 &&&&& **[0.58\%] & **[0.02\%] & 639278  &  614575 & 962 \\
    \hline
  \end{tabular}
  \label{tab:Exp3-more-congested}
\end{table}

\begin{table}[ht]
  \renewcommand{\arraystretch}{1.2}
  \centering
  \caption{The Strength of LP Relaxations of MIP Formulations in the Third Experiment Under a More Congested Demand Setting}
  \setlength\tabcolsep{5pt}
  \fontsize{8}{12}\selectfont
  \begin{tabular}{|c|c|*{10}{c}|}
    \hline
    \multirow{6}{*}{IGap} & Instance & 1 & 2 & 3 & 4 & 5 & 6 & 7 & 8 & 9 & 10\\
    \cline{2-12}
     & \F1^+ & 2.33\% & 2.61\% & 2.53\% & 2.24\% & 2.43\% & 2.40\% & 2.48\% & 2.60\% & 2.56\% & 2.48\% \\
     & \F1^+-X & 0.21\% & 0.21\% & 0.26\% & 0.20\% & 0.21\% & 0.17\% & 0.23\% & 0.32\% & 0.29\% & 0.26\% \\
    \cline{2-12}
     & Instance & 11 & 12 & 13 & 14 & 15 & 16 & 17 & 18 & 19 & 20\\
    \cline{2-12}
     & \F1^+ & 2.59\% & 2.76\% & 2.61\% & 2.41\% & 2.56\% & 2.56\% & 2.35\% & 2.25\% & 2.41\% & 2.48\% \\
     & \F1^+-X & 0.24\% & 0.22\% & 0.23\% & 0.30\% & 0.26\% & 0.22\% & 0.24\% & 0.26\% & 0.25\% & 0.25\% \\
    \hline
    \multirow{4}{*}{Pct. reduction} & Instance & 1 & 2 & 3 & 4 & 5 & 6 & 7 & 8 & 9 & 10\\
    \cline{2-12}
    & \F1^+-X & 90.81\% & 91.93\% & 89.87\% & 91.24\% & 91.36\% & 93.04\% & 90.54\% & 87.50\% & 88.79\% & 89.52\% \\
    \cline{2-12}
    & Instance & 11 & 12 & 13 & 14 & 15 & 16 & 17 & 18 & 19 & 20\\
    \cline{2-12}
    & \F1^+-X & 90.62\% & 91.88\% & 91.33\% & 87.41\% & 89.93\% & 91.36\% & 89.80\% & 88.40\% & 89.72\% & 89.76\%\\
    \hline
  \end{tabular}
  \label{tab:Exp3-more-congested-2}
\end{table}
}


\subsection{Results of the Three Experiments by Disabling the Smart Features of the Solver}

In this section, we run the three experiments in Section~\ref{sec:comp-exper} by disabling the smart features of CPLEX, including the presolve, heuristics, and default cut generation, to examine the strength of our proposed strong valid inequalities to help solve UC formulations.
As the numbers of variables and constraints remain unchanged for all tested formulations in these three experiments, they are not reported in the following tables.

The results of the first experiment are summarized in Tables~\ref{tab:apxB:Exp1-1} and \ref{tab:apxB:Exp1-2}.
Using formulation F1, CPLEX is unable to solve any of these 20 instances to optimality within the one-hour time limit.
In Table~\ref{tab:apxB:Exp1-1}, the ``---" rows denote the instances for which a feasible solution cannot be found by CPELX within the time limit using formulation F1.
Using formulation F1-X, CPLEX can solve 2 instances to optimality within the time limit.
The number of nodes explored by F1-X is much smaller than that of formulation F1.
The number of user cuts added by F1-X in the solution process is smaller compared with the total number of constraints in formulations F1 and F1-X.
For the instances that cannot be solved to optimality using formulation F1-X, the terminating gaps are all within $0.05\%$ except for instances 2 and 7, and are much smaller than those using formulation F1.
The integrality gaps generated by formulation F1-X are substantially smaller than those of F1, with reductions of at least $55\%$.
These results demonstrate that the proposed strong valid inequalities significantly tighten the single-binary formulation, thereby enhancing the efficiency of the solution process.
\afterpage{
\begin{table}[ht]
  \renewcommand{\arraystretch}{1.2}
  \centering
  \caption{\tred{Performance of MIP Formulations in the First Experiment by Disabling Smart Features}}
  \setlength\tabcolsep{7pt}
  \fontsize{8}{12}\selectfont
  \begin{tabular}{|*{6}{c|}}
    \hline
    \multirow{2}{*}{Instance} & \multicolumn{2}{c|}{CPU time [TGap]} 
    & \multicolumn{2}{c|}{\# nodes} & \# user cuts \\
    \cline{2-6}
     & F1           & F1-X       & F1      & F1-X    & F1-X \\
     \hline
     1 & **[0.25\%] & **[0.04\%] & 3243082 & 1331163 & 428 \\
     2 & **[0.33\%] & **[0.10\%] & 1687199 & 906171  & 768  \\
     3 & **[0.35\%] & **[0.05\%] & 1398467 & 403775  & 937 \\
     4 & **[0.30\%] & **[0.01\%] & 1217117 & 1592032 & 490 \\
     5 & **[0.48\%] & **[0.03\%] & 1218680 & 494229  & 823 \\
     6 & **[0.53\%] & **[0.02\%] & 2492793 & 837237  & 823 \\
     7 & **[0.31\%] & **[0.06\%] & 1647817 & 576308  & 472 \\
     8 & **[0.53\%] & **[0.05\%] & 1431881 & 547503  & 732 \\
     9 & **[0.44\%] & **[0.04\%] & 1062217 & 481563  & 899 \\
    10 & **[0.62\%] & **[0.03\%] & 1232270 & 552396  & 790 \\
    11 & **[0.47\%] &   671.2    & 397815  & 17587   & 1050 \\
    12 & **[0.66\%] & **[0.03\%] & 345767  & 132849  & 2423 \\
    13 & **[0.71\%] & **[0.02\%] & 300637  & 50590   & 1907 \\
    14 & ---        & **[0.03\%] & ---     & 70987   & 3195 \\
    15 & ---        & **[0.03\%] & ---     & 17541   & 2050 \\
    16 & **[0.70\%] & 935.7      & 262321  & 15358   & 1347 \\
    17 & ---        & **[0.02\%] & ---     & 68966   & 2168 \\
    18 & **[1.63\%] & **[0.01\%] & 155156  & 106438  & 2116 \\
    19 & ---        & **[0.02\%] & ---     & 61927   & 2338 \\
    20 & ---        & **[0.02\%] & ---     & 66462   & 1810 \\
    \hline
  \end{tabular}
  \label{tab:apxB:Exp1-1}
\end{table}

\begin{table}[ht]
  \renewcommand{\arraystretch}{1.2}
  \centering
  \caption{\tred{The Strength of LP Relaxations of MIP Formulations in the First Experiment by Disabling Smart Features}}
  \setlength\tabcolsep{5pt}
  \fontsize{8}{12}\selectfont
  \begin{tabular}{|c|c|*{10}{c}|}
    \hline
    \multirow{6}{*}{IGap} & Instance & 1 & 2 & 3 & 4 & 5 & 6 & 7 & 8 & 9 & 10\\
    \cline{2-12}
     & F1 & 0.45\% & 0.39\% & 0.41\% & 0.36\% & 0.51\% & 0.62\% & 0.34\% & 0.55\% & 0.51\% & 0.64\% \\
     & F1-X & 0.20\% & 0.14\% & 0.08\% & 0.06\% & 0.05\% & 0.05\% & 0.07\% & 0.06\% & 0.05\% & 0.05\% \\
    \cline{2-12}
     & Instance & 11 & 12 & 13 & 14 & 15 & 16 & 17 & 18 & 19 & 20\\
    \cline{2-12}
     & F1 & 0.32\% & 0.32\% & 0.40\% & 0.39\% & 0.52\% & 0.36\% & 0.45\% & 0.42\% & 0.37\% & 0.44\% \\
     & F1-X & 0.05\% & 0.02\% & 0.02\% & 0.03\% & 0.03\% & 0.08\% & 0.01\% & 0.01\% & 0.02\% & 0.02\% \\
    \hline
    \multirow{4}{*}{Pct. reduction} & Instance & 1 & 2 & 3 & 4 & 5 & 6 & 7 & 8 & 9 & 10\\
    \cline{2-12}
    & F1-X & 55.7\% & 63.9\% & 81.6\% & 84.5\% & 90.2\% & 92.4\% & 78.7\% & 89.5\% & 89.5\% & 92.8\% \\
    \cline{2-12}
    & Instance & 11 & 12 & 13 & 14 & 15 & 16 & 17 & 18 & 19 & 20\\
    \cline{2-12}
    & F1-X & 83.1\% & 94.1\% & 95.1\% & 92.3\% & 94.0\% & 77.1\% & 96.0\% & 96.6\% & 94.2\% & 94.4\%\\
    \hline
  \end{tabular}
  \label{tab:apxB:Exp1-2}
\end{table}
}

Tables~\ref{tab:apx:Exp2-1} and \ref{tab:apx:Exp2-2} present the results for the second experiment.
Using formulation F2, no instance can be solved to optimality by CPLEX within the time limit.
In contrast, using formulation F2-X, the two-binary formulation with our proposed valid inequalities, CPLEX can solve 5 instances to optimality within the time limit.
Using formulation F2-Y, the two-binary formulation with valid inequalities from \cite{pan2016polyhedral}, CPLEX is able to solve 2 instances to optimality within the time limit.
Using formulation F2-Z, the two-binary formulation with our and \citeauthor{pan2016polyhedral}'s inequalities, CPLEX can solve only 1 instance to optimality within the time limit.
This can be attributed to the increase in the size of the LP relaxation at each node during the branching process, resulting from the addition of a large number of user cuts, which adversely affects the solution time. 
Fewer nodes are explored by the three strong formulations F2-X, F2-Y, and F2-Z, compared to formulation F2.
The number of user cuts added during the solution process by these three strong formulations is smaller compared to the total number of constraints in formulations F2, F2-X, F2-Y, and F2-Z.
For nearly all of the unsolved instances, the terminating gaps of the three strong formulations are within $0.05\%$.
The integrality gaps of these three strong formulations are also much smaller than those using formulation F2.
Furthermore, formulations F1-X and F2-Y have similar performance, suggesting that the single-binary formulation with our proposed valid inequalities exhibits comparable results to the two-binary formulation with strong valid inequalities.
In addition, formulation F2-X outperforms both F2-Y and F2-Z, demonstrating that the effectiveness of our proposed valid inequalities in tightening the two-binary formulation is more significant than that of valid inequalities from \cite{pan2016polyhedral}.
\afterpage{
\begin{table}[ht]
  \renewcommand{\arraystretch}{1.2}
  \centering
  \caption{\tred{Performance of MIP Formulations in the Second Experiment by Disabling Smart Features}}
  \setlength\tabcolsep{5pt}
  \fontsize{7}{12}\selectfont
  \begin{tabular}{|*{12}{c|}}
    \hline
    \multirow{2}{*}{Instance} & \multicolumn{4}{c|}{CPU time [TGap]} 
    & \multicolumn{4}{c|}{\# nodes} & \multicolumn{3}{c|}{\# user cuts} \\
    \cline{2-12}
       & F2         & F2-X    & F2-Y   & F2-Z        & F2      & F2-X     & F2-Y    & F2-Z     & F2-X   & F2-Y   & F2-Z\\
    \hline
     1 & **[0.26\%]  & 2429.2     & 184.4       & **[0.03\%]  & 3749277 & 4217040 & 266731  & 1948319  & 461  & 211  & 571\\
     2 & ** [0.29\%] & **[0.05\%] & **[0.09\%]  & **[0.05\%]  & 3496986 & 2166004 & 1841074 & 1683486  & 521  & 304  & 576\\
     3 & **[0.34\%]  & **[0.06\%] & **[0.05\%]  & **[0.05\%]  & 1936293 & 1603280 & 1303901 & 520019   & 730  & 293  & 426\\
     4 & ** [0.26\%] & **[0.02\%] & **[0.02\%]  & **[0.02\%]  & 2231992 & 2161428 & 1690409 & 1537977  & 396  & 171  & 338\\
     5 & **[0.46\%]  & **[0.04\%] & **[0.04\%]  & **[0.03\%]  & 1786842 & 1790680 & 1151283 & 715773   & 780  & 238  & 535\\
     6 & **[0.51\%]  & 3241.8     & **[0.01\%]  & **[0.01\%]  & 4239872 & 3331564 & 1556754 & 1175628  & 782  & 298  & 563\\
     7 & **[0.27\%]  & **[0.05\%] & 332.8       & **[0.06\%]  & 1824764 & 1492391 & 161419  & 915783   & 504  & 177  & 450\\
     8 & ** [0.46\%] & **[0.05\%] & **[0.05\%]  & **[0.05\%]  & 1955325 & 1620740 & 703513  & 631110   & 542  & 220  & 444\\
     9 & **[0.43\%]  & **[0.03\%] & **[0.05\%]  & **[0.05\%]  & 2294145 & 1603961 & 661026  & 403170   & 1007 & 351  & 608\\
    10 & **[0.60\%]  & **[0.02\%] & **[0.04\%]  & **[0.03\%]  & 2024929 & 1983320 & 875849  & 950878   & 869  & 299   & 720\\
    11 & **[0.42\%]  & **[0.05\%] & **[0.05\%]  & **[0.05\%]  & 647225  & 479026  & 448558  & 184850   & 1066 & 643   & 1392\\
    12 & **[0.38\%]  & **[0.02\%] & **[0.06\%]  & **[0.03\%]  & 558586  & 410723  & 66363   & 42267    & 1983 & 639   & 1077\\
    13 & **[0.41\%]  & **[0.02\%] & **[0.02\%]  & **[0.02\%]  & 562676  & 437129  & 148136  & 148136   & 1420 & 1650  & 1483\\
    14 & **[0.57\%]  & **[0.04\%] & **[0.03\%]  & **[0.03\%]  &  478871 & 317574  &  412241 & 164262   & 2285 & 3048  & 1140\\
    15 & **[0.82\%]  & 797.1      & **[0.02\%]  & **[0.03\%]  & 617261  & 70838   & 181056  & 41661    & 1827 & 658   & 1739\\
    16 & **[0.48\%]  & 1812.6     & **[0.03\%]  & **[0,02\%]  & 521416  & 283080  & 236032  & 107088   & 1576 & 721   & 1372\\
    17 & **[0.80\%]  & 788.0      & **[0.02\%]  & 2570.1      & 314340  & 85493   & 204695  & 34590    & 1446 & 563   & 1239\\
    18 & **[0.81\%]  & **[0.01\%] & **[0.01\%]  & **[0,02\%]  & 482334  &  304661 & 195106  & 125472   & 1987 & 679   & 1501\\
    19 & **[0.69\%]  & **[0.02\%] & **[0.02\%]  & **[0.02\%]  & 403289  &  370957 & 173920  & 94781    & 2124 & 600 & 1238\\
    20 & **[0.81\%] & **[0.01\%] & **[0.02\%]   & **[0.04\%]  &  452740 & 290555  & 149267  & 22281    & 1764 & 496 & 1315\\
    \hline
  \end{tabular}
  \label{tab:apx:Exp2-1}
\end{table}

\begin{table}[ht]
  \renewcommand{\arraystretch}{1.2}
  \centering
  \caption{\tred{The Strength of LP Relaxations of MIP Formulations in the Second Experiment by Disabling Smart Features}}
  \setlength\tabcolsep{5pt}
  \fontsize{8}{12}\selectfont
  \begin{tabular}{|c|c|*{10}{c}|}
    \hline
    \multirow{10}{*}{IGap} & Instance & 1 & 2 & 3 & 4 & 5 & 6 & 7 & 8 & 9 & 10\\
    \cline{2-12}
    & F2 & 0.37\% & 0.31\% & 0.25\% & 0.20\% & 0.26\% & 0.25\% & 0.25\% & 0.29\% & 0.27\% & 0.36\% \\
    & F2-X & 0.20\% & 0.14\% & 0.08\% & 0.06\% & 0.05\% & 0.04\% & 0.07\% & 0.06\% & 0.06\% & 0.05\% \\
    & F2-Y & 0.20\% & 0.14\% & 0.08\% & 0.05\% & 0.05\% & 0.05\% & 0.08\% & 0.06\% & 0.06\% & 0.05\% \\
    & F2-Z & 0.20\% & 0.14\% & 0.08\% & 0.05\% & 0.05\% & 0.04\% & 0.07\% & 0.06\% & 0.06\% & 0.05\%\\
    \cline{2-12}
    & Instance & 11 & 12 & 13 & 14 & 15 & 16 & 17 & 18 & 19 & 20\\
    \cline{2-12}
    & F2 & 0.23\% & 0.19\% & 0.20\% & 0.20\% & 0.26\% & 0.20\% & 0.22\% & 0.19\% & 0.17\% & 0.22\% \\
    & F2-X & 0.05\% & 0.02\% & 0.02\% & 0.03\% & 0.03\% & 0.02\% & 0.03\% & 0.02\% & 0.02\% & 0.04\% \\
    & F2-Y & 0.06\% & 0.03\% & 0.02\% & 0.03\% & 0.03\% & 0.02\% & 0.02\% & 0.02\% & 0.02\% & 0.03\%\\
    & F2-Z & 0.05\% & 0.02\% & 0.02\% & 0.03\% & 0.03\% & 0.02\% & 0.02\% & 0.02\% & 0.02\% & 0.03\%\\
    \hline
    \multirow{8}{*}{Pct. reduction} & Instance & 1 & 2 & 3 & 4 & 5 & 6 & 7 & 8 & 9 & 10\\
    \cline{2-12}
    & F2-X & 45.5\% & 54.8\% & 69.2\% & 69.6\% & 80.0\% & 83.4\% & 72.1\% & 78.0\% & 77.4\% & 86.9\% \\
    & F2-Y & 45.5\% & 54.7\% & 67.9\% & 73.2\% & 81.8\% & 80.6\% & 69.6\% & 78.8\% & 78.6\% & 85.9\%\\
    & F2-Z & 45.5\% & 54.8\% & 69.4\% & 73.2\% & 81.8\% & 84.1\% & 72.1\% & 80.3\% & 78.6\% & 86.9\%\\
    \cline{2-12}
    & Instance & 11 & 12 & 13 & 14 & 15 & 16 & 17 & 18 & 19 & 20\\
    \cline{2-12}
    & F2-X & 78.2\% & 87.5\% & 89.3\% & 86.9\% & 88.6\% & 91.4\% & 88.2\% & 89.7\% & 85.6\% & 83.2\%\\
    & F2-Y & 76.2\% & 84.9\% & 89.0\% & 86.0\% & 89.3\% & 87.8\% & 89.3\% & 91.8\% & 88.1\% & 84.3\%\\
    & F2-Z & 78.1\% & 87.7\% & 89.5\% & 87.0\% & 89.4\% & 91.4\% & 89.3\% & 91.8\% & 88.1\% & 84.4\%\\
    \hline
  \end{tabular}
  \label{tab:apx:Exp2-2}
\end{table}
}
The results for the third experiment are presented in Tables~\ref{tab:axp:Exp3-1} and \ref{tab:axp:Exp3-2}.
Using formulation \F1^+, CPLEX cannot solve any of these 20 instances to optimality within the time limit while using formulation \F1^+-X, CPLEX can solve 1 instance to optimality.
For the instance that cannot be solved to optimality within the time limit, the terminating gaps of formulation \F1^+-X are much smaller than those of formulation \F1^+.
The number of nodes explored by formulation \F1^+-X is smaller than that of formulation \F1^+.
The number of user cuts added in the solution process is also smaller than the total number of constraints in formulations \F1^+ and \F1^+-X.
The integrality gaps of formulation \F1^+-X are also much smaller than those using formulation \F1^+, which indicates that our proposed valid inequalities can tighten the single-binary formulation significantly.

\afterpage{
\begin{table}[ht]
  \renewcommand{\arraystretch}{1.2}
  \centering
  \caption{\tred{Performance of MIP Formulations in the Third Experiment by Disabling Smart Features}}
  \fontsize{8}{12}\selectfont
  \setlength\tabcolsep{7pt}
  \begin{tabular}{|*{6}{c|}}
    \hline
    \multirow{2}{*}{Instance} & \multicolumn{2}{c|}{CPU time [TGap]} 
    & \multicolumn{2}{c|}{\# nodes} & \# user cuts \\
    \cline{2-6}
    & \F1^+ & \F1^+-X & \F1^+ & \F1^+-X & \F1^+-X \\
    \hline
     1 & **[1.07\%] & **[0.31\%] & 678548  & 267276 & 1636 \\
     2 & **[1.14\%] & **[0.33\%] & 812095  & 204492 & 1470 \\
     3 & **[0.85\%] & **[0.23\%] & 755989  & 276532 & 1500 \\
     4 & **[1.04\%] & **[0.45\%] & 828645  & 157520 & 1421 \\
     5 & **[1.32\%] & 3398.7     & 601218  & 298525 & 1943 \\
     6 & **[1.23\%] & **[0.17\%] & 511022  & 324346 & 1594 \\
     7 & **[1.42\%] & **[0.23\%] & 962694  & 491768 & 1712 \\
     8 & **[1.20\%] & **[0.37\%] & 747792  & 294910 & 1424 \\
     9 & **[1.10\%] & **[0.27\%] & 880796  & 215853 & 1512 \\
    10 & **[0.91\%] & **[0.26\%] & 1046090 & 285761 & 1862 \\
    11 & **[1.11\%] & **[0.31\%] & 870860  & 382853 & 1457 \\
    12 & **[1.59\%] & **[0.26\%] & 765177  & 275495 & 1955 \\
    13 & **[0.96\%] & **[0.22\%] & 902408  & 206672 & 1464 \\
    14 & **[1.09\%] & **[0.47\%] & 902171  & 285107 & 1594 \\
    15 & **[1.27\%] & **[0.35\%] & 704154  & 267693 & 1580 \\
    16 & **[1.23\%] & **[0.30\%] & 850453  & 101352 & 1338 \\
    17 & **[1.17\%] & **[0.38\%] & 569100  & 276125 & 1412 \\
    18 & **[0.96\%] & **[0.27\%] & 710371  & 282015 & 1551 \\
    19 & **[1.23\%] & **[0.27\%] & 639002  & 274153 & 1569 \\
    20 & **[1.20\%] & **[0.25\%] & 608163  & 414109 & 1420 \\
    \hline
  \end{tabular}
  \label{tab:axp:Exp3-1}
\end{table}

\begin{table}[ht]
  \renewcommand{\arraystretch}{1.2}
  \centering
  \caption{\tred{The Strength of LP Relaxations of MIP Formulations in the Third Experiment by Disabling Smart Features}}
  \setlength\tabcolsep{5pt}
  \fontsize{8}{12}\selectfont
  \begin{tabular}{|c|c|*{10}{c}|}
    \hline
    \multirow{6}{*}{IGap} & Instance & 1 & 2 & 3 & 4 & 5 & 6 & 7 & 8 & 9 & 10\\
    \cline{2-12}
     & \F1^+ & 0.94\% & 1.10\% & 0.81\% & 1.04\% & 0.86\% & 0.89\% & 1.15\% & 1.08\% & 1.03\% & 0.88\% \\
     & \F1^+-X & 0.45\% & 0.46\% & 0.43\% & 0.63\% & 0.32\% & 0.37\% & 0.46\% & 0.58\% & 0.48\% & 0.46\% \\
    \cline{2-12}
     & Instance & 11 & 12 & 13 & 14 & 15 & 16 & 17 & 18 & 19 & 20\\
    \cline{2-12}
     & \F1^+ & 0.97\% & 1.23\% & 0.94\% & 1.07\% & 1.08\% & 0.93\% & 0.96\% & 0.89\% & 1.04\% & 0.90\% \\
     & \F1^+-X & 0.47\% & 0.53\% & 0.44\% & 0.59\% & 0.49\% & 0.42\% & 0.49\% & 0.48\% & 0.48\% & 0.40\% \\
    \hline
    \multirow{4}{*}{Pct. reduction} & Instance & 1 & 2 & 3 & 4 & 5 & 6 & 7 & 8 & 9 & 10\\
    \cline{2-12}
    & \F1^+-X & 52.2\% & 58.4\% & 46.6\% & 39.7\% & 62.3\% & 59.1\% & 59.7\% & 46.4\% & 52.8\% & 47.4\% \\
    \cline{2-12}
    & Instance & 11 & 12 & 13 & 14 & 15 & 16 & 17 & 18 & 19 & 20\\
    \cline{2-12}
    & \F1^+-X & 51.6\% & 57.1\% & 53.5\% & 44.8\% & 54.5\% & 55.4\% & 48.9\% & 46.3\% & 54.0\% & 56.1\%\\
    \hline
  \end{tabular}
  \label{tab:axp:Exp3-2}
\end{table}
}


\end{APPENDICES}
\end{document}